\documentclass[a4paper,10pt]{amsart}
\usepackage[utf8]{inputenc}
\usepackage{amsmath,amssymb,amsthm}
\usepackage[utf8]{inputenc}
\usepackage{tikz-cd}
\usetikzlibrary{decorations.pathreplacing}
\usepackage{todonotes}
\usepackage[foot]{amsaddr}
\usepackage[pagebackref=true]{hyperref}
\setuptodonotes{color=red!40,fancyline}
\usepackage{color}
\usepackage{subcaption}
\usepackage{longtable}

\renewcommand{\theenumi}{(\alph{enumi})}
\renewcommand{\labelenumi}{\theenumi}

\newtheorem{theorem}{Theorem}[section]
\newtheorem{proposition}[theorem]{Proposition}
\newtheorem{corollary}[theorem]{Corollary}
\newtheorem{lemma}[theorem]{Lemma}
\newtheorem{definition}[theorem]{Definition}

\theoremstyle{definition}
\newtheorem{remark}[theorem]{Remark}
\newtheorem{example}[theorem]{Example}
\newtheorem{assumption}{Assumption}

\usepackage[linesnumbered,ruled,vlined,hanginginout]{algorithm2e}
\SetKwInput{KwInput}{Input}
\SetKwInput{KwOutput}{Output}

\title[Differential Galois Theory for the Classical Groups]{On the Direct Problem in Differential Galois Theory for the Classical Groups}
\date{\today}
\author{Daniel Robertz}
\address{Lehrstuhl f\"ur Algebra und Zahlentheorie, RWTH Aachen University, D--52056 Aachen, Germany} 
\email{daniel.robertz@rwth-aachen.de}
\author{Matthias Sei\ss}
\address{Institut f\"ur Mathematik, Universit\"at Kassel, D--34109 Kassel, Germany}
\email{mseiss@mathematik.uni-kassel.de}

\newcommand{\N}{\mathbb{N}}
\newcommand{\Z}{\mathbb{Z}}
\newcommand{\Q}{\mathbb{Q}}

\DeclareMathOperator{\Stab}{Stab}
\DeclareMathOperator{\Aut}{Aut}
\DeclareMathOperator{\kernel}{ker}
\DeclareMathOperator{\image}{im}
\DeclareMathOperator{\Frac}{Frac}
\DeclareMathOperator{\trdeg}{trdeg}

\newcommand{\gauge}[2]{#1 . #2}

\newcommand{\dlog}{\ell\delta}

\newcommand{\diag}{{\rm diag}}
\newcommand{\group}{G}

\newcommand{\GalConn}{H_{\rm con}}
\newcommand{\Hred}{H_{\rm red}}

\newcommand{\GL}{\mathrm{GL}}
\newcommand{\SL}{\mathrm{SL}}
\newcommand{\SP}{\mathrm{Sp}}
\newcommand{\SO}{\mathrm{SO}}
\newcommand{\Gzwei}{\mathrm{G}_2}

\newcommand{\cocomp}[1]{#1^{\circ}}
\newcommand{\liealg}{\mathfrak{g}}
\newcommand{\gl}{\mathfrak{gl}}

\newcommand{\field}{C}
\newcommand{\difffield}{F}
\newcommand{\difffieldalg}{F_{\mathrm{alg}}}

\newcommand{\generalext}{\mathcal{E}}
\newcommand{\Gal}{{\rm Gal}}
\newcommand{\extfieldspec}{\mathcal{\overline{E}}}
\newcommand{\extfieldred}{E_{\rm red}}
\newcommand{\KPV}{K_{\rm PV}}

\newcommand{\roots}{\Phi}
\newcommand{\rootbasis}{\Delta}
\newcommand{\weyl}{\mathcal{W}}
\newcommand{\height}{\mathrm{ht}}
\newcommand{\cartan}{\mathfrak{h}}
\newcommand{\torus}{T}
\newcommand{\borel}{B}
\newcommand{\unipotent}{U}
\newcommand{\parabolic}{P}
\newcommand{\levi}{L}
\newcommand{\Lred}{L_{\rm red}}
\newcommand{\radic}{R}
\newcommand{\unirad}{R_u}
\newcommand{\urelem}{u}
\newcommand{\leviroots}{\Psi}

\newcommand{\ANFcomp}{A_{\group}^{\rm comp}}
\newcommand{\BNFcomp}{B_{\group}}

\newcommand{\Aprered}{A_{\rm red}^{\rm pre}}
\newcommand{\Aprerad}{A_{\rm rad}^{\rm pre}}
\newcommand{\Ared}{A_{\rm red}}
\newcommand{\Arad}{A_{\rm rad}}

\DeclareMathOperator{\Lie}{Lie}

\DeclareMathOperator{\Ad}{Ad}
\DeclareMathOperator{\ad}{ad}

\newcommand{\wronski}{\mathrm{wr}}

\newcommand{\coord}{Y}

\newcommand{\vq}{\overline{v}}

\newcommand{\baa}{\boldsymbol{a}}

\newcommand{\bvq}{\overline{\boldsymbol{v}}}
\newcommand{\bvh}{\widehat{\boldsymbol{v}}}

\newcommand{\bff}{\boldsymbol{f}}

\newcommand{\bee}{\boldsymbol{e}}
\newcommand{\bsq}{\overline{\boldsymbol{s}}}
\newcommand{\bss}{\boldsymbol{s}}
\newcommand{\sq}{\overline{s}}
\newcommand{\bpp}{\boldsymbol{p}}
\newcommand{\btt}{\boldsymbol{t}}
\newcommand{\buu}{\boldsymbol{u}}
\newcommand{\bvv}{\boldsymbol{v}}
\newcommand{\bww}{\boldsymbol{w}}
\newcommand{\bxx}{\boldsymbol{x}}
\newcommand{\byy}{\boldsymbol{y}}
\newcommand{\bzz}{\boldsymbol{z}}

\newcommand{\bvvbase}{\boldsymbol{v}_{\rm base}}
\newcommand{\bvvqbase}{\overline{\boldsymbol{v}}_{\rm base}}
\newcommand{\bvvext}{\boldsymbol{v}_{\rm ext}}

\newcommand{\locali}{\mathcal{R}}

\newcommand{\bZ}{\boldsymbol{Z}}
\newcommand{\bexp}{\boldsymbol{\exp}}
\newcommand{\bint}{\boldsymbol{\mathrm{int}}}

\newcommand{\intn}{\mathrm{int}}
\newcommand{\intrad}{\mathrm{intrad}}
\newcommand{\bintrad}{\boldsymbol{\mathrm{intrad}}}

\newcommand{\expsolass}{\exp^{\rm ass}}
\newcommand{\exprforexp}{\mathrm{EXP}}
\newcommand{\exprforint}{\mathrm{INT}}

\newcommand{\ratexp}{\widehat{\exp}}
\newcommand{\ratint}{\widehat{\mathrm{int}}}\newcommand{\ratv}{\widehat{v}}
\newcommand{\ratf}{\widehat{f}}
\newcommand{\ratfm}{\widehat{\fm}}
\newcommand{\bratexp}{\boldsymbol{\widehat{\exp}}}
\newcommand{\bratint}{\boldsymbol{\widehat{\mathrm{int}}}}
\newcommand{\bratv}{\boldsymbol{\widehat{v}}}
\newcommand{\bratf}{\boldsymbol{\widehat{f}}}

\newcommand{\Sol}{\mathrm{Sol}}
\newcommand{\RatSol}{\mathrm{RatSol}}
\newcommand{\sol}{\mathrm{sol}}
\newcommand{\bsol}{\boldsymbol{\mathrm{sol}}}

\newcommand{\fm}{\mathcal{Y}}
\newcommand{\fmspec}{\mathcal{\overline{Y}}}

\newcommand{\fml}{\mathcal{Y}_{\mathrm{Liou}}}
\newcommand{\fmlspec}{\mathcal{\overline{Y}}_{\mathrm{Liou}}}
\newcommand{\ALiou}{A_{\rm Liou}}

\newcommand{\fmuq}{\underline{\mathcal{Y}}}
\newcommand{\fmred}{\mathcal{Y}_{\rm red}}
\newcommand{\fmrad}{\mathcal{Y}_{\rm rad}}
\newcommand{\fmredq}{\overline{\mathcal{Y}}_{\rm red}}
\newcommand{\fmradq}{\overline{\mathcal{Y}}_{\rm rad}}
\newcommand{\fmreduq}{\underline{\mathcal{Y}}_{\rm red}}
\newcommand{\fmraduq}{\underline{\mathcal{Y}}_{\rm rad}}
\newcommand{\fmredqhat}{\widehat{\overline{\mathcal{Y}}}_{\rm red}}
\newcommand{\fmradqhat}{\widehat{\overline{\mathcal{Y}}}_{\rm rad}}
\newcommand{\fmreduqhat}{\widehat{\underline{\mathcal{Y}}}_{\rm red}}
\newcommand{\fmraduqhat}{\widehat{\underline{\mathcal{Y}}}_{\rm rad}}

\DeclareMathOperator{\LCLM}{LCLM}

\newcommand{\Ric}[3]{\mathrm{Ric}_{#1}(#2, #3)}
\newcommand{\idealspec}{S}
\newcommand{\idealinter}{S_{\mathrm{inter}}}
\newcommand{\idealmax}{S_{\mathrm{max}}}
\newcommand{\idealric}{S_{\mathrm{Ric}}}
\newcommand{\Imax}{I_{\mathrm{max}}}

\newcommand{\Iuni}{I_{\mathrm{uni}}}

\newcommand{\sigred}{\sigma_{\rm red}}
\newcommand{\siginter}{\sigma_{\rm inter}}
\newcommand{\sigul}{\underline{\sigma}}
\newcommand{\sigmax}{\sigma_{\rm max}}
\newcommand{\sigPV}{\sigma_{\rm PV}}

\newcommand{\ulphi}{\underline{\varphi}}
\newcommand{\ovphi}{\overline{\varphi}}
\newcommand{\ovphired}{\overline{\varphi}_{\rm red}}
\newcommand{\ulQ}{\underline{Q}}
\newcommand{\ovQ}{\overline{Q}}
\newcommand{\ovQred}{\overline{Q}_{\rm red}}

\def\addots{\mathinner{\mkern1mu\raise1pt\vbox{\kern7pt\hbox{.}}
  \mkern2mu\raise4pt\hbox{.}\mkern2mu\raise7pt\hbox{.}\mkern1mu}}

\definecolor{DarkGreen}{rgb}{0, 0.5, 0}

\begin{document}

\begin{abstract}
   Let $\group$ be a classical group of Lie rank $l$ and let $\field$ be an algebraically closed field of characteristic zero. For $l$ differential indeterminates $\bvv=(v_1,\dots,v_l)$ over $\field$ we constructed in \cite{Seiss_Generic} a general Picard-Vessiot extension $\generalext$ of the differential field $\field\langle \bss(\bvv)\rangle$ having differential Galois group $\group(\field)$. Here $\bss(\bvv)=(s_1(\bvv),\dots,s_l(\bvv))$ are certain differential polynomials in $\field\{\bvv \}$ which are differentially algebraically independent over $\field$. The linear differential equation defining $\generalext$ is defined by the normal form matrix $A_{\group}(\bss(\bvv))$ lying in the Lie algebra of $\group$.\\
    In the first part of this paper we analyze the structure of $\generalext$ induced by the action of the standard parabolic subgroups of $\group(\field)$ on $\generalext$.
   In the second part we consider specializations $A_{\group}(\bss(\bvv)) \to A_{\group}(\bsq)$ with $\bsq \in \field(z)^l$ of the normal form matrix for $\group$ of type $A_l$, $B_l$, $C_l$ or $\Gzwei$ (here $l=2$). We show how one can combine the results of the first part with known algorithms for the computation of the differential Galois group and its Lie algebra to determine the differential Galois group of certain specialized equations
   $\partial(\byy) = A_{\group}(\bsq)\byy$ over $\field(z)$ with $\field$ a computable algebraically closed field of characteristic zero.
\end{abstract}

\maketitle

\section{Introduction}\label{sec:intro}
\subsection*{Differential Galois Theory}

Differential Galois theory is a generalization of the well-known classical Galois theory for polynomial equations to linear ordinary  differential equations.
We recall briefly that in the classical theory one considers a polynomial equation  
\[
p(x) \, := \, x^n + a_{n-1} x^{n-1} + \dots + a_1 x+ a_0 \, = \, 0
\]
with coefficients in some field $k$ and that one wants to study the symmetries of the roots of this equation.
These symmetries are described by the group of all $k$-automorphisms of a splitting field for $p(x)=0$, i.e.\ a smallest 
extension field of $k$ such that $p(x)$ splits into linear factors. This group of $k$-automorphism is called the Galois group.
In differential Galois theory we study instead the solutions of a linear differential equation
\begin{equation} \label{eqn:firstlineardifferentialeq}
    L(y) \, := \, y^{(n)}+a_{n-1} y^{(n-1)} + \dots + a_1 y' + a_0 y \, = \, 0
\end{equation}
with coefficients in some differential field $F$ with algebraically closed field of constants $\field$.
Here, the analogous object to the splitting field is the so-called Picard-Vessiot field, which is a certain differential field extension $E$ of $F$ containing a full set of solutions for the equation $L(y)=0$.
The differential algebraic relations of the solutions are now described by the properties of the group of all differential $F$-automorphisms of $E$, meaning the $F$-automorphisms of $E$ which commute with the derivation of $F$.
While the classical Galois group has a representation as a permutation group, the differential Galois group acts by $\field$-linear transformations on the full set of solutions and so has a representation as a linear algebraic group defined over $\field$.
The direct problem is concerned with the computation of the differential Galois group for a given linear differential equation \eqref{eqn:firstlineardifferentialeq}.
Even though some progress has been made in some particular cases, the general case is still very challenging, in particular effective computations. This paper makes a contribution to this problem based on recent progress for the case of generic equations.

\subsection*{State of the Art of the Direct Problem}
The first algorithmic contribution to the direct problem in differential Galois theory (for a general introduction to the topic see \cite{vanderPutSinger,Magid}) was given by J.~Kovacic in 1986 in  \cite{KovacicAlgorithmSL2}. Kovacic uses the classification of the algebraic subgroups of $\SL_2(\field)$ to design an effective algorithm to compute the Liouvillian solutions of an order two homogeneous linear differential equation of type
\[
L(y) \, := \, y'' + a_0 y \, = \, 0 \qquad \text{with} \  a_0 \in \field(z)\,,
\]
where $\field$ is an algebraically closed field of characteristic zero and $\difffield := \field(z)$ is the rational function field in $z$ endowed with the usual derivation $\partial=\frac{d}{dz}$.
Knowing the Liouvillian solutions of the equation $L(y)=0$, one can deduce its differential Galois group $H \leq \SL_2(\field)$ and vice versa. M.~F.\ Singer and F.~Ulmer generalized this idea to order three differential equations in \cite{SingerUlmerThirdOrder,SingerUlmerSecondThirdOrder,SingerUlmerProducts}.

For a homogeneous linear differential equation
\[
 L(y) \, := \, y^{(n)} + a_{n-1} y^{(n-1)} + \dots + a_1 y' + a_0 y \, = \, 0 \qquad \text{with} \  a_i \in \field(z),
\]
of order $n$
whose corresponding first order system is completely reducible, i.e.\ a direct sum of irreducible systems, 
E.~Compoint and M.~F.\ Singer showed in \cite{SingerCompoint} how to compute the differential Galois group $H\leq \GL_n(\field)$.
Their approach is based on E.~Compoint's result that if the differential Galois group is unimodular and reductive, then the ideal defining the Picard-Vessiot extension is generated by the invariants of the differential Galois group it contains.
Degree bounds for generating invariants of reductive groups give a degree bound for the generating polynomials of the ideal defining the Picard-Vessiot extension.
These polynomials can be computed from the differential equation without knowing the differential Galois group using the methods of \cite{HoeijWeilInvariant}. 

An algorithm for the general case was given by E.\ Hrushovski in \cite{Hrushovski} using model theory. His approach was elaborated on and improved by R.~Feng in \cite{Feng}, by  M.~Sun in \cite{Sun} and by D.~Rettstadt in his PhD thesis \cite{Rettstadt}. An improvement for the degree bounds of the defining equations of the proto-Galois group was given by  E.~Amzallag,
A.~Minchenko and G.~Pogudin in \cite{Pogudin}.
But it is still very hard to actually compute the Galois group using E.\ Hrushovski's algorithm.

Instead of computing the differential Galois group $H$ of an order $n$ equation one can focus on the computation of the Lie algebra $\Lie(H)$ of $H$. This approach is based on the Kolchin-Kovacic reduction theorems appearing in  \cite{KovacicInvProb,KovacicOntheInverseProblem} and \cite{Kolchin,BuiumCassidy}.
Roughly speaking they state that over an algebraic extension of the base field $\field(z)$ one can gauge transform the defining matrix $A$ of the corresponding first order system into the Lie algebra of the differential Galois group.
Such a matrix is called a \emph{reduced form} of $A$.
Computing the smallest Lie algebra containing a Wei-Norman decomposition (cf.\ \cite{WeiNorman}) of a reduced form determines $\Lie(H)$.
However, this approach only yields a full solution if the differential Galois group is connected.
An algorithm to compute a reduced form of a first order system is presented in the paper \cite{DreyfusWeil} by T.~Dreyfus and J.-A.\ Weil. It is based on the works \cite{CompointWeil,BarkatouEigenrings,vanderHoevenGalois} and \cite{BarkatouCluzeauWeilDiVizio,MonfordReducedForms} by 
A.\ Aparicio-Monforte, E.~Compoint, M.~Barkatou, T.~Cluzeau, L.\ Di Vizio, J.-A.\ Weil and J.\ van der Hoeven. 

\subsection*{General Extension Fields for the Classical Groups}
In this paper we consider the direct problem in differential Galois theory of a special family of linear differential equations with differential Galois group a subgroup $H$ of a classical group $\group$ of type $A_l$, $B_l$, $C_l$ or $\Gzwei$ (here $l=2$).
We make essential use of the general differential equations and their general Picard-Vessiot extensions constructed in  \cite{Seiss_Generic,Seiss}, meaning that the general extension field $\generalext$ has generators which depend on $l$ differential indeterminates $\bvv=(v_1,\dots,v_l)$ over $\field$, where $l$ is the Lie rank of $\group$. Since in this paper we consider specializations into the differential field $\difffield$, we here perform our construction of $\generalext$ over $\difffield \supset \field$ as in \cite{RobertzSeissNormalForms}, meaning that $\bvv$ are differential indeterminates over $\difffield$. 
The construction relies on the geometric structure of $\group$ expressed by a Chevalley basis of the Lie algebra $\Lie(\group)$ and the Bruhat decomposition of $\group$.
The first one facilitates the construction of the \emph{normal form matrix} $A_{\group}(\bss(\bvv))$ defining the differential equation 
\begin{equation}\label{eqn:introduction2}
\partial(\byy) \, = \, A_{\group}(\bss(\bvv)) \, \byy
\end{equation}
for $\generalext$. The matrix $A_{\group}(\bss(\bvv))$ depends on  certain $l$ differential polynomials 
\[
\bss(\bvv) \, = \, (s_1(\bvv),\dots,s_l(\bvv)) \quad \text{with} \quad s_i(\bvv) \in \field\{\bvv\},
\]
which are differentially algebraically independent over $\difffield$.
In this framework we can construct a fundamental matrix $\fm$ for \eqref{eqn:introduction2} as a parametrization of the double coset in the Bruhat decomposition corresponding to the longest Weyl group element $\overline{w}$.  
In other words, with a representative $n(\overline{w})$ of $\overline{w}$ in the normalizer of a maximal torus we let
\begin{equation}\label{eq:fundmatrixgen}
\fm \, = \, \buu(\bvv,\bff) \, n(\overline{w}) \, \btt(\bexp) \, \buu(\bint)\,,
\end{equation}
 where
\begin{itemize}
    \item $\buu(\bvv,\bff)$ is a parametrization of the product of all negative root groups by $\bvv$ and  $m-l$ differential polynomials $\bff=(f_{l+1},\dots,f_m)$ in $\field\{\bvv\}$ with $m=|\roots^-|$ where $\roots$ denotes the root system of $\group$,
    \item $\btt(\bexp)$ is the product of torus elements parametrized by exponentials $\bexp=(\exp_1,\dots, \exp_l)$ satisfying $\frac{\exp_i'}{\exp_i} = g_i(\bvv)$ with $g_i(\bvv) \in \Z[\bvv]$ homogeneous of degree one,
    \item $\buu(\bint)$ is a parametrization of the product of all negative root groups by certain iterated integrals $\bint=(\intn_1,\dots,\intn_m)$.
\end{itemize}
It is shown in \cite[Section~7]{RobertzSeissNormalForms} that the field
\[
\generalext \, = \, \difffield\langle \bss(\bvv) \rangle (\bvv,\bff,\bexp,\bint) 
\]
is a Picard-Vessiot extension of $\difffield\langle \bss(\bvv)\rangle$ with differential Galois group $\group$.
Note that the whole construction depends on the differential indeterminates $\bvv$. A summary of our construction and the main notation are presented in Part~\ref{part:I}.

\subsection*{The Contribution of the Paper}
The goal of this paper is to present an algorithm to compute the differential Galois groups $H(\field)\leq \group(\field)$ of certain well-behaving specializations
\[
\partial(\byy) \, = \,A_{\group}(\bsq) \, \byy 
\]
of the normal form matrix $A_{\group}(\bss(\bvv))$ with 
\[
\sigma_0\colon \difffield\{ \bss(\bvv) \} \to \difffield, \quad \bss(\bvv) \mapsto \bsq \in \difffield^l, \quad \difffield \, = \, \field(z)\,.
\]
This is achieved by combining the above two mentioned algorithms of E.~Compoint and M.~F.\ Singer respectively T.~Dreyfus and J.-A.\ Weil with our construction of a general extension and in particular of our fundamental matrix $\fm$. We are going to elaborate on our approach.

\subsection*{The General Extension and Parabolic Subgroups} Before developing our algorithm for the specialized equation we analyze in Part~\ref{part:II} intermediate extensions of our general extension field $\generalext$ over $\difffield\langle \bss(\bvv) \rangle$ corresponding, according to the Fundamental Theorem of Differential Galois Theory, to parabolic subgroups of $\group$.
A key ingredient here is the parametrized Bruhat decomposition of the general fundamental matrix $\fm$ in \eqref{eq:fundmatrixgen},
which allows us to connect explicitly standard parabolic subgroups $\parabolic_J \leq \group$ for $J \subseteq \{1, \dots, l \}$, where $\rootbasis=\{ \alpha_1,\dots,\alpha_l \}$ is a basis of the root system $\roots$ of $\group$, and their Levi decompositions
\[
\parabolic_J \, = \, \levi_J \ltimes \unirad(\parabolic_J)
\]
with intermediate extensions. Here $\unirad(\parabolic_J)$ denotes the unipotent radical and $\levi_J$ the standard Levi group of $\parabolic_J$ (cf.\ \eqref{eqn:standardLevi} on page \pageref{eqn:standardLevi}). The Galois action of $g \in \group(\field)$ on the elements $\bvv$, $\bff$, $\bexp$ and $\bint$ is induced by recomputing the Bruhat decomposition of
\[
\fm \, g \, = \, \buu(\bvv,\bff) \, n(\overline{w}) \, \btt(\bexp) \, \buu(\bint) \, g\,.
\]
We establish a bijection between standard parabolic subgroups $\parabolic_J$ and partitions $\bvv =  \bvvext \cup \bvvbase $ of the differential indeterminates $\bvv$, where $\bvvbase$ consists of those indeterminates which are fixed under the action of $\parabolic_J$. 
Correspondingly we obtain  partitions of the set of indices $I=\{1,\dots,l\}$ of $\bvv$ into $I=I'\cup I''$, where the indices in $I'$ correspond to the indeterminates in $\bvvext$ and the ones in $I''$ to the indeterminates in $\bvvbase$. 
We show that the fixed field $\generalext^{\parabolic_J}$ for a parabolic subgroup $\parabolic_J$ is generated as differential field by $\bss(\bvv)$ and $\bvvbase$ over $\difffield$, that is
\[
\generalext^{\parabolic_J} \, = \, \difffield\langle \bss(\bvv), \bvvbase\rangle\,.
\]
We denote the negative roots of $\roots^-$ by $\beta_1,\dots,\beta_m$. Let $\leviroots$ be the root subsystem of $\roots$ generated by all simple roots $\alpha_i$ with $i \in J$.
Then $\leviroots$ is the root system of $\levi_J$ and $\roots^- \setminus \leviroots^-$ is the set of roots $\beta$ whose corresponding root groups $U_{\beta} \leq \unipotent^-$ generate the unipotent radical $\unirad(\parabolic_J)$.    
Combining the structure of the root system with the induced action of $\unirad(\parabolic_J)$ on the parameters we prove that the fixed field $\generalext^{\unirad(\parabolic_J)}$ of the unipotent radical is generated over $\difffield\langle \bss(\bvv), \bvvbase\rangle$ as a differential field by
\begin{equation}\label{eqn:paramtesreductivepart}
\bvvext, \, \bexp, \, \intn_i \quad \text{with} \ 1\leq i \leq m \ \text{such that} \ \beta_i \in \leviroots^-.
\end{equation}
Moreover, the Fundamental Theorem of Differential Galois Theory implies that the differential Galois group of
$\generalext^{\unirad(\parabolic_J)}$ over $\difffield\langle \bss(\bvv), \bvvbase\rangle$ is isomorphic to the Levi group $\levi_J$.
 Levi groups of $\parabolic_J$ are maximal reductive subgroups justifying to call the elements in \eqref{eqn:paramtesreductivepart} the \emph{parameters of the 
 reductive part} of $\parabolic_J$. In turn we call the integrals $\intn_i$ with $\beta_i \in \roots^- \setminus \leviroots^-$ the \emph{parameters of the unipotent radical part}.
 
 The intermediate extensions $\generalext/\generalext^{\parabolic_J}$ and  
 $\generalext^{\unirad(\parabolic_J)}/\generalext^{\parabolic_J}$ and the decompositions $\bvv=\bvvbase\cup \bvvext$ are linked to
 solutions of certain differential operators.
 Denote by
\[
L_{\group}(\bss(\bvv),\partial) \in \field\langle \bss(\bvv) \rangle[\partial]
\]
the \emph{normal form operator} corresponding to the normal form matrix $A_{\group}(\bss(\bvv))$. 
It turns out that for each indeterminate $v_i$ there is an associated equation
\begin{equation}\label{eqn:Ldet(i)}
L^{\det(i)}(\bss(\bvv),y) \, = \, 0\,, \qquad i=1, \ldots, l\,,
\end{equation}
with solution $\exp(\int b_i v_i)$ with $b_i \in \field^{\times}$. Thus, $b_i v_i$ is a solution of the Riccati equation $\Ric{i}{\bss(\bvv)}{y} = 0$ corresponding to \eqref{eqn:Ldet(i)}.
The Riccati equations will help us later to compute for the specialized equation the specialized elements $\bvvqbase$ in $\difffield$ and so to determine $\parabolic_J$ containing the Galois group. (This is a generalization of \emph{Case~1} in Kovacic's Algorithm.)
Moreover, each partition $\bvv = \bvvbase \cup \bvvext$, and thus each standard parabolic subgroup, induces an irreducible factorization of the normal form operator
\begin{equation}\label{eqn:introduction3}
L_{\group}(\bss(\bvv),\partial) \, = \, L_{1}(\bss(\bvv),\bvvbase,\partial) \cdots L_{k}(\bss(\bvv),\bvvbase, \partial)
\end{equation}
over $\difffield\langle \bss(\bvv),\bvvbase \rangle$.
We prove that the extension $\generalext^{\unirad(\parabolic_J)}/\generalext^{\parabolic_J}$ with differential Galois group isomorphic to $\levi_J$ is a Picard-Vessiot extension for the least common left multiple
\begin{equation}\label{eqn:introduction4}
\LCLM(\bss(\bvv),\bvvbase,\partial) \, := \, \LCLM(L_{1}(\bss(\bvv),\bvvbase,\partial), \dots, L_{k}(\bss(\bvv),\bvvbase,\partial))
\end{equation}
 of these irreducible factors.  
On the one hand, the basis elements 
\[
y^{I''}_1,\dots,y^{I''}_{n^{I''}} \in \generalext^{\unirad(\parabolic_J)} \quad \text{(fixed basis)}
\]
of the solution space of the 
least common left multiple in $\generalext^{\unirad(\parabolic_J)}$ are differential rational functions in the parameters of the reductive part.
On the other hand, since $\generalext^{\unirad(\parabolic_J)}$ is generated as a differential field by this basis over $\generalext^{\parabolic_J}$, we can express the parameters of the reductive part as differential rational functions in these basis elements. We can determine differential rational functions 
$\exprforexp^{I''}_i(\bZ)$, $V^{I''}_i(\bZ)$ and $\exprforint^{I''}_j(\bZ)$ in $\generalext^{\parabolic_J}\langle \bZ \rangle$ for $i=1,\dots,l$ and $j=1,\dots,m$ with $\beta_j \in \leviroots^-$ such that
\begin{equation}\label{eqn:introduction1}
\begin{array}{c}
\exprforexp^{I''}_i(y^{I''}_1,\dots,y^{I''}_{n^{I''}}) \, = \, \exp_i, \quad
V^{I''}_i(y^{I''}_1,\dots,y^{I''}_{n^{I''}}) \, = \, v_i 
\quad \text{and} \\[0.5em] \exprforint^{I''}_j(y^{I''}_1,\dots,y^{I''}_{n^{I''}}) \, = \, \intn_j \, ,
\end{array}
\end{equation}
where $\bZ = (Z_1,\dots, Z_{n^{I''}})$ are differential indeterminates over $\generalext^{\parabolic_J}$.  
The functions $\exprforexp^{I''}_i(\bZ)$, $V^{I''}_i(\bZ)$ and $\exprforint^{I''}_j(\bZ)$ are important, since they will allow us later to determine a specialization of the parameters of the reductive part from a basis of a solution space of the specialized least common left multiple.  
In the last section of Part~\ref{part:II} we show that the fundamental matrix $\fm$ factors into $\fm = \fmred \, \fmrad$ with 
\begin{eqnarray*}
\fmred & = & \buu(\bvv,\bff) \, n(\overline{w}) \, \btt(\bexp) \prod_{\beta_i \in \leviroots^-} u_{\beta_i}(\intn_i) \,,\\
\fmrad & = & \prod_{\beta_i \in \roots^- \setminus \leviroots^-} u_{\beta_i}(y_{\beta_i}) \, \in \, \unirad(\parabolic_J)(\generalext \setminus\generalext^{\unirad(\parabolic_J)})\,,
\end{eqnarray*}
where $y_{\beta_i}$ are polynomials in $\field[\bint] \setminus \field[\intn_i \mid \beta_i \in \leviroots^- ]$.
Moreover, we prove that we can determine a reduction matrix $g_1 \in \group(\generalext^{\parabolic_J})$ such that $g_1\fm$ lies in the parabolic subgroup $\parabolic_J(\generalext)$ and $g_1\fmred$ is contained in the Levi group $\levi_J(\generalext^{\unirad(\parabolic_J)})$.
In other words, $g_1\fm$ admits the factorization
\[
g_1\fm \, = \, (g_1 \fmred) \cdot \fmrad \, \in \, \levi_J(\generalext^{\unirad(\parabolic_J)}) \cdot \unirad(\parabolic_J)(\generalext\setminus \generalext^{\unirad(\parabolic_J)})\,,
\]
splitting $g_1 \fm$ into its reductive part and unipotent radical part.

 \subsection*{The Three Extensions of $\sigma_0$}
The main part of our paper, that is Part~\ref{part:III}, employs the above results to compute the differential Galois group of the specialized normal form matrix $A_{\group}(\bsq)$ with 
\[
\sigma_0\colon \difffield\{ \bss(\bvv) \} \to \difffield, \quad \bss(\bvv) \mapsto \bsq \in \difffield^l, \quad \difffield \, = \, \field(z) \,.
\]
 We are going to extend $\sigma_0$ in three steps to a specialization 
 \begin{equation}\label{eq:introduction1}
\begin{array}{rcl}
  \sigPV \colon D^{-1} \difffield\{ \bvv \}[\bexp, \bexp^{-1}, \bint] & \to & \overline{\generalext},\\[0.2em]
   (\bvv, \bexp,\bint) &\mapsto & (\bvq, \overline{\bexp},\overline{\bint})
\end{array}
\end{equation}
for a Picard-Vessiot extension $\overline{\generalext}$ of $\difffield$ for $A_{\group}(\bsq)$, where $D$ is a certain multiplicatively closed subset. In each step
we add new algebraic relations between the parameters $\bvv, \bexp$ and $\bint$ until we obtain the Picard-Vessiot extension $\overline{\generalext}$ with fundamental matrix $\overline{\fm}$ being a specialization of $\fm$.
\begin{figure}[h]
\centering
\begin{tikzcd}
D^{-1} \difffield\{ \bvv \}[\bexp, \bexp^{-1}, \bint]
\arrow{rr}{\sigPV} & &
\overline{\generalext}\\
D^{-1} \difffield\{ \bvv \}[\bexp, \bexp^{-1}, \intn_i \mid \beta_i \in \leviroots^-]
\arrow[hookrightarrow]{u}{}
\arrow{rr}{\sigred} & &
\extfieldred
\arrow[hookrightarrow]{u}{}\\
\difffield\{ \bss(\bvv), \bvvbase \}
\arrow[hookrightarrow]{u}{}
\arrow{rr}{\siginter} & &
\difffield
\arrow[hookrightarrow]{u}{}\\
\difffield\{ \bss(\bvv) \}
\arrow[hookrightarrow]{u}{}
\arrow{rr}{\sigma_0} & &
\difffield
\arrow[hookrightarrow]{u}{}
\end{tikzcd}
\caption{ {Overview of successive extensions of the specialization $\sigma_0$}}\label{fig:extendsigma}
\end{figure}

\subsection*{A Parabolic Bound for a Specialization} The first step of the extension of $\sigma_0$ is supposed to map as many $v_i$ to rational functions in $\difffield$ as possible, provided that a construction of a Picard-Vessiot extension $\overline{\generalext}$ of $\difffield$ is still possible based on that choice.
More precisely, we compute with the known algorithms the rational solutions in $\difffield$ of the specialized Riccati equations $\Ric{i}{\bsq}{v_i}=0$ corresponding to the specialized associated equations $L^{\det(i)}(\bsq,y) = 0$ for $i=1,\dots,l$. Here we make the assumption on $\sigma_0$ that these specializations are defined.
Using the differential Thomas decomposition (cf.\ \cite{RobertzHabil}) we construct a longest tuple $\bvvqbase$ of rational solutions of the various Riccati equations such that the differential ideal
\[
\idealinter \, = \, \langle \bss(\bvv)-\bsq, \bvvbase - \bvvqbase \rangle \lhd \difffield\{\bvv\}
\]
is proper.
Being a proper ideal guarantees that we are able to continue from here the construction of a Picard-Vessiot extension $\overline{\generalext}/\difffield$ and the extended specialization \eqref{eq:introduction1}.
We obtain the first extended specialization
\[
\siginter\colon \difffield\{ \bss(\bvv), \bvvbase \} \to \difffield, \ (\bss(\bvv),\bvvbase) \mapsto (\bsq,\bvvqbase) \, .
\]
We denote by $H$ the differential Galois group of $\overline{\generalext}/\difffield$, whose representation depends on the fixed choice of $\idealinter$ as well as the subsequent choices made to complete the construction of $\overline{\generalext}$.
The choice of $\bvvqbase$ induces a partition of the indices $I=I'\cup I''$ or equivalently a partition of the indeterminates $\bvv=\bvvext \cup  \bvvbase$. 
According to the bijection established in Part~\ref{part:II}, this partition
 determines a standard parabolic subgroup $\parabolic_J$ of $\group$.
We will prove that $H$ is contained in $\parabolic_J$.
Since the tuple $\bvvqbase$ has as many entries as possible, we also prove
that $\parabolic_J$ is minimal among the standard parabolic subgroups  containing $H$ with respect to inclusion.
This implies that for every Levi group $\levi$ of $H$, there is a Levi group 
$\widetilde{\levi}$ of $\parabolic_J$ with $\levi \leq \widetilde{\levi}$ such that $\levi$ is $\widetilde{\levi}$-irreducible,
i.e., $\levi$ is not contained in any proper parabolic subgroup of $\widetilde{\levi}$. We also prove that the unipotent radical $\unirad(H)$ of $H$ is contained in the unipotent radical $\unirad(\parabolic_J)$ of $\parabolic_J$.

\subsection*{The Reductive Part}
Next we perform an extension of the specialization $\siginter$ to the parameters of the reductive part. 
To this end, we show that under certain natural assumptions on $\siginter$ the generic irreducible factorization \eqref{eqn:introduction3} corresponding to $\bvvbase$ specializes to an irreducible factorization 
\[
L_{\group}(\bsq,\partial) \, = \, L_{1}(\bsq,\bvvqbase,\partial) \cdots L_{k}(\bsq,\bvvqbase, \partial)
\]
of the specialized normal form operator.
We denote by $\overline{\LCLM}(\bsq,\bvvqbase,\partial)$ the
specialization of the generic least common left multiple.
With a further assumption on $\siginter$ we prove that 
\[
\overline{\LCLM}(\bsq,\bvvqbase,\partial) \, = \, \LCLM(L_{1}(\bsq,\bvvqbase,\partial), \dots, L_{k}(\bsq,\bvvqbase,\partial))\,.
\]
We apply now the algorithm of E.~Compoint and M.~F.\ Singer developed in \cite{SingerCompoint} to compute the generators of a maximal differential ideal
\[
Q \lhd \difffield[\GL_{n_{I''}}] \, = \, \difffield[X_{i,j},\det(X_{i,j})^{-1}]
\]
for the specialized least common left multiple  
and obtain a Picard-Vessiot extension
\[
\extfieldred \, := \, \Frac(\difffield[\GL_{n_{I''}}]/Q) ,
\]
whose differential Galois group is the stabilizer $\Stab(Q) \leq \GL_{n_{I''}}(\field)$ of $Q$, which is isomorphic to any Levi group $\levi$ of $H$.

We continue the construction of our Picard-Vessiot extension $\overline{\generalext}$ of $\difffield$ and our extended specialization \eqref{eq:introduction1} by matching the parameters of the reductive part with their respective counterparts in $\extfieldred$.
More precisely, the idea is to find for the parameters $\bvv$, $\bexp$ and $\intn_i$ with $\beta_i \in \leviroots^-$ rational functions
$\widehat{\bvv}$, 
$\bratexp$ and $\ratint_i$ in $\Frac(\difffield[\GL_{n_{I''}}])$ such that $\siginter$ extends to a differential homomorphism
\[
\begin{array}{rcl}
    \sigred\colon D^{-1}\difffield\{ \bvv \}[\bexp,\bexp^{-1} , \intn_i \mid \beta_i \in \leviroots^- ] & \to & \extfieldred\,,\\[0.2em]
    \bvv & \mapsto & \bvh + Q \, =: \, \bvq\,,\\[0.2em]
    \bexp & \mapsto & \bratexp + Q \, =: \, \overline{\bexp}\,,\\[0.2em]
    \intn_i & \mapsto & \ratint_i + Q \, =: \, \overline{\intn}_i \, . 
\end{array}
\]
Fixing a basis $\overline{X}_{1,1}:=X_{1,1}+Q,\dots,\overline{X}_{1,n_{I''}} := X_{1,n_{I''}}+Q$ of the solution space in $\extfieldred$ of $\overline{\LCLM}(\bsq,\bvvqbase,\partial) \, y = 0$ 
we develop a necessary and sufficient condition such that the map 
\begin{eqnarray*}
\eta\colon D^{-1} \field\{ \bss(\bvv) , \bvvbase \} \{ y^{I''}_1,\dots,y^{I''}_{n_{I''}} \} & \to & \difffield [\GL_{n_{I''}}] / Q\,,\\
(\bss(\bvv),\bvvbase) & \mapsto & (\bsq,\bvvqbase)\,,\\
y^{I''}_i & \mapsto & \sum_{j=1}^{n_{I''}} \overline{c}_{j,i} \overline{X}_{1,j} 
\end{eqnarray*}
sending the generic basis $y^{I''}_1,\dots,y^{I''}_{n_{I''}}$ to $\field$-linear combinations of $\overline{X}_{1,1},\dots,\overline{X}_{1,n_{I''}}$ is a differential ring homomorphism.
We use the differential Thomas decomposition to determine $\overline{c}_{i,j} \in \field$ satisfying this condition.
The rational functions
$\bratexp$, $\widehat{\bvv}$ and $\ratint_i$ are then obtained by applying this homomorphism to 
\[
\exprforexp^{I''}_i(y_1^{I''},\dots,y^{I''}_{n_{I''}}), \ V^{I''}_i(y_1^{I''},\dots,y^{I''}_{n_{I''}})  \quad \text{and} \quad \exprforint^{I''}_j(y_1^{I''},\dots,y^{I''}_{n_{I''}}) 
\quad \text{(cf.\ \eqref{eqn:introduction1})} .
\]

Having computed the \emph{matching} $\sigred$, we extend the field $\extfieldred$ allowing us to specialize also the remaining parameters $\intn_i$ with $\beta_i \in \roots^- \setminus \leviroots^-$ of the unipotent radical part.
More precisely, we
adjoin for each $\intn_i$ with $\beta_i \in \roots^- \setminus \leviroots^-$ an element $\underline{\intrad}_i$ to $\extfieldred$ which has the appropriate derivative and all these elements are algebraically independent over $\extfieldred$.
We obtain the differential ring  
\[
\underline{R} \, := \, \extfieldred[\underline{\intrad}_i \mid \beta_i \in \roots^- \setminus \leviroots^-]
\]
such that $\sigred$ extends to a differential homomorphism
\begin{equation}\label{eqn:sigmaintroduction}
\begin{array}{rcl}
\sigul\colon D^{-1}\difffield\{ \bvv \}[\bexp,\bexp^{-1},\bint ] & \to & \Frac(\underline{R}) =: \underline{E}\\[0.5em]
\intn_i & \mapsto & \underline{\intrad}_i \quad \text{for} \   \beta_i \in \roots^- \setminus \leviroots^-\,.
\end{array}
\end{equation}
Note that $\underline{E}$ is not necessarily a Picard-Vessiot extension of $\difffield$ for $A_{\group}(\bsq)$.
We construct our final Picard-Vessiot extension $\overline{\generalext}$ of $\difffield$ as the field of fractions of $\underline{R}/\Imax$ for a choice of a maximal differential ideal $\Imax \lhd \underline{R}$. 

Applying $\sigul$ to the general fundamental matrix $\fm$ and its decomposition $\fm = \fmred \, \fmrad$ we obtain
\[
\fmuq \, = \, \fmreduq \, \fmraduq \in \group(\underline{R})
\]
with $\fmreduq \in \group(\extfieldred)$ and $\fmraduq \in \unirad(\parabolic_J)(\underline{R})$.
Note that the subsequent steps of constructing $\overline{\generalext}$
will not affect $\fmreduq$, whose construction is completed at this point.
We prove that the logarithmic derivative
\[
\Aprered \, := \, \dlog (\fmreduq) \in \Lie(\group)(\difffield)
\]
has entries in $\difffield$ and that $\extfieldred$ is a Picard-Vessiot extension of $\difffield$ for $\Aprered$.
The differential Galois group $\Lred$ of $\extfieldred/\difffield$,
in its representation induced by the fundamental matrix $\fmreduq$,
is contained in the standard Levi group of $\parabolic_J$ and will be used later.

The group of differential $\difffield$-automorphisms of $\underline{E}$
has a linear representation $\underline{H}$ induced by $\fmuq$.
We prove that $H\leq \underline{H}\leq \parabolic_J$ and that $\underline{H}$ has Levi decomposition $\underline{H} = \underline{\levi} \ltimes \unirad(\parabolic_J)$, that is, its unipotent radical coincides with the unipotent radical of $\parabolic_J$.
It turns out that $\Lred$ and any Levi group $\levi$ of $H$ are Levi groups of $\underline{H}$, implying that $\underline{\levi}$, $\levi$ and $\Lred$ are conjugate by elements in $\unirad(\parabolic_J)$.
Moreover, we will show that different choices of $\Imax \lhd \underline{R}$ lead to differential $\difffield$-isomorphic extension fields $\overline{\generalext}$.
We prove that these isomorphisms are induced by elements of $\unirad(\parabolic_J)(\field)$, as are the automorphisms of $\underline{R}$ corresponding to
transitions between different $\Imax$.
Moreover, for every Levi group $\underline{\levi}$ of $\underline{H}$, there exists $\Imax$ such that $\underline{\levi}$ is a Levi group of $H$, in particular, there exists $\Imax$ realizing $\Lred$ as a Levi group of $H$.
We present an algorithm which computes a generating set of the defining ideals $I_{\Lred}$ and $I_{\underline{H}}$ of the respective groups $\Lred$ and $\underline{H}$. 

As a conclusion, note that all algebraic relations between the parameters of the reductive part are given by the ideal $Q$. It is left to determine the algebraic relations over $\extfieldred$ among the indeterminates $\underline{\intrad}_i$.

\subsection*{The Unipotent Radical Part} 
The third and last extension takes care of the specialization of the parameters $\intn_i$ with $\beta_i \in \roots^- \setminus \leviroots^-$. The basic idea is to compute the Lie algebra of the unipotent radical $\unirad(H) = \unirad(\cocomp{H})$ of the connected component of a potential differential Galois group $H$ by reduction of $A_{\group}(\bsq)$.
Knowing the Lie algebra we can compute generators of the defining ideal of $\unirad(H)$ using the exponential map. 
Applying the same reduction to $\fmuq$, the unipotent radical part of the resulting Levi decomposition is required to lie in $\unirad(H)(\underline{R})$.
Evaluating the above generators in order to express this condition, gives
a generating set for a maximal differential ideal $\Imax$ in $\underline{R}$. We are going to explain this in more detail.

In order to determine the Lie algebra of a potential differential Galois group we compute a reduced form of $A_{\group}(\bsq)$ using the algorithm of 
T.~Dreyfus and J.-A.\ Weil presented in \cite{DreyfusWeil}. Their algorithm can be divided into three main steps: transforming $A_{\group}(\bsq)$ into block triangular form, reducing the diagonal part and finally reducing the off-diagonal part.
Since we already made choices for the construction of our fundamental matrix $\fmspec$, we replace the first two steps by a reduction procedure specific to our choices. We apply the third reduction step of their algorithm to our intermediate result to reduce the part belonging to the Lie algebra of $\unirad(\parabolic_J)$.

For our first reduction step recall from Part~\ref{part:II} the matrix $g_1$ and the Levi decomposition of $g_1 \fm$.
Applying $\sigul$ defined in \eqref{eqn:sigmaintroduction} to them we obtain a matrix $\overline{g}_1 \in \group(\difffield)$ and the Levi decomposition
\[
\overline{g}_1 \fmuq \, = \, (\overline{g}_1 \fmreduq) \cdot \fmraduq \in \parabolic_J(\underline{R})
\]
with $\overline{g}_1 \fmreduq \in \levi_J(\extfieldred)$ and $\fmraduq \in \unirad(\parabolic_J)(\underline{R})$. 

For our second reduction we possibly need to algebraically extend the base field $\difffield$ depending on the connectedness of the differential Galois group of the reductive part.
By the fundamental theorem, the fixed field in $\extfieldred$ of the connected component $\cocomp{\Stab(Q)}$ is an algebraic extension $\difffieldalg$ of $\difffield$ and the differential Galois group of $\extfieldred$ over $\difffieldalg$ is $\cocomp{\Stab(Q)}$.
This implies that the differential Galois group of $\extfieldred$ over $\difffieldalg$ for $\Aprered$ is the connected component $\cocomp{\Lred}$ of $\Lred$. Connectedness and the fact that with $\difffield$ also $\difffieldalg$ is a $C_1$-field guarantee that there exists a matrix $\overline{g}_2 \in \levi_J(\difffieldalg)$ gauge transforming $\Aprered$ into the Lie algebra of $\cocomp{\Lred}(\difffieldalg)$.
In order to compute $\overline{g}_2$ we use invariant theory for reductive groups to determine first a primitive element in $\extfieldred$ for the algebraic extension $\difffieldalg$ of $\difffield$. 
Using the ideal $Q$ and the fundamental matrix $\overline{g}_1 \fmreduq$ we can compute with Gr\"obner basis methods a generating set of a maximal differential ideal in $\difffield[\GL_n]$ defining a Picard-Vessiot extension for $\Aprered$. Based on the existence of $\overline{g}_2$ and on the primitive element for $\difffieldalg$ we develop an algorithm which computes an $\difffieldalg$-rational point $\overline{g}_2 \in \levi_J(\difffieldalg)$ of this maximal differential ideal. 
The matrix $\overline{g}_2$ achieves the reduction of the reductive part, meaning that it satisfies $\overline{g}_2 \overline{g}_1 \fmreduq \in \cocomp{\Lred}(\extfieldred)$ and $\Aprered$ is reduced into the Lie algebra: 
\[
\gauge{\overline{g}_2\overline{g}_1}{\Aprered} \, =: \, \Ared \in \Lie(\cocomp{\Lred})(\difffieldalg)\,.
\]
We prove that the effect of gauge transforming $A_{\group}(\bsq)$ by $\overline{g}_2 \overline{g}_1$ is the direct sum decomposition
\begin{equation}\label{eqn:introsemireducedA_G(bsq)}
\gauge{\overline{g}_2 \overline{g}_1}{A_{\group}(\bsq)} \, = \, \Ared + \Aprerad \in \Lie(\cocomp{\Lred})(\difffieldalg) \oplus \Lie(\unirad(\parabolic_J))(\difffieldalg)  
\end{equation}
with reduced reductive part $\Ared$. 

For the computation of our third reduction matrix $\overline{g}_3 \in \unirad(\parabolic_J)(\difffieldalg)$ we intend to use the third step of the algorithm of T.~Dreyfus and J.-A.\ Weil.
Note that their algorithm at this step requires a triangular block structure with reduced diagonal part. In case $\group = \SL_{l+1}$ or $\group = \SP_{2l}$ the decomposition \eqref{eqn:introsemireducedA_G(bsq)} fulfills this condition.
For the remaining groups the authors wonder if an adaptation of this part of their algorithm using the structure of the Lie algebra instead of the block triangularity is possible (cf.\ Proposition~\ref{prop:droppingtraingularblock}).
Neglecting effectivity at this point, we show that such a reduction using the Lie algebra exists.  
Having computed $\overline{g}_3$ with their algorithm, we achieve a complete reduction
\[
\gauge{\overline{g}_3\overline{g}_2 \overline{g}_1}{A_{\group}(\bsq)} \, = \, \Ared + \Arad 
\]
into the Lie algebra of a connected group $\GalConn$. Theoretically the defining ideal of $\GalConn$ would allow us to construct a Picard-Vessiot extension of $\difffieldalg$ for $\Ared + \Arad$ with connected differential Galois group $\GalConn$.
We prove that $\cocomp{\Lred}$ is a Levi group of $\GalConn$ and that the unipotent radical $R_1$ of $\GalConn$ is a subgroup of $\unirad(\parabolic_J)$ so that $\GalConn$ has the Levi decomposition $\GalConn=\cocomp{\Lred} \ltimes R_1$.

Using the Wei-Norman decomposition and the exponential map we can compute from $\Arad$ generators $f_1,\dots,f_a$ of the defining ideal $I_{R_1}$ of the unipotent radical $R_1$. Moreover, we explain how one can compute the Levi decomposition of the reduced fundamental matrix
\[
\overline{g}_3\overline{g}_2 \overline{g}_1 \fmuq \, = \, \fmreduqhat \fmraduqhat
\]
with $\fmreduqhat \in \cocomp{\Lred}(\extfieldred)$ and $\fmraduqhat \in \unirad(\parabolic_J)(\underline{R})$.
As a crucial step, we show that from evaluating the generators $f_1,\dots,f_a$ of $I_{R_1}$  at the matrix $\fmraduqhat$ we obtain generators
\[
f_1(\fmraduqhat), \quad \dots, \quad f_a(\fmraduqhat) \in \underline{R}
\]
of a maximal differential ideal $\Imax$ in $\underline{R}$. We construct with respect to this $\Imax$ our final Picard-Vessiot extension $\overline{\generalext}$ of $\difffield$ with Galois group $H$ and fundamental matrix $\fmspec$.
The factors of the Levi decomposition of the reduced fundamental matrix
\[
\overline{g}_3 \overline{g}_2 \overline{g}_1 \fmspec \, = \, \fmredqhat \fmradqhat
\]
then satisfy $\fmredqhat \in \cocomp{\Lred}(\extfieldred)$ and more importantly $\fmradqhat \in R_1(\overline{\generalext})$
with entries in $\overline{\generalext} \setminus \extfieldred$.
We prove that the differential Galois group $H$ has Levi decomposition $H=\levi \ltimes R_1$ with a Levi group $\levi$ such that $\cocomp{\Lred} = \cocomp{\levi}$.
In other words, we only know that the connected component of $\Lred$ is a Levi group of $\cocomp{H}$. We do not know if $\Lred$ is always automatically a Levi group of $H$. Thus, we cannot determine $I_{H}$ by computing the defining ideal of the product variety of $\Lred$ and $R_1$ from $I_{\Lred}$ and $I_{R_1}$. Instead, we actually need to compute the generators of the defining ideal $I_{H}$ of $H$ from the known algebraic relations between the parameters of $\fmspec$. 
More precisely, choosing elements $\widetilde{f}_1,\dots , \widetilde{f}_a$ in 
\[
\difffield(\GL_{n_{I''}})[\underline{\intrad}_i \mid \beta_i \in \leviroots^-]
\]
which are equal to $f_1(\fmraduqhat),\dots,f_a(\fmraduqhat)$ modulo $Q$, we present an algorithm which computes from the algebraic relations defined by the ideal $Q \lhd \difffield[\GL_{n_{I''}}]$ and by the ideal  
\[
(\widetilde{f}_1,\dots , \widetilde{f}_a ) \lhd  \difffield(\GL_{n_{I''}})[\underline{\intrad}_i \mid \beta_i \in \leviroots^-]
\]
generators of the defining ideal $I_H$ of the differential Galois group $H$. 
\subsection*{The Structure of the Paper and Conclusions} 
The paper is divided into three main parts \ref{part:I}, \ref{part:II}, \ref{part:III} each of which consists of several sections.
Part~\ref{part:I} deals with the introduction of the classical groups and the normal form equations defining the general extension field.
In Part~\ref{part:II} we analyze the action of the standard parabolic subgroups on $\generalext$ and we derive the consequences of the Galois correspondence in the generic case. The main part is \ref{part:III} in which we develop the algorithm for the computation of the differential Galois group of a specialization of the normal form matrix.  Throughout the paper we use statements whose proofs can be found in the appendix. Due to the length of the paper we included at the end a table of notation. 
  
The algorithm in its present form has certainly room for improvement. For example, the computation of the second reduction matrix $\overline{g}_2$ in Section~\ref{sec:unipotentradical} is very costly and one might use an adapted version of the algorithm of T.\ Dreyfus and J.-A.\ Weil. 
Moreover, we did not perform a detailed complexity analysis of our algorithm,  which uses Gr\"obner basis and Thomas decomposition computations and the algorithms of M.~F.\ Singer and E.\ Compoint as well as of T.\ Dreyfus and J.-A.\ Weil. Furthermore, we did not investigate yet the consequences of the assumptions on the specializations of our normal form matrices. This and the above mentioned questions are topics for future investigation. 
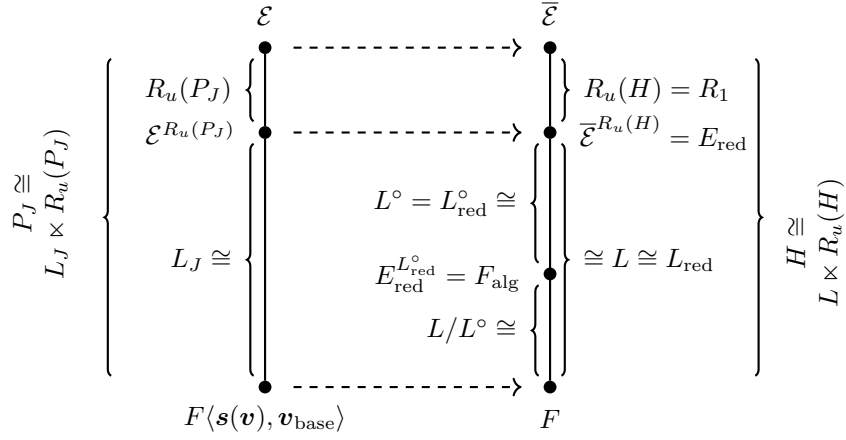
\begin{figure}[h]
\centering
\begin{tikzpicture}[thick,scale=0.75]
\node[circle, draw, fill=black, inner sep=0pt, minimum width=4pt] (CSV) at (-1,-2) {};
\node[circle, draw, fill=black, inner sep=0pt, minimum width=4pt] (RUP) at (-1,2.5) {};
\node[circle, draw, fill=black, inner sep=0pt, minimum width=4pt] (EE) at (-1,4) {};
\node[circle, draw, fill=black, inner sep=0pt, minimum width=4pt] (CZ) at (4,-2) {};
\node[circle, draw, fill=black, inner sep=0pt, minimum width=4pt] (RUH) at (4,2.5) {};
\node[circle, draw, fill=black, inner sep=0pt, minimum width=4pt] (ES) at (4,4) {};
\node[circle, draw, fill=black, inner sep=0pt, minimum width=4pt] (FALG) at (4,0) {};
\node (G1) at (-1,0.25) {};
\node (G2) at (-1,3.25) {};
\node (H1) at (4,0.25) {};
\node (H2) at (4,3.25) {};
\node (GG) at (-3.7,3) {};
\node (HH) at (7.3,-1) {};
\node (LREDCC) at (4,1.25) {};
\node (LHCC) at (4,-1) {};
\draw (CSV) -- (RUP);
\draw (RUP) -- (EE);
\draw (CZ) -- (RUH);
\draw (RUH) -- (ES);
\node (LC) [below=0.05em of CSV] {$\difffield\langle \bss(\bvv),\bvvbase\rangle$};
\node (RCZ) [below=0.25em of CZ] {$\difffield$};
\node (REE) [above=0.25em of EE] {$\generalext$};
\node (RES) [above=0.25em of ES] {$\extfieldspec$};
\node (LRUP) [left=0.5em of RUP] {$\generalext^{\unirad(\parabolic_J)}$};
\node (LFALG) [left=0.5em of FALG] {$\extfieldred^{\cocomp{\Lred}}=\difffieldalg$};
\node (RRUH) [right=0.5em of RUH] {$\extfieldspec^{\unirad(H)}=\extfieldred$};
\node (LG1) [left=0.5em of G1] {$\levi_J \cong$};
\node (LG2) [left=0.5em of G2] {$\unirad(\parabolic_J)$};
\node (RH1) [right=0.5em of H1] {$\cong \levi \cong \Lred$};
\node (RH2) [right=0.5em of H2] {$\unirad(H)=R_1$};
\node (LL) [left=2.25em of GG,rotate=90] {$\begin{array}{c} \parabolic_J\cong \\ \levi_J \ltimes \unirad(\parabolic_J) \end{array}$};
\node (RR) [right=2.5em of HH,rotate=90] {$\begin{array}{c} H\cong\\ \levi \ltimes \unirad(H) \end{array}$};
\node (LLREDCC) [left=0.5em of LREDCC] {$\cocomp{\levi} = \cocomp{\Lred} \cong$};
\node (LLHCC) [left=0.5em of LHCC] {$\levi / \cocomp{\levi} \cong$};
\draw [decorate,
    decoration = {brace}] (-1.2,-1.8) -- (-1.2,2.3);
\draw [decorate,
    decoration = {brace}] (-1.2,2.7) -- (-1.2,3.8);
\draw [decorate,
    decoration = {brace}] (4.2,2.3) -- (4.2,-1.8);
\draw [decorate,
    decoration = {brace}] (4.2,3.8) -- (4.2,2.7);
\draw [decorate,
    decoration = {brace}] (-3.7,-1.8) -- (-3.7,3.8);
\draw [decorate,
    decoration = {brace}] (7.6,3.8) -- (7.6,-1.8);
\draw [decorate,
    decoration = {brace}] (3.8,0.2) -- (3.8,2.3);
\draw [decorate,
    decoration = {brace}] (3.8,-1.8) -- (3.8,-0.2);
\draw[->, dashed] (-0.5,-2) to (3.5,-2);
\draw[->, dashed] (-0.5,2.5) to (3.5,2.5);
\draw[->, dashed] (-0.5,4) to (3.5,4);
\end{tikzpicture}
\caption{ {Decomposition of Picard-Vessiot extension as tower of fixed fields for the generic case and its specialization.}}\label{fig:fields}
\end{figure}

\part{The Classical Groups and Their Normal Forms}\label{part:I}
\section{The Classical Groups}\label{sec:classicalgroups}
Let $\field$ be an algebraically closed field of characteristic zero and let 
\[
\group(\field) \leq \GL_n(\field)
\]
be one of the  classical groups of Lie rank $l$ of type $A_l$, $B_l$, $C_l$, $D_l$ or $\Gzwei$ (here $l=2$). 
We denote by $\roots$ the root system of the type corresponding to $\group$ and we write 
$$\rootbasis=\{ \alpha_1,\dots,\alpha_l\}$$ for a basis of $\roots$ with simple roots $\alpha_i$. Each root in $\roots$ can be written uniquely as a $\Z$-linear combination of simple roots with all coefficients either positive or negative leading to a disjoint decomposition
\[
\roots \, = \, \roots^+ \cup \roots^-
\]
of $\roots$ into a set of positive roots $\roots^+$ and a set of negative roots $\roots^-$. The cardinalities of the positive and negative roots are equal and we denote them by
\[
m \, = \, |\roots^+| \, = \, |\roots^-|.
\]
The height $\height(\alpha)$ of a root $\alpha \in \roots$ is defined as the sum of all coefficients in the $\Z$-linear combination of $\alpha$ with respect to the basis $\rootbasis$. We write 
\[
\roots^+ \, = \, \{\alpha_1,\dots,\alpha_m \},
\]
where we enumerate the positive roots  
in such a way that $\alpha_1,\dots,\alpha_l$ are the simple roots from above and $\height(\alpha_r) \leq \height(\alpha_s)$ for all $r \leq s$. We write 
\[\label{eqn:numberingnegativeroots}
\roots^- \, = \, \{ \beta_1,\dots,\beta_m \}
\]
for the negative roots, where the enumeration is chosen such that $\beta_i=-\alpha_i$.
Then $\beta_1,\dots,\beta_l$ are the negative simple roots and for $i \leq j$ we have $|\height(\beta_i)|\leq |\height(\beta_j)|$.

Let $\weyl$ be the Weyl group of $\roots$. It is well known that $\weyl$ is generated by the $l$ simple reflections $w_{\alpha_i}$ for the simple roots $\alpha_i \in \rootbasis$. Each element $w$ of $\weyl$ can be written as a finite product of simple reflections.
The minimal number of simple reflections needed to express $w \in \weyl$ as such a product is called the length $l(w)$ of $w$.
The length of $w$ is equal to the number of roots $\alpha \in \roots^+$ such that $w(\alpha) \in \roots^-$. There is a unique Weyl group element of maximal length which we denote by $\overline{w}$. It induces a bijection between $\roots^+$ and $\roots^-$ and it maps the simple roots $\rootbasis$ bijectively to the negative simple roots $-\rootbasis$.

We denote the Lie algebra of $\group$ by $\liealg \subset \gl_n (\field)$ and we fix a Cartan subalgebra $\cartan$ of $\liealg$ consisting of diagonal matrices. With respect to the Cartan subalgebra $\cartan$ we consider the Cartan decomposition   
\begin{equation}
\liealg \, = \, \cartan \oplus \bigoplus_{\alpha \in \roots} \liealg_{\alpha}
\end{equation}
 of $\liealg$ with one-dimensional root spaces $\liealg_{\alpha} \subset \liealg$ for the roots $\alpha \in \roots$. We choose a Chevalley basis 
\begin{equation}\label{eqn:Chevbasis}
\{ H_i  \mid 1 \leq i \leq l \} \cup \{ X_{\alpha} \mid \alpha \in \roots  \}
\end{equation}
according to this decomposition, where $H_1,\dots,H_l$ span the Cartan subalgebra $\cartan$ and each $X_{\alpha}$ the root space $\liealg_{\alpha}$. We write
\[
\mathfrak{u}^+ \, = \, \bigoplus_{i=1}^m  \liealg_{\alpha_i} \quad \text{and} \quad  \mathfrak{u}^- \, = \, \bigoplus_{i=1}^m \liealg_{\beta_i}
\]
for the maximal nilpotent Lie subalgebras being the direct sums of all positive and all negative root spaces, respectively. We denote by 
\[
\mathfrak{b}^+ \, = \, \cartan \oplus \mathfrak{u}^+ \quad \text{and} \quad \mathfrak{b}^- \, = \, \cartan \oplus \mathfrak{u}^-
\]
the maximal solvable Lie subalgebras containing the Cartan subalgebra $\cartan$. 

For a root $\alpha \in \roots$ we denote by $U_{\alpha}$ the one-dimensional root group whose Lie algebra is the root space $\liealg_{\alpha}$. For $x \in \field$ we write $u_{\alpha}(x) \in U_{\alpha}$ for the image of $x X_{\alpha}$ under the exponential map 
\begin{equation*}\label{eqn:exopnentialmap}
\exp\colon \liealg_{\alpha} \to U_{\alpha}, \ X_{\alpha} \mapsto \sum_{j \geq 0} \frac{1}{j!} X_{\alpha}^j  
\end{equation*}
and we call $u_{\alpha}(x)$ a \emph{parametrization of the root group} $U_{\alpha}$.
For $\bxx=(x_1,\dots,x_m) \in \field^m$ we denote by 
\[
\buu(\bxx) \, = \, u_{\beta_1}(x_1) \cdots u_{\beta_m}(x_m)  
\]
the product of all parametrized root group elements for the negative roots in this fixed order.
Let $\unipotent^+$ and $\unipotent^-$ be the maximal unipotent subgroups generated by all root groups corresponding to the positive and negative roots, respectively. The Lie algebra of $\unipotent^+$ (resp.\ $\unipotent^-$) is then $\mathfrak{u}^+$ (resp.\ $\mathfrak{u}^-$). 

We denote by $\torus$ the maximal torus whose Lie algebra is $\cartan$. Moreover, for $i=1,\dots,l$ consider the one-dimensional subtorus $T_i$ of $\torus$ with Lie algebra spanned by $H_i$.
For $x \in \field^{\times}$ let $t_i(x)$ be the image of 
$$\begin{pmatrix} x & 0 \\ 0 &x^{-1} \end{pmatrix}$$ under the isomorphism
\[
 \SL_2 \to \langle U_{\alpha_i},U_{\beta_i} \rangle, \ \begin{pmatrix}
    1 & x \\ 0 & 1
\end{pmatrix} \mapsto \exp(x X_{\alpha_i}), \ \begin{pmatrix}
    1 & 0 \\ x & 1
\end{pmatrix} \mapsto \exp(x X_{\beta_i}) \, .
\]
Then $t_i(x)$ parametrizes $T_i$. For $\bxx = (x_1,\dots,x_l) \in (\field^{\times})^l$ we write
\[
\btt(\bxx) \, = \, t_1(x_1) \cdots t_l(x_l)
\]
for the product of all $t_i(x_i)$, and $\btt(\bxx)$ parametrizes the full torus $\torus$. We denote the normalizer of $\torus$ in $\group$ by $N_{\group}(\torus)$. The Weyl group $\weyl$ is isomorphic to the quotient $N_{\group}(\torus)/\torus$ and for each $w \in \weyl$ we fix a representative $n(w)$ of $w$ in $N_{\group}(\torus)$.

We denote by $\borel^+$ and $\borel^-$ the Borel subgroups containing $\torus$ and the maximal unipotent subgroups $\unipotent^+$ and $\unipotent^-$, respectively.
Clearly, we have $\borel^+ = \torus \unipotent^+$ and $\borel^- = \torus \unipotent^-$ and their Lie algebras are $\mathfrak{b}^+$ and $\mathfrak{b}^-$, respectively. We denote a parabolic subgroup of $\group$, that is a subgroup which contains a Borel subgroup, by $\parabolic$.
In this paper the standard parabolic subgroups are those which contain the Borel subgroup $\borel^-$. Each parabolic subgroup of $\group$ is conjugate to one and only one standard parabolic subgroup. For a subset
$J \subset \{1,\dots,l \}$ let $\weyl_J$ be the subgroup of $\weyl$ generated by the simple reflections $w_{\alpha_j}$ with $j \in J$. The groups
\begin{equation}\label{eqn:standardparabolic}
    \parabolic_J \, := \, \bigcup_{w \in \weyl_J} \borel^- \, n(w) \, \borel^-
\end{equation}
are the standard parabolic subgroups of $\group$, i.e., the map $J\mapsto \parabolic_J$ defines a bijection between the subsets of $\{1,\dots,l\}$ and the standard parabolic subgroups of $\group$. The roots of $\parabolic_J$ relative to $\torus$ are the roots in $\roots^-$ and $\leviroots^+ := \leviroots \cap \roots^+$, where
\begin{equation}\label{eqn:definitionLeviroots}
\leviroots \, := \, \roots \cap \langle \alpha_j \mid j \in J \rangle_{\Z\mathrm{-span}} \, .
\end{equation}

To shorten notation we will omit in the following the plus sign in the notation of Lie subalgebras and subgroups, i.e., we will write $\mathfrak{u}$ for $\mathfrak{u}^+$ and so on.

Our results (including the previous ones in \cite{Seiss,Seiss_Generic,RobertzSeissNormalForms}) are based on geometric structure theorems of algebraic groups (cf.\ Theorems~\ref{thm:bruhat1}--\ref{thm:levidecomposition}). They are used to establish a connection between the defining matrix for the differential equation, the fundamental solution matrix and the differential Galois group. 

For the reductive group $\group$ the following two theorems provide a normal form for elements in $\group$ parametrized by $\borel$ and $\weyl$. For their proofs we refer to \cite[28.3 Theorem and 28.4 Theorem]{HumGroups}.

\begin{theorem}[Bruhat decomposition]\label{thm:bruhat1}
We have 
\[
\group \, = \, \biguplus_{w \in \weyl} \borel \, n(w) \, \borel \qquad \text{(disjoint  union)}
\]
with $\borel n(w) \borel = \borel n(\widetilde{w}) \borel$ if and only if $w=\widetilde{w}$ in $\weyl$.
\end{theorem}

\begin{theorem}\label{thm:bruhat2}
Each element $g \in \group$ can be written in the form 
\[
g \, = \, u' \, n(w) \, t \, u \, ,
\]
where $w \in \weyl$, $t \in \torus$, $u \in U$ and $u' \in U'_{w} := U \cap n(w) \,\unipotent^- \, n(w)^{-1}$ are all determined uniquely by $g$.
\end{theorem}

 The next two theorems provide us with a decomposition of a subgroup  of $\group$ into a reductive and a unipotent subgroup. For details and the proofs of the theorems we refer to \cite[Chapter~6, $4^{\circ}$]{OnishchikVinberg}.
The radical $\radic(H)$ of a subgroup $H$ of $\group$ is the largest connected normal solvable subgroup of $H$. The subgroup $\unirad(H)\leq \radic(H)$ of unipotent elements of $\radic(H)$ is normal in $H$ and is called the unipotent radical of $H$.

\begin{theorem}[Levi decomposition]\label{thm:levidecomposition}
There is a reductive subgroup $\levi$ of $H$, called a \emph{Levi group} of $H$, such that
\[
H \, = \, \levi \ltimes \unirad(H) \qquad \text{(semidirect product).}
\]
Each element $g \in H$ can be written in the form $g = g' \, u$ with $g' \in \levi$ and $u \in \unirad(H)$ uniquely determined by $g$.
\end{theorem}

\begin{theorem}\label{thm:levidecompositionconj}
    Let $H = \levi \ltimes \unirad(H)$ be a Levi decomposition of $H$ and let $\widetilde{H}$ be a reductive subgroup of $H$. Then there exists $u \in \unirad(H)$ such that $u \widetilde{H} u^{-1} \leq \levi$. In particular, if $\widetilde{\levi}$ is another Levi group of $H$, then $\levi$ and $\widetilde{\levi}$ are conjugate by an element of $\unirad(H)$.
\end{theorem}

For a standard parabolic subgroup $\parabolic_J$ the unipotent radical $\unirad(\parabolic_J)$ is generated by those root groups $U_{\beta}$ for which $\beta \in \roots^- \setminus \leviroots^-$ with $\leviroots^-:=\roots^- \cap \leviroots$. 
Its Lie algebra is the direct sum of root spaces
\[
\Lie(\unirad(\parabolic_J)) \, = \, \bigoplus_{\beta \in \roots^- \setminus \leviroots^-} \liealg_{\beta} \, .
\]
The group $\parabolic_J$ has many Levi groups and by Theorem~\ref{thm:levidecompositionconj} any two Levi groups of $\parabolic_J$ are conjugate by an element of $\unirad(\parabolic_J)$.
We denote by $\levi_J$ the Levi group of $\parabolic_J$ whose Lie algebra is
\[
\Lie(\levi_J) \, = \, \cartan \oplus \bigoplus_{\alpha \in \leviroots} \liealg_{\alpha} \, .
\]
We call $\levi_J$ the \emph{standard Levi group of} $\parabolic_J$ and
\begin{equation}\label{eqn:standardLevi}
\parabolic_J \, = \, \levi_J\ltimes \unirad(\parabolic_J) 
\end{equation}
the \emph{standard Levi decomposition of} $\parabolic_J$.

The following definition is taken from \cite[Section~3.2.1]{SerreIrreductiblite}.

\begin{definition}
    Let $H$ be a connected reductive linear algebraic group. A closed subgroup $\widetilde{H}$ of $H$ is called $H$-irreducible, if $\widetilde{H}$ is not contained in any proper parabolic subgroup of $H$.
\end{definition}

\begin{proposition}\label{prop:LirreducibleForLevi}
    Let $H$ be a closed reductive subgroup of $\group$.
    Then there exists a parabolic subgroup $P$ of $\group$ which is minimal with respect to containing $H$, such that $H$ is $\levi$-irreducible for a Levi group $\levi$ of $P$. Every parabolic subgroup which is minimal with respect to containing $H$ has this property.
\end{proposition}

\begin{proof}
In characteristic zero for a closed subgroup $H$ of $\group$ the notion of being $\group$-completely reducible and reductive coincide (cf.\ \cite[Section~2.2]{BateMartinRoehrle}). Thus the statement follows from \cite[Corollary~3.5]{BateMartinRoehrle}.
\end{proof}

\section{A Normal Form for the Classical Groups}\label{sec:normalform}

Let $\difffield$ be the rational function field $\field(z)$ with standard derivation $\frac{\mathrm{d}}{\mathrm{d}z}$. Let $\boldsymbol{\mathrm{t}}=(\mathrm{t}_1,\dots,\mathrm{t}_l)$ be differential indeterminates over $\difffield$. Let $\group$ be one of the classical groups of Section~\ref{sec:classicalgroups}.
In \cite{Seiss} we constructed a matrix $A_{\group}(\boldsymbol{\mathrm{t}})$ in $\liealg( \field \langle \boldsymbol{\mathrm{t}}\rangle)$ such that 
$\partial(\byy) = A_{\group}(\boldsymbol{\mathrm{t}}) \, \byy$ defines a Picard-Vessiot extension of $\field \langle \boldsymbol{\mathrm{t}} \rangle$ with differential Galois group $\group(\field)$.
For the construction of $A_{\group}(\boldsymbol{\mathrm{t}})$ we considered the structure of the Lie algebra $\liealg$.
More precisely, there are $l$ positive roots $\gamma_1,\dots,\gamma_l \in \roots^+$ such that we obtain a direct sum decomposition
\[
\mathfrak{b}^+ \, = \, \ad(A_0^-)(\mathfrak{u}^+) + \sum_{i=1}^l \liealg_{\gamma_i}\,,
\]
where 
\[
A_0^- \, = \, \sum_{i=1}^l X_{\beta_i}
\]
is the sum of all basis elements of the root spaces corresponding to the simple negative roots. The roots $\gamma_1,\dots,\gamma_l$ are called the \emph{complementary roots} and their heights $\height(\gamma_i)$ are equal to the exponents of $\liealg$.
Interchanging the roles of the negative with the positive roots, we define 
\[
A_0^+ \, = \, \sum_{i=1}^l X_{\alpha_i}\,.
\]
\begin{definition}
We call the matrix
\[
A_{\group}(\boldsymbol{\mathrm{t}}) \, := \, A_0^+ + \sum_{i=1}^l \mathrm{t}_i X_{-\gamma_i}
\]
the \emph{normal form matrix} for the classical group $\group$.
\end{definition}

For the construction of our general extension field we consider further differential indeterminates $\bvv=(v_1,\dots,v_l)$ over $\difffield$.
Recall that $n(\overline{w}) \in N_{\group}(\torus)$ denotes a representative of the longest Weyl group element $\overline{w} \in \weyl$.

\begin{theorem}[\cite{RobertzSeissNormalForms}]\label{thm:RobertzSeissNormalForms}
    There are 
    \begin{enumerate}
    \renewcommand{\theenumi}{(\roman{enumi})}
    \renewcommand{\labelenumi}{\theenumi}
    \item non-zero constants $c_1,\dots,c_l$ in $\field$,
    \item polynomials $g_1(\bvv),\dots,g_l(\bvv) \in \Z[\bvv]$ that are $\field$-linearly independent and homogeneous  of degree one,
    \item differential polynomials $\bss(\bvv)=(s_1(\bvv),\dots,s_l(\bvv))$ with $s_i(\bvv) \in \field\{ \bvv\}$ which are differentially algebraically independent over $\difffield$ and 
    \item differential polynomials $\bff=(f_{l+1},\dots,f_m)$ with $f_i \in \field\langle \bvv \rangle$
    \end{enumerate}
    having the following properties:
    \begin{enumerate}
    \item The matrix
    \[
    \ALiou(\bvv) \, := \, \sum_{i=1}^l g_i(\bvv) H_i + \sum_{i=1}^l c_i X_{\beta_i} \in \mathfrak{b}^-(\field\langle \bvv\rangle),
    \]
    where $ \beta_i = -\alpha_i$ for $i=1,\dots,l$, defines a Picard-Vessiot extension $\generalext$ of $\difffield\langle \bvv \rangle$ with differential Galois group $\borel^-(\field)$ and has fundamental solution matrix 
    \[
    \fml \, := \, \btt(\bexp) \, \buu(\bint) \in \borel^-(\generalext)\,.
    \]
    The parameters $\bexp=(\exp_1,\dots,\exp_l) \in \generalext^l$ for the diagonal torus in $\borel^-(\generalext)$ satisfy
    \[
    \partial(\exp_i) \exp_i^{-1} \, = \, g_i(\bvv)
    \]
    and the parameters $\bint = (\intn_1, \dots, \intn_m) \in \generalext^m$ for the maximal unipotent subgroup of $\borel^-(\generalext)$ are successive integrals, meaning that the integral $\intn_i$ corresponds to the root $\beta_i$ and depends on $\bexp$ and on those integrals $\intn_j$ with $|\height(\beta_j)| < |\height(\beta_i)|$. 
    \item The logarithmic derivative of 
    \[
    \fm \, := \, \buu(\bvv,\bff) \, n(\overline{w}) \, \fml
    \]
    is the normal form matrix $A_{\group}(\bss(\bvv))$.
    The differential field $\generalext=\difffield\langle \bvv \rangle(\fml)$ is a Picard-Vessiot extension of $\difffield\langle \bss(\bvv) \rangle$ for $A_{\group}(\bss(\bvv))$ and the differential Galois group of $\generalext$ over $\difffield\langle \bss(\bvv) \rangle$ is $\group(\field)$.
    \end{enumerate}
\end{theorem}

The action of $\group(\field)$ on the parameters $\bvv$, $\bff$, $\bexp$ and $\bint$, which generate $\generalext$ over $\difffield \langle \bss(\bvv)\rangle$ as a field, is induced by the Bruhat decomposition, namely each $g \in \group(\field)$ induces a differential $\difffield \langle \bss(\bvv)\rangle$-automorphism 
\[
\gamma_g \in \Gal_\partial(\generalext / \difffield \langle \bss(\bvv) \rangle)
\]
of $\generalext$ by multiplying $\fm$ from the right with $g$.
Thus we can determine the effect of $\gamma_g$ on the parameters $\bvv$, $\bff$, $\bexp$ and $\bint$ by means of the Bruhat decomposition of the product
\[
\fm \, g \, = \, \buu(\bvv,\bff ) \, n(\overline{w}) \, \btt (\bexp) \, \buu(\bint) \, g\,.
\]
More precisely, if $\bvv^g = (v^g_1, \dots , v^g_l)$, $\bff^g=(f^g_{l+1},\dots,f^g_m)$, $\bexp^g=(\exp^g_1,\dots, \exp^g_l)$ and $\bint^g=(\intn^g_1, \dots , \intn^g_m)$ are the parameters of the Bruhat decomposition
\[
\fm \, g \, = \, \buu(\bvv^g,\bff^g) \, n(\overline{w}) \, \btt (\bexp^g) \, \buu(\bint^g)\,,
\] 
then the images of $\bvv$, $\bff$, $\bexp$ and $\bint$ under $\gamma_g$ are 
\[
\begin{array}{rcl}
\gamma_g(\bvv) \! & \! = \! & \! \bvv^g\,,\\[0.2em]
\gamma_g(\bff) \! & \! = \! & \! \bff^g\,,\\[0.2em]
\gamma_g(\bexp) \! & \! = \! & \! \bexp^g\,,\\[0.2em]
\gamma_g(\bint) \! & \! = \! & \! \bint^g\,.
\end{array}
\]
Recall from Theorem~\ref{thm:RobertzSeissNormalForms} that the differential polynomials $s_i(\bvv) \in \difffield\{ \bvv \} \subset \generalext$ are invariant under the action of the differential Galois group $\group(\field)$.

\begin{definition}\label{de:normalformequation}
The matrix differential equation defined by the normal form matrix $A_{\group}(\bss(\bvv))$ corresponds to the scalar linear differential equation
\[
L_{\group}(\bss(\bvv),\partial) \, y \, = \, 0
\] 
introduced in \cite{Seiss} with suitable operator
\[
L_{\group}(\bss(\bvv),\partial) \in \field\{ \bss( \bvv) \}[\partial]\,.
\]
We call $L_{\group}(\bss(\bvv),\partial) \, y = 0$ the \emph{normal form (scalar) equation} for $\group$ and the linear operator 
$L_{\group}(\bss(\bvv),\partial)$  
the \emph{normal form operator} for $\group$.
\end{definition}

Given the normal form equation, we consider its associated equations (cf.\ Appendix~\ref{app:associated&riccati} for their definition and construction) and their corresponding Riccati equations in case $\group$ is of type $A_l$, $B_l$, $C_l$ or $\Gzwei$ (here $l=2$).

\begin{definition}\label{def:associatedRic}
For $i=1,\dots,l$ we denote the
\emph{$i$-th associated equation} for the normal form equation $L_{\group}(\bss(\bvv),\partial) \, y = 0$ by
\[
L^{\det(i)}(\bss(\bvv),\partial) \, y \, = \, 0 \, . 
\]
Moreover, we denote by 
\[
\Ric{i}{\bss(\bvv)}{y} \, = \, 0
\]
the Riccati equation for the $i$-th associated equation.
\end{definition}

\begin{proposition}\label{prop:exponentialandRiccati}
The $i$-th associated equation has the exponential
\[
\expsolass_i \, := \, e^{\int b_i v_i} \in \generalext  
\] 
for some $b_i \in \{\pm 1,-2 \}$
as a solution and so $\Ric{i}{\bss(\bvv)}{y} = 0$ has the solution $b_i v_i \in \field \{ \bvv \}$. Moreover, $b_1v_1,\dots,b_lv_l$ and $g_1(\bvv), \dots,g_l(\bvv)$ generate the same $\Z$-module.
\end{proposition}

\begin{proof}
See Proposition~\ref{prop:assoc_operator} and Corollary~\ref{cor:riccati_assoc_operator} in  Appendix~\ref{app:associated&riccati}.
\end{proof}

Over $\field\langle \bss(\bvv) \rangle$ the normal form operator is irreducible. 
In the subsequent sections we will consider irreducible factorizations of the normal form operator over intermediate differential fields
of $\field\langle \bss(\bvv) \rangle \subset \field\langle \bvv \rangle$ obtained by adjoining subsets of the indeterminates $\bvv$ to $\field\langle \bss(\bvv) \rangle$.
In the other extreme case when the base field is $\field\langle \bvv \rangle$ we have the following factorization.
\begin{proposition}
    The normal form operator
    \[
    L_{\group}(\bss(\bvv),\partial) \in \field \langle \bss(\bvv) \rangle[\partial]
    \]
    of order $n$ factors over $\field\langle \bvv \rangle$ into a product of first order operators
    \[
    L_{\group}(\bss(\bvv),\partial) \, = \, \prod_{i=1}^n (\partial - a_i)\,,
    \]
    where $a_i \in \field [ \bvv]$ is homogeneous of degree one for all $i=1,\dots,n$.
\end{proposition}

\begin{proof}
See Propositions~\ref{prop:factorization} and \ref{prop:factorizationG2} in Appendix~\ref{sec:exponentialsol}.
\end{proof}

\section{The Gauge Transformation}\label{sec:gauge}

Let $K$ be a differential field with field of constants $\field$ and derivation $\partial_K$.
Two matrices $A_1$ and $A_2 \in \gl_n(K)$ are called \emph{gauge equivalent} over $K$ if there exists $g \in \GL_n(K)$ such that  
\begin{equation}\label{eqn:gaugeequiv}
\gauge{g}{A_1} \, := \, g A_1 g^{-1} + \partial_K(g)g^{-1} \, = \, A_2 \, . 
\end{equation}
Consider now the adjoint action $\Ad$ of $\group$ on $\liealg$, that is for $g \in \group$ the automorphism 
\[
\Ad(g)\colon \liealg \to \liealg, \quad X \mapsto gXg^{-1}  \, .
\]
Moreover, let
\[
\dlog\colon \GL_n(K) \to \gl_n(K), \quad g \mapsto \partial_k(g)g^{-1}
\]
be the  logarithmic derivative.
Then we have
\begin{equation}
   \gauge{g}{A_1} \, = \, \Ad(g)(A_1) + \dlog (g) \, = \, A_2\,,
\end{equation}
and so the following two remarks enable us to describe the gauge transformation of an element in $\liealg(K)$ by a root group element in terms of the root system of $\group$.  
\begin{remark}\label{remark3}
For linearly independent $\alpha$, $\beta \in \roots$ and $r, q \in \N$ let
$\alpha - r \beta, \dots ,\alpha + q \beta$ be the $\beta$-string through $\alpha$ and let 
$\langle \alpha, \beta \rangle$ be the Cartan integer. Then, with respect to the Chevalley basis defined in \eqref{eqn:Chevbasis}, we have \cite[Section~4.3]{Carter}
\begin{eqnarray*}
\Ad(u_{\beta}(x))(X_{\alpha}) & = & \sum\nolimits_{i=0}^q c_{\beta, \alpha,i} x^i X_{\alpha + i \beta}\,,\\
\Ad(u_{\beta}(x))(H_{\alpha}) & = & H_{\alpha} - \langle \alpha, \beta \rangle x X_{\beta}\,,\\
\Ad(u_{\beta}(x))(X_{-\beta}) & = & X_{-\beta} + x H_{\beta} - x^2  X_{\beta}\,,
\end{eqnarray*}
where $c_{\beta,\alpha,0} = 1$ and $c_{\beta,\alpha,i} = \pm \binom{r+i}{i}$.
\end{remark}

The next remark is taken from \cite{KovacicInvProb}.
\begin{remark}\label{remark4}
 Let $\group \leq \GL_n$ be a linear algebraic group. Then the restriction of $\dlog$ to $\group$ maps
 $\group(K)$ to its Lie algebra $\liealg(K)$, i.e., we have
 \begin{equation*}
  \dlog|_{\group}\colon \group(K) \to \liealg(K).
 \end{equation*}
\end{remark}

\part{The General Extension Field and the Standard Parabolic Subgroups}\label{part:II}

\section{The Fixed Field of a Standard Parabolic Subgroup}\label{sec:fixedfieldparabolic}
Since up to conjugation every reductive subgroup of $\group$ is contained $\levi_J$-irreducibly in the standard Levi group $\levi_J$ of a standard parabolic subgroup $\parabolic_J$, we investigate in this section the fixed field $\generalext^{\parabolic_J}$ of the general extension field under $\parabolic_J$.
It will turn out that it is differentially generated over $\difffield$ by $\bss(\bvv)$ and by the differential indeterminates of a uniquely determined subset of $\{ v_1,\dots, v_l \}$.  
We introduce the following notation which will be used throughout the paper.
\begin{definition}\label{def:partitions}
We denote by 
\[
I' \cup I'' \, = \, \{i_1, \dots, i_r \} \cup \{ i_{r+1},\dots, i_l \} 
\]
a partition of $I := \{ 1, \dots, l \}$ and we define the set  
\[
J \, = \, \{ j \in I \mid \alpha_j = \overline{w} \, (-\alpha_i) \ \text{for some} \ i \in I' \}\,.
\]
The partition $I' \cup I''$ defines a partition  
\[
\bvvbase \, = \, (v_{i_{r+1}}, \dots, v_{i_l}), \quad 
\bvvext \, = \, (v_{i_{1}}, \dots, v_{i_r})
\]
of the differential indeterminates $\bvv = (v_1,\dots,v_l)$.
\end{definition}

Recall from Section~\ref{sec:normalform} that the Galois action $\gamma_g$ for $g \in \group(\field)$ on the parameters $\bvv$, $\bff$, $\bexp$ and $\bint$ in the Bruhat decomposition of $\fm$ is induced by the Bruhat decomposition of $\fm \, g$. For a standard parabolic subgroup
\[
\parabolic_J \, = \, \bigcup_{w \in \weyl_J} \borel^- \, n(w) \, \borel^-
\]
the following lemma determines the action of $n(w_{\alpha_j}) \in N_{\group}(\torus)$ on the parameters for the simple reflections $w_{\alpha_j}$ with $j \in J$ which generate $\weyl_J$.

\begin{lemma}\label{lem:invariantreflections}
Let $\bxx=(x_1,\dots,x_m)$, $\bee=(e_1,\dots,e_l)$ and $\byy=(y_1,\dots,y_m)$ be indeterminates over $\field$. For a simple root $\alpha_j \in \rootbasis$ let $\alpha_i \in \rootbasis$ be the unique simple root such that $\overline{w} \, (-\alpha_i)=\alpha_j$. Then there exist $x \in \field(\bxx,\bee,\byy) \setminus \field (\bxx )$ and $b \in \borel^-(\field(\bxx,\bee,\byy))$ such that
\[
\buu(\bxx) \, n(\overline{w}) \, \btt(\bee) \, \buu(\byy) \, n(w_{\alpha_j}) \, = \, \buu(\bxx) \, u_{-\alpha_i}(x) \, n(\overline{w}) \, b\,.
\]
\end{lemma}

\begin{proof}
For the given simple root $\alpha := \alpha_j \in \rootbasis$ the maximal unipotent group 
$\unipotent^-$ can be written as
\begin{equation}\label{eq:decompositionUneg}
\unipotent^- \, = \, \unipotent_{-\alpha} \, \unipotent^-_{w_{\alpha}} \, = \, \unipotent^-_{w_{\alpha}} \, \unipotent_{-\alpha}\,,
\end{equation}
where $U^-_{w_{\alpha}} := \unipotent^- \cap n(w_{\alpha}) \, \unipotent^- \, n(w_{\alpha})^{-1}$,
by \cite[end of 28.1]{HumGroups}, exchanging the roles played by the positive and negative roots. 
Moreover, the subgroups $\unipotent^-_{w_{\alpha}}$ and $\torus$ are normalized by 
$n(w_{\alpha})$  
and so we obtain
\[
\begin{array}{rcl}
\unipotent^- \, n(\overline{w}) \, \borel^- \, n(w_{\alpha})
\! & \! = \! & \! \unipotent^- \, n(\overline{w}) \, \torus \, \unipotent^- \,  n(w_{\alpha}) \, = \, \unipotent^- \, n(\overline{w}) \, \unipotent^- \, \torus \, n(w_{\alpha})\\[0.5em]
\! & \! = \! & \!
\unipotent^- \, n(\overline{w}) \, \unipotent^- n(w_{\alpha}) \, \torus \, = \, \unipotent^- \, n(\overline{w}) \, \unipotent_{-\alpha} \, \unipotent^-_{w_{\alpha}} \, n(w_{\alpha}) \, \torus\\[0.5em]
\! & \! = \! & \!
\unipotent^- \, n(\overline{w}) \, \unipotent_{-\alpha} \, n(w_{\alpha}) \, \unipotent^-_{w_{\alpha}} \, \torus\,.   
\end{array}
\]
Applying this to the product of $\buu(\bxx) \, n(\overline{w}) \, \btt(\bee) \, \buu(\byy)$ with $n(w_{\alpha})$, we obtain 
\begin{equation}\label{eq:fundmatrixlevel}
\begin{array}{rcl}
\! & \! \! & \! \buu(\bxx) \, n(\overline{w}) \, \btt(\bee) \, \buu(\byy) \, n(w_{\alpha}) \\[0.5em]
\! & \! = \! & \!
\buu(\bxx) \, n(\overline{w}) \, u_2 \, t_2 \, n(w_{\alpha}) = \,
 \buu(\bxx) \, n(\overline{w}) \, u_2 \, n(w_{\alpha}) \, t_3\\[0.5em]
\! & \! = \! & \! \buu(\bxx) \, n(\overline{w}) \, u_{-\alpha} \, u_{w_{\alpha}} \, n(w_{\alpha}) \, t_3 \, = \,
  \buu(\bxx) \, n(\overline{w}) \, u_{-\alpha} \, n(w_{\alpha}) \, \widetilde{u}_{w_{\alpha}} \, t_3\,,
\end{array}
\end{equation}
where $t_2$, $t_3 \in \torus$, $u_2 = u_{-\alpha} u_{w_{\alpha}} \in \unipotent^-$ with $u_{-\alpha} \in \unipotent_{-\alpha}$, $u_{w_{\alpha}} \in \unipotent^-_{w_{\alpha}}$ and $\widetilde{u}_{w_{\alpha}} \in U^-_{w_{\alpha}}$.
Since we only rewrote the product 
\[
\btt(\bee) \, \buu(\byy) \, n(w_{\alpha}) \, ,
\] 
whose factors are parametrized by $\bee$ and $\byy$, we conclude that the parameter of $u_{-\alpha}$ lying in the one-parameter subgroup $U_{-\alpha}$ is an element of $\field(\bxx,\bee,\byy) \setminus \field(\bxx)$.
In the last product of \eqref{eq:fundmatrixlevel} we investigate further the factor $\buu(\bxx ) \, n(\overline{w}) \, u_{-\alpha} \, n(w_{\alpha})$. From \cite[end of 28.1]{HumGroups} (now without exchanging the roles played by the positive and negative roots) it follows that on group level we have
\begin{equation}\label{eq:grouplevel}
\unipotent^- \, n(\overline{w}) \, \unipotent_{-\alpha} \, n(w_{\alpha}) \, = \, n(\overline{w}) \, \unipotent^+ \, \unipotent_{-\alpha} \, n(w_{\alpha}) \, = \, n(\overline{w}) \, \unipotent_{w_{\alpha}} \, \unipotent_{\alpha} \, \unipotent_{-\alpha} \, n(w_{\alpha})
\end{equation}
with $n(\overline{w})^{-1} \, \unipotent^- \, n(\overline{w}) = \unipotent^+$ and $\unipotent^+ = \unipotent_{w_{\alpha}}^+ \, \unipotent_{\alpha}$ and $U^+_{w_{\alpha}} := \unipotent^+ \cap n(w_{\alpha}) U^+n(w_{\alpha})^{-1}$.  
Since $\unipotent_{\alpha} \, \unipotent_{-\alpha} \, n(w_{\alpha})$ is contained in the centralizer $C_{\group}(\kernel(\alpha))$ (the root $\alpha$ is here considered as the map $\alpha: \torus \to \field^{\times}$), the computation in $\SL_2$
\begin{equation}\label{eq:sl2computation}
\begin{pmatrix} 1 & a_1 \\ 0 & 1 \end{pmatrix} \begin{pmatrix} 1 & 0 \\ a_2 & 1 \end{pmatrix}
\begin{pmatrix} 0 & 1 \\ -1 & 0 \end{pmatrix} = 
\begin{pmatrix} 1 & \frac{a_1 a_2+1}{a_2} \\ 0 & 1 \end{pmatrix}
\begin{pmatrix} \frac{1}{a_2} & 0 \\ -1 & a_2 \end{pmatrix} 
\end{equation}
with $a_1$, $a_2 \in \field(\bxx,\bee,\byy)$ and $a_2 \neq 0$
combined with \eqref{eq:grouplevel} implies that there exist 
$u_+ \in \unipotent^+$ and $u_{w_{\alpha}} \in \unipotent^+_{w_{\alpha}}$, $u_{\alpha} \in \unipotent_{\alpha}$
with $u_+=u_{w_{\alpha}} u_{\alpha}$ and $\widetilde{u}_{\alpha} \in \unipotent_{\alpha}$ and $b \in \borel^-$ 
such that
\begin{equation}\label{eq:importantpart}
\begin{array}{rcl}
\buu(\bxx) \, n(\overline{w}) \, u_{-\alpha} \, n(w_{\alpha}) \! & \! = \! & \!
n(\overline{w}) \, u_+ \, u_{-\alpha} \, n(w_{\alpha})\\[0.5em]
\! & \! = \! & \!  
 n(\overline{w}) \, u_{w_{\alpha}} \, u_{\alpha} \, u_{-\alpha} \, n(w_{\alpha}) \, = \, n(\overline{w}) \, u_{w_{\alpha}} \, \widetilde{u}_{\alpha} \, b\,,
\end{array}
\end{equation}
where $\widetilde{u}_{\alpha} \, b =  u_{\alpha} \, u_{-\alpha} \, n(w_{\alpha})$ according to \eqref{eq:sl2computation}.
The longest Weyl group element $\overline{w}$ induces a bijection between $\rootbasis$ and $\rootbasis^-$ and so there exists a unique $-\alpha_i \in \rootbasis^-$ such that $\overline{w}(-\alpha_i) = \alpha$. Thus, from the decomposition
\[
\buu(\bxx ) \, = \, u_{w_{\alpha_i}} u_{-\alpha_i}(x_i) \, = \,
\Big( \prod_{-\alpha_k \in \rootbasis^- \setminus \{ -\alpha_i \}} u_{-\alpha_k}(x_k) \prod_{\scriptsize \begin{array}{c} \beta \in \roots^- \\ \height(\beta)< -1 \end{array}} u_{\beta} \Big) \, u_{-\alpha_i}(x_i)\,,
\]
which we obtain from applying \eqref{eq:decompositionUneg} for the simple root $\alpha_i$ to $\buu(\bxx )$, 
it follows that
\[
\begin{array}{rcl}
n(\overline{w})^{-1} \, \buu(\bxx ) \, n(\overline{w})
\! & \! = \! & \! u_+ \, = \, u_{w_{\alpha}} \, u_{\alpha}(x_i)\\[0.75em]
\! & \! = \! & \! \displaystyle
\Big( \prod_{  \alpha_{j_k} \in \rootbasis \setminus \{ \alpha \} } u_{\alpha_{j_k}}(x_k) \prod_{\scriptsize \begin{array}{c} \beta \in \roots^+ \\ \height(\beta)> 1 \end{array}} u_{\beta} \Big) \, u_{\alpha}(x_i),
\end{array}
\]
where $\overline{w} \, (-\alpha_k) = \alpha_{j_k}$ and $u_{w_{\alpha_i}} \in \unipotent^-_{w_{\alpha_i}}$.
Since the parameter of $u_{\alpha}$ is $x_i$ and the parameter of $u_{-\alpha}$ is an element of $\field(\bxx,\bee,\byy) \setminus \field(\bxx)$, say $1/x$, it follows from \eqref{eq:sl2computation} that the parameter $\widetilde{x}$ of $\widetilde{u}_{\alpha}$ is 
\[
\widetilde{x} \, = \, x_i+x .
\]
Further developing \eqref{eq:importantpart} by inserting $n(\overline{w})^{-1} n(\overline{w})$ between $\widetilde{u}_{\alpha}$ and $b$ we obtain
\[
\begin{array}{l}
\buu(\bxx ) \, n(\overline{w}) \, u_{-\alpha} \, n(w_{\alpha}) \, = \\[0.75em]
\qquad\qquad
\displaystyle
\Big( \prod_{-\alpha_k \in \rootbasis^- \setminus \{ -\alpha_i \}} u_{-\alpha_k}(x_k)  \prod_{\scriptsize \begin{array}{c} \beta \in \roots^- \\ \height(\beta) < -1 \end{array}} u_{\beta} \Big) \, u_{-\alpha_i}(\widetilde{x}) \, n(\overline{w}) \, b \,.
\end{array}
\]
Finally, combining \eqref{eq:fundmatrixlevel} and the last equation together with
\[
u_{-\alpha_i} (\widetilde{x}) \, = \, u_{-\alpha_i}(x_i) \, u_{-\alpha_i}(x)\,,
\]
we obtain for some $\widetilde{b} \in \borel^-$ the decomposition
\[ 
\buu(\bxx) \, n(\overline{w}) \, \btt(\bee) \, \buu(\byy) \, n(w_{\alpha}) \, = \, \buu(\bxx) \, u_{-\alpha_i}(x) \, n(\overline{w}) \, \widetilde{b}\,. 
\]
\end{proof}

\begin{theorem}\label{thm:fixedfieldparabolic}
Let $\parabolic_J$ be a standard parabolic subgroup.
According to Definition~\ref{def:partitions}
the set $J$ uniquely determines the partition
\[ 
I' \cup I'' \, = \, \{i_1, \dots, i_r\} \cup \{ i_{r+1}, \dots, i_l \} 
\]
of $I = \{ 1, \dots, l \}$.
Let  
\[
q \, := \, |\roots^- \setminus \langle \alpha_{i_1},\dots, \alpha_{i_r} \rangle_{\Z-\mathrm{span}}|.
\]
Then there exist elements $\bpp=(p_1,\dots,p_q)$ with $p_i \in \field[\bvv,\bff] \subset \field\{\bvv \}$ such that
\[
\generalext^{\parabolic_J} \, = \, \difffield\langle \bss(\bvv)\rangle ( \bpp ) .
\]
Moreover, we may choose the first $l-r$ entries
of $\bpp$ to be the indeterminates 
$\bvvbase = (v_{i_{r+1}},\dots, v_{i_l})$.
The indeterminates $\bvvext=(v_{i_{1}},\dots, v_{i_r})$ are not fixed by $\parabolic_J$ and $\parabolic_J$ is the largest subgroup of $\group$ fixing the indeterminates $\bvvbase$.
\end{theorem}

\begin{proof}
Generally, if $u_1 \, n(w) \, b_1$ is the Bruhat decomposition of an element of $\group$, then the Bruhat decomposition after multiplication on the right with an element $b_2 \in \borel^-$ is $u_1 \, n(w) \, b_3$ with $b_3 = b_1 \, b_2$.
Hence, the factor $u_1$ does not change.

Let 
\[
b_1 \, n(w) \, b_2 \in \parabolic_J \, = \, \bigcup_{w \in \weyl_J} \borel^- \, n(w) \, \borel^-
\]
with $b_1$, $b_2 \in \borel^-$ and $w \in \weyl_J = \langle w_{\alpha_j} \mid j \in J \rangle$.
Moreover, let $\widetilde{U}$ be the unipotent subgroup of $\unipotent^-$ generated by the root groups 
\[
\unipotent_{\beta_{i_1}}, \, \dots, \, \unipotent_{\beta_{i_r}} \,  
\] 
corresponding to the simple negative roots $\beta_{i_1},\dots,\beta_{i_r}$.
Then the above observation and Lemma~\ref{lem:invariantreflections}, applied to each factor in the product of simple reflections for $w$ separately, imply that
\begin{eqnarray*}
\fm \, b_1 \, n(w) \, b_2 
\! & \! = \! & \! \buu(\bvv, \bff) \, n(\overline{w}) \, \btt(\bexp) \, \buu(\bint)  \, b_1 \, n(w) \, b_2\\
\! & \! = \! & \! \buu(\bvv, \bff) \, n(\overline{w}) \, b_3 \, n(w) \, b_2\\
\! & \! = \! & \! \buu(\bvv,\bff) \, u \, n(\overline{w}) \, b_4 \, b_2\\
\! & \! = \! & \! \buu(\bvv, \bff) \, u \, n(\overline{w}) \, b_5\,,
\end{eqnarray*}
where $b_3 = \btt(\bexp) \, \buu(\bint) \, b_1$ and $u \in \widetilde{U}(\generalext)$ obtained from applying successively Lemma~\ref{lem:invariantreflections} to the factors $n(w_{\alpha_j})$ with $j \in J$ in the product expression for $n(w)$, $b_4 \in \borel^-$ and $b_5 = b_4 \, b_2$. 
Thus the Galois action of $b_1 \, n(w) \, b_2$ on $\bvv$ and $\bff$ is determined by the Bruhat decomposition of the product
\[
\buu(\bvv, \bff) \, u \, .  
\]
Since $\widetilde{U}$ is generated by the root groups corresponding to the negative simple roots $\beta_{i_1},\dots,\beta_{i_r}$, the elements $\bvvbase=(v_{i_{r+1}},\dots,v_{i_l})$ are among the invariants of this action and the elements $\bvvext=(v_{i_1},\dots,v_{i_r})$ are not. First we are going to determine the invariants for the action of $\widetilde{U}(\field)$ on $\unipotent^-(\field)$.
Therefore, let $\field[\unipotent^-]$ be the coordinate ring of $\unipotent^-(\field)$. Since $\unipotent^-(\field)$ is unipotent, the quotient $(\unipotent^-/\widetilde{U})(\field)$ is affine by \cite[II.6, Corollary~6.9 (b)]{Borel}. 
 Moreover $(\unipotent^-/\widetilde{U})(\field)$ is isomorphic to $\field^{q}$ and so the ring of invariants 
\[
\field[\unipotent^-/\widetilde{U}] \, \cong \, \field[\unipotent^-]^{\widetilde{U}(\field)}
\]
is generated by $q$ algebraically independent elements $\widetilde{\bpp}=(\widetilde{p}_1,\dots,\widetilde{p}_q)$. We extend now the field of definition from $\field$ to $\generalext$, that is we consider now the action of $\widetilde{U}(\generalext)$ on 
$\unipotent^-(\generalext)$.
Since $\generalext[\widetilde{U}]\cong \generalext \otimes_{\field} \field[\widetilde{U}]$, we have 
\[
(\field[\unipotent^-] \otimes_{\field} \generalext)^{\widetilde{U}(\generalext)} \, \cong \, \field[\unipotent^-]^{\widetilde{U}(\field)} \otimes_{\field} \generalext\,,
\]
i.e., the invariant ring for the action of $\widetilde{U}(\generalext)$ on $\unipotent^-(\generalext)$ is generated by the 
$q$ algebraically independent polynomials 
\[
\widetilde{p}_1 \otimes 1,\dots,\widetilde{p}_q \otimes 1 \in \field[\unipotent^-]^{\widetilde{U}(\field)}\otimes_{\field} \generalext \, .
\]
For the point $\buu(\bvv,\bff) \in \unipotent^-(\generalext)$, the invariants have constant values on the orbit of $\buu(\bvv,
\bff)$ under $\widetilde{U}(\generalext)$.
Evaluating the above invariants $\widetilde{p}_1,\dots,\widetilde{p}_q$ at 
$\buu(\bvv,\bff)$ we obtain $\bpp=(p_1,\dots,p_q)$ with $p_i \in \field[ \bvv,\bff ] \subset \generalext$ which are 
invariant under the action of $\widetilde{U}(\generalext)$.
Since $\bvvbase$ are generators of the polynomial ring $\field[\bvv,\bff]$ and they are invariant, we can choose the first components of $\bpp$ to be these elements. We  show that $\bpp$ are algebraically independent over 
$\difffield\langle \bss(\bvv)\rangle$. 
The differential field extension $\generalext^{\borel^-}=\difffield\langle \bss(\bvv)\rangle(\bvv,\bff)$ of $\difffield
\langle \bss(\bvv)\rangle$ has transcendence degree $m$ by the Fundamental Theorem of Differential Galois Theory and so the $m$ elements $\bvv,\bff$ 
are algebraically independent over $\difffield\langle \bss(\bvv)\rangle$. Thus $\bvv,\bff$ are algebraically 
independent over $\difffield$ and so the ring homomorphism 
\begin{equation}\label{eqn:isomforgroupaction}
\begin{array}{rcl}
\difffield \otimes_{\field} \field[\unipotent^-] = \difffield \otimes_{\field} \field[\coord, \det(\coord)^{-1}] / 
I_{\unipotent^-} \! & \! \to \! & \! \difffield[\bvv ,\bff]\,,\\[0.5em]
\overline{\coord}_{i,j} := \coord_{i,j} + I_{\unipotent^-} \! & \! \mapsto \! & \! (\buu(\bvv,\bff))_{i,j}
\end{array}
\end{equation}
is a ring isomorphism, where $I_{\unipotent^-}$ is the defining ideal in $\field [\coord, \det(\coord)^{-1}]$ of $\unipotent^-$.
Since 
$\widetilde{p}_1,\dots,\widetilde{p}_q$ are algebraically independent over $\difffield$, their images $p_1,\dots,p_q$ 
are also algebraically independent over $\difffield$.
Assume that there exists 
\[
P(Z_1,\dots,Z_q) \in \difffield \{ \bss(\bvv)\}[Z_1,\dots,Z_q]
\] 
such that $P(p_1,\dots,p_q)=0$. Since the elements $\bss(\bvv)$ are algebraically independent over $\difffield[Z_1,\dots,Z_q]$, the 
coefficients in $\difffield[Z_1,\dots,Z_q]$ of $P(Z_1,\dots,Z_q)$ considered as a polynomial in $\bss(\bvv)$ have to 
vanish when substituting $p_1,\dots,p_q$ for $Z_1,\dots,Z_q$. Since $p_1,\dots ,p_q$ are algebraically independent over 
$\difffield$, these coefficients are the zero polynomials implying that $P(Z_1,\dots,Z_q)$ is also the zero polynomial.

By the Fundamental Theorem of Differential Galois Theory, the transcendence degree of $\generalext^{\parabolic_J}$ over $\difffield\langle \bss(\bvv) \rangle$ is
\[
2m+l- \dim(\parabolic_J) \, = \, 2m+l-(m+l+(m-q)) \, = \, q\,.
\]
Since the transcendence degree of $\difffield\langle \bss(\bvv) \rangle (\bpp)$ over $\difffield\langle \bss(\bvv) \rangle$ is $q$ and $\difffield\langle \bss(\bvv) \rangle (\bpp) \subset \generalext^{\parabolic_J}$,  
we conclude that $\generalext^{\parabolic_J}=\difffield\langle \bss(\bvv) \rangle (\bpp)$ (note that this implies that $\difffield\langle \bss(\bvv) \rangle (\bpp)$ is a differential field, i.e., all derivatives of the elements in 
$\bpp$ are in $\difffield\langle \bss(\bvv) \rangle (\bpp)$).

Let $g \in \group \setminus \parabolic_J$. Then by Theorem~\ref{thm:bruhat2}, formulated for $\borel^-$ instead of $\borel^+$, there is a unique $w \in \weyl \setminus \weyl_J$ such that 
\[
g \in (\unipotent^-_{w})' \, n(w) \, \borel^-\,.
\]
Since $w \notin \weyl_J$, there is at least one simple reflection $w_{\alpha_i}$ appearing in the product of simple reflections for $w$ with $i \notin J$.
Let $\alpha_s \in \Delta$ such that $\overline{w}(-\alpha_s)=\alpha_i$. Since $\overline{w}$ induces a bijection between $\rootbasis$ and $\rootbasis^-$, we conclude that $s \notin I'$ and so $s \in I''$. We conclude with Lemma~\ref{lem:invariantreflections} that the indeterminate $v_s$ is not fixed by $g$ and so $\parabolic_J$ is the largest group fixing $\bvvbase$.
\end{proof}

\begin{remark}\label{rem:correspondencepartitionparabolic}
We have the one-to-one correspondences 
\[
\left\{ \begin{array}{c}
\text{partitions}\\
I = I' \cup I''
\end{array} \right\}
\stackrel{1:1}{\longleftrightarrow}
\left\{ \begin{array}{c}
\text{standard parabolic}\\
\text{subgroups} \, \parabolic_J
\end{array} \right\}
\stackrel{1:1}{\longleftrightarrow}
\left\{ \begin{array}{c}
 \bvvbase \, \text{fixed by} \, \parabolic_J \,  \text{and }    \\
 \bvvext \, \text{not fixed by} \, \parabolic_J
\end{array} \right\}
\]
The first correspondence follows from the definition of the partition $I' \cup I''$ of $I$ and the set $J$. 
The second correspondence is a consequence of Theorem~\ref{thm:fixedfieldparabolic}, since it states that the 
fixed field $\difffield\langle \bss(\bvv) \rangle (\bpp)$ of $\parabolic_J$ contains the differential indeterminates $\bvvbase$ and does not contain the differential indeterminates $\bvvext$. 
\end{remark}

\begin{corollary}\label{cor:fixedfieldgenbyvbase}
We have $\difffield \langle \bss(\bvv)\rangle (\bpp)=\difffield \langle \bss(\bvv), \bvvbase \rangle$.
\end{corollary}

\begin{proof}
According to Theorem~\ref{thm:fixedfieldparabolic} the indeterminates $\bvvbase$ are among the differential polynomials $\bpp$ and since $\difffield \langle \bss(\bvv)\rangle (\bpp)$ is a differential field, also all derivatives of $\bvvbase$ are contained in $\difffield \langle \bss(\bvv)\rangle (\bpp)$.
Thus we have inclusions of differential fields
\[
\difffield \langle \bss(\bvv)\rangle \subset \difffield \langle \bss(\bvv),\bvvbase\rangle \subset \difffield \langle \bss(\bvv)\rangle (\bpp)=\generalext^{\parabolic_J} \subset \generalext.
\]
By the Fundamental Theorem of Differential Galois Theory \cite[Theorem~6.3.8 (1)]{CrespoBook}, there exists a closed subgroup $\widetilde{G}$ of $\group$
such that $\generalext^{\widetilde{G}}=\difffield \langle \bss(\bvv),\bvvbase\rangle$ and $\widetilde{G} \geq \parabolic_J$.
Since $\widetilde{G}$ also fixes $\bvvbase$ and $\parabolic_J$ is according to Theorem~\ref{thm:fixedfieldparabolic} the largest subgroup of $\group$ fixing $\bvvbase$, we conclude that 
$\widetilde{G} = \parabolic_J$ and so $\difffield \langle \bss(\bvv)\rangle (\bpp) = \difffield \langle \bss(\bvv), \bvvbase \rangle$.
\end{proof}

\begin{example}
    Let $\group = \SL_4(\field)$. The root system $\roots$ of $\SL_4$ is of type $A_3$ and $\roots^-$ consists of the six roots 
    \begin{gather*}
    \beta_1 \, := \, -\alpha_1, \quad \beta_2 \, := \, -\alpha_2, \quad
    \beta_3 \, := \, -\alpha_3, \\ \beta_4 \, := \, -\alpha_1-\alpha_2, \quad
    \beta_5 \, := \, -\alpha_2-\alpha_3, \quad \beta_6 \, := \, -\alpha_1-\alpha_2-\alpha_3\,.
    \end{gather*} 
     For six indeterminates $\boldsymbol{x}=(x_1,\dots,x_6)$ the respective parametrized root group elements are
     \begin{gather*}
         u_i(x_i) = E_4 + x_i E_{i+1,i} \quad \text{for} \ i=1,2,3, \\
         u_4(x_4) = E_4 + x_4 E_{3,1}, \ u_5(x_5) = E_4 + x_5 E_{4,2} \quad \text{and} \quad 
         u_6(x_6) = E_4 + x_6 E_{4,1}.
     \end{gather*}
     We consider here the case where $I' = \{ 1,3 \}$ and $I''=\{ 2 \}$, that is $\bvvext=(v_1,v_3)$ and $\bvvbase=(v_2)$.
     In other words, the indeterminate $v_2$ will be fixed by the Galois action of $\parabolic_J$ with $J = \{1,3\}$. The group $\widetilde{U}$ (cf.\ Theorem~\ref{thm:fixedfieldparabolic}) is the product of the root groups $U_{\beta_1}$ and $U_{\beta_3}$. 
  For two new indeterminates $\boldsymbol{y}=(y_1,y_3)$ the action of $\widetilde{U}$ on  $\field[\boldsymbol{x}]$ is described by recomputing the standard decomposition of  
  \[
  \boldsymbol{u}(\boldsymbol{x}) u_{\beta_1}(y_1) u_{\beta_3}(y_3).
  \]
  We find that the ring of invariants $\field[\boldsymbol{x}]^{\widetilde{U}}$ is generated over $\field$ by 
  \begin{equation}\label{eqn:examplefirstInvariants}
x_{2},\  -x_{1}x_{2} + x_{4}, \ x_5, \ -x_1 x_5 + x_3 x_4 + x_6 .
  \end{equation}
  Following the construction presented in \cite{Seiss_Generic} for $A_{\SL_4}(\bss(\bvv))$ and its fundamental matrix $\fm$ we find the following differential polynomials 
  in $\field\{ \bvv\}$: 
  \begin{eqnarray*}
  s_1(\bvv) \! & \! = \! & \!
v_3^2 + v_2^2 - v_1 v_2 - v_2 v_3 + v_1^2 + v_1' + v_2' + v_3'\,,\\[0.2em]
s_2(\bvv) \! & \! = \! & \!
4v_1v_1' - v_2v_1' - 2v_1v_2' + 2v_2v_2' - v_2v_3' - v_1v_2^2 \, +\\[0.2em]
\! & \! \! & \! v_2^2v_3 + v_1^2v_2 - v_2v_3^2 + 2v_1'' + v_2''\,,\\[0.2em]
s_3(\bvv) \! & \! = \! & \!
v_1''' + 2v_1v_1'' - v_1v_2'' + 2(v_1')^2 + 2v_1v_2v_1' - v_2^2v_1' \, +\\[0.2em]
\! & \! \! & \! v_2v_3v_1' - v_3^2v_1' - v_1'v_2' - v_1'v_3' + v_1^2v_2' - 2v_1v_2v_2' \, -\\[0.2em]
\! & \! \! & \! v_1^2v_3' + v_1v_2v_3' + v_1^2v_2v_3 - v_1^2v_3^2 - v_1v_2^2v_3 + v_1v_2v_3^2\,,\\[0.2em]
f_4(\bvv) \! & \! = \! & \!
v_1^2 + v_1'\,,\\[0.2em]
f_5(\bvv) \! & \! = \! & \!
v_2^2 + v_1^2 + v_1' + v_2' - v_2 v_1\,,\\[0.2em]
f_6(\bvv) \! & \! = \! & \!
v_1^3 - v_1^2 v_3 + 3 v_1 v_1' - v_1' v_3 + v_1''\,.
\end{eqnarray*}
 Substituting $(\bvv,\bff)$ for $\bxx$ in the invariants in \eqref{eqn:examplefirstInvariants} we obtain the invariants in $\field[\bvv , \bff]$ under the Galois action of $\parabolic_J$.
We find  
\begin{gather*}
\mathrm{Inv}_1 \, := \, v_2\,, \quad
\mathrm{Inv}_2 \, := \, -v_2 v_1 + v_1^2 + v_1'\,, \quad
\mathrm{Inv}_3 \, := \, v_2^2 + v_1^2 + v_1' + v_2' - v_2 v_1\,,\\
\mathrm{Inv}_4 \, := \, -v_2^2 v_1 + v_1^2 v_2 + 2 v_1 v_1' - v_2'v_1 + v_1''\,.
\end{gather*}
 According to Corollary~\ref{cor:fixedfieldgenbyvbase}, we have $\generalext^{\parabolic_J} = \field\langle \bss(\bvv), v_2 \rangle$.
 We are going to check that $\mathrm{Inv}_i \in \field\langle \bss(\bvv), v_2 \rangle$ for $i=1,\ldots,4$.
 Clearly, $\mathrm{Inv}_1 \in \field\{ \bss(\bvv), v_2 \}$ and we observe that 
 \[
 \mathrm{Inv}_4 \, = \, \frac{1}{2}(s_1(\bvv) - v_2' - v_2^2+v_2(s_2(\bvv) - v_2'' - 2 v_2' v_2)) \in \field\{ \bss(\bvv),v_2 \}\,.
 \]
 For the remaining two invariants consider the three differential polynomials
 \begin{equation*}
 \begin{array}{rcl}
     w_1 \! & \! := \! & \!
     v_2^4 - 2 v_2^2 s_1 + 2v_2' v_2^2 + s_1^2 - 2 s_1 v_2' +(v_2')^2 + 4s_3 - 4 \ \mathrm{Inv}_4' \in \field\{ \bss(\bvv) , v_2 \}, \\[0.5em]
     w_2 \! & \! := \! & \! -2 (s_1 - v_2' - v_2^2)' +2(s_2 - v_2'' - 2 v_2'  v_2) \in \field\{ \bss(\bvv) , v_2 \},  \\[0.5em]
     w_3 \! & \! := \! & \! v_1^2 - v_2 v_1 + v_3 v_2 - v_3^2 + v_1' - v_3' \in \field\{ \bvv \}.
     \end{array}
 \end{equation*}
  The element $w_1$ is actually the square $w_3^2$ and the derivative of $w_3$ is 
  \[
  w_3' \, = \, \frac{1}{2}w_2  -v_2 w_3\,.
  \]
  Differentiating $w_1=w_3^2$ and using the expression for $w_3'$ and $w_1=w_3^2$ we obtain
  \[
  w_1' \, = \, 2 w_3 (\frac{1}{2}w_2  -v_2 w_3) \, = \, w_3w_2 -2v_2 w_3^2 \, = \, w_3w_2 -2v_2 w_1\,,
  \]
  which is equivalent to $w_1'+2v_2 w_1=w_3w_2 $.
  The element $w_1'+2v_2 w_1 \in \field\{\bss(\bvv),v_2\}$ factors in $\field\{ \bvv\}$ as $w_2 w_3$ and, since $w_2 \in \field\{\bss(\bvv),v_2\}$, we conclude that  
  \[
  w_3 \, = \, \frac{w_1'+2v_2 w_1}{w_2} \in \field\langle \bss(\bvv), v_2 \rangle.
  \]
  We obtain that
  \[
  \mathrm{Inv}_2 \, = \, \frac{1}{2}(w_3+(s_1 - v_2' - v_2^2))  \quad \text{and} \quad  \mathrm{Inv}_3 \, = \, \mathrm{Inv}_2 +v_2^2 + v_2' 
  \]
  are elements of $\field\langle \bss(\bvv), v_2 \rangle$. Note that  it is in general not true that the subring of all differential polynomial invariants of $\field\langle \bss(\bvv), v_2 \rangle$ is equal to $\field\{ \bss(\bvv),v_2 \}$. 
\end{example}

\begin{remark}\label{rem:computerepofp_i}
    According to Corollary~\ref{cor:fixedfieldgenbyvbase} the invariants $p_1,\dots,p_q$ in $\field\{ \bvv \}$ have a representation as a differential rational function in $\field\langle \bss(\bvv), \bvvbase \rangle$. We are going to explain how one can compute $p_{i,1}$ and $p_{i,2}$ in $\field \{ \bss(\bvv), \bvvbase  \}$ such that $p_i=p_{i,1}/p_{i,2} $ for $1 \leq i \leq q$. For a differentiation order $d_1\geq0$ and a degree $d_2\geq 1$ one considers all monomials $\mathrm{mon}_1,\dots,\mathrm{mon}_k$
    up to degree $d_2$ (including also the monomial $1$) in the new indeterminates 
    \begin{equation}\label{eqn:computeinvariantsindeterminates}
      \partial^j(\widehat{s}_1), \dots,\partial^j(\widehat{s}_l), \partial^j(\widehat{v}_{i_{r+1}}),\dots, \partial^j(\widehat{v}_{i_{l}}) \quad \text{for} \ 0 \leq j \leq d_1.
    \end{equation} 
    For constant indeterminates $c_{1,1},\dots,c_{1,k}$ and $c_{2,1},\dots,c_{2,k}$ one makes the ansatz
    \[
    \widehat{p}_{i,1} = \sum_{j=1}^k c_{1,j} \mathrm{mon}_j \quad \text{and} \quad  \widehat{p}_{i,2} = \sum_{j=1}^k c_{2,j} \mathrm{mon}_j
    \]
    and substitutes in $\widehat{p}_{i,2} p_i - \widehat{p}_{i,1}=0$ for the variables in \eqref{eqn:computeinvariantsindeterminates}
    the respective elements
    \begin{equation}\label{eqn:computeinvariantselements}
        \partial^j(s_1(\bvv)),\dots,\partial^j(s_l(\bvv)),   \partial^j(v_{i_{r+1}}),\dots, \partial^j(v_{i_{l}}) \quad \text{for} \ 0 \leq j \leq d_1
    \end{equation}
    in $\field \{ \bvv \}$. Comparing coefficients in $\field \{ \bvv \}$, one obtains a linear system in the indeterminates $c_{1,1},\dots,c_{1,k}$ and $c_{2,1},\dots,c_{2,k}$ over $\field$. 
    If this system has a solution in $\field^{2k}$, 
    then one
    substitutes in $\widehat{p}_{i,1}$ and $\widehat{p}_{i,2}$ for the constant indeterminates this solution and for the variables in 
    \eqref{eqn:computeinvariantsindeterminates} the
    elements in \eqref{eqn:computeinvariantselements}
    and obtains $p_{i,1}$ 
    and  $p_{i,2}$ in $\field \{ \bss(\bvv), \bvvbase  \}$ such that $p_i=p_{i,1}/p_{i,2} $. If the linear system has no solution, then one increases $d_1$ and $d_2$ and repeats the computation. Since by Corollary~\ref{cor:fixedfieldgenbyvbase} such a representation of $p_i $ exists, this process has to stop after finitely many iterations.
\end{remark}

\section{The Structure Theorem for Parabolic Subgroups}\label{sec:structureparabolic}
Let $\parabolic_J$ be a proper standard parabolic subgroup of $\group$. Then the fixed field of $\generalext$ under $\parabolic_J(\field)$ is
\[
\generalext^{\parabolic_J} \, = \, \difffield\langle \bss(\bvv),\bvvbase \rangle
\]
as we have seen in the previous section. By the Fundamental Theorem of Differential Galois Theory, the extension 
$\generalext$ of $\difffield\langle \bss(\bvv),\bvvbase \rangle$ is a Picard-Vessiot extension for the normal form equation with differential Galois group $\parabolic_J(\field)$.
Moreover, if
\[
\parabolic_J(\field) \, = \, \levi_J(\field) \ltimes \unirad(\parabolic_J)(\field)
\]
is the standard Levi decomposition of $\parabolic_J$, then $\generalext^{\unirad(\parabolic_J)}$ is a Picard-Vessiot extension of $\difffield\langle \bss(\bvv),\bvvbase \rangle$ with differential Galois group isomorphic to 
\[
\levi_J(\field) \, \cong \, \big( \parabolic_J / \unirad(\parabolic_J) \big)(\field)\,.
\]
In this section we determine generators and a defining equation for $\generalext^{\unirad(\parabolic_J)}$.

Over $\difffield\langle \bss(\bvv),\bvvbase \rangle$ the normal form operator is not irreducible anymore. Let
\begin{equation}\label{eqn:irreduciblefactorization}
L_{\group}(\bss(\bvv),\partial) \, = \, L_{1}(\bss(\bvv),\bvvbase,\partial) \cdots L_{k}(\bss(\bvv),\bvvbase, \partial)
\end{equation}
be an irreducible factorization of $L_{\group}(\bss(\bvv),\partial)$ over $\difffield\langle \bss(\bvv),\bvvbase \rangle$, 
where each factor in this product is of order at least $1$ and monic.
The dependence of $L_{i}(\bss(\bvv),\bvvbase,\partial)$ on $\bvvbase$ indicates the ground field $\difffield\langle \bss(\bvv),\bvvbase \rangle$ for the factorization. If the ground field is clear from the context we shortly write 
\begin{equation}\label{eqn:shortdefofirredfactors}
  L_1(\partial) \, := \, L_{1}(\bss(\bvv),\bvvbase,\partial), \quad \dots, \quad
L_k(\partial) \, := \, L_{k}(\bss(\bvv),\bvvbase,\partial) .
\end{equation}

\begin{definition}\label{def:lclm}
We denote by 
\[
\LCLM(\bss(\bvv),\bvvbase,\partial) \, := \, \LCLM(L_{1}(\bss(\bvv),\bvvbase,\partial), \dots, L_{k}(\bss(\bvv),\bvvbase,\partial)) 
\]  
the \emph{least common left multiple} of the irreducible factors
\[
L_{1}(\bss(\bvv),\bvvbase,\partial), \quad \dots, \quad
    L_{k}(\bss(\bvv),\bvvbase,\partial)
\]
of $L_{\group}(\bss(\bvv),\partial)$ over $\difffield\langle \bss(\bvv),\bvvbase \rangle $. We denote its order by $n_{I''}$.
\end{definition}

The unipotent radical $\unirad(\parabolic_J)$ is generated by those root groups which correspond to the roots in $\roots^- \setminus \leviroots^-$. The following lemma implies that those $\intn_i$ in $\bint=(\intn_1,\dots, \intn_m)$ with index $i$ such that $\beta_i \in \leviroots^-$ are fixed by the action of $\unirad(\parabolic_J)(\field)$.

\begin{lemma}\label{lem:bruhatunipotent}
Let $\beta_j \in \roots^-\setminus \leviroots^-$ and let
$\bxx=(x_1,\dots,x_m)$ and $y$ be indeterminates over $\field$.
Then the coefficients $\tilde{x}_1,\dots,\tilde{x}_m \in \field[\bxx,y]$ in the standard decomposition 
\[
u_{\beta_1}(\tilde{x}_1) \cdots u_{\beta_m}(\tilde{x}_m)
\]
of the product
\begin{equation}\label{eqn:productrootgroups}
\buu(\bxx) \, u_{\beta_j}(y) \, = \, \big(   u_{\beta_1}(x_1) \cdots  u_{\beta_m}(x_m)\big) \, u_{\beta_j}(y)  
\end{equation}
are uniquely determined.
Moreover, they satisfy $\tilde{x}_j\neq x_j$ and $\tilde{x}_i = x_i$ for all $i \in \{ 1,\dots,m \}$ such that $\beta_i \in \leviroots^-$.
\end{lemma}

\begin{proof}
The first part follows from \cite[Theorem~5.3.3 (ii)]{Carter} applied for the set of negative roots $\roots^-$.

The second part involves more work.
From \cite[Theorem~5.2.2]{Carter} we have for two roots $\beta$, $\beta' \in \roots^-$ the exchange formula
\begin{equation}\label{eqn:exchangeformula}
   u_{\beta'}(x') \, u_{\beta}(x) \, = \, u_{\beta}(x) \, u_{\beta'}(x') \prod_{a',a>0}  u_{a'\beta' + a \beta} (c_{a',a,\beta,\beta'}(-x)^{a'}x'^{a})\,,
\end{equation}
where the product is taken over all positive integers $a'$, $a$ such that $a' \beta' + a \beta \in \roots^-$ and where
$c_{a',a,\beta,\beta'} \in \Q$. We apply formula \eqref{eqn:exchangeformula} to the product \eqref{eqn:productrootgroups} until we moved $u_{\beta_j}(y)$ to the $j$-th factor $u_{\beta_j}(x_j)$ of the product $u_{\beta_1}(x_1) \cdots u_{\beta_m}(x_m)$. More precisely, we obtain
\begin{eqnarray}
\nonumber
\buu(\bxx) \, u_{\beta_j}(y) \! & \! = \! & \!
\big( u_{\beta_1}(x_1) \cdots  u_{\beta_m}(x_m)\big) u_{\beta_j}(y)\\
\nonumber
     \! & \! = \! & \! u_{\beta_1}(x_1) \cdots u_{\beta_{j-1}}(x_{j-1}) u_{\beta_j}(x_j)u_{\beta_j}(y)\prod_{i=j+1}^m u_{\beta_i}(x_i) \, \widetilde{u}_i\\ \label{eqn:rewriting2}
     \! & \! = \! & \! u_{\beta_1}(x_1) \cdots u_{\beta_{j-1}}(x_{j-1}) u_{\beta_j}(x_j+ y)\prod_{i=j+1}^m u_{\beta_i}(x_i) \, \widetilde{u}_i \, ,
\end{eqnarray}
where $\tilde{u}_i \in \unipotent^-(\field[\bxx,y])$. Moreover, each $\widetilde{u}_i$ is a product of elements of root groups belonging to roots of $\roots^- \setminus \leviroots^-$, since with $\beta_j \in \roots^- \setminus \leviroots^-$ also
$a' \beta + a \beta_j$
belongs to $\roots^- \setminus \leviroots^-$ for every root $\beta \in \roots^-$ and  $a', \, a > 0$.
Furthermore, the roots $a' \beta + a \beta_j$ are ranked higher than the roots $\beta_1,\dots,\beta_j$. Thus in the product 
\begin{equation}\label{eqn:product1}
\prod_{i=j+1}^m u_{\beta_i}(x_i) \, \widetilde{u}_i
\end{equation}
only elements of the root groups $U_{\beta_{j+1}},\dots,U_{\beta_m}$ can appear as factors. We apply recursively formula \eqref{eqn:exchangeformula} to these factors in increasing order of roots.
The above arguments show that this process creates only new factors which belong to higher ranked roots and so the process of bringing the factors into order stops after finitely many steps and we conclude that the product in \eqref{eqn:product1} is rewritten as 
\[
u_{\beta_{j+1}}(\tilde{x}_{j+1})\cdots u_{\beta_{m}}(\tilde{x}_{m}).
\]
The rewriting process defined by \eqref{eqn:exchangeformula} and \eqref{eqn:rewriting2} determines unique polynomials
\[
\tilde{x}_1 , \ldots, \tilde{x}_m \in \field[\bxx,y]\,.
\]
This proves $\tilde{x}_j=x_j+y \neq x_j$. 

Assume we are in step $k=j+1,\dots,m$ of the recursion, i.e., we have introduced the product
\[
 u_{\beta_1}(x_1) \cdots u_{\beta_{j-1}}(x_{j-1}) \, u_{\beta_j}(x_j+ y) u_{\beta_{j+1}}(\tilde{x}_{j+1}) \cdots u_{\beta_k}(\tilde{x}_k) \, \widetilde{u} \,,
\]
where $\widetilde{u}$ is a product of root group elements belonging to $U_{\beta_{k+1}}, \dots, U_{\beta_m}$.
Every time we applied formula \eqref{eqn:exchangeformula}, the newly created factors are root group elements belonging to roots in $\roots^- \setminus \leviroots^-$.
Thus only these root group elements can appear more than once as factors in the new product.
If $\beta_{k+1} \in \leviroots^-$, then $u_{\beta_{k+1}}(x_{k+1})$ is the only factor in the product $\tilde{u}$ lying in the root group $U_{\beta_{k+1}}$. Swapping $u_{\beta_{k+1}}(x_{k+1})$ to position $k+1$ yields
\[
 u_{\beta_1}(x_1) \cdots u_{\beta_{j-1}}(x_{j-1}) u_{\beta_j}(x_j+ y) u_{\beta_{j+1}}(\tilde{x}_{j+1}) \cdot u_{\beta_k}(\tilde{x}_k)  u_{\beta_{k+1}}(x_{k+1}) \, \widehat{u} \,,
\]
where $\widehat{u}$ is a product of root group elements belonging to $U_{\beta_{k+2}}, \dots, U_{\beta_m}$. Thus, for all $i \in \{ 1,\dots,m \}$ with $\beta_i \in \leviroots^-$ we have $\tilde{x}_i = x_i$.

We further mention that if $\beta_{k+1} \in \roots^- \setminus \leviroots^-$, then there might appear several factors
\[
u_{\beta_{k+1}}(y_1), \quad \dots, \quad u_{\beta_{k+1}}(y_t)
\]
in the product $\tilde{u}$ belonging to the root group $U_{\beta_{k+1}}$, where without loss of generality $y_1=x_{k+1}$ and
$y_2,\dots,y_t \in \field[\bxx,y]$. 
Collecting all these factors in $\tilde{u}$ towards the left, we obtain
\begin{gather*}
    u_{\beta_1}(x_1) \cdots u_{\beta_{j-1}}(x_{j-1}) \, u_{\beta_j}(x_j+ y) \, u_{\beta_{j+1}}(\tilde{x}_{j+1}) \cdots u_{\beta_k}(\tilde{x}_k) \, u_{\beta_{k+1}}(y_1) \cdots u_{\beta_{k+1}}(y_t) \, \widehat{u} \\
    = u_{\beta_1}(x_1) \cdots u_{\beta_{j-1}}(x_{j-1}) \, u_{\beta_j}(x_j+ y) \, u_{\beta_{j+1}}(\tilde{x}_{j+1}) \cdots u_{\beta_k}(\tilde{x}_k) \, u_{\beta_{k+1}}(\tilde{x}_{k+1}) \, \widehat{u} \, ,
\end{gather*}
where $\tilde{x}_{k+1}=y_1+\dots +y_t$ and $\widehat{u}$ is a product of root group elements belonging to $U_{\beta_{k+2}}, \dots, U_{\beta_m}$. Hence, for all $\beta_k$ with $j+1 \leq k \leq m$ and $\beta_k \in \roots^- \setminus \leviroots^-$ we have either
$\tilde{x}_k \neq x_k$ or $\tilde{x}_k = x_k$.
\end{proof}

\begin{theorem}\label{cor:extensionsbyPI}
For a standard parabolic subgroup $\parabolic_J$ we consider its standard Levi decomposition
\[
\parabolic_J \, = \,  \levi_J \ltimes \unirad(\parabolic_J).
\] 
Let $\generalext^{\parabolic_J} = \difffield\langle \bss(\bvv)\rangle ( \bpp )$ be as in Theorem~\ref{thm:fixedfieldparabolic}. Then the following statements hold:
\begin{enumerate}
    \item\label{cor:extensionsbyPI(a)} The general extension field $\generalext$ 
    is a Picard-Vessiot extension of $\difffield\langle \bss(\bvv)\rangle ( \bpp )$ with differential Galois group $\parabolic_J(\field)$ (cf.\ Figure~\ref{fig:koerper} (a)).
    \item\label{cor:extensionsbyPI(b)} 
    We have 
    \[
    \generalext^{\unirad(\parabolic_J)} = \difffield\langle \bss(\bvv)\rangle ( \bpp ) 
    \langle \bexp, \bvvext ,   \intn_{i} \mid \beta_i \in \leviroots^- \rangle.
    \]
    \item\label{cor:extensionsbyPI(c)} 
    The differential field $\generalext^{\unirad(\parabolic_J)}$ 
    is a Picard-Vessiot extension of $\difffield\langle \bss(\bvv)\rangle ( \bpp )$ with differential Galois group isomorphic to $\levi_J$ and is defined by the linear differential equation
\[
\LCLM(\bss(\bvv),\bvvbase,\partial) \, y \, = \, 0
\]
    in the differential indeterminate $y$ over $\difffield\langle \bss(\bvv)\rangle ( \bpp )$ (cf.\ Figure~\ref{fig:koerper} (b)). 
\end{enumerate}
\end{theorem}

\begin{figure}
\centering
\begin{subfigure}[b]{0.27\linewidth}
\begin{tikzpicture}[thick,scale=0.75]
\node[circle, draw, fill=black, inner sep=0pt, minimum width=4pt] (A0) at (0,4) {};
\node[circle, draw, fill=black, inner sep=0pt, minimum width=4pt] (B0) at (0,1) {};
\node[circle, draw, fill=black, inner sep=0pt, minimum width=4pt] (C0) at (0,0) {};
\draw (A0) -- (B0);
\draw (B0) -- (C0);
\node (RA0) [right=0.5em of A0] {$\generalext$};
\node (RC0) [right=0.5em of C0] {$\difffield \langle \bss(\bvv) \rangle$};
\node (RB0) [right=0.5em of B0] {$\generalext^{\parabolic_J}$};
\node (D0) at (0,2.5) {};
\node (LD0) [left=0.5em of D0] {$\parabolic_J$};
\draw [decorate, decoration = {brace,mirror}] (-0.2,3.8) --  (-0.2,1.2);
\end{tikzpicture}
\subcaption*{(a)\label{subfig:3a}}
\end{subfigure}
\begin{subfigure}[b]{0.27\linewidth}
\begin{tikzpicture}[thick,scale=0.75]
\node[circle, draw, fill=black, inner sep=0pt, minimum width=4pt] (A1) at (0,4) {};
\node[circle, draw, fill=black, inner sep=0pt, minimum width=4pt] (B1) at (0,3) {};
\node[circle, draw, fill=black, inner sep=0pt, minimum width=4pt] (C1) at (0,1) {};
\node[circle, draw, fill=black, inner sep=0pt, minimum width=4pt] (D1) at (0,0) {};
\draw (A1) -- (B1);
\draw (B1) -- (C1);
\draw (C1) -- (D1);
\node (RA1) [right=0.5em of A1] {$\generalext$};
\node (RD1) [right=0.5em of D1] {$\difffield \langle \bss(\bvv) \rangle$};
\node (RC1) [right=0.5em of C1] {$\generalext^{\parabolic_J}$};
\node (RB1) [right=0.5em of B1] {$\generalext^{\unirad(\parabolic_J)}$};
\node (E1) at (0,2) {};
\node (F1) at (0,3.5) {};
\node (LE1) [left=0.5em of E1] {$\levi_J\cong$};
\node (LF1) [left=0.5em of F1] {$\unirad(\parabolic_J)$};
\draw [decorate, decoration = {brace,mirror}] (-0.2,2.8) --  (-0.2,1.2);
\draw [decorate, decoration = {brace,mirror}] (-0.2,3.8) --  (-0.2,3.2);
\end{tikzpicture}
\subcaption*{(b)\label{subfig:3b}}
\end{subfigure}

\caption{The different subextensions of $\generalext/ \difffield\langle \bss(\bvv)\rangle$.\label{fig:koerper}}
\end{figure}

\begin{proof}
\ref{cor:extensionsbyPI(a)}\
According to Theorem~\ref{thm:fixedfieldparabolic} the fixed field of $\generalext$ under the action of $\parabolic_J(\field)$ is $\generalext^{\parabolic_J} = \difffield \langle \bss(\bvv) \rangle (\bpp)$ and so by the Fundamental Theorem of Differential Galois Theory $\generalext$ is a Picard-Vessiot extension of $\difffield \langle \bss(\bvv) \rangle (\bpp)$ with differential Galois group $\parabolic_J(\field)$.\\
\ref{cor:extensionsbyPI(b)}\
Multiplying the fundamental matrix 
\[
\fm \, = \, \buu(\bvv,\bff) \, n(\overline{w}) \, \btt(\bexp) \, \buu(\bint)
\]
from the right by an element of $\unipotent^-(\field)$ and determining the Bruhat decomposition of the product only has an effect on the factor $\buu(\bint )$. Thus the parameters $\bvv$, $\bff$ and $\bexp$ are left fixed by $\unipotent^-(\field)$ and so by the unipotent radical $\unirad(\parabolic_J)(\field) \leq \unipotent^-(\field)$, too.
Recall that the unipotent radical $\unirad(\parabolic_J)(\field)$ is generated by those root groups which correspond to the roots in $\roots^-\setminus \leviroots^-$. We apply now Lemma~\ref{lem:bruhatunipotent} successively for each factor $u_j(y_j) \in U_{\beta_j}(\field)$ to
\[
\buu(\bint) \, \prod_{\beta_j \in \roots^- \setminus \leviroots^- } u_j(y_j)
\]
and conclude that $\intn_{i}$ is left fixed by $\unirad(\parabolic_J)$ for every $\beta_i \in \leviroots^-$ and that $\intn_{j}$ is not left fixed by $\unirad(\parabolic_J)$ for every $\beta_j \in \roots^-\setminus \leviroots^-$.
With $\difffield\langle \bss(\bvv)\rangle ( \bpp )=\difffield \langle \bss(\bvv),\bvvbase \rangle$
we conclude that 
\begin{equation}\label{eq:inclusionK}
K \, := \, \difffield\langle \bss(\bvv)\rangle ( \bpp ) 
    \langle \bexp, \bvvext, \intn_{i} \mid \beta_i \in \leviroots^- \rangle \subset \generalext^{\unirad(\parabolic_J)} \, .
\end{equation}

By the Fundamental Theorem of Differential Galois Theory, $\generalext$ is a purely transcendental extension of 
$\difffield\langle \bss(\bvv) \rangle$ and of $\generalext^{\unirad(\parabolic_J)}$ with respective transcendence degrees
\[
\trdeg_{\difffield\langle \bss(\bvv) \rangle}(\generalext) \, = \, 2m+l 
\quad \mbox{and} \quad 
\trdeg_{\generalext^{\unirad(\parabolic_J)}}(\generalext) \, = \, |\roots^- \setminus \leviroots^-| \, .
\]
Then the transcendence degree of $\generalext^{\unirad(\parabolic_J)}$ over $\difffield\langle \bss(\bvv) \rangle$ is
\begin{eqnarray*}
\trdeg_{\difffield\langle \bss(\bvv) \rangle}(\generalext^{\unirad(\parabolic_J)}) & = & \trdeg_{\difffield\langle \bss(\bvv) \rangle}(\generalext) - \trdeg_{\generalext^{\unirad(\parabolic_J)}}(\generalext) \\
& = & 2m+l - |\roots^- \setminus \leviroots^-| \\
& = & m + l + |\leviroots^-|\,,
\end{eqnarray*}
where we used $|\roots^- \setminus \leviroots^-| = m - |\leviroots|$. 
We prove that
\[
\trdeg_{\difffield\langle \bss(\bvv) \rangle}(K) \, = \, m + l + |\leviroots^-| \, .
\]
Since $\difffield\langle \bss(\bvv) \rangle (\bpp)=\difffield\langle \bss(\bvv),\bvvbase \rangle$ according to Corollary~\ref{cor:fixedfieldgenbyvbase}, we have on the one hand that
\[
\difffield\langle \bss(\bvv) \rangle (\bpp)\langle \bvvext \rangle \, = \, \difffield\langle \bvv \rangle\,.
\]
On the other hand, the proofs of \cite[Proposition~7.1]{RobertzSeissNormalForms} and \cite[Corollary~7.2]{RobertzSeissNormalForms} imply that 
\[
\difffield\langle \bss(\bvv) \rangle (v_1,v_1',\dots,v_1^{(d_1-1)},\dots, v_l,v_l',\dots,v_l^{(d_l-1)}) \, = \, \difffield\langle \bvv \rangle 
\]
with $d_1+\dots+d_l=m$ and so combined we have that  
\[
\difffield\langle \bss(\bvv) \rangle (\bpp)\langle \bvvext \rangle \, = \, \difffield\langle \bss(\bvv) \rangle (v_1,v_1',\dots,v_1^{(d_1-1)},\dots, v_l,v_l',\dots,v_l^{(d_l-1)})\,.
\]
We conclude that the $2m+l$ elements 
\begin{equation}\label{eqn:generatorsfixedfieldrad}
v_1, \, v_1', \, \dots, \, v_1^{(d_1-1)}, \, \dots, \, v_l, \, v_l', \, \dots, \,v_l^{(d_l-1)}, \,\bexp , \, \bint
\end{equation}
are algebraically independent over $\difffield\langle \bss(\bvv) \rangle$, since they generate the general extension $\generalext$ of $\difffield\langle \bss(\bvv) \rangle$ with differential Galois group $\group(\field)$ of transcendence degree 
\[
\trdeg_{\difffield\langle \bss(\bvv) \rangle}( \generalext)= \dim(\group) \, = \, 2m+l\,.
\]
Since $K$ is generated over $\difffield\langle \bss(\bvv) \rangle$ as a field by the generators of $\generalext$ in \eqref{eqn:generatorsfixedfieldrad} without the $m - |\leviroots^-|$ generators $\intn_i$ in $\bint=(\intn_1,\dots, \intn_m) $ corresponding to the roots $\beta_i \in \roots^- \setminus \leviroots^-$,
we conclude that
\[
\trdeg_{\difffield\langle \bss(\bvv) \rangle}(K) \, = \, m+l+|\leviroots^-|\,.
\]
Because the transcendence degrees coincide it follows with \eqref{eq:inclusionK} that $K = \generalext^{\unirad(\parabolic_J)}$.\\
\ref{cor:extensionsbyPI(c)} Recall from Definition~\ref{def:lclm} that 
\[
\LCLM(\bss(\bvv),\bvvbase,\partial) 
\]  
is the least common left multiple of the irreducible factors $L_1(\partial),\dots,L_k(\partial)$ 
of $L_{\group}(\bss(\bvv),\partial)$ over $\difffield\langle \bss(\bvv) \rangle (\bpp) = \difffield \langle \bss(\bvv), \bvvbase \rangle$.
We denote by $n_1,\dots,n_k$ the orders of the irreducible factors $L_1(\partial),\dots,L_k(\partial)$ and 
we define the integers
\[
n_k' \, = \, n_k, \ n_{k-1}' \, = \, n_k+n_{k-1}, \  \dots , \ n_2' \, = \, n_k+ \dots + n_2, \ n_1' \, = \, n_k+\dots+n_1\,.
\]
According to the proof of \cite[Proposition~4.2]{SingerCompoint} there exists a basis $y_1,\dots,y_n$ in $\generalext$ of the solution space of 
$L_{\group}(\bss(\bvv),\partial) \, y = 0$ such that the elements 
\begin{equation}\label{eqn:generatorsforLCLM}
\begin{array}{l}
     y_1, \ \dots, \ y_{n_k'}, \\[0.5em] L_k(\partial)y_{n_k'+1}, \ \dots , \ L_k(\partial)y_{n_{k-1}'},\\[0.5em]
     (L_{k-1}(\partial) \circ L_k(\partial))y_{n_{k-1}'+1}, \  \dots , \  (L_{k-1}(\partial) \circ L_k(\partial))y_{n_{k-2}'}, \  \dots, \\[0.5em]
     (L_{2}(\partial) \circ \cdots \circ L_k(\partial))y_{n_{2}'+1}, \ \dots , \  (L_{2}(\partial) \circ \cdots \circ L_k(\partial))y_{n_{1}'}
\end{array}
\end{equation}
span the solution space of the equation  
\[
\LCLM(\bss(\bvv),\bvvbase,\partial) \, y \, = \, 0\,.
\]
It is also shown there that the Picard-Vessiot extension generated as a differential field over $\difffield\langle \bss(\bvv) \rangle (\bpp)$ by the elements in \eqref{eqn:generatorsforLCLM} is equal to the fixed field $\generalext^{\unirad(\parabolic_J)}$.
Hence, by the Fundamental Theorem of Differential Galois Theory $\generalext^{\unirad(\parabolic_J)}$ is a Picard-Vessiot extension of $\difffield\langle \bss(\bvv) \rangle (\bpp)$ with differential Galois group 
\[
(\parabolic_J / \unirad(\parabolic_J))(\field) \, \cong \, \levi_J(\field)\,.
\]
\end{proof}

\begin{remark}\label{rem:basisLCLM}
    A $\field$-basis of the solution space for the least common left multiple can be computed using the $\field$-basis of the solution space of the normal form equation 
    \[
    L_{\group}(\bss(\bvv),\partial) \, y \, = \, 0\,.
    \]
    A $\field$-basis of the latter is formed by the entries $y_1,\dots,y_n$ of the first row of the matrix $\BNFcomp\fm$,
    where $\BNFcomp$ is the matrix we choose to gauge transform 
    $A_{\group}(\bss(\bvv))$ to the companion matrix which corresponds to the normal form equation. 
    The matrix $\BNFcomp$ describes a change of a basis of a differential module for $A_{\group}(\bss(\bvv))$ to a basis defined by a cyclic vector which can be taken from \cite{Seiss}. It is shown in Proposition~\ref{prop:PropertiesOfBcomp} that $\BNFcomp \in \GL_n(\field \{ \bss(\bvv) \})$. 
    Let $(c_{i,j})$ be an $n \times n$ matrix of constant indeterminates and define 
    \[
    (\tilde{y}_1,\dots,\tilde{y}_n)^{tr} \, = \, (c_{i,j}) (y_1,\dots,y_n)^{tr}.
    \]
    Let $n'_k,\dots,n'_1 $ be as in the proof of Theorem~\ref{cor:extensionsbyPI}.
    One determines the first $n_k'$ rows of $(c_{i,j})$ such that 
    $\{ \tilde{y}_1, \, \dots, \, \tilde{y}_{n_k'} \}$ is a basis of the solution space of $L_k(\partial)$.
    Then one computes $c_{i,j}$ with $n_k'+1 \le i \le n_{k-1}'$ and $1 \le j \le n$ such that 
    $$\{ L_k(\partial)\tilde{y}_{n_k'+1},\dots , \,L_k(\partial)\tilde{y}_{n_{k-1}'} \}$$
    is a $\field$-basis of the solution space of $L_{k-1}(\partial)$. Continuing in this way we find bases
    \begin{equation}\label{eqn:basesLCLM}
    \begin{array}{l}
     \{ \tilde{y}_1, \, \dots, \, \tilde{y}_{n_k'} \},   
     \{ L_k(\partial)\tilde{y}_{n_k'+1},\dots , \,L_k(\partial)\tilde{y}_{n_{k-1}'}\}, \\[0.5em]
     \{ (L_{k-1}(\partial) \circ L_k(\partial))\tilde{y}_{n_{k-1}'+1}, \, \dots , (L_{k-1}(\partial) \circ L_k(\partial))\tilde{y}_{n_{k-2}'}\}, \, \dots, \\[0.5em]
    \{ (L_{2}(\partial) \circ \cdots \circ L_k(\partial))\tilde{y}_{n_{2}'+1}, \, \dots , \, (L_{2}(\partial) \circ \cdots \circ L_k(\partial))\tilde{y}_{n_{1}'}\}
    \end{array}
     \end{equation}
     of the solution spaces of the irreducible factors $L_k(\partial),L_{k-1}(\partial)\dots,L_1(\partial)$. All these elements together generate the solution space of the least common left multiple, but need not be linearly independent. 
Using Gaussian elimination we find a basis $y^{I''}_1,\dots,y^{I''}_{n_{I''}}$ for the solution space of    
\begin{equation} \label{eqn:lclmforbasis}
    \LCLM(\bss(\bvv),\bvvbase,\partial) \, y \, = \, 0 .
    \end{equation} 
\end{remark}

We fix a $\field$-basis 
\[
y^{I''}_1,\dots,y^{I''}_{n_{I''}}
\]
in $\generalext^{\unirad(\parabolic_J)}$ of the solution space of the linear differential equation \eqref{eqn:lclmforbasis}.
The next step is to express the parameters $\bvv$, $\bexp$ and the entries $\intn_i$ of $\bint=(\intn_1,\dots,\intn_m)$ with indices $i$ such that $\beta_i \in \leviroots^-$ in differential algebraic terms with respect to $y^{I''}_1,\dots,y^{I''}_{n_{I''}}$. It is an immediate consequence of Theorem~\ref{cor:extensionsbyPI} that this is possible.

\begin{proposition}\label{prop:BruhatParametersofXmainpart}
Let
\[
\field[\group] \, = \, \field[\overline{Y}_{i,j} \mid i,j=1, \dots, n] = \field[Y_{i,j}\mid i,j=1, \dots, n]/I_{\group}
\]
be the coordinate ring of $\group$, where $I_{\group}$ is the defining ideal of $\group$ and $n$ denotes the dimension of the representation of $\group$. Then there exist $e_1,\dots,e_l \in \field[\group]$ and
\[
\boldsymbol{z} \, = \, (\prod_{j=1}^l e_j^{a_{1,j}},\dots, \prod_{j=1}^l e_j^{a_{l,j}}) \quad \text{with} \ a_{i,j} \in \Z
\]
and $\bxx:=(x_1,\dots,x_m)$ and $\bww:=(w_1,\dots,w_m)$ in the localization $\mathcal{M}^{-1}\field[\group]$ of $\field[\group]$ by the multiplicatively closed subset  $\mathcal{M}$ generated by $e_1,\dots,e_l$ such that 
    \[
    \overline{Y} \, = \, \buu(\bxx) \, n(\overline{w}) \, \btt(\boldsymbol{z}) \, \buu(\bww)  
    \]
    is the Bruhat decomposition of $\overline{Y}:=(\overline{Y}_{i,j})$.  
\end{proposition}
\begin{proof}
    For a proof see Propositions~\ref{prop:BruhatParametersofX} and \ref{prop:BruhatParametersofXforG2} in Appendix~\ref{sec:denom}.
\end{proof}

\begin{proposition}\label{cor:exprforexpandint}
Let $y^{I''}_1,\dots,y^{I''}_{n_{I''}}$ be a $\field$-basis of the solution space of the least common left multiple and let $\bZ=(Z_1,\dots,Z_{n_{I''}})$ be differential indeterminates over $\field\langle \bss(\bvv),\bvvbase \rangle$. 
\begin{enumerate}
\item\label{cor:exprforexp}
There exist $l$ differential rational functions
\[
\exprforexp^{I''}_1(\bZ), \quad \dots, \quad \exprforexp^{I''}_l(\bZ)
\]
in $\field\langle \bss(\bvv),\bvvbase \rangle \langle \bZ \rangle$ such that 
\[
\exprforexp^{I''}_i(y^{I''}_1,\dots,y^{I''}_{n_{I''}}) \, = \, \exp_i \, .
\]
In other words, we have
\[
\frac{\exprforexp^{I''}_i(y^{I''}_1,\dots,y^{I''}_{n_{I''}})'}{ \exprforexp^{I''}_i(y^{I''}_1,\dots,y^{I''}_{n_{I''}})} \, = \,g_i(\bvv)\,.
\]
\item\label{cor:exprforv}
There exist $l$ differential rational functions 
\[
V_i^{I''}(\bZ) \in \field\langle \bss(\bvv) ,\bvvbase \rangle \langle \bZ \rangle
\]
such that 
\[
V_i^{I''}(y^{I''}_1,\dots,y^{I''}_{n_{I''}}) \, = \, v_i \, .
\]
\item\label{cor:exprforint}
For each index $i$ such that $\beta_i \in \leviroots^-$ there exist differential rational functions
\[
\exprforint_i^{I''}(\bZ) \in \field \langle \bss(\bvv),\bvvbase \rangle  \langle \bZ \rangle
\]
such that
\[
\exprforint_i^{I''}(y^{I''}_1,\dots,y^{I''}_{n_{I''}}) \, = \, \intn_i \, .
\]
\end{enumerate}
Moreover, all these rational functions can be chosen to be contained in a localization $D_{\exp}^{-1} \field\langle \bss(\bvv),\bvvbase \rangle\{ \bZ \}$ by a multiplicatively closed subset $D_{\exp}$ generated by $l$ differential polynomials 
$$ E_1(\bZ), \ \dots , \
E_l(\bZ) \in \field\langle \bss(\bvv),\bvvbase \rangle\{ \bZ \}$$ having the additional property that their evaluations
\[
E_1(y^{I''}_1, \dots, y^{I''}_{n_{I''}}) \, = \, \expsolass_1, 
\quad \dots,\quad 
E_l(y^{I''}_1,\dots,y^{I''}_{n_{I''}}) \, = \, \expsolass_l 
\]
are the exponential solutions of the associated equations (cf.\ Proposition~\ref{prop:exponentialandRiccati}).
\end{proposition}

\begin{proof}
According to Theorem~\ref{cor:extensionsbyPI} we have that $\exp_j \in \generalext^{\unirad(\parabolic_J)}$ for $1\leq j \leq l$ and that $\intn_i \in \generalext^{\unirad(\parabolic_J)}$ for all indices $i$ such that $\beta_i \in \leviroots^-$.
Moreover, Theorem~\ref{cor:extensionsbyPI} also implies that 
\[
\generalext^{\unirad(\parabolic_J)} \, = \, \field \langle \bss(\bvv),\bvvbase \rangle \langle y^{I''}_1,\dots,y^{I''}_{n_{I''}} \rangle \,. 
\]
This proves \ref{cor:exprforexp} and \ref{cor:exprforint}.
For \ref{cor:exprforv} we use Gaussian elimination to determine $\Q$-linear combinations of $g_1(\bvv),\dots,g_l(\bvv)$ to express $v_1,\dots,v_l$.
In these expressions we substitute for $g_1(\bvv),\dots,g_l(\bvv)$ the differential rational functions
$$\frac{\exprforexp^{I''}_j(\bZ)'}{\exprforexp^{I''}_j(\bZ)} \qquad \text{for} \ j=1,\dots,l$$
and obtain $V_i^{I''}(\bZ)$.

To prove the supplement we consider the differential $\field\langle \bss(\bvv) \rangle$-isomorphism of Picard-Vessiot rings 
\[
\psi\colon \field \langle \bss(\bvv) \rangle \otimes_{\field} \field[\group] \rightarrow  \field \langle \bss(\bvv) \rangle [\fm], \ \overline{Y}_{i,j} \mapsto \fm_{i,j}
\]
for $A_{\group}(\bss(\bvv))$ (cf.\ \cite[Theorem 1.28]{vanderPutSinger} and note that here the torsor is trivial) and extend it to the differential $\field \langle \bss(\bvv) \rangle$-isomorphism of Picard-Vessiot fields
\[
\psi\colon \field \langle \bss(\bvv) \rangle \otimes_{\field} \field(\group) \rightarrow  \field \langle \bss(\bvv) \rangle (\fm) = \generalext, \ \overline{Y}_{i,j} \mapsto \fm_{i,j} \, ,
\]
which we also denote by $\psi$.
By the uniqueness of the Bruhat decomposition the inverse $\psi^{-1}$ maps the Bruhat decomposition of $\fm$ to the Bruhat decomposition of $\overline{Y}$ from Proposition~\ref{prop:BruhatParametersofXmainpart}.
More precisely, we have   
\[
\begin{array}{rcl}
\psi^{-1}(\fm)  \! & \! = \! & \!
\psi^{-1}(\buu(\bvv,\bff) \, n(\overline{w}) \, \btt(\bexp) \, \buu(\bint)) \\[0.5em] \! & \! = \! & \!
\buu(\psi^{-1}(\bvv,\bff)) \, n(\overline{w}) \, \btt(\psi^{-1}(\bexp)) \, \buu(\psi^{-1}(\bint))\\[0.5em]
\! & \! = \! & \! \buu(\bxx) \, n(\overline{w}) \, \btt(\bzz) \, \buu(\bww) = \overline{Y}\,,
\end{array}
\]
from which we conclude that $\psi^{-1}(\bvv)=(x_1,\dots,x_l)$, $\psi^{-1}(\bexp)=\bzz$ and that $\psi^{-1}(\intn_i)=w_i$ fo all $1 \leq i \leq m$  with $\beta_i \in \leviroots^-$. 

We are going to prove that there exist exponents $a_{j,k}$, $b_{j,k}$, $c_{i,k}$ in $\Z_{\geq 0}$ such that
\[
\begin{array}{ll}
\check{z}_j \, := \, \exp_j \prod_{k=1}^l (\expsolass_k)^{a_{j,k}} &
\quad \text{for} \ 1 \leq j \leq l\,, \\[0.5em]
\check{x}_j \, := \, v_j \prod_{k=1}^l (\expsolass_k)^{b_{j,k}} & \quad \text{for} \ 1 \leq j \leq l\,,\\[0.5em]
\check{w}_i \, := \, \intn_i \prod_{k=1}^l (\expsolass_k)^{c_{i,k}} & \quad \text{for} \ 1 \leq i \leq m \ \text{with} \ \beta_i \in \leviroots^-
\end{array}
\]
are in the Picard-Vessiot ring $\field\langle \bss(\bvv)\rangle[\fm]$.
To this end, we refer to the proof of Proposition~\ref{prop:BruhatParametersofX}  and \ref{prop:BruhatParametersofXforG2}, where $e_1,\dots,e_l$ of Proposition~\ref{prop:BruhatParametersofXmainpart} are defined as the images of $\expsolass_1, \dots, \expsolass_l$ under $\psi^{-1}$. 
According to Proposition~\ref{prop:BruhatParametersofXmainpart}, the elements $e_1,\dots,e_l$ are contained in $\field[\group]$ and $z_1,\dots,z_l$ are products of powers of $e_1,\dots,e_l$ with exponents in $\Z$.
Moreover, the same proposition yields that $x_1,\dots,x_l$ as well as those $w_i$ with $1 \leq i \leq m$ such that $\beta_i \in \leviroots^-$ are
in the localization $\mathcal{M}^{-1}\field[\group]$, where  $\mathcal{M}$ is generated by $e_1,\dots,e_l$.
We conclude that we can multiply $z_1,\dots,z_l$ and $x_1,\dots,x_l$ as well as $w_i$ with suitable non-negative powers of $e_1,\dots,e_l$ to obtain elements in $\field[\group]$. Applying now $\psi$ to these products yields the elements $\check{z}_j$, $\check{x}_j$ and $\check{w}_i$ in $\field\langle \bss(\bvv)\rangle[\fm]$ as claimed.

Finally, we are going to prove that the $\check{z}_j$, $\check{x}_j$ and $\check{w}_i$ as well as the $\expsolass_j$ lie in the Picard-Vessiot ring  
\[
\field \langle \bss(\bvv), \bvvbase \rangle \{ y^{I''}_1,\dots,y^{I''}_{n_{I''}} \} \subset \generalext^{\unirad(\parabolic_J)}
\]
for the least common left multiple. This will then show the assertion of the supplement.
Theorem~\ref{cor:extensionsbyPI} implies that the $v_j$, $\exp_j$ and $\intn_i$ are elements of $\generalext^{\unirad(\parabolic_J)}$ and with the $\exp_j$ also the $\expsolass_j$ are in $\generalext^{\unirad(\parabolic_J)}$ by Proposition~\ref{prop:exponentialandRiccati}. Hence, the elements $\check{z}_j$, $\check{x}_j$ and $\check{w}_i$ are in $\generalext^{\unirad(\parabolic_J)}$.
Since $\check{z}_j$, $\check{x}_j$ and $\check{w}_i$  lie in $\field\langle \bss(\bvv)\rangle[\fm]$, they satisfy a linear differential equation over $\field\langle \bss(\bvv)\rangle$ according to \cite[Corollary~1.38]{vanderPutSinger}. Clearly, the $\expsolass_j$ also satisfy a linear differential equation over $\field\langle \bss(\bvv)\rangle$. 
Since $\field\langle \bss(\bvv)\rangle \subset \field\langle \bss(\bvv), \bvvbase \rangle$, the elements $\check{z}_j$, $\check{x}_j$, $\check{w}_i$ and the $\expsolass_j$
trivially satisfy a linear differential equation over $\field\langle \bss(\bvv), \bvvbase \rangle$ and so 
the previous reference implies that they lie in the Picard-Vessiot ring  
$\field \langle \bss(\bvv), \bvvbase \rangle \{ y^{I''}_1,\dots,y^{I''}_{n_{I''}} \}$  
for the least common left multiple. 
\end{proof} 

\begin{remark} \label{rem:introThomas}
We briefly recall that for any differential system
\begin{equation}\label{eq:differentialsystem}
p_1 = 0, \quad \ldots, \quad p_r = 0, \quad q_1 \neq 0, \quad \ldots, \quad q_s \neq 0,
\end{equation}
defined over a differential field $K$ of characteristic zero,
where $p_1, \ldots, p_r, q_1, \ldots, q_s \in K\{ x_1, \ldots, x_m \}$
are differential polynomials, a \emph{Thomas decomposition} can be
computed in finitely many steps, which consists of finitely many
simple differential systems $S_1$, \ldots, $S_k$ such that the
solution set of \eqref{eq:differentialsystem} in formal power series
is the disjoint union
of the solutions set of $S_1$, \ldots, $S_k$. This decomposition depends, in particular, on a chosen ranking on $K\{ x_1, \ldots, x_m \}$.
Using an elimination ranking $x_1, \ldots, x_k \gg x_{k+1}, \ldots, x_m$ produces simple differential systems with the property that those equations that do not
involve $x_1, \ldots, x_k$ generate all equations not involving $x_1, \ldots, x_k$ implied by the system.
\end{remark}

\begin{remark}\label{rem:computeEXPVINT}
  We are going to explain how one can compute the rational functions of Proposition~\ref{cor:exprforexpandint} \ref{cor:exprforexp}, \ref{cor:exprforv} and \ref{cor:exprforint} such that they have the properties stated in the supplement.
  
Let $\BNFcomp$ and $y_1,\dots,y_n$ be as in Remark~\ref{rem:basisLCLM}. The process explained in that remark yields a basis $y^{I''}_1,\dots,y^{I''}_{n_{I''}}$ of the solution space for 
\[
\LCLM(\bss(\bvv),\bvvbase,\partial) \, y \, = \, 0
\]
from $y_1,\dots,y_n$.
  We repeat now the same construction with differential indeterminates $\hat{y}_1,\dots,\hat{y}_n$ over $\field \langle \bss(\bvv),\bvvbase \rangle$  instead of $y_1,\dots,y_n$. 
  More precisely,  let
  \[
  h_1(\hat{y}_1,\dots,\hat{y}_n), \quad \dots, \quad h_n(\hat{y}_1,\dots,\hat{y}_n)
  \]
  from Remark~\ref{rem:basisLCLM} be the linear differential polynomials in $\hat{y}_1,\dots,\hat{y}_n$ with coefficients in $\field\langle \bss(\bvv),\bvvbase \rangle$ such that their evaluations $h_1(y_1,\dots,y_n),\dots,h_n(y_1,\dots,y_n)$ are equal to the basis elements in \eqref{eqn:basesLCLM}. 
  We renumber $h_1(\hat{y}_1,\dots,\hat{y}_n),\dots, h_n(\hat{y}_1,\dots,\hat{y}_n)$ such that the first $n_{I''}$ linear differential polynomials 
 are the ones defining the basis
 \[
 y_1^{I''} = h_1(y_1,\dots,y_n), \quad \dots, \quad y_{n_{I''}}^{I''} = h_{n_{I''}}(y_1,\dots,y_n)
 \]
from Remark~\ref{rem:basisLCLM}.
We define with new differential indeterminates $\hat{y}^{I''}_1,\dots,\hat{y}^{I''}_{n_{I''}}$ over $\field \langle \bss(\bvv), \bvvbase \rangle$ the differential polynomials
\begin{gather*}
\hat{y}^{I''}_1 - h_1(\hat{y}_1,\dots,\hat{y}_n), \quad \ldots, \quad
\hat{y}^{I''}_{n_{I''}}-h_{n_{I''}}(\hat{y}_1,\dots,\hat{y}_n)\\
\in \field\langle \bss(\bvv),\bvvbase \rangle \{ \hat{y}_1,\dots,\hat{y}_n , \hat{y}^{I''}_1,\dots,\hat{y}^{I''}_{n_{I''}}\}.
\end{gather*}
Next we apply the linear change of variables
\begin{equation}\label{eqn:linearchangeofvariables}
\BNFcomp^{-1} \wronski(\hat{y}_1,\dots, \hat{y}_n) \, = \, \coord, \quad \text{where} \ \coord=(\coord_{i,j}),
\end{equation}
to the defining ideal $I_{\group} \unlhd \field[\coord,\det(\coord)^{-1}]=\field[\GL_n]$ of $\group$ and obtain the ideal $$I_{\group}^{\rm comp} \unlhd \field\{ \bss(\bvv)\}[\hat{y}_1,\dots,\hat{y}_1^{(n-1)},\dots, \hat{y}_n, \dots , \hat{y}_n^{(n-1)}] .$$
Moreover, let $\overline{\coord} = (\overline{\coord}_{i,j})$ be the matrix whose entries are the residue classes of $\coord_{i,j}$ in
\[
\field[\coord,\det(\coord)^{-1}]/I_{\group}.
\]
Compute the Bruhat decomposition of $\overline{\coord}$ and obtain for its coefficients $\bxx$, $\bzz$ and $\bww$ rational functions  in $\overline{\coord}_{i,j}$
as in Proposition~\ref{prop:BruhatParametersofXmainpart}.
Recall from the proof of 
Proposition~\ref{cor:exprforexpandint} that $\psi$ maps these rational functions to their corresponding counterpart in the Bruhat decomposition of $\fm$. 
Moreover, the images under $\psi$ of $e_1,\dots,e_l$ and 
the numerators of $x_1,\dots,x_l$, $z_1,\dots,z_l$ and of $w_i$ for $1 \leq i \leq m$ with $\beta_i \in \leviroots^-$  are 
$\expsolass_1, \dots, \expsolass_l$ and $\check{x}_1, \dots,\check{x}_l$, $\check{z}_1,\dots,\check{z}_l$ and $\check{w}_i$ in $\field\langle \bss(\bvv), \bvvbase\rangle \{y_1^{I''}, \dots, y_{n_{I''}}^{I''} \}$ (cf.~the proof of Proposition~\ref{cor:exprforexpandint}).
We are going to explain how one can compute a representation  
of the latter elements 
in $\field\langle \bss(\bvv), \bvvbase\rangle \{y_1^{I''}, \dots, y_{n_{I''}}^{I''} \}$ using their preimages in $\field[\overline{Y}_{i,j}]$. 
To this purpose, let $f$ be one of the above elements in $\field[\overline{Y}_{i,j}]$ and let $\check{f}$ be its corresponding image in $\field\langle \bss(\bvv), \bvvbase\rangle \{y_1^{I''}, \dots, y_{n_{I''}}^{I''} \}$.
We apply the linear change of variables \eqref{eqn:linearchangeofvariables} to $f$ and obtain $\tilde{f} \in \field \langle \bss(\bvv) \rangle \{\hat{y}_1, \dots, \hat{y}_n\}$. We compute now by differential elimination the intersection of the differential ideal generated by the generators of $I_{\group}^{\rm comp}$ and by
\[
\begin{array}{ll}
x - \tilde{f}(\hat{y}_1,\dots,\hat{y}_n)\,, & \\[0.5em]
\hat{y}^{I''}_i - h_i(\hat{y}_1,\dots,\hat{y}_n)\,, & \quad i = 1, \ldots, n_{I''}\,,\\[0.5em]
L_{\group}(\bss(\bvv),\partial) \, \hat{y}_j\,, & \quad j = 1, \ldots, n\,,\\[0.5em]
L_s(\partial) \, \hat{y}^{I''}_i\,, & \quad \text{if} \ \hat{y}^{I''}_i \ \text{is a solution of the factor} \  L_s(\partial), \, s = 1, \ldots, k\,,
\end{array}
\]
with the differential polynomial ring $\field\langle \bss(\bvv),\bvvbase \rangle \{ x, \hat{y}^{I''}_1,\dots,\hat{y}^{I''}_{n_{I''}}\}$.
Using an appropriate elimination ranking as
$\hat{y}_1,\dots,\hat{y}_n \gg x \gg \hat{y}^{I''}_1,\dots,\hat{y}^{I''}_{n_{I''}}$
we obtain for each ideal a differential polynomial of the form 
    \[
    x-d
    \]
    with $d \in \field\langle \bss(\bvv),\bvvbase \rangle \{ \hat{y}^{I''}_1,\dots,\hat{y}^{I''}_{n_{I''}}\}$.
    Note that in general the differential Thomas decomposition returns several simple systems. 
     By substituting the basis elements $y_1,\dots, y_n$, $y_1^{I''},\dots , y_{n_{I''}}^{I''}$ of Remark~\ref{rem:basisLCLM} for $\hat{y}_1,\dots, \hat{y}_n$, $\hat{y}_1^{I''},\dots , \hat{y}_{n_{I''}}^{I''}$   into the equations and inequations of the simple systems and taking one without a contradiction, one finds a valid relation $ x -d=0$. 
     The differential polynomial $d$ defines a differential polynomial in $\field \langle \bss(\bvv), \bvvbase \rangle \{ \bZ \}$ with the property that when we substitute for $\bZ$ the basis elements $y_1^{I''},\dots , y_{n_{I''}}^{I''}$ we obtain $\check{f}$. 
     If $f$ was $e_j$ for $1 \leq j \leq l$, then $d$ defines the differential polynomial $E_j(\bZ)$ in the supplement of Proposition~\ref{cor:exprforexpandint}.
     If $f$ was the numerator of $x_j$, $z_j$ or $w_i$, then dividing the differential polynomial in $\field \langle \bss(\bvv), \bvvbase \rangle \{ \bZ \}$ defined by $d$  by the respective product of powers of $E_1(\bZ), \dots,E_l(\bZ)$ with exponents in $\Z_{\geq 0}$ yields the rational function 
  $\exprforexp^{I''}_j(\bZ)$, $V_j(\bZ)$ or $\exprforint_i^{I''}(\bZ)$
of Proposition~\ref{cor:exprforexpandint} having the properties of the supplement.
\end{remark}

\begin{example}
For $\group = \SL_4$ the normal form operator is
\[
L_{\SL_4}(\bss(\bvv), \partial) \, = \, \partial^4 - s_1(\bvv) \partial^2 - s_2(\bvv) \partial - s_3(\bvv)
\]
with coefficients 
\[
\begin{array}{rcl}
s_1(\bvv) \! & \! = \! & \!
v_3^2 + v_2^2 - v_1 v_2 - v_2 v_3 + v_1^2 + v_1' + v_2' + v_3',\\[0.2em]
s_2(\bvv) \! & \! = \! & \!
4v_1v_1' - v_2v_1' - 2v_1v_2' + 2v_2v_2' - v_2v_3' - v_1v_2^2 \, +\\[0.2em]
\! & \! \! & \! v_2^2v_3 + v_1^2v_2 - v_2v_3^2 + 2v_1'' + v_2'',\\[0.2em]
s_3(\bvv) \! & \! = \! & \!
v_1''' + 2v_1v_1'' - v_1v_2'' + 2(v_1')^2 + 2v_1v_2v_1' - v_2^2v_1' \, +\\[0.2em]
\! & \! \! & \! v_2v_3v_1' - v_3^2v_1' - v_1'v_2' - v_1'v_3' + v_1^2v_2' - 2v_1v_2v_2' \, -\\[0.2em]
\! & \! \! & \! v_1^2v_3' + v_1v_2v_3' + v_1^2v_2v_3 - v_1^2v_3^2 - v_1v_2^2v_3 + v_1v_2v_3^2.
\end{array}
\]
We consider here the case where $I'=\{2,3\}$ and $I''=\{ 1\}$. The longest Weyl group element $\overline{w}$ maps $-\alpha_2$, $-\alpha_3$ to $\alpha_2$, $\alpha_1$ and so we have $J=\{1,2 \}$. Over $\field \langle \bss(\bvv), v_1\rangle$ 
the normal form operator has the irreducible factorization $L_{\SL_4}(\bss(\bvv), \partial)=L_1(\partial) L_2(\partial)$ with 
\[
\begin{array}{l}
L_1(\partial) := \partial^3 + v_1 \partial^2 - (s_1(\bvv) - v_1^2 - 3 v_1') \partial - s_2(\bvv) - v_1 s_1(\bvv) + 3 v_1'' + 5 v_1 v_1' + v_1^3, \\[0.2em]
L_2(\partial):= \partial - v_1 .
\end{array}
\]
Moreover, with $J=\{1,2 \}$ we find that among the coefficients $\bint$ of the Bruhat decomposition of $\fm$ the elements $\intn_3$, $\intn_5$ and $\intn_6$ are not in the fixed field $\generalext^{\unirad(\parabolic_J)}$. 
From the coefficients of the Bruhat decomposition of $\overline{Y}$ we obtain the following rational functions in $\field(\overline{Y}_{i,j})$ which correspond to $\expsolass_1,\expsolass_2,\expsolass_3$, $\bexp$ and $\intn_1$, $\intn_2$, $\intn_4$, respectively:
\[
\begin{array}{rcl}
    e_1 \! & \! = \! & \! \overline{Y}_{1,4}\,, \quad
    e_2 \, = \, \overline{Y}_{1,3} \overline{Y}_{2,4} - \overline{Y}_{1,4} \overline{Y}_{2,3}\,,\\[0.5em]
    e_3 \! & \! = \! & \! \overline{Y}_{1,2} \overline{Y}_{2,3} \overline{Y}_{3,4} - \overline{Y}_{1,2}\overline{Y}_{2,4}\overline{Y}_{3,3} - \overline{Y}_{1,3}\overline{Y}_{2, 2}\overline{Y}_{3,4} + \overline{Y}_{1,3}\overline{Y}_{2,4}\overline{Y}_{3,2} \, +\\[0.2em]
    \! & \! \! & \! \overline{Y}_{1,4}\overline{Y}_{2,2}\overline{Y}_{3,3} - \overline{Y}_{1,4}\overline{Y}_{2,3}\overline{Y}_{3,2}\,,\\[0.5em]
    z_1 \! & \! = \! & \! 1 / e_3\,, \quad
    z_2 \, = \, 1 / e_2\,, \quad
    z_3 \, = \, 1 / e_1\,,\\[0.5em]
    w_1 \! & \! = \! & \! \displaystyle
    \frac{1}{e_3}(\overline{Y}_{1,1}\overline{Y}_{2,3}\overline{Y}_{3,4} - \overline{Y}_{1,1}\overline{Y}_{2,4}\overline{Y}_{3,3} - \overline{Y}_{1,3}\overline{Y}_{2,1}\overline{Y}_{3,4} + \overline{Y}_{1,3}\overline{Y}_{2,4}\overline{Y}_{3,1} \, + \\[0.5em]
    \! & \! \! & \! \overline{Y}_{1,4}\overline{Y}_{2,1}\overline{Y}_{3,3} - \overline{Y}_{1, 4} \overline{Y}_{2,3} \overline{Y}_{3,1})\,,\\[0.5em]
    w_2 \! & \! = \! & \! \displaystyle
    \frac{\overline{Y}_{1, 2} \overline{Y}_{2,4} - \overline{Y}_{1, 4} \overline{Y}_{2, 2}}{e_2}\,, \quad
    w_4 \, = \, \frac{\overline{Y}_{1,1} \overline{Y}_{2, 4} - \overline{Y}_{1,4} \overline{Y}_{2, 1}}{e_2}\,.
\end{array}
\]
Let $f$ be one of the $e_i$ or one of the numerators of
$w_j$ and perform the substitution of variables $\overline{Y} \mapsto \wronski(\hat{y}_1,\dots,\hat{y}_n)$ to it.  
Then the next step is the computation of the differential Thomas decomposition of the differential ideal in 
$$\field \{ \bss(\bvv),v_1 \} \{x,\hat{y}_i,\hat{y}^{I''}_i \mid i=1,2,3,4 \}$$ generated by
\begin{gather*}
\begin{array}{l}
x - \tilde{f}(\hat{y}_1,\dots,\hat{y}_n),  \\[0.5em]
\hat{y}^{I''}_1 - L_2(\partial)\hat{y}_1, \ 
\hat{y}^{I''}_2 - L_2(\partial)\hat{y}_2, \
\hat{y}^{I''}_3 - L_2(\partial)\hat{y}_3, \
\hat{y}^{I''}_4 - \hat{y}_4, \
\\[0.5em]
L_{\group}(\bss(\bvv),\partial) \, \hat{y}_1, \ L_{\group}(\bss(\bvv),\partial) \, \hat{y}_2, L_{\group}(\bss(\bvv),\partial) \, \hat{y}_3, L_{\group}(\bss(\bvv),\partial) \, \hat{y}_4,
\\[0.5em]
L_1(\partial) \, \hat{y}^{I''}_1,  
L_1(\partial) \, \hat{y}^{I''}_2, 
L_1(\partial) \, \hat{y}^{I''}_3,
L_2(\partial) \, \hat{y}^{I''}_4  \quad \text{and} \\[0.5em]
\det(\wronski(\hat{y}_1,\hat{y}_2,\hat{y}_3,\hat{y}_4))-1 .
\end{array}
\end{gather*}
The numbering of the $\hat{y}_i^{I''}$ is chosen according to the numbering of the first row $y_1,\dots,y_4$ in $\fm$. Note that $y_4$ is a solution of $L_2(\partial)$ and that $L_2(\partial) \, y_1$, $L_2(\partial) \, y_2$ and $L_2(\partial) \, y_3$ are solutions of $L_1(\partial)$. Using the elimination ranking defined by
\[
\hat{y}_1 > \hat{y}_2 > \hat{y}_3 > \hat{y}_4 \gg x \gg \hat{y}^{I''}_1 > \hat{y}^{I''}_2 > \hat{y}^{I''}_3 > \hat{y}^{I''}_{4}\,,
\]
the Thomas decompositions for these ideals consist of a single simple differential system. 
In this system we find the linear relation of the form $x-d$ with $d \in \field\{ \bss(\bvv),v_1 \}  \{ \hat{y}_1^{I''},\hat{y}_2^{I''},\hat{y}_3^{I''},\hat{y}_4^{I''} \}$. For the different choices of $f$, we find $d$ as  
\[
\begin{array}{cl}
 \hat{y}_4^{I''}    & \text{if} \, f = e_1\,, \\[0.5em]
  - \hat{y}_3^{I''} \hat{y}_4^{I''}    & \text{if} \, f = e_2\,, \\[0.5em]
   \hat{y}_2^{I''} (\hat{y}_3^{I''})' \hat{y}_4^{I''} - (\hat{y}_2^{I''})' \hat{y}_3^{I''} \hat{y}_4^{I''}   & \text{if} \, f=e_3\,, \\[0.5em]
 \hat{y}_1^{I''} (\hat{y}_3^{I''})' \hat{y}_4^{I''} - (\hat{y}_1^{I''})'  \hat{y}_3^{I''} \hat{y}_4^{I''}    & \text{if} \, f=w_1\,, \\[0.5em]
  - \hat{y}_2^{I''} \hat{y}_4^{I''}   & \text{if} \, f=w_2\,,  \\[0.5em]
  - \hat{y}_1^{I''} \hat{y}_4^{I''}   & \text{if} \, f=w_4\,.
\end{array}
\]
Replacing in the first three differential polynomials the indeterminates $\hat{y}_i^{I''}$ by $Z_i$, we obtain 
$E_1(\bZ)$, $E_2(\bZ)$ and $E_3(\bZ)$. Performing the same replacement in the last three differential polynomials and dividing by $E_3 (\bZ)$, $E_2(\bZ)$ and $E_2(\bZ)$ respectively, we obtain $\exprforint_1(\bZ)$, $\exprforint_2(\bZ)$ and $\exprforint_4(\bZ)$. Moreover, we have 
\[
\exprforexp_1(\bZ) \, = \, \frac{1}{E_3(\bZ)}, \quad
\exprforexp_2(\bZ) \, = \, \frac{1}{E_2(\bZ)}, \quad
\exprforexp_3(\bZ) \, = \, \frac{1}{E_1(\bZ)}
\]
and as $V_1(\bZ)$, $V_2(\bZ)$ and $V_3(\bZ)$ we take the logarithmic derivative of $E_1(\bZ)$, $E_2(\bZ)$ and $E_3(\bZ)$ respectively.
\end{example}

\begin{remark}\label{rem:continueThomas}
Continuing the discussion of Remark~\ref{rem:introThomas},
each simple differential system $S_i$ in a Thomas decomposition
admits an effective membership test to its associated radical differential ideal. More precisely,
let $\mathcal{I}(S_i^{=})$ be the differential ideal of the differential polynomial ring
$K\{ x_1, \ldots, x_m \}$ which is generated by the left hand sides of the equations in $S_i$.
Let $q$ be the product of the initials and separants of the equations in $S_i$ (determined by the ranking).
Then iterated pseudo-reductions of a given  differential polynomial
$p \in K\{ x_1, \ldots, x_m \}$ modulo the left hand sides of the equations in $S_i$ decides whether $p$ belongs to the radical differential ideal $\mathcal{I}(S_i^{=}) : q^{\infty}$.
For more details we refer to \cite{RobertzHabil}.
\end{remark}

\begin{proposition}\label{prop:computationRELnew}
Let $\varphi$ be the differential homomorphism
\[
\begin{array}{rcl}
\varphi\colon \field\{  \bss(\bvv),\bvvbase \} \{ \bZ \} & \to & \field\{ \bss(\bvv),\bvvbase \} \{ y_1^{I''},\dots, y^{I''}_{n_{I''}}\}\,,\\[0.5em]
\bZ=(Z_1,\dots,Z_{n_{I''}}) & \mapsto & (y_1^{I''},\dots, y^{I''}_{n_{I''}})\,.
\end{array}
\]
Then there exist finitely many $\mathrm{REL}_1,\dots,\mathrm{REL}_k \in \field\{  \bss(\bvv),\bvvbase \} \{ \bZ \}$ such that
\[
\ker(\varphi) \, = \, \mathcal{I} : (\mathrm{REL}^{\neq})^{\infty}\,,
\]
where $\mathcal{I}$ is the differential ideal generated by
$\mathrm{REL}_1,\dots,\mathrm{REL}_k$ and where $\mathrm{REL}^{\neq}$ is the product of the initials and separants of these differential polynomials, defined with respect to a chosen
ranking on $\field\{  \bss(\bvv),\bvvbase \} \{ \bZ \}$.
\end{proposition}

\begin{proof}
Since $\field\{ \bss(\bvv),\bvvbase \} \{ y_1^{I''},\dots, y^{I''}_{n_{I''}}\}$ is  contained in $\generalext$, it is an integral domain and so $\ker(\varphi)$ is a prime differential ideal.
 
According to the proof of Theorem~\ref{cor:extensionsbyPI}, a basis 
$y_1^{I''},\dots,y^{I''}_{n_{I''}}$ of the solution space of  
\begin{equation*} 
\LCLM(\bss(\bvv),\bvvbase,\partial) \, y \, = \, 0
\end{equation*}
is obtained by applying the operators 
\[
L_k(\partial), \quad L_{k-1}(\partial) \circ L_k(\partial), \quad \dots, \quad L_{2}(\partial) \circ \cdots \circ L_k(\partial)  \in \field\langle \bss(\bvv),\bvvbase \rangle [\partial]
\]
to the basis elements $y_1,\dots,y_n$ of the solution space of the normal form operator. Let $\bexp^{-1}=(\exp_1^{-1},\dots,\exp_l^{-1})$.
Since $y_1,\dots,y_n$ are elements of
\[
\field\langle \bss(\bvv),\bvvbase \rangle \{ \bvvext , \bexp,\bexp^{-1}, \bint \},  
\]
which is closed under these operators, the basis elements 
$y_1^{I''},\dots,y^{I''}_{n_{I''}}$ are also elements of this differential ring and so
Theorem~\ref{cor:extensionsbyPI} \ref{cor:extensionsbyPI(b)} implies that they are even elements of 
\[
\field\langle \bss(\bvv),\bvvbase \rangle \{ \bvvext , \bexp,\bexp^{-1}, \intn_i \mid \beta_i \in \leviroots^- \}\,.
\]
Hence, we have 
\begin{equation}\label{eqn:PVRingcontainedindiffring}
\begin{array}{l}
    \field \{ \bss(\bvv),\bvvbase \} \{ y_1^{I''},\dots, y^{I''}_{n_{I''}}\}\\[0.5em]
    \qquad \subset \field\langle \bss(\bvv),\bvvbase \rangle \{ \bvvext , \bexp,\bexp^{-1}, \intn_i \mid \beta_i \in \leviroots^- \}\,.
\end{array}
\end{equation}

Let $\widetilde{\bvv} = (\widetilde{v}_1,\dots,\widetilde{v}_l)$, $\widetilde{\bexp}=(\widetilde{\exp}_1,\dots,\widetilde{\exp}_l)$ and $\widetilde{\intn}_i$ with $\beta_i \in \leviroots^-$ be differential indeterminates over $\field\langle \bss(\bvv),\bvvbase \rangle$ and denote by $R$ the differential ring
\[
\field\langle \bss(\bvv),\bvvbase \rangle \{ \widetilde{\bvv} , \widetilde{\bexp},\widetilde{\bexp}^{-1}, \widetilde{\intn}_i \} \, .
\] 
The integrand of $\intn_i$ with index $i$ such that $\beta_i \in \leviroots^-$ is a polynomial expression 
\[
\mathrm{integrand}_i(\bvv,\bexp,\bexp^{-1},\intn_j  )
\]
in the elements $\bvv$, $\bexp$, $\bexp^{-1}$ and $\intn_j$ with indices $j$ such that $\beta_j \in \leviroots^-$ and 
$|\height(\beta_j)|<|\height(\beta_i)|$.
Then the differential ideal $Q$ in $R$ generated by  
\[
\begin{array}{rl}
\widetilde{\exp}'_i \, \widetilde{\exp}^{-1}_i - g_i(\widetilde{\bvv}) & \quad \text{for} \ i=1,\dots,l\,,\\[0.2em]
\widetilde{v}_i - v_i & \quad \text{for} \  v_i \in \bvvbase\,,\\[0.2em]
s_i(\widetilde{\bvv}) - s_i(\bvv) & \quad \text{for} \ i=1,\dots, l\,,\\[0.2em]
\widetilde{\intn}_i' -  \mathrm{integrand}_i(\widetilde{\bvv},\widetilde{\bexp},\widetilde{\bexp}^{-1},\widetilde{\intn}_j) & \quad 
\text{for} \ \beta_i \in \leviroots^-\,,
\end{array}
\]
    is the kernel of the surjective differential $\field\langle \bss(\bvv),\bvvbase \rangle$-homomorphism
    \begin{eqnarray*}
    R & \to & \field\langle \bss(\bvv),\bvvbase \rangle \{ \bvvext , \bexp,\bexp^{-1}, \intn_i \mid \beta_i \in \leviroots^- \}  \\
    (\widetilde{\bvv},\widetilde{\bexp},\widetilde{\intn}_i \mid \beta_i \in \leviroots^-) &\mapsto &(\bvv,\bexp,\intn_i \mid \beta_i \in \leviroots^- ).
    \end{eqnarray*} 
    Thus $Q$ is a prime differential ideal, since $R/Q$ is isomorphic to an integral domain, and we obtain a differential isomorphism 
    \[
    \iota\colon \field\langle \bss(\bvv),\bvvbase \rangle \{ \bvvext , \bexp,\bexp^{-1}, \intn_i \mid \beta_i \in \leviroots^- \} \to R/Q\,.
    \]
    We conclude with \eqref{eqn:PVRingcontainedindiffring} that the problem reduces to computing the kernel of the differential homomorphism
    \[
    \begin{array}{rcl}
    \varphi\colon \field\{ \bss(\bvv),\bvvbase \} \{ \bZ \} & \to & R/Q,\\[0.5em]
    \bZ=(Z_1,\dots,Z_{n_{I''}}) & \mapsto & (\bar{y}_1,\dots, \bar{y}_{n_{I''}})\,,
    \end{array}
    \]
    where $\bar{y}_1,\dots, \bar{y}_{n_{I''}}$ are the images of $y_1^{I''},\dots,y_{n_{I''}}^{I''}$ under $\iota$.
    Let $\widetilde{y}_1,\dots, \widetilde{y}_{n_{I''}}$ be the expressions obtained by replacing 
    in $y_1^{I''},\dots, y_{n_{I''}}^{I''}$ the elements $\bvv$, $\bexp$, $\bexp^{-1}$, $\intn_i$ by $\widetilde{\bvv}$, $\widetilde{\bexp}$, $\widetilde{\bexp}^{-1}$, $\widetilde{\intn}_i$ for $\beta_i \in \leviroots^-$.
    Moreover, let $\widetilde{Q}$ be the differential ideal in 
    \[
    \field\{ \bss(\bvv),\bvvbase \} \{ \bZ, \widetilde{\bvv} , \widetilde{\bexp},\widetilde{\bexp}^{-1}, \widetilde{\intn}_i \} 
    \]
    generated by the generators of $Q$ and the numerators of
    \[
    Z_1 - \widetilde{y}_1, \quad \ldots , \quad Z_{n_{I''}} - \widetilde{y}_{n_{I''}}\,.
    \]
    Since $Q$ is a prime differential ideal and the other generators of $\widetilde{Q}$ are linear in the differential indeterminates $Z_1$, \ldots, $Z_{n_{I''}}$, we conclude that $\widetilde{Q}$ is also prime.
    We compute a Thomas decomposition
    of the differential system defined by $\widetilde{Q}$  
    with respect to an elimination ranking on $\field\{ \bss(\bvv),\bvvbase \} \{ \bZ, \widetilde{\bvv} , \widetilde{\bexp},\widetilde{\bexp}^{-1}, \widetilde{\intn}_i \}$ with
    \[
    (\widetilde{\bvv},\widetilde{\bexp},\widetilde{\bexp}^{-1},\widetilde{\intn}_i \mid \beta_i \in \leviroots^-) \gg (\bss(\bvv),\bvvbase,\bZ)\,,
    \]
    extending a chosen ranking on the second block.
    Since $\widetilde{Q}$ is prime, the resulting Thomas decomposition contains a
    uniquely determined generic simple system (cf.\ \cite[Subsect.~2.2.3]{RobertzHabil}).
    Now $\widetilde{Q}$ is the intersection of the radical differential ideals defined by the simple differential systems of the Thomas decomposition.
    Moreover, each of these differential ideals contains the differential ideal defined by the generic simple system.
    Hence, $\widetilde{Q}$ is equal to the latter ideal.
    Now $\ker(\varphi)$ is the intersection of this ideal with $\field\{ \bss(\bvv),\bvvbase \} \{ \bZ \}$.
    Due to the choice of the elimination ranking, the
    left hand sides $\mathrm{REL}_1,\dots,\mathrm{REL}_k$ of the equations
    in the generic simple system which only involve the indeterminates $\bss(\bvv), \bvvbase$ and $\bZ$ yield differential polynomials as required.
\end{proof}

\section{Reduction of the Normal Form Matrix into the Lie Algebra of a Parabolic Subgroup}\label{sec:intoLieAlgParabolic}

In Sections~\ref{sec:fixedfieldparabolic} and \ref{sec:structureparabolic} we have seen that for any standard parabolic subgroup $\parabolic_J$ the general extension field 
$\generalext$ is a Picard-Vessiot extension of
\[
\generalext^{\parabolic_J} \, = \, \difffield\langle \bss(\bvv), \bvvbase\rangle
\]
for $A_{\group}(\bss(\bvv))$ with differential Galois group $\parabolic_J(\field)$. In this section we will show how to compute a matrix 
$g_1 \in \group(\generalext^{\parabolic_J})$ such that $g_1\fm \in \parabolic_J(\generalext)$ and 
\[
\gauge{g_1}{A_{\group}(\bss(\bvv))} \in \Lie(\parabolic_J)(\generalext^{\parabolic_J}).
\]
This achieves a reduction of the normal form matrix.
Let
\[
\rootbasis'_1 \cup \ldots \cup \rootbasis'_d \, = \, \{\alpha_{i_1},\dots,\alpha_{i_r} \}  
\]
be the unique partition of the subset $\{ \alpha_i \mid i \in I' \} \subset \rootbasis$ of simple roots
such that $\rootbasis'_1,  \dots ,\rootbasis'_d$ are bases of maximal irreducible root subsystems
$\roots'_1, \dots, \roots'_d$ of $\roots$.
Furthermore, we denote the roots of
\[
\roots^-\setminus ( \Phi_1'^- \cup \dots \cup \Phi_d'^- ) \subset \{ \beta_1,\dots, \beta_m \} \, = \, \roots^-
\]
by $\beta_{k_1},\dots,\beta_{k_s}$, where we choose the numbering such that $k_i < k_j$ for $i < j$.  

\begin{lemma}\label{lem:notaroot}
For $i$, $j \in \{1,\dots,d\}$ with $i\neq j$ let $\alpha \in \roots_i'^-$ and $\beta \in \roots_j'^-$. Then $\alpha + \beta$ is not a root of $\roots^-$.
\end{lemma}

\begin{proof}
Assume that $\alpha + \beta$ is a root of $\roots^-$. By \cite[Corollary~10.2]{HumLie}, $\alpha  + \beta$ can be written as a sum of negative simple roots in such a way that each partial sum is a root. Hence, after exchanging the roles of $i$ and $j$ if necessary, there are simple roots $\alpha^i_{1}, \dots, \alpha^i_{u} \in \rootbasis_i'$ occurring in the representation of $\alpha$ as linear combination of simple roots, and there is a simple root $\alpha^j \in \rootbasis'_j$ occurring in $\beta$ such that
\[
-\alpha^i_{1} - \dots - \alpha^i_{u} - \alpha^j
\]
is a root of $\roots^-$. We claim that $\alpha^i_{1}, \dots, \alpha^i_{u}$ are orthogonal to $\alpha^j$. Suppose that there is $\alpha^i \in \{\alpha^i_{1}, \dots, \alpha^i_{u} \}$ such that $\alpha^i$ and $\alpha^j$ are not orthogonal.
Since the irreducibility of $\roots_i'$ implies the irreducibility of $\rootbasis_i'$, we conclude that if we adjoin $\alpha^j$ to $\rootbasis_i'$ it is not possible to write $\{ \alpha^j \} \cup \rootbasis_i'$ as a disjoint union of two sets such that each simple root in one set is orthogonal to each root in the other. 
This means that $\{ \alpha^j \} \cup \rootbasis_i'$ is irreducible, contradicting the fact that $\rootbasis_i'$ was maximally irreducible.
Hence, $\alpha^i_{1}, \dots, \alpha^i_{u}$ are orthogonal to $\alpha^j$ and so $\langle \alpha^i_{1} + \dots + \alpha^i_{u}, \alpha^j \rangle=0$. Then the image of $\alpha^i_{1} + \dots + \alpha^i_{u} + \alpha^j$ under the reflection $\sigma_{\alpha^j}$ is
\[
\begin{array}{rcl}
\sigma_{\alpha^j} (\alpha^i_{1}+\dots + \alpha^i_{u} + \alpha^j)
\! & \! = \! & \! \alpha^i_{1}+\dots + \alpha^i_{u} + \alpha^j - \langle \alpha^i_{1}+ \dots + \alpha^i_{u} + \alpha^j , \alpha^j\rangle \, \alpha^j\\[0.2em] 
\! & \! = \! & \! \alpha^i_{1}+\dots + \alpha^i_{u} + \alpha^j - \langle \alpha^j , \alpha^j \rangle \, \alpha^j\\[0.2em]
\! & \! = \! & \! \alpha^i_{1}+\dots + \alpha^i_{u} - \alpha^j 
\end{array}
\]
by linearity of the first argument of $\langle \cdot, \cdot \rangle$ and $\langle \alpha^j , \alpha^j \rangle = 2$.
Since the simple root $\alpha^j$ is not among the simple roots $\alpha^i_{1}+\dots + \alpha^i_{u}$, 
we obtain a root of $\roots$ which is the sum of simple roots whose coefficients have different signs,
in contradiction to the properties of a basis,
cf.\ \cite[Chapter~10.1]{HumLie}.
\end{proof}

\begin{proposition}\label{prop:transformationintoPJgen}
    Recall from above that
    \[
    \{\beta_{k_1},\dots,\beta_{k_s}\} \, = \, \roots^-\setminus ( \Phi_1'^- \cup \dots \cup \Phi_d'^- ) \, .
    \]
    There exist $x_{k_1},\dots,x_{k_s} \in \generalext^{\parabolic_J}$ such that
    \[
    g_1 \, := \, n(\overline{w})^{-1} u_{\beta_{k_s}}(x_{k_s})\cdots u_{\beta_{k_1}}(x_{k_1})
    \]
    satisfies $g_1\fm \in \parabolic_J(\generalext)$ and 
    $\gauge{g_1}{A_{\group}(\bss(\bvv))} \in \Lie(\parabolic_J)(\generalext^{\parabolic_J})$.
    Moreover, we have 
    \[
    g_1 \, \buu(\bvv,\bff) \, n(\overline{w}) \in  \unipotent_{\leviroots^+}(\generalext) \leq \parabolic_J(\generalext) \, ,
    \]
    where $\buu(\bvv,\bff)$ is the first factor in the Bruhat decomposition of $\fm$ and $\unipotent_{\leviroots^+}$ is the product of root groups corresponding to the roots in $\leviroots^+$.
\end{proposition}

\begin{proof}
Recall that the negative roots $\beta_1,\dots,\beta_m$ are numbered in such a way that $|\height(\beta_i) |\leq |\height(\beta_j)|$ for $i\leq j$ and so the same holds for the roots $\beta_{k_1},\dots,\beta_{k_s}$.
We prove now by induction on $i=1,\dots,s$ that there are elements $x_{k_1},\dots,x_{k_i} \in \generalext^{\parabolic_J}$ such that in the standard decomposition 
\[
u_{\beta_{k_i}}(x_{k_i})\cdots u_{\beta_{k_1}}(x_{k_1}) \, \buu(\bvv,\bff) \, = \, u_{\beta_1}(y_1) \cdots u_{\beta_m}(y_m)
\]
as the product of elements of all root groups $U_{\beta_1}$, \ldots, $U_{\beta_m}$ (in that order) 
the parameters $y_{k_1},\dots, y_{k_i}$ are all zero.
Before we start with the induction, note that by Lemma~\ref{lem:invariantreflections}, for any $g \in \parabolic_J(\field)$ we have 
\begin{equation}\label{eq:actionPJ}
\buu(\bvv,\bff) \, n(\overline{w}) \, \btt(\bexp) \, \buu(\bint) \, g \, = \, \buu(\bvv,\bff) \, u_g \, n(\overline{w}) \, b\,,
\end{equation}
where the matrix $u_g$ is a product of root group elements corresponding to the roots in $\Phi'^-_1\cup \dots \cup \Phi'^-_d$  with parameter values in $\generalext$ and where $b \in \borel^-(\generalext)$. We will also use the  exchange formula (cf.\ \cite[Theorem~5.2.2]{Carter})
\begin{equation}\label{eqn:exchangeformula2}
u_{\beta'}(x') \, u_{\beta}(x) \, = \, u_{\beta}(x) \, u_{\beta'}(x') \prod_{a',a>0}  u_{a' \beta' + a \beta} (c_{a',a,\beta,\beta'}(-x)^{a'}x'^{a})\, 
\end{equation}
for two roots $\beta$, $\beta' \in \roots^-$, 
where the product is taken over all positive integers $a'$, $a $ such that $a' \beta' + a \beta \in \roots^-$ and where
$c_{a',a,\beta,\beta'} \in \Q$.
Let $i=1$. Then $\beta_{k_1}$ is a negative simple root. By \eqref{eq:actionPJ} the parameter value $\tilde{y}_{k_1}$ in the standard decomposition
\[
\buu(\bvv,\bff) \, u_g \, = \, u_{\beta_1}(\tilde{y}_1) \cdots u_{\beta_m}(\tilde{y}_m)
\]
of $\buu(\bvv,\bff) \, u_g$ is $(\bvv,\bff)_{k_1}$ for every $g \in \parabolic_J$ showing that $(\bvv,\bff)_{k_1}$ is in $\generalext^{\parabolic_J}$, where $(\bvv,\bff)_{k_1}$ means the $k_1$-th entry in the tuple $(\bvv,\bff)$.
We apply now successively the exchange formula to $u_{\beta_{k_1}}(x_{k_1}) \buu(\bvv,\bff)$ with $x_{k_1} = -(\bvv,\bff)_{k_1}$ until we obtain 
\[
u_{\beta_{k_1}}(x_{k_1}) \buu(\bvv,\bff) \, = \, u_{\beta_1}(y_1) \cdots u_{\beta_m}(y_m)\,,
\]
where $y_{k_1}=0$ and $y_j=(\bvv,\bff)_j$ for all $\beta_j \in  \Phi'^-_1\cup \dots \cup \Phi'^-_d$. 
    
Suppose the induction hypothesis holds for $i-1$, that is there exist elements $x_{k_1},\dots,x_{k_{i-1}} \in \generalext^{\parabolic_J}$ such that in the standard decomposition
\[
u_{\beta_{k_{i-1}}}(x_{k_{i-1}})\cdots u_{\beta_{k_1}}(x_{k_1}) \buu(\bvv,\bff) \, = \, u_{\beta_1}(y_1) \cdots u_{\beta_m}(y_m)
\]
the parameters $y_{k_1},\dots, y_{k_{i-1}}$ are all zero. For $\sigma_g \in \Gal_{\partial}(\generalext/\generalext^{\parabolic_J})$ we compute with \eqref{eq:actionPJ} that
\[
\begin{array}{rcl}
u_{\beta_1}(\sigma_g(y_1)) \cdots u_{\beta_m}(\sigma_g(y_m))
\! & \! = \! & \!
\sigma_g(u_{\beta_1}(y_1) \cdots u_{\beta_m}(y_m))\\[0.5em]
\! & \! = \! & \!
\sigma_g(u_{\beta_{k_{i-1}}}(x_{k_{i-1}})\cdots u_{\beta_{k_1}}(x_{k_1}) \buu(\bvv,\bff))\\[0.5em]
\! & \! = \! & \!
u_{\beta_{k_{i-1}}}(x_{k_{i-1}})\cdots u_{\beta_{k_1}}(x_{k_1}) \sigma_g(\buu(\bvv,\bff))\\[0.5em]
\! & \! = \! & \!
u_{\beta_{k_{i-1}}}(x_{k_{i-1}})\cdots u_{\beta_{k_1}}(x_{k_1}) \buu(\bvv,\bff) u_g\\[0.5em]
\! & \! = \! & \! 
u_{\beta_1}(y_1) \cdots u_{\beta_m}(y_m) \, u_g\,.
\end{array}
\]
The element $u_g$ is a product of root group elements corresponding to the roots in $\Phi'^-_1\cup \dots \cup \Phi'^-_d$ (cf.\ \eqref{eq:actionPJ}) and by Lemma~\ref{lem:notaroot} no intermediate product involves a root group element corresponding to a root in $\roots^- \setminus (\Phi'^-_1\cup \dots \cup \Phi'^-_d)$. 
Let $u_{\beta}$ be the first factor in this product and let 
\[
u_{\beta_1}(y_1) \cdots u_{\beta_m}(y_m) \, u_{\beta} \, = \, u_{\beta_1}(\tilde{y}_1) \cdots u_{\beta_m}(\tilde{y}_m)
\]
be the standard decomposition obtained by applying successively formula~\eqref{eqn:exchangeformula2}.
The parameter values among $y_1,\dots,y_m$ of the root group elements corresponding to all roots in $\roots^- \setminus(\Phi'^-_1\cup \dots \cup \Phi'^-_d)$ with height less than $|\height(\beta_{k_i})|$ are zero by induction hypothesis.
Therefore, applying successively formula~\eqref{eqn:exchangeformula2} only affects the parameter values of the root group elements belonging to the roots in $\Phi'^-_1\cup \dots \cup \Phi'^-_d$ and to the roots in $\roots^- \setminus( \Phi'^-_1\cup \dots \cup \Phi'^-_d)$ with height greater than $|\height(\beta_{k_i})|$ in absolute value.
Using induction on the number of factors and the same reasoning as for the first factor shows that $\tilde{y}_{k_i}=y_{k_i}$ in the standard decomposition 
\[
u_{\beta_1}(y_1) \cdots u_{\beta_m}(y_m) \, u_{g} \, = \, u_{\beta_1}(\tilde{y}_1) \cdots u_{\beta_m}(\tilde{y}_m)
\]
and so $y_{k_i} \in \generalext^{\parabolic_J}$. Now we apply formula \eqref{eqn:exchangeformula2} to
\[
u_{\beta_{k_i}}(-y_{k_i}) \, u_{\beta_{k_{i-1}}}(x_{k_{i-1}})\cdots u_{\beta_{k_1}}(x_{k_1}) \, \buu(\bvv,\bff) \, = \, u_{\beta_{k_i}}(-y_{k_i}) \, u_{\beta_1}(y_1) \cdots u_{\beta_m}(y_m).
\]
As above the successive application of formula \eqref{eqn:exchangeformula2} until one reaches the standard decomposition creates only new parameter values in root group elements belonging to roots in $\roots^- \setminus( \Phi'^-_1\cup \dots \cup \Phi'^-_d)$ with height greater than $|\height(\beta_{k_i})|$ and to $\beta_{k_i}$. In the latter case the parameter value becomes zero. This completes the induction. 

The induction statement for $i=s$ implies that in the standard decomposition
\begin{equation}\label{eqn:reductiontoPJ}
u_{\beta_{k_s}}(x_{k_s})\dots u_{\beta_{k_1}}(x_{k_1}) \, \buu(\bvv,\bff) \, = \, u_{\beta_1}(y_1) \cdots u_{\beta_m}(y_m)
\end{equation}
the parameter values $y_{k_1},\dots,y_{k_s}$ are zero and so the right hand side of \eqref{eqn:reductiontoPJ} is a product of root group elements corresponding only to roots in $\Phi'^-_1\cup \dots \cup \Phi'^-_d $. Since the reflection $\sigma_{\overline{w}}$ maps the roots in $ \Phi'^-_1\cup \dots \cup \Phi'^-_d$ to roots in $\leviroots^+$, we conclude with \eqref{eqn:reductiontoPJ} that 
\[
n(\overline{w})^{-1} \, u_{\beta_{k_s}}(x_{k_s})\cdots u_{\beta_{k_1}}(x_{k_1}) \, \buu(\bvv,\bff) \, n(\overline{w}) \in \unipotent_{\leviroots^+}(\generalext) \leq \parabolic_J(\generalext).
\]
Since $\btt(\bexp) \, \buu(\bint)$ is clearly an element of $\borel^-(\generalext)\subset \parabolic_J(\generalext)$, we have that
\[
g_1 \fm \, = \,
n(\overline{w})^{-1} \, u_{\beta_{k_s}}(x_{k_s})\cdots u_{\beta_{k_1}}(x_{k_1}) \, \buu(\bvv,\bff) \, n(\overline{w}) \, \btt(\bexp) \, \buu(\bint) \in \parabolic_J(\generalext) 
\]
with $g_1:=n(\overline{w}) \, u_{\beta_{k_s}}(x_{k_s})\cdots u_{\beta_{k_1}}(x_{k_1}) \in \group(\generalext^{\parabolic_J})$, where we recall that according to the induction statement $x_{k_s},\dots,x_{k_1} \in \generalext^{\parabolic_J}$.
Finally, we conclude with $g_1 \fm \in \parabolic_J(\generalext)$ and Remark~\ref{remark4} that
\begin{gather*}
\dlog (g_1\fm) \, = \, g_1. \dlog(\fm) \, = \, g_1.A_{\group}(\bss(\bvv)) \in \Lie(\parabolic_J)(\generalext^{\parabolic_J}).
\end{gather*}
\end{proof}

\begin{remark}\label{rem:transformationintoPJgen}
    The parameter values 
    $$x_{k_1},\dots,x_{k_s} \in \generalext^{\parabolic_J} = 
    \field\langle \bss(\bvv),\bvvbase \rangle
    $$
    for the root groups $U_{\beta_{k_1}},\dots,U_{\beta_{k_s}}$ of Proposition~\ref{prop:transformationintoPJgen} can be algorithmically determined.
    Indeed, one successively multiplies $\boldsymbol{u}(\bvv,\bff)$ from the left with 
    \[
    u_{\beta_{k_1}}(x_{k_1}),\dots, u_{\beta_{k_s}}(x_{k_s})
    \]
    and applies effectively in each multiplication step multiple times the exchange formula until one obtains a standard decomposition. If 
    \[
    u_{\beta_{k_{i-1}}}(x_{k_{i-1}}) \dots u_{\beta_{k_1}}(x_{k_1}) \boldsymbol{u}(\bvv,\bff) \, = \, u_{\beta_1}(y_1)\cdots 
    u_{\beta_m}(y_m)
    \]
    is the standard decomposition in the $(i-1)$-st multiplication step, then according to the proof of Proposition~\ref{prop:transformationintoPJgen} the parameter value $x_{k_i}$ for the $i$-th multiplication step is simply $-y_{k_i}$.
    Since applying the exchange formula only involves operations in $\field[\bvv,\bff]$ and the $x_{k_1},\dots,x_{k_s}$ are invariant, we conclude that 
    \[
x_{k_1},\dots,x_{k_s} \in \field[\bpp] \, = \, \generalext^{\parabolic_J} \cap \field[\bvv,\bff] \subset \field\{ \bvv \} \, .
\]
Now we use Thomas decomposition to determine  representations of $x_{k_1},\dots,x_{k_s}$ in $\field[\bpp]$.
More precisely, for differential indeterminates $\widehat{\bpp}=(\widehat{p}_1,\dots,\widehat{p}_q)$ compute the normal form of $x_{k_1},\dots,x_{k_s}$ with respect to the differential ideal in $\field \{ \bvv \}$ generated by 
\[
\widehat{p}_1 - p_1, \quad \dots, \quad \widehat{p}_q-p_q \in \field\{ \bvv, \widehat{\bpp} \}
\]
and an elimination ranking $\bvv \gg \widehat{\bpp}$. We obtain expressions for  $x_{k_1},\dots,x_{k_s}$ in $\field [\widehat{\bpp}]$.
According to Remark~\ref{rem:computerepofp_i} we can compute for $1 \leq i \leq q$ elements $p_{1,i}$ and $p_{2,i}$ in $\field\{ \bss(\bvv), \bvvbase \}$ such that $p_i = p_{1,i}/p_{2,i}$.
Thus, if we substitute in $x_{k_1},\dots,x_{k_s}$ for the variables
$\widehat{p}_i$ the rational functions $p_{1,i}/p_{2,i}$, we obtain representations of $x_{k_1},\dots,x_{k_s}$ as rational functions in $\field \langle \bss(\bvv), \bvvbase\rangle$. Note that the denominators of $x_{k_1},\dots,x_{k_s}$ are in the multiplicatively closed subset of $\field \{ \bss(\bvv), \bvvbase\}$ generated by $p_{2,1},\dots,p_{2,q}.$
\end{remark}

We determine the Levi decomposition of the matrix $g_1\fm \in \parabolic_J(\generalext)$, where $g_1$ is as in Proposition~\ref{prop:transformationintoPJgen}. It will be the uniquely determined product 
\[
g_1\fm \, = \, (g_1\fmred) \cdot \fm_{\rm rad}
\]
with $g_1\fmred \in \levi_J(\generalext^{\unirad(\parabolic_J)})$ and $\fm_{\rm rad} \in \unirad(\parabolic_J)(\generalext)$ with entries in $\generalext \setminus \generalext^{\unirad(\parabolic_J)}$, where $\levi_J$ is the standard Levi group of $\parabolic_J$.  
To this end, we denote the roots of $\leviroots^-$ by
\begin{equation}\label{eqn:rootsofpsiminus}
\{ \beta_{j_1},\dots,\beta_{j_k}\} \, = \, \leviroots^-
\end{equation}
and the roots of the complement $\roots^- \setminus \leviroots^-$ by 
\begin{equation}\label{eqn:rootsofradical}
\{ \beta_{j_{k+1}},\dots,\beta_{j_m} \} \, = \, \roots^- \setminus \leviroots^-\,.
\end{equation}
\begin{lemma}\label{lem:separationradicalnew}
Let $\bxx=(x_1,\dots,x_m)$ be indeterminates over $\field$. We have a unique factorization 
\begin{equation}\label{eqn:radicaltotheleft}
u_{\beta_1}(x_1)\cdots u_{\beta_m}(x_m) \, = \, u_{\beta_{j_1}}(y_{j_1}) \cdots u_{\beta_{j_k}}(y_{j_k}) \cdot u_{\beta_{j_{k+1}}}(y_{j_{k+1}}) \cdots u_{\beta_{j_m}}(y_{j_m})
\end{equation}
with $y_{j_1}=x_{j_1}, \dots , y_{j_k}=x_{j_k}$ and $y_{j_{k+1}},\dots,y_{j_m} \in \field[\bxx]\setminus \field[x_{j_1},\dots, x_{j_k}]$.
\end{lemma}

\begin{proof}
 Let $\beta \in \roots^- \setminus \leviroots^-$. Moreover, let $\widetilde{\beta}_1 \in \roots^- \setminus \leviroots^-$ and $\widetilde{\beta}_2 \in \leviroots^-$. If for $a',a>0  $ the sum $a' \beta + a \widetilde{\beta}_i$ is a root, then $a' \beta + a \widetilde{\beta}_i \in \roots^- \setminus \leviroots^-$.
Thus applying the exchange formula \eqref{eqn:exchangeformula2} to a root group element corresponding to a root in $\roots^- \setminus \leviroots^-$ until all factors belonging to roots in $\leviroots^-$ have been moved to the left of it, creates only new factors belonging again to roots in  $\roots^- \setminus \leviroots^-$.
Hence, in each step the parameters of the factors corresponding to the roots of $\leviroots^-$ are unchanged.
The exchange formula \eqref{eqn:exchangeformula2} implies that the parameters of the newly created factors are monomials of degree greater than one.
We conclude that the parameter $y_j$ of a factor in the final product \eqref{eqn:radicaltotheleft} belonging to a root $\beta_j$ with $j \in \{ j_{k+1},\dots,j_m\}$ is the sum of the indeterminate $x_j$ and a polynomial of degree greater than one in $\field[\bxx]$ and so $y_j \in \field[\bxx]\setminus\field[x_{j_1},\dots, x_{j_k}]$.

The uniqueness of the factorization follows from the fact that the product map 
\begin{gather*}
    U_{\beta_{k_1}} \times \dots \times U_{\beta_{k_j}} \times  U_{\beta_{k_{j+1}}} \times  U_{\beta_{k_m}} \to \unipotent^- , \\
    (u_{\beta_{k_1}},\dots, u_{\beta_{k_j}},u_{\beta_{k_{j+1}}},\dots, u_{\beta_{k_m}}) \mapsto u_{\beta_{k_1}} \cdots  u_{\beta_{k_j}} u_{\beta_{k_{j+1}}}\cdots u_{\beta_{k_m}}
\end{gather*}
is an isomorphism of varieties (cf.\ \cite[end of Section~28.5]{HumGroups}). 
\end{proof}

\begin{remark}\label{rem:genericYredYraddecomp}
Applying Lemma~\ref{lem:separationradicalnew} to the matrix $\buu( \bint)$,
we obtain the decomposition
\[
    \buu(\bint) \, = \, u_{\beta_{j_1}}(\intn_{j_1}) \cdots u_{\beta_{j_k}}(\intn_{j_k}) \cdot u_{\beta_{j_{k+1}}}(y_{j_{k+1}}) \cdots u_{\beta_{j_m}}(y_{j_m})
\]
with $y_{j_{k+1}},\dots,y_{j_m} \in \field[\bint] \setminus \field[\intn_{j_1},\dots,\intn_{j_k}]$.
Since $\{\beta_{j_1},\dots,\beta_{j_k} \} = \leviroots^-$, it follows that 
\[
u_{\beta_{j_1}}(\intn_{j_1}) \cdots u_{\beta_{j_k}}(\intn_{j_k}) \in U_{\leviroots^-}
\]
and from Theorem~\ref{cor:extensionsbyPI} \ref{cor:extensionsbyPI(b)} that $\intn_{j_1},\dots ,\intn_{j_k}$ are elements of $\generalext^{\unirad(\parabolic_J)}$. 
 We define
\[
\fmred \, := \, \buu(\bvv,\bff) \, n(\overline{w}) \, \btt(\bexp) \, u_{j_1}(\intn_{j_1}) \cdots u_{j_k}(\intn_{j_k}) \in \group(\generalext^{\unirad(\parabolic_J)}).
\]
Moreover, $g_1\buu(\bvv,\bff) n(\overline{w})\in \unipotent_{\leviroots^+}(\generalext^{\unirad(\parabolic_J)})$ according to Proposition~\ref{prop:transformationintoPJgen} and so, because $\leviroots$ is the root system of the Levi group $\levi_J$, we have  
\[
g_1 \fmred  
\in \levi_J(\generalext^{\unirad(\parabolic_J)}).
\]
Since the unipotent radical $\unirad(\parabolic_J)$ is the direct product of root groups corresponding to the roots in $\roots^- \setminus \leviroots^- = \{ \beta_{j_{k+1}},\dots,\beta_{j_m}\}$, we conclude that  
\[
\fm_{\rm rad} \, := \, u_{\beta_{j_{k+1}}}(y_{j_{k+1}}) \cdots u_{\beta_{j_m}}(y_{j_m}) \in \unirad(\parabolic_J)(\generalext \setminus \generalext^{\unirad(\parabolic_J)}) \, ,
\]
because $\field[\bint] \setminus \field[\intn_{j_1},\dots,\intn_{j_k}] \subset \generalext \setminus \generalext^{\unirad(\parabolic_J})$.
Clearly the Levi decomposition of $g_1\fm$ is
\[
g_1 \fm  \, = \, g_1 \fmred \cdot \fm_{\rm rad}\,.
\]     
\end{remark}

\part{Computing the Galois Group of a Specialized Normal Form}\label{part:III}

\section{A Parabolic Bound for the Differential Galois Group of a Specialization}\label{sec:bound}

Let $\bsq =(\sq_1,\dots,\sq_l )$ be $l$ rational functions in $\difffield=\field(z)$ and let
\begin{equation*}\label{eqn:defsigma0}
 \sigma_0\colon \field \{ \bss(\bvv) \} \to \difffield, \, \bss(\bvv) \mapsto \bsq 
\end{equation*}
be the differential homomorphism   
which specializes the differentially algebraically independent polynomials $\bss(\bvv)$ to $\bsq$.
We consider now the specialized normal form matrix
\begin{equation*}\label{eqn:defspecnormalform}
  \sigma_0 (A_{\group}(\bss(\bvv)))=A_{\group}(\bsq)  .
\end{equation*}

\begin{assumption}\label{assumption1}
We assume that the rational functions $\bsq$ are chosen such that 
no denominator of the coefficients of the associated operators and of their Riccati equations of Definition~\ref{def:associatedRic} specializes under $\sigma_0$ to zero. 
\end{assumption}

\begin{remark}
     Assumption~\ref{assumption1} guarantees that we can also apply $\sigma_0$ to the normal form equation $L_{\group}(\bss(\bvv),\partial) \, y = 0$.
     Indeed, the coefficients of $L_{\group}(\bss(\bvv),\partial) \, y = 0$ are the entries of the last row of the companion matrix $\BNFcomp.A_{\group}(\bss(\bvv))$.
     Since we have $\BNFcomp \in \GL_n(\field\{ \bss(\bvv)\})$ by Proposition~\ref{prop:PropertiesOfBcomp}, it follows that the coefficients of $L_{\group}(\bss(\bvv),\partial) \, y = 0$  are elements of the differential ring $\field\{ \bss(\bvv)\}$.
\end{remark}

In this section we are going to determine a standard parabolic subgroup $\parabolic_J(\field)$ which will contain the differential Galois group $H(\field)$. To this end, we will present an algorithm which determines a partition $I=I'\cup I''$ as introduced in Part~\ref{part:II} for the specialized normal form matrix $A_{\group}(\bsq)$.   
The algorithm uses $\difffield$-rational solutions of the specialized Riccati equations corresponding to the specialized associated operators.

Consider the two differential ideals 
\[
S \, := \, \langle s_1(\bvv)- \sq_1,\dots, s_l(\bvv)- \sq_l \rangle
\]
and
\[
S_{\mathrm{Ric}} \, := \, \langle \Ric{1}{\bsq}{v_1},\dots , \Ric{l}{\bsq}{v_l} \rangle
\]
of $\difffield\{ \bvv \}$, where
\[
\Ric{1}{\bsq}{v_1}, \quad \dots, \quad \Ric{l}{\bsq}{v_l} \in \difffield\{ \bvv \}
\]
are the differential polynomials obtained by replacing in the $i$-th Riccati polynomial 
$\Ric{i}{\bss(\bvv)}{y}$ (cf.\ Definition~\ref{def:associatedRic}) the elements $\bss(\bvv)$ and $y$ by $\bsq$ and $v_i$, respectively.

\begin{lemma}\label{lem:idealSdiffprim}
The differential ideal $\idealspec$ is prime.
\end{lemma}

\begin{proof}
   The proof of \cite[Proposition~7.1 (b)]{RobertzSeissNormalForms} is easily adapted to the case of the ideal $\idealspec$, using the fact that each $s_i(\bvv)$ involves one term which is a derivative of a certain indeterminate $v_j$ with constant coefficient. 
   Thus $\difffield\{ \bvv \} / \idealspec$ is isomorphic to a polynomial ring in finitely many variables. Therefore, $\difffield\{ \bvv \} / \idealspec$ is an integral domain and so $\idealspec$ is prime. 
\end{proof}

The relation between the solutions of the Riccati equations and the differential ideal $\idealspec$ is given by the following proposition.

\begin{proposition}\label{SRicContainedInS}
  The differential ideal $\idealric$ is contained in $\idealspec$, i.e.\ $\idealric \subset \idealspec$.
\end{proposition}

\begin{proof}
We show that the images of $\Ric{1}{\bsq}{v_1},\dots, \Ric{l}{\bsq}{v_l}$ under
\[
\pi\colon \difffield\{\bvv\} \to \difffield\{ \bvv\} / \idealspec, \quad v_i \mapsto v_i+S
\]
are zero. But this follows easily from the fact that substituting the differential indeterminate $v_i$ for $y$ in $\Ric{i}{\bss(\bvv)}{y}$ yields $0$.
Indeed, we then have
\[
\Ric{i}{\bsq}{v_i+S} \, = \,
\Ric{i}{\bss(\bvv)+\idealspec}{v_i + \idealspec} 
\, = \, \Ric{i}{\bss(\bvv)}{v_i} + \idealspec \, = \, 0 + \idealspec\,.
\]
\end{proof}

Roughly speaking Proposition~\ref{SRicContainedInS} means that the variety defined by $S$ is contained in the variety defined by $S_{\mathrm{Ric}}$.
We are interested in common points of both varieties having the property that as many coordinates as possible are in the rational function field $\difffield$.
To this end, we compute all rational solutions in $\difffield$ of the Riccati equations
\[
\Ric{1}{\bsq}{v_1} \, = \, 0\,, \quad \dots, \quad
\Ric{l}{\bsq}{v_l} \, = \, 0
\]
using the known algorithms. More precisely, it is possible to decide algorithmically whether or not a Riccati equation has a solution in $\difffield$ (cf.\ \cite[Proposition~4.9]{vanderPutSinger} and \cite{Beke}, \cite{Schlesinger}).
It is also possible to determine algorithmically all rational solutions of a Riccati equation with coefficients in $\difffield$.
To be more specific, if the Riccati equation has rational solutions, then one can compute 
solutions $u_i \in \field(z)$ with $i=1,\dots,s$ of the Riccati equation and for each $i$ a finite dimensional $\field$-vector space $W_i \subset \field[z]$ containing $\field$ and a basis $\{ w_{i,1} , \dots, w_{i,m_i}\}$ of $W_i$ such that the set 
\begin{equation}\label{eq:rationalSol}
  \bigcup_{i=1}^s \left\{ u_i + \frac{c_{i,1} (w_{i,1})' + \dots + c_{i,m_i} (w_{i,m_i})'}{c_{i,1} w_{i,1} + \dots + c_{i,m_i} w_{i,m_i}} \, \middle| \, c_{i,k} \in \field \ \text{not all zero}
 \right\} 
\end{equation}
is the set of all solutions in $\difffield$ of the Riccati equation (cf.\ \cite[Proposition~4.9]{vanderPutSinger}).  
We are going to combine with Algorithm~\ref{alg:ratv} as many rational solutions $\vq_i \in \difffield$ as possible into a proper differential ideal 
\[ \label{defn:idealSinter}
\idealinter \, := \, \langle s_1(\bvv)- \sq_1,\dots, s_l(\bvv)- \sq_l , v_{i_{r+1}}- \vq_{i_{r+1}},\dots, v_{i_{l}}- \vq_{i_{l}} \rangle \subset \difffield\{ \bvv \}
\]
using the differential Thomas decomposition (cf.\ \cite{RobertzHabil}), i.e., we determine a partition
\[
I \, = \, I' \cup I'' \, = \, \{i_1,\dots,i_r\} \cup \{i_{r+1},\dots, i_l \}
\]
with $r$ minimal such that $\idealinter \unlhd \difffield\{ \bvv \}$ is proper. 
Algorithm~\ref{alg:ratv} starts with the computation of all solutions in $\difffield$ of all $l$ Riccati equations.
This is done because if a specialization of some $v_i$ in a Picard-Vessiot field for $A_{\group}(\bsq)$ lies in $\difffield$, then the corresponding Riccati equation must have at least one solution in $\difffield$.
But not all found solutions in $\difffield$ can be used to define a proper differential ideal $\idealinter$.
It may be possible that one needs to leave out potential $\vq_i \in \difffield$ and needs to keep instead the respective differential indeterminate $v_i$. 
It may also be the case that only particular combinations of rational solutions to the Riccati equations lead to a proper differential ideal by specialization.
To overcome these problems Algorithm~\ref{alg:ratv} simply checks with the differential Thomas decomposition all finitely many possibilities and one takes a consistent one, that is $\idealinter$ is a proper differential ideal,
for which the number of Riccati equations $\Ric{i}{\bsq}{v_i} = 0$ admitting rational solutions is maximal.
In case a Riccati equation $\Ric{i}{\bsq}{v_i} = 0$ has infinitely many rational solutions, that is the vector space $W_i$ is not trivial, we add the additional differential equations $c_{i,k}' = 0$ to $\idealinter$ and with a suitable ranking the differential Thomas decomposition delivers the conditions on the $c_{i,k}$ such that $\idealinter$ is consistent.

\begin{algorithm} 
\DontPrintSemicolon
\KwInput{$\Ric{1}{\bsq}{v_1}, \, \dots, \, \Ric{l}{\bsq}{v_l} $
and $s_1(\bvv)-\overline{s}_1,\dots,s_l(\bvv)-\overline{s}_l$.}
\KwOutput{A partition
\[
I \, = \, \{ 1, \dots, l \} \, = \, I' \cup I'' \, = \, \{ i_1, \dots, i_r\} \cup \{ i_{r+1}, \dots, i_l \}
\]
 with minimal $r$ and rational solutions $\vq_{i_{r+1}},\dots, \vq_{i_l}$ in $\difffield$ of the Riccati equations $\Ric{i_{r+1}}{\bsq}{v_{i_{r+1}}}=0, \, \dots, \, \Ric{i_l}{\bsq}{v_{i_l}}=0 $ such that 
\[
\idealinter \, = \, \langle 
s_1(\bvv)-\overline{s}_1, \quad \ldots, \quad s_l(\bvv)-\overline{s}_l, \quad v_{i_{r+1}} - \vq_{i_{r+1}}, \quad \dots , \quad v_{i_{l}}-\vq_{i_{l}} \rangle
\]
of $\difffield\{ \bvv \}$ is proper.}  
$I'' \gets \emptyset$\\
\For{$k = 1$, $2$, \ldots, $l$} 
{
$\Sol_k \gets \{ v_k \}$\\
\If{$\Ric{k}{\bsq}{v_k}=0$ has rational solutions}
{
$I'' \gets I'' \cup{\{ k \}}$\\
Compute the set of rational solutions
\[
 \RatSol_k \, = \, \left\{ u^k_i + \frac{c^k_{i,1} (w^k_{i,1})' + \dots + c^k_{i,m_i} (w^k_{i,m_i})'}{c^k_{i,1} w^k_{i,1}+ \dots + c^k_{i,m_i} w^k_{i,m_i}} \, \middle| \, i=1,\dots,r_k 
 \right\}
\]
of $\Ric{k}{\bsq}{v_k}=0$, where $u_i^k$, $w_{i,j}^k$ are as in \eqref{eq:rationalSol}, $r_k \in \Z_{\geq 0}$ and $c_{i,j}^k$ are constant indeterminates.\\ 
$\Sol_k \gets \Sol_k \cup \RatSol_k$
}
}
Define $P$ as the list of elements of the Cartesian product $\Sol_1 \times \dots \times \Sol_l$.\\
Sort $P$ in ascending order comparing the number of components of tuples that are differential indeterminates $\bvv$.\\
\While{$P \neq \emptyset$}{
Take the first element $\bsol=(\sol_{1},\dots,\sol_{l})$ of $P$ and remove it from $P$.\\
Consider the differential system $\mathcal{J}$ corresponding to the differential ideal of $\difffield[c_{i,j}^k]\{ \bvv \}$ that is obtained from 
\[  
\langle s_1(\bvv) - \overline{s}_1,\dots, s_l(\bvv) - \overline{s}_l, \, v_1 - b_1^{-1} \sol_1, \dots, v_l - b_l^{-1} \sol_l \rangle \subset \difffield(c_{i,j}^k)\{ \bvv \}
\]
by clearing denominators involving $c_{i,j}^k$ where $b_i$ as in Proposition~\ref{prop:exponentialandRiccati}.\\
Treat the constants $\boldsymbol{c}=(c_{i,j}^k)$ as differential indeterminates with vanishing derivatives, i.e., add the equations 
$\partial c_{i,j}^k=0$ to $\mathcal{J}$.\\
Compute a Thomas decomposition of $\mathcal{J}$ together with the above cleared denominators as inequations with respect to a ranking of $\difffield\{ \bvv, \mathbf{c} \}$ for which $c_{i,j}^k$ are ranked lowest, and let $\mathcal{S}_1$, \ldots, $\mathcal{S}_t$ be the resulting simple differential systems.\\
\If{$t > 0$}{
Remove indices $k$ from $I''$ where $\sol_{k}$ is a variable $v_{k}$.\\
Set $I' = \{ 1, \dots, l \} \setminus I''$.\\
Choose a simple system $\mathcal{S}$ of the Thomas decomposition and 
read off 
the equations and inequations which only involve the indeterminates $\boldsymbol{c}=(c_{i,j}^k)$. Choose a constant point $\boldsymbol{\overline{c}}=(\overline{c}_{i,j}^k)$ satisfying these equations and inequations and determine the explicit solution $\overline{\sol}_{k}$ in 
$\difffield$ defined by $\boldsymbol{\overline{c}}$.  \\
\Return($I'$, $I''$, $\overline{v}_{k}=\overline{\sol}_{k}$ for $k \in I''$)
}
}
\caption{Compute Consistent System\label{alg:ratv}}
\end{algorithm}

\begin{proposition}
Algorithm~\ref{alg:ratv} is correct and terminates. 
\end{proposition}

\begin{proof}
    Since all sets $\Sol_k$ are finite, the list $P$ also consists only of finitely many elements.
    The sorted list $P$ defined in step~8 and 9 contains always as the last element the tuple which  consists of all differential variables $\bvv$. For this tuple the ideal considered in step~12 becomes the differential ideal $\idealspec$ which is prime by Lemma~\ref{lem:idealSdiffprim}, and therefore a Thomas decomposition of the corresponding differential system $\mathcal{J}$ contains at least one simple system $\mathcal{S}$, that is $t>0$. Since the last tuple of the list is the tuple of all differential indeterminates, we obtain $I'' = \emptyset$ and $I' = \{ 1, \dots, l \}$ in step~16 and 17 and there are no rational solutions. Thus the algorithm always terminates. 

    The minimality of $r$ is guaranteed by the sorting of the list $P$. Note that since we added the denominators as inequations to the input of the differential Thomas decomposition the choice of any simple system cannot lead into loosing a rational solution.
    
    Since we choose a specialization $\boldsymbol{\overline{c}}=(\overline{c}_{i,j}^k)$ of the indeterminates $\boldsymbol{c}=(c_{i,j}^k)$ according to the equations and inequations of the simple system $\mathcal{S}$ only involving the $\boldsymbol{c}$, the linear differential polynomials $v_{k}-\overline{\sol}_{k}$ for $k \in I''$ together with the differential polynomials 
    \[
    s_1(\bvv) - \overline{s}_1, \quad \dots, \quad s_l(\bvv) - \overline{s}_l
    \]
    form a proper differential ideal $\idealinter$ of $\difffield\{\bvv\}$.
\end{proof}

We are going to use $\idealinter$ to construct a Picard-Vessiot extension $\overline{\generalext}$ of $\difffield$ for $A_{\group}(\bsq)$. Our respective fundamental matrix $\overline{\fm}$ will have the Bruhat decomposition 
\[
\fmspec \, = \, \buu(\overline{\bvv},\overline{\bff}) \, n(\overline{w}) \, \btt(\overline{\bexp}) \, \buu(\overline{\bint}) 
\]
with $\overline{v}_i \in \difffield$ for $i\in I''$ and $\overline{v}_i \in \overline{\generalext} \setminus \difffield$ for $i \in I'$ which will force the differential Galois group $H$ of $\overline{\generalext}$ over $\difffield$ to be a subgroup of the standard parabolic group $\parabolic_J$. 

\begin{proposition}\label{prop:parabolicbound}
Let $I=I' \cup I''$ and $\idealinter$ be the result of Algorithm~\ref{alg:ratv} for $A_{\group}(\bsq)$ and let $J$ be as in Definition~\ref{def:partitions} for the partition $I'\cup I''$. 
\begin{enumerate}
\item\label{prop:E(a)}
Then there exists a Picard-Vessiot extension $\overline{\generalext}$ of $\difffield$ for
$A_{\group}(\overline{\bss})$ with fundamental matrix
\[
\fmspec \, = \, \buu(\overline{\bvv},\overline{\bff}) \, n(\overline{w}) \, \btt(\overline{\bexp}) \, \buu(\overline{\bint}) 
\]
such that $\overline{\bvv}$ extends the tuple of $\overline{v}_i$ with $i\in I''$ given by Algorithm~\ref{alg:ratv} with certain $\overline{v}_i \in \overline{\generalext} \setminus \difffield$ for $i \in I'$ and where $\overline{\bexp} \in (\overline{\generalext}^{\times})^l$ and $\overline{\bint} \in \overline{\generalext}^m$.
\item\label{prop:E(b)}
The differential Galois group $H$ of $\overline{\generalext}$ over $\difffield$ with representation induced by $\fmspec$ is contained in $\parabolic_J$.  
\item\label{prop:E(c)}
Let $\bexp^{-1}=(\exp_1^{-1},\dots,\exp_l^{-1})$. The map
\begin{gather*}
\sigPV\colon \difffield \{ \bvv \}[\bexp,\bexp^{-1}, \bint ] \to \overline{\generalext}\,,\\
\bvv \mapsto \overline{\bvv}, \, 
\bexp \mapsto \overline{\bexp}, \, 
\bint \mapsto \overline{\bint}
\end{gather*}
is a differential homomorphism and the kernel of $\sigPV|_{\difffield\{\bvv\}}$ contains $\idealinter$. 
\item\label{prop:E(d)}
If $H = \levi \ltimes \unirad(H)$ 
is a Levi decomposition of $H$ for some Levi group $\levi$, then $\levi$ is $\widetilde{\levi}$-irreducible for a Levi group $\widetilde{\levi}$ of $\parabolic_J$ and $\unirad(H)\leq \unirad(\parabolic_J)$.
\end{enumerate}
\end{proposition}

\begin{proof}
\ref{prop:E(a)}\
Since $\idealinter$ is a proper differential ideal of $\difffield\{\bvv\}$, we can choose a maximal differential ideal $\idealmax$ of $\difffield\{ \bvv \}$ containing $\idealinter$. The quotient $\difffield\{ \bvv \}/\idealmax$ is differentially simple and finitely generated over $\difffield$, namely by the residue classes
\[
v_1 + \idealmax, \quad \dots, \quad v_1^{(d_1)} + \idealmax, \quad \dots, \quad v_l + \idealmax, \quad \dots, \quad v_l^{(d_l)} + \idealmax\,.
\]
Hence, by \cite[Lemma~1.17]{vanderPutSinger} the field of fractions 
\[
\Frac(\difffield\{ \bvv\}/\idealmax)
\]
has constants $\field$. Now one uses the standard method to construct a Picard-Vessiot extension $\overline{\generalext}$ of
    $
    \Frac(\difffield\{ \bvv\}/\idealmax)
    $
for the matrix $\ALiou(\bvv+ \idealmax)$ defining the Liouvillian part. More precisely,
we consider the differential ring
\[
\Frac(\difffield\{ \bvv\}/\idealmax)[X_{i,j},\det(X_{i,j})^{-1}] ,
\]
where the derivation on $X_{i,j}$ is defined by $\partial(X_{i,j}) = \ALiou(\bvv+ \idealmax)(X_{i,j})$ (cf.\ Theorem~\ref{thm:RobertzSeissNormalForms}).
Since $\ALiou(\bvv+ \idealmax) \in \mathfrak{b}^-(\Frac(\difffield\{ \bvv\}/\idealmax))$, we can choose a maximal differential ideal $\Imax$ in 
this differential ring such that $\Imax$ contains the defining ideal of $\borel^-$. Then 
\[
\overline{\generalext} \, := \, \Frac(  \Frac(\difffield\{ \bvv\}/\idealmax)[X_{i,j},\det(X_{i,j})^{-1}]/I_{\rm max})
\]
is a Picard-Vessiot extension for $\ALiou(\bvv+ \idealmax)$ and the fundamental matrix $\fmlspec := X + \Imax$ has the property that $\fmlspec \in \borel^-(\overline{\generalext})$. Thus there exist unique elements 
$\overline{\bexp} = (\overline{\exp}_1,\dots,\overline{\exp}_l)$ and $\overline{\bint}=(\overline{\intn}_1,\dots, \overline{\intn}_m)$ in $\overline{\generalext}$ such that $\fmlspec$ has the Bruhat decomposition 
\[
\fmlspec \, = \, \btt(\overline{\bexp}) \, \buu(\overline{\bint}) \in \borel^-( \overline{\generalext})\,.
\]
    It follows from the construction that the matrix 
    \begin{equation}\label{eqn:BruhatDecompspecY}
    \fmspec \, = \, \buu(\overline{\bvv}) \, n( \overline{w}) \, \btt(\overline{\bexp}) \, \buu(\overline{\bint}) 
    \end{equation}
    is a fundamental matrix for $A_{\group}(\overline{\bss})$ with the property that $\overline{v}_i \in \difffield$ for $i\in I''$. 
    Clearly, we have $\difffield (\fmspec_{i,j}) \subset \overline{\generalext}$. 
    By the uniqueness of the Bruhat decomposition the Bruhat decomposition of
    $\overline{\fm} \in \group(\difffield (\overline{\fm}_{i,j}))$ over $\difffield (\overline{\fm}_{i,j})$ has to coincide with the one in \eqref{eqn:BruhatDecompspecY} over $\overline{\generalext} \supset \difffield (\fmspec_{i,j})$.
    Thus $\bvq$, $\overline{\bexp}$ and $\overline{\bint}$ are elements of $\difffield (\overline{\fm}_{i,j})$ and so $\difffield (\overline{\fm}_{i,j}) =  \overline{\generalext}$.
    We conclude that $\overline{\generalext}$ is a Picard-Vessiot extension of $\difffield$ for $A_{\group}(\bsq)$.
    The remaining $\overline{v}_i$ with $i \in I'$ are not in $\difffield$, since otherwise the set $I''$ returned by Algorithm~\ref{alg:ratv} would have been a proper superset.
    
    \ref{prop:E(b)}\
    Let $H$ be the differential Galois group of $\overline{\generalext}$ over $\difffield$ in the representation induced by $\fmspec$. Then $H$ has to fix all $\overline{v}_i \in \difffield$ with $i \in I''$. This means that $H$ has to fix all indeterminates
    $v_i$ with $i \in I''$, since $\idealinter$ contains the linear differential polynomials
    \[
    v_{i_{r+1}} - \overline{v}_{i_{r+1}},\dots ,v_{i_{l}}- \overline{v}_{i_{l}}\, .
    \]
    According to Theorem~\ref{thm:fixedfieldparabolic}, the largest subgroup of $\group$ fixing these elements is the standard parabolic subgroup $\parabolic_J$ and so $H$ is contained in $\parabolic_J$.

    \ref{prop:E(c)}\
    Clearly $\sigmax\colon \difffield\{ \bvv \} \to \Frac(\difffield \{ \bvv\}/\idealmax)$ is a differential homomorphism whose kernel contains $\idealinter$. The elements $\bexp$ and $\bint$ are algebraically independent over $\difffield \langle \bvv \rangle$
    and the elements in $\overline{\bexp}$ are non-zero, i.e., they are invertible. Thus 
    the differential homomorphism $\sigmax$ extends uniquely in the obvious way to a homomorphism of rings
    \[
    \sigPV\colon \difffield\{ \bvv\}[\bexp, \bexp^{-1}, \bint] \to \Frac(\difffield \{ \bvv\}/\idealmax)[\overline{\bexp},\overline{\bexp}^{-1},\overline{\bint}] \, .
    \]
    The derivative of $\exp_i$ is $g_i(\bvv) \exp_i$ (cf.\ Theorem~\ref{thm:RobertzSeissNormalForms}) and by construction the derivative of $\overline{\exp}_i$ is $\sigmax(g_i(\bvv)) \, \overline{\exp}_i$.
    Thus we have $\partial(\sigPV(\exp_i))=\sigPV(\partial(\exp_i))$. Recall that the integrand of $\intn_i$ is a polynomial expression in $\bexp$, $\bexp^{-1}$ and those $\intn_j$ with $|\height(\beta_j)|<|\height(\beta_i)|$.
    Since $\fml = \btt(\bexp) \, \buu(\bint)$ is the fundamental matrix for $\ALiou(\bvv)$ and $\fmlspec = \btt(\overline{\bexp}) \, \buu(\overline{\bint})$ the one for $\sigmax(\ALiou(\bvv))$, the derivative of $\overline{\intn}_i$ is the same polynomial expression but now in $\overline{\bexp}$, $\overline{\bexp}^{-1}$ and $\overline{\intn}_j$. 
    We conclude that $\sigPV$ is a differential homomorphism.
    
    \ref{prop:E(d)}\
    Since $H \leq \parabolic_J$, the Levi group $\levi$ is contained in $\parabolic_J$. We are going to show that $\parabolic_J$ is minimal among the parabolic subgroups of $\group$ with respect to containing $\levi$. Since $\levi$ is reductive it will then follow from 
    Proposition~\ref{prop:LirreducibleForLevi} that $\levi$ is $\widetilde{\levi}$-irreducible for some Levi group $\widetilde{\levi}$ of $\parabolic_J$. 
    Let $P$ be a further parabolic subgroup of $\group$ such that $\levi$ is contained in $P$ and such that $P\leq \parabolic_J$.
There is a unique standard parabolic subgroup $\parabolic_{\widetilde{J}}$ and
    $g \in \group$ such that $g \parabolic g^{-1} = \parabolic_{\widetilde{J}}$. Since by Algorithm~\ref{alg:ratv} the set $J$ has the property that as many indeterminates as possible of $\bvv$ are fixed by $\parabolic_J$, the group $\parabolic_{\widetilde{J}}$ can only fix the same number or a smaller number of indeterminates $\bvv$. We conclude by Theorem~\ref{thm:fixedfieldparabolic} that $|\tilde{J}|\geq |J|$ and so
    \[
    \dim (P) \, = \, \dim(\parabolic_{\widetilde{J}}) \, \geq \, \dim(\parabolic_J)
    \]
    forcing $P = \parabolic_J$, since $P \leq \parabolic_J$.    

     The radical $\unirad(H)$ is a connected unipotent subgroup and so it is contained in a Borel subgroup of $\group$. We are going to show that this Borel subgroup is $\borel^-$.
     It then follows from \cite[Corollary~A, Section~30.3]{HumGroups} and the last sentence of the proof of \cite[Proposition~30.3]{HumGroups} that there is a standard parabolic subgroup $P_{\widetilde{J}}$ such that $\unirad(H) \leq \unirad(P_{\widetilde{J}})$ and $N_{\group}(\unirad(H)) \leq P_{\widetilde{J}}$ which implies that $H\leq P_{\widetilde{J}}$.
     Since $\unirad(H)$ is unipotent the Picard-Vessiot extension $\overline{\generalext}$ of $\overline{\generalext}^{\unirad(H)}$ corresponds to a tower of one-dimensional anti-derivative extensions with Galois groups $\mathbb{G}_{a}$.
     Because every $\overline{\exp}_i \in \overline{\generalext}$ defines an exponential extension of $\difffield\langle \bvq \rangle$ with Galois group a subgroup of $\mathbb{G}_{m}$, we conclude that $\overline{\exp}_i \in \overline{\generalext}^{\unirad(H)}$, implying that its logarithmic derivative $\vq_i \in \overline{\generalext}^{\unirad(H)}$.
     By the construction of $\overline{\generalext}$, we have that
     \[
     \overline{\generalext}^{\borel^-} \, = \, \Frac(\difffield\{ \bvv\}/\idealmax) \, = \, \difffield \langle \bvq \rangle
     \]
     and so we obtain the inclusion $\overline{\generalext}^{\borel^-} \subseteq \overline{\generalext}^{\unirad(H)}$.
     It follows from the Fundamental Theorem of Differential Galois Theory that $\unirad(H) \leq \borel^-$ as claimed.
     Since $H \leq \parabolic_J$ and $H \leq P_{\widetilde{J}}$, we conclude that $H \leq P_{J \cap \widetilde{J}}$. Since $\parabolic_J$ is minimal among the standard parabolic subgroups with respect to containing $H$ and we have $P_{J \cap \widetilde{J}} \leq \parabolic_J$, it follows that $J\cap \widetilde{J} = J$ and so $\parabolic_J \leq \parabolic_{\widetilde{J}}$. Because $\unirad(\parabolic_J)$ and $\unirad(P_{\widetilde{J}})$ are generated by the root groups corresponding to roots in $\roots^- \setminus \Psi^-_J$ and in $\roots^- \setminus \Psi^-_{\widetilde{J}}$ respectively, it follows from $\Psi^-_{J} \subseteq \Psi^-_{\widetilde{J}}$ that $\unirad(\parabolic_J) \geq \unirad(\parabolic_{\widetilde{J}})$, where 
     \[
     \leviroots_J \, := \, \roots \cap \langle \alpha_j \mid j \in J \rangle_{\Z\mathrm{-span}} 
     \]
     and $\Psi_{\widetilde{J}}$ is defined similarly.
     Thus with the above we obtain that $\unirad(\parabolic_J) \geq \unirad(H)$.
\end{proof}

\section{Specializing the Parameters of the Reductive Part}\label{sec:specializingreductivepart}

We extend the specialization
\[
\sigma_0\colon \field \{ \bss(\bvv) \} \to \difffield, \, \bss(\bvv) \mapsto \bsq
\]
from Section~\ref{sec:bound} to a specialization of $\field\{ \bss(\bvv),\bvvbase \}$ using the ideal $\idealinter$.
More precisely, we define   
\[ \label{defn:sigmainter}
\begin{array}{rcl}
\siginter\colon \field\{ \bss(\bvv),\bvvbase \} & \to & \difffield \, \cong \, \difffield \{ \bss(\bvv),\bvvbase \}/\idealinter\,,\\[0.2em]
\bss(\bvv) & \mapsto & \bsq\,,\\[0.2em]
\bvvbase=(v_{i_{r+1}},\dots, v_{i_l}) & \mapsto & \bvvqbase = (\overline{v}_{i_{r+1}},\dots, \overline{v}_{i_l})\,.
\end{array}
\]
\begin{assumption}\label{assumption2}
Recall that in $\field\langle \bss(\bvv),\bvvbase \rangle [\partial]$ the generic operator $L_{\group}(\bss(\bvv),\partial)$ has the irreducible factorization
\[
L_{\group}(\bss(\bvv),\partial) \, = \, L_1(\bss(\bvv),\bvvbase,\partial) \cdots L_k(\bss(\bvv),\bvvbase,\partial)
\]
and we have determined the least common left multiple 
\[
\LCLM(\bss(\bvv),\bvvbase,\partial) \, = \, \LCLM( L_1(\bss(\bvv),\bvvbase,\partial), \dots , L_k(\bss(\bvv),\bvvbase,\partial))  
\] 
of these irreducible factors. In addition to  Assumption~\ref{assumption1} we further assume the following:
\begin{enumerate}
    \item\label{item1:assumption2}
    The denominators of the coefficients in the irreducible factors and in their least common left multiple do not specialize to zero under $\siginter$.
    \item\label{itemNew:assumption2}
    The denominators $p_{2,1},\dots,p_{2,q} \in \field \{ \bss(\bvv),\bvvbase \}$ of the representation of the invariants $p_1,\dots,p_q$ in $\field \langle \bss(\bvv), \bvvbase \rangle$ computed in Remark~\ref{rem:computerepofp_i} do not specialize to zero under $\siginter$.
    \item\label{itemNew2:assumption2}
    The denominators of the coefficients of $E_1(\bZ),\dots,E_l(\bZ)$ from Proposition~\ref{cor:exprforexpandint} do not specialize to zero under $\siginter$.
    \item\label{itemNew3:assumption2}
    The denominators of the coefficients of the numerators of $\exprforexp_j(\bZ)$, $V_j(\bZ)$ and $\exprforint_i(\bZ)$ do not specialize to zero under $\siginter$.
    \item\label{item3:assumption2}
    The order of the least common left multiple of the specialized factors
    \[
    \overline{L}_1(\partial) \, := \, L_1(\siginter(\bss(\bvv),\bvvbase),\partial), \, \dots , \, \overline{L}_k(\partial) \, := \, L_k(\siginter(\bss(\bvv),\bvvbase),\partial)
    \]
    (cf.\ Definition~\ref{def:specfactorsLCLM} below) is equal to the order of $\LCLM(\bss(\bvv),\bvvbase,\partial)$.
\end{enumerate}
\end{assumption}

According to Assumption~\ref{assumption2} \ref{item1:assumption2} we are able to specialize the generic irreducible factors and their least common left multiple.
\begin{definition}\label{def:specfactorsLCLM}
For the irreducible factorization  
\[
L_{\group}(\bss(\bvv),\partial) \, = \, L_1(\bss(\bvv),\bvvbase,\partial) \cdots L_k(\bss(\bvv),\bvvbase,\partial)
\]
in $\field\langle \bss(\bvv),\bvvbase \rangle [\partial]$ we denote its specialization under $\siginter$ by 
\[
L_{\group}(\bsq,\partial) \, = \, \overline{L}_1(\partial) \cdots \overline{L}_k(\partial)
\]
in $\difffield[\partial]$. Moreover, for the least common left multiple 
\[
\LCLM(\bss(\bvv),\bvvbase,\partial) \, = \, \LCLM( L_1(\bss(\bvv),\bvvbase,\partial), \dots , L_k(\bss(\bvv),\bvvbase,\partial))  
\] 
in $\field\langle \bss(\bvv),\bvvbase \rangle [\partial]$ we denote its specialization under $\siginter$ by
\[
\overline{\LCLM}(\bsq,\bvvqbase,\partial) \, .
\] 
\end{definition}

Let $y_1,\dots,y_n$ be a basis of the solution space $V$ of $L_{\group}(\bss(\bvv),\partial) \, y=0$ in $\generalext$ such that for $i=1,\dots,k$ the elements $\widehat{y}_{i,1},\dots,\widehat{y}_{i,\widehat{n}_i}$ defined as
\begin{equation}\label{eqn:basisforirredfactors} 
\begin{array}{c}
    \widehat{y}_{k,1} = y_1, \ \dots, \  \widehat{y}_{k,\widehat{n}_k}= y_{n_k'}, \\[0.5em]  \widehat{y}_{k-1,1} = L_k(\partial)y_{n_k'+1}, \ \dots , \  \widehat{y}_{k-1,\widehat{n}_{k-1}} = L_k(\partial)y_{n_{k-1}'},\\[0.5em]
    \widehat{y}_{k-2,1}= (L_{k-1}(\partial) \circ L_k(\partial))y_{n_{k-1}'+1}, \  \dots , \qquad\qquad\qquad\qquad\qquad\qquad\qquad \\[0.5em] \qquad\qquad\qquad \qquad \qquad \qquad \qquad \widehat{y}_{k-2,\widehat{n}_{k-2}}= (L_{k-1}(\partial) \circ L_k(\partial))y_{n_{k-2}'}, \\[0.5em]  \vdots \\[0.5em]
    \widehat{y}_{1,1} = (L_{2}(\partial) \circ \cdots \circ L_k(\partial))y_{n_{2}'+1}, \ \dots , \ \widehat{y}_{1,\widehat{n}_1} = (L_{2}(\partial) \circ \cdots \circ L_k(\partial))y_{n_{1}'}
\end{array}
\end{equation}
 form a basis of the solution space $V_i$ of $L_i(\bss(\bvv),\bvvbase,\partial) \, y=0$  in $\generalext$ (cf.\ the proof of Theorem~\ref{cor:extensionsbyPI} and Remark~\ref{rem:basisLCLM}).

\begin{remark}\label{rem:extensigma}
    Let $D \subset \field \{ \bss(\bvv),\bvvbase \}$ be the multiplicatively closed subset generated by the denominators addressed in Assumption~\ref{assumption2} 
    \ref{item1:assumption2}--\ref{itemNew3:assumption2}.
    Then the set $D$ specializes under $\siginter$ to a multiplicatively closed subset of $\difffield$ which does not contain zero. 
    Thus, the differential homomorphism $\sigPV$ of Proposition~\ref{prop:parabolicbound} extends to the localization by $D$, i.e.\ to the differential homomorphism 
    \[
    \sigPV\colon D^{-1} \difffield \{ \bvv \}[\bexp,\bexp^{-1}, \bint ] \to \overline{\generalext}\,,
    \]
    which we also denote by $\sigPV$. Indeed,
    since the kernel of $\sigPV$ contains the ideal $\idealinter$, 
    we have $\ker(\sigPV) \cap D = \emptyset$. Since for $i=1,\dots,k$ the basis elements $ \widehat{y}_{i,1},\dots,\widehat{y}_{i,\widehat{n}_i}$ of the solution space $V_i$ of  $L_i(\bss(\bvv),\bvvbase,\partial) \, y = 0$ and the basis elements $y_1^{I''},\dots,y_{n_{I''}}^{I''}$ of the solution space of $\LCLM(\bss(\bvv),\bvvbase,\partial) \, y = 0$
    are contained in
    \[
     D^{-1} \difffield \{ \bvv \}[\bexp,\bexp^{-1}, \bint ] \, ,
    \]
    we can specialize them to elements in $\overline{\generalext}$ via $\sigPV$.
\end{remark}
    
\begin{lemma}\label{lem:specializesolutionspace}
Let $\sigPV$ be as in Remark~\ref{rem:extensigma}. 
    \begin{enumerate}
    \item\label{item:specializesolutionspace1}
    The specialization $\sigPV$ induces a $\field$-vector space isomorphism 
    \[
    \sigPV|_{V}\colon V \to \overline{V}, \, y_i \mapsto \sigPV(y_i)
    \]
    from $V$ to the solution space $\overline{V}$ of $L_{\group}(\bsq,\partial) \, y = 0$ in $\overline{\generalext}$.
    \item\label{item:specializesolutionspace2}
    For $i=1,\dots,k$ the specialization $\sigPV$ induces a $\field$-vector space isomorphism 
    \[
    \sigPV|_{V_i}\colon V_i \to \overline{V}_i, \, \widehat{y}_{i,j} \mapsto \sigPV(\widehat{y}_{i,j}) \quad \mathrm{with} \ j=1,\dots,\widehat{n}_i
    \]
    from $V_i$ to the solution space $\overline{V}_i$ of $\overline{L}_i(\partial)y=0$ in $\overline{\generalext}$.
    \item\label{item:specializesolutionspace3}
    The spaces $V_i$ and $\overline{V}_i$ are invariant under $\parabolic_J(\field)$ and $H(\field)$, respectively.
    For $g \in H(\field)$ the homomorphism $\sigPV$ is compatible with the induced isomorphisms $\gamma_g\colon V_i \to V_i$ and $\overline{\gamma}_g\colon \overline{V}_i \to\overline{V}_i$, i.e.\
    $\sigma(\gamma_g(v))=\overline{\gamma}_g (\sigPV(v))$ for all $v \in V_i$.
    \end{enumerate}
\end{lemma}

\begin{proof}
    \ref{item:specializesolutionspace1}\
    We have that $\BNFcomp \fm$ is a Wronskian matrix for $L_{\group}(\bss(\bvv),\partial)y=0$ and so the elements $y_1,\dots,y_n$ of its first row form a basis for $V$. Then Proposition~\ref{prop:parabolicbound} \ref{prop:E(a)}\ and \ref{prop:E(c)}\ imply that 
    \[
    \sigPV(\BNFcomp\fm) \, = \, \sigPV(\BNFcomp) \, \fmspec
    \]
    is a Wronskian matrix for $L_{\group}(\bsq,\partial) \, y = 0$ with entries in $\overline{\generalext}$. Since $\det(\fmspec)=1$ and $\det(\sigPV(\BNFcomp)) \in \field^{\times}$, the entries in the first row 
    $\sigPV(y_1),\dots,\sigPV(y_n)$ are $\field$-linearly independent by \cite[Lemma~1.12]{vanderPutSinger}.\\
    \ref{item:specializesolutionspace2}\
    Let $i \in \{1,\dots,k\}$ and let $y_1,\dots,y_n$ be a basis of $V$ with the property that the elements in \eqref{eqn:basisforirredfactors} form bases of $V_1,\dots,V_k$. Since $\sigPV$ is a differential homomorphism, the images $\sigPV(\widehat{y}_{i,1}),\dots,\sigPV(\widehat{y}_{i,\widehat{n}_i})$ are solutions of $\overline{L}_i(\partial) \, y = 0$. According to \ref{item:specializesolutionspace1}\
    the elements $\sigPV(y_1),\dots,\sigPV(y_n)$ are $\field$-linearly independent. One can use this to show by induction on $i=k,\dots,1$ that $\sigPV(\widehat{y}_{i,1}),\dots,\sigPV(\widehat{y}_{i,\widehat{n}_i})$ are $\field$-linearly independent.  Since the order of $\overline{L}_i(\partial)$ is the same as the order of $L_i(\bss(\bvv),\bvvbase,\partial)$ (recall that $L_i(\bss(\bvv),\bvvbase,\partial)$ is monic), we conclude that $\sigPV(\widehat{y}_{i,1}),\dots,\sigPV(\widehat{y}_{i,\widehat{n}_i})$ form a $\field$-basis of $\overline{V}_i$.\\
    \ref{item:specializesolutionspace3}\
    The first part follows from \cite[Lemma~2.2]{Singer_Reducibility} applied to the generic and specialized situation separately.
    For $g \in H(\field)\leq \parabolic_J(\field)$ let $\gamma_g \in \Gal_{\partial}(\generalext/\generalext^{\parabolic_J})$ and
    $\overline{\gamma}_g \in \Gal_{\partial}(\overline{\generalext}/\difffield)$ be given by 
    $\gamma_g(\fm) = \fm \, g$ and $\overline{\gamma}_g (\fmspec) = \fmspec \, g$.
    Since $\BNFcomp \in \GL_n(\field\{ \bss(\bvv) \})$ and $\sigPV(\BNFcomp) \in \GL_n(\difffield)$, the action of $\gamma_g$ respectively $\overline{\gamma}_g$ on $V$ and $\overline{V}$ is induced by
    \[
    \gamma_g (\BNFcomp \, \fm) \, = \, \gamma_g(\BNFcomp) \, \gamma_g(\fm) \, = \, \BNFcomp \, \fm \, g
    \]
    and
    \[
    \overline{\gamma}_g(\sigPV(\BNFcomp) \, \fmspec) \, = \, \overline{\gamma}_g(\sigPV(\BNFcomp)) \, \overline{\gamma}_g(\overline{\fm} ) \, = \, \sigPV(\BNFcomp) \, \overline{\fm} \, g
    \]
    respectively.  According to Proposition~\ref{prop:parabolicbound} we have $\sigPV(\fm) = \fmspec$ and so we conclude that 
    \[
    \sigPV ( \gamma_g (\BNFcomp \fm)) \, = \, \sigPV(\BNFcomp \, \fm \, g) \, = \, \sigPV(\BNFcomp) \, \fmspec \, g \, = \, \overline{\gamma}_g(\sigPV(\BNFcomp) \, \overline{\fm})\,.
    \]
    By restricting the action on $V$ and $\overline{V}$ to $V_i$ and $\overline{V}_i$, respectively, the claim follows.
\end{proof}

\begin{proposition}\label{prop:specializefactorization} $\,$
\begin{enumerate}
    \item\label{item:specializefactorization1}
    The specialized factorization
    \[
    L_{\group}(\bsq,\partial) \, = \, \overline{L}_1(\partial) \cdots \overline{L}_k(\partial)
    \]
    is an irreducible factorization in $\difffield[\partial]$. 
    \item\label{item:specializefactorization2}
    The specialized least common left multiple $\overline{\LCLM}(\bsq,\bvvqbase,\partial)$ is the least common left multiple of $\overline{L}_1(\partial), \dots ,\overline{L}_k(\partial)$ in $\difffield[\partial]$, that is
    \[
     \overline{\LCLM}(\bsq,\bvvqbase,\partial) \, = \, \LCLM(\overline{L}_1(\partial), \dots ,\overline{L}_k(\partial)) \,.
    \] 
    \item\label{item:specializefactorization3}
    The differential homomorphism $\sigPV$ of Remark~\ref{rem:extensigma} induces a $\field$-vector space isomorphism 
    \[
    \sigPV\colon V_{I''} \to \overline{V}_{I''}, \, (y_1^{I''},\dots,y_{n_{I''}}^{I''}) \mapsto (\sigPV(y_1^{I''}),\dots,\sigPV(y_{n_{I''}}^{I''}))
    \]
    between the solution space $V_{I''}$ of 
    $\LCLM(\bss(\bvv),\bvvbase,\partial)y = 0$ in $\generalext$ and the solution space of 
    $\overline{\LCLM}(\bsq,\bvvqbase,\partial) \, y = 0$ in $\overline{\generalext}$.
    \end{enumerate}
\end{proposition}

\begin{proof}
\ref{item:specializefactorization1}
Recall that the irreducible factorization of $L_{\group}(\bss(\bvv),\partial) \, y = 0$
induces a flag
\begin{equation}\label{eqn:flagirredcublefactoization}
V'_k\subset V'_{k-1} \subset \dots \subset V'_1=V
\end{equation}
on the solution space $V$ of $L_G(\bss(\bvv),\partial) \, y = 0$, where the subspaces $V'_k, \dots, V'_1$  are the solution spaces of the operators
\begin{gather*}
    L_k(\partial) , \ L_{k-1}(\partial) \circ L_k(\partial), \ \dots, \
    L_1(\partial) \circ \cdots \circ L_k(\partial)=L_G(\bss(\bvv),\partial) ,
\end{gather*}
respectively. Since $\parabolic_J(\field)$ is the Galois group of $\generalext$ over $\field\langle \bss(\bvv),\bvvbase\rangle$, the parabolic subgroup $\parabolic_J(\field)$ stabilizes the flag \eqref{eqn:flagirredcublefactoization} by \cite[Lemma~2.2]{Singer_Reducibility}.
Moreover, $\parabolic_J(\field)$ does not stabilize any refinement of the flag in \eqref{eqn:flagirredcublefactoization}.
Indeed, assume that there is a subspace $\tilde{V}'$ such that $\tilde{V}'$ is a proper subspace of $V_i'$, $V_{i+1}'$ is a proper subspace of $\tilde{V}'$ and $\parabolic_J(\field)$ stabilizes the flag 
\[
V'_k\subset \dots \subset V'_{i+1} \subset \tilde{V}' \subset V'_i \subset \dots \subset V'_1=V.
\]
Then by \cite[Lemma~2.2]{Singer_Reducibility} there exists an operator $\tilde{L}(\partial) \in \field\langle \bss(\bvv), \bvvbase \rangle[\partial]$
with solution space $\tilde{V}'$.
Since $\tilde{V}' \subset V'_i$ and $V'_{i+1} \subset \tilde{V}'$, there exist operators $Q_1(\partial), Q_2(\partial) \in \field\langle \bss(\bvv), \bvvbase \rangle[\partial]$ of order at least one such that
\[
L_i(\partial) \circ \dots \circ L_k(\partial) \, = \, Q_1(\partial) \circ \tilde{L}(\partial) \quad \text{and} \quad 
\tilde{L}(\partial) \, = \, Q_2(\partial) \circ L_{i+1}(\partial) \circ \dots \circ L_k(\partial) .
\]
We obtain a contradiction to the irreducibility of the operator $L_i(\partial)$.  

Suppose that there is at least one generic irreducible factor $L_i(\bss(\bvv),\bvvbase,\partial)$ such that its specialization $\overline{L}_i(\partial)$ is not irreducible over $\difffield$, i.e., there exist operators $\overline{L}_{i,1}(\partial)$ and $\overline{L}_{i,2}(\partial)$ in $\difffield[\partial]$ of order at least one such that $\overline{L}_i(\partial) = \overline{L}_{i,1}(\partial) \circ \overline{L}_{i,2}(\partial)$. 
According to Lemma~\ref{lem:specializesolutionspace} the generic flag specializes under $\sigma$ to a flag
\[
\overline{V}'_k \subset \overline{V}'_{k-1} \subset \dots \subset \overline{V}'_1 = \overline{V}
\]
of the solution space $\overline{V}$ of $L_{\group}(\bsq, \partial) \, y = 0$ and this flag is stabilized by $H(\field)$, since $H(\field)\leq \parabolic_J(\field)$.
Our assumption implies now that this flag becomes finer, i.e., there exists a subspace $\overline{V}'_{i,2}$ of positive dimension such that $\overline{V}'_{i,2}$ is a proper subspace of $\overline{V}'_i$ and $\overline{V}'_{i+1}$ is a proper subspace of $\overline{V}'_{i,2}$ and $H(\field)$ stabilizes the flag
\[
\overline{V}'_k \subset \dots \subset \overline{V}'_{i+1} \subset \overline{V}'_{i,2} \subset \overline{V}'_i \subset \dots \subset \overline{V}'_1 = \overline{V}.
\]
The stabilizer of this flag is a parabolic subgroup $P(\field)$ of $\group(\field)$ containing $H(\field)$. 
Since this flag is finer than the one in \eqref{eqn:flagirredcublefactoization} and $\parabolic_J(\field)$ does not stabilize any refinement of it, $P(\field)$ is a proper subgroup of $\parabolic_J(\field)$.
But this contradicts the fact that $\parabolic_J(\field)$ is the smallest parabolic subgroup containing the Levi groups of $H(\field)$ (cf.\ the proof of Proposition~\ref{prop:parabolicbound} \ref{prop:E(d)}). We conclude that
\[
L_{\group}(\bsq,\partial) \, = \, \overline{L}_1(\partial) \dots 
\overline{L}_k(\partial)
\]
is an irreducible factorization over $\difffield$.

\ref{item:specializefactorization2}\
Since $L_i(\bss(\bvv),\bvvbase,\partial)$ divides $\LCLM(\bss(\bvv),\bvvbase,\partial)$ on the right and $\siginter$ is a differential homomorphism and so can be extended to
\[
D^{-1} \difffield\{ \bss(\bvv), \bvvbase\}[\partial] \to \difffield[\partial] \, ,
\]
we obtain that each irreducible operator $\overline{L}_i(\partial)$ divides $\overline{\LCLM}(\bsq,\bvvqbase,\partial)$ on the right for every $i \in \{1,\dots,k\}$.
By Assumption~\ref{assumption2} \ref{item3:assumption2}, the order of $\overline{\LCLM}(\bsq,\bvvqbase,\partial)$ coincides with the order of
$\LCLM(\overline{L}_1(\partial),\dots, \overline{L}_k(\partial))$.
Hence, the statement follows from the uniqueness of the least common left multiple. 

\ref{item:specializefactorization3}\
By \cite[Lemma~2.12]{Singer_Reducibility} the solution spaces of $\LCLM(\bss(\bvv),\bvvbase,\partial)$ and $\overline{\LCLM}(\bsq,\bvvqbase,\partial)$ are $V_1 + \dots + V_k$ and $\overline{V}_1+ \dots + \overline{V}_k$, respectively.
Since the least common left multiples have the same order, their solution spaces have the same dimension. The differential homomorphism $\sigPV$ induces a $\field$-linear map between these solution spaces.
Let $\overline{v} \in \overline{V}_1+ \dots + \overline{V}_k$.
Then there exist $\overline{v}_i \in \overline{V}_i$ with $i=1,\dots,k$ such that $\overline{v}=\overline{v}_1+\dots+\overline{v}_k$.
Since $\sigPV$ restricts to isomorphisms $V_i \to \overline{V}_i$ by Lemma~\ref{lem:specializesolutionspace} \ref{item:specializesolutionspace2}, there exist $v_i \in V_i$ such that $\sigPV(v_i)=\overline{v}_i$ and so $\sigPV(v_1+\dots+ v_k)=\overline{v}_1+\dots+\overline{v}_k = \overline{v}$.
Thus the $\field$-linear map between the solution spaces is surjective and so an isomorphism.\qedhere
\end{proof}

\begin{remark}
Note that Assumption~\ref{assumption2} \ref{item1:assumption2} might imply Assumption~\ref{assumption2} \ref{item3:assumption2}.
For this claim one has to prove that the Wronskian determinant of $y_1^{I''},\dots,y_{n_{I''}}^{I''}$ is the denominator of the coefficient of the second highest order term in $\LCLM(\bss(\bvv),\bvvbase,\partial)$.
Then Assumption~\ref{assumption2} \ref{item1:assumption2} implies that the basis $y_1^{I''},\dots,y_{n_{I''}}^{I''}$ specializes to a $\field$-linearly independent set in $\overline{\generalext}$ which can be shown to be a basis of the least common left multiple of $\overline{L}_1(\partial), \dots,\overline{L}_k(\partial)$.
This forces this least common left multiple to have the same order as $\LCLM(\bss(\bvv),\bvvbase,\partial)$. 
\end{remark}

\begin{definition}\label{def:Q}
We denote by $Q$ a maximal differential ideal in the differential ring
\[
\difffield[X_{i,j},\det(X_{i,j})^{-1} \mid i,j=1,\dots,n_{I''}] \, = \, \difffield[\GL_{n_{I''}}]\,,
\]
where the derivation on $X_{i,j}$ is defined by
\[
\partial(X) \, = \, A_{\rm comp} \, X\,, \qquad X = (X_{i,j})\,,
\]
where $A_{\rm comp}$ denotes the companion matrix corresponding to the differential equation $\overline{\LCLM}(\bsq,\bvvqbase,\partial) \, y = 0$ (cf.\ Proposition~\ref{prop:specializefactorization}).
Moreover, we define
\[
\extfieldred \, := \, \Frac(\difffield[\GL_{n_{I''}}]/Q)\,,
\]
which is -- by the standard construction method -- a Picard-Vessiot extension of $\difffield$ for $\overline{\LCLM}(\bsq,\bvvqbase,\partial) \, y = 0$. 
\end{definition}

\begin{remark}\label{rem:computationQ}
    The maximal differential ideal $Q$ of Definition~\ref{def:Q} can be computed 
    using the results presented in \cite{SingerCompoint} and \cite{HoeijWeilInvariant}. More precisely, according to 
    \cite[Proposition~4.2]{SingerCompoint} the differential Galois group of $\overline{\LCLM}(\bsq,\bvvqbase, \partial) \, y = 0$ is a Levi group of the differential Galois group $H$ for the specialized normal form equation $L_{\group}(\bsq,\partial) \, y = 0$ and so is reductive. Hence, we can use the algorithm presented in \cite[Section~4.1]{SingerCompoint} to compute generators $q_1,\dots,q_s$ of $Q$.
\end{remark}

\begin{proposition}\label{prop:fixedfield} \,
\begin{enumerate}
\item\label{item:fixedfield1}
The differential homomorphism
\[
\sigPV\colon D^{-1} \difffield \{ \bvv \}[\bexp,\bexp^{-1}, \bint ] \to \overline{\generalext}
\]
from Remark~\ref{rem:extensigma} restricts to a differential homomorphism 
\[
\sigPV\colon D^{-1} \difffield \{ \bvv \}[\bexp,\bexp^{-1}, \intn_{i} \mid \beta_i \in \leviroots^- ] \to \overline{\generalext}^{\unirad(H)} \, .
\]
\item\label{item:fixedfield2} We have
\[
\overline{\generalext}^{\unirad(H)} \, = \, \difffield \langle \overline{\bvv} , \overline{\bexp}, \overline{\intn}_i \mid \beta_i \in \leviroots^- \rangle \, = \, \difffield \langle \sigPV(y_1^{I''}),\dots,\sigPV(y_{n_{I''}}^{I''}) \rangle\,.
\]
\end{enumerate}
\end{proposition}

\begin{proof}
    It follows from Theorem~\ref{cor:extensionsbyPI} describing the generic situation that $\unirad(\parabolic_J)$ fixes the elements $\bvv$, $\bexp$ and $\intn_i$ with $\beta_i \in \leviroots^-$.
    Since by Proposition~\ref{prop:parabolicbound} \ref{prop:E(d)}\ we have the inclusion $\unirad(H) \leq \unirad(\parabolic_J)$, we conclude that the specialized elements $\overline{\bvv}$, $\overline{\bexp}$ and $\overline{\intn}_i$ with $\beta_i \in \leviroots^-$ are fixed by $\unirad(H)$. Thus, we obtain the inclusions
    \[
    \overline{\generalext}^{\unirad(H)} \supseteq \difffield \langle \overline{\bvv} , \overline{\bexp}, \overline{\intn}_i \mid \beta_i \in \leviroots^- \rangle \supset \difffield \{ \overline{\bvv} \}[\overline{\bexp},\overline{\bexp}^{-1}, \overline{\intn}_{i} \mid \beta_i \in \leviroots^- ] \, .
    \]
    This shows \ref{item:fixedfield1}\
    and one inclusion in \ref{item:fixedfield2}.
    
    By Proposition~\ref{prop:specializefactorization} \ref{item:specializefactorization3}\ the elements $\sigPV(y_1^{I''}),\dots,\sigPV(y_{n_{I''}}^{I''})$ in $\overline{\generalext}$ form a $\field$-basis of the solution space of $\overline{\LCLM}(\bsq,\bvvqbase, \partial) \, y = 0$ and so  
    \[
    \KPV:= \difffield (\wronski( \sigPV(y_1^{I''}),\dots,\sigPV(y_{n_{I''}}^{I''}))) 
    \]
    is a Picard-Vessiot extension of $\difffield$ for that equation. Since the generic elements $y_1^{I''} ,\dots, y_{n_{I''}}^{I''}$ are fixed by $\unirad(H)\leq \unirad(\parabolic_J)$ (cf.\ Theorem~\ref{cor:extensionsbyPI}), the specialized elements
    \[
    \sigPV(y_1^{I''}), \quad \dots, \quad \sigPV(y_{n_{I''}}^{I''})
    \]
    are fixed by $\unirad(H)$ and so
    $\KPV \subseteq \overline{\generalext}^{\unirad(H)}$.
    Since $\overline{\generalext}^{\unirad(H)}$  is also a Picard-Vessiot extension of $\difffield$ for the same equation (cf.\ \cite[Proposition~4.2]{SingerCompoint}), we conclude that
    $\KPV = \overline{\generalext}^{\unirad(H)}$.
    By Remark~\ref{rem:extensigma} and Theorem~\ref{cor:extensionsbyPI} the elements $y_1^{I''} ,\dots, y_{n_{I''}}^{I''}$ are contained in 
    \[
     D^{-1} \difffield \{ \bvv \}[\bexp,\bexp^{-1}, \bint ]  \cap \difffield \langle \bss(\bvv), \bvvbase \rangle \langle \bexp, \bvvext, \intn_i \mid \beta_i \in \leviroots^- \rangle.
    \]
     Hence, their specializations $\sigPV(y_1^{I''}),\dots,\sigPV(y_{n_{I''}}^{I''})$ are contained in 
     \[ 
     \difffield \langle \overline{\bvv} , \overline{\bexp}, \overline{\intn}_i \mid \beta_i \in \leviroots^- \rangle \, .
     \] 
     From this we obtain the inclusion
    \[
      \overline{\generalext}^{\unirad(H)}  = \KPV
     \subseteq \difffield \langle \overline{\bvv} , \overline{\bexp}, \overline{\intn}_i \mid \beta_i \in \leviroots^- \rangle \, .
    \]
\end{proof}

\begin{proposition}\label{prop:sigmared}
There are rational functions
\[
\widehat{\bvv} \, := \, (\widehat{v}_1, \ldots, \widehat{v}_l), \quad \bratexp \, := \, (\ratexp_1, \ldots, \ratexp_l) \quad \text{and} \quad \widehat{\intn}_i \quad (\beta_i \in \leviroots^-)
\]
in $\difffield(\GL_{n_{I''}})$ such that the map 
\[
\begin{array}{rcl}
    \siginter\colon \field\{ \bss(\bvv),\bvvbase \} & \to & \difffield\,,\\[0.2em]
    \bss(\bvv) & \mapsto & \bsq\,,\\[0.2em]
    \bvvbase = (v_{i_{r+1}},\dots, v_{i_l}) & \mapsto & \bvvqbase = (\overline{v}_{i_{r+1}},\dots, \overline{v}_{i_l})
\end{array}
\]
extends to a differential homomorphism 
\[
\begin{array}{rcl}
\sigred\colon
D^{-1} \difffield \{ \bvv \}[\bexp,\bexp^{-1}, \intn_i \mid \beta_i \in \leviroots^- ]
& \to & \extfieldred\,,\\[0.2em]
\bvv & \mapsto & \bvh + Q\,,\\[0.2em]
\bexp & \mapsto & \bratexp + Q\,,\\[0.2em]
\intn_i & \mapsto & \ratint_i + Q
\end{array}
\]
mapping the basis $y_1^{I''},\dots,y_{n_{I''}}^{I''}$ of $V_{I''}$ to the basis $\sigred(y_1^{I''}),\dots,\sigred(y_{n_{I''}}^{I''})$ of
the solution space $\overline{V}_{I''}$ of $\overline{\LCLM}(\bsq,\bvvqbase, \partial) \, y = 0$ in $\extfieldred$.
\end{proposition}

\begin{proof}
Since by Proposition~\ref{prop:specializefactorization} \ref{item:specializefactorization3}, $\sigPV(y_1^{I''}),\dots,\sigPV(y_{n_{I''}}^{I''})$ form a basis of the solution space $\overline{V}_{I''}$ of $\overline{\LCLM}(\bsq,\bvvqbase, \partial) \, y = 0$, the differential field
\[
\KPV:=\difffield(\wronski(\sigPV(y_1^{I''}),\dots,\sigPV(y_{n_{I''}}^{I''})))
\]
is a Picard-Vessiot extension of $\difffield$ for $\overline{\LCLM}(\bsq,\bvvqbase, \partial) \, y = 0$, where $\wronski$ denotes the Wronskian matrix. Indeed, as a subfield of $\overline{\generalext}$ its constants are $\field$. Trivially, there exists a fundamental matrix in $\GL_{n_{I''}}(\KPV)$
for $A_{\rm comp}$ and 
$\KPV$ is 
generated as field over $\difffield$ by the entries of this fundamental matrix.
According to Proposition~\ref{prop:fixedfield} \ref{item:fixedfield1}\ 
and \ref{item:fixedfield2}\
the differential homomorphism $\sigPV$
of Remark~\ref{rem:extensigma} restricts to a differential homomorphism 
\[
\sigPV\colon D^{-1} \difffield \{ \bvv \}[\bexp,\bexp^{-1}, \intn_i \mid \beta_i \in \leviroots^- ] \to  
\KPV \, . 
\]

Since both differential fields $\extfieldred$ and 
$\KPV$
are Picard-Vessiot extensions of $\difffield$ for the same differential equation $\overline{\LCLM}(\bsq,\bvvqbase, \partial) \, y = 0$, there exists a differential $\difffield$-isomorphism 
\[
\begin{array}{rcl}
   \varphi \colon  
   \KPV 
   & \! \to \! & \! \Frac(\difffield[\GL_{n_{I''}}]/Q) \, = \, \extfieldred \\[0.5em]
   \wronski(\sigPV(y_1^{I''}),\dots,\sigPV(y_{n_{I''}}^{I''})) \! & \! \mapsto \! & \! \wronski(\overline{X}_{1,1},\dots,\overline{X}_{1,{n_{I''}}}) \, M
\end{array}
\]
for a constant matrix $M \in \GL_{n_{I''}}(\field)$. Composing $\sigPV$ with $\varphi$ we obtain a differential homomorphism 
\begin{gather*}
    \sigred\colon D^{-1} \difffield \{ \bvv \}[\bexp,\bexp^{-1}, \intn_i \mid \beta_i \in \leviroots^- ] \to \extfieldred \, .
\end{gather*}
The rational functions
\[
\widehat{\bvv} \, := \, (\widehat{v}_1, \ldots, \widehat{v}_l), \quad \bratexp \, := \, (\ratexp_1, \ldots, \ratexp_l) \quad \mathrm{and} \quad \widehat{\intn}_i \quad (\beta_i \in \leviroots^-)
\]
can be chosen as preimages in $\Frac(\difffield[\GL_{n_{I''}}])$ of  
$\sigred(\bexp)$, $\sigred(\bvv)$ and $\sigred(\intn_i)$ in
$\Frac(\difffield[\GL_{n_{I''}}]/Q)$ under the canonical map
\[
\difffield[\GL_{n_{I''}}]_{Q} \to \Frac(\difffield[\GL_{n_{I''}}]/Q)
\]
from the localization of $\difffield[\GL_{n_{I''}}]$ at the prime ideal $Q$.
The images of the basis elements $y_1^{I''},\dots,y_{n_{I''}}^{I''}$ under $\sigred$ are the entries in the vector $(\overline{X}_{1,1},\dots,\overline{X}_{1,n_{I''}}) M$ and they form a basis of the solution space $\overline{V}_{I''}$ of $\overline{\LCLM}(\bsq,\bvvqbase, \partial) \, y = 0$ in $\extfieldred$, because $M \in \GL_{n_{I''}}(\field)$.
\end{proof}

\begin{remark}\label{rem:specializeEXPVINT}
Note that we can restrict the domain of definition of $\sigred$ to  
$$R_{\rm rest} \, := \, D^{-1} \field\{ \bss(\bvv) , \bvvbase \} \{ y_1^{I''},\dots,y_{n_{I''}}^{I''} \}$$
and obtain a differential homomorphism $\eta$. Recall that 
in Proposition~\ref{cor:exprforexpandint} and Remark~\ref{rem:computeEXPVINT} we determined 
\[
\exprforexp_j^{I''}(\bZ), \quad V_j^{I''}(\bZ) \quad \text{and} \quad \exprforint_i^{I''}(\bZ) 
\]
in the localization $D_{\rm exp}^{-1}\field \langle \bss(\bvv),\bvvbase \rangle \{ \bZ \} $ such that 
\begin{gather*}
\exprforexp_j^{I''}(y_1^{I''},\dots,y_{n_{I''}}^{I''} ) \, = \, \exp_i, \quad V_j^{I''}(y_1^{I''},\dots,y_{n_{I''}}^{I''} ) \, = \, v_i \quad \text{and} \\ \exprforint_i^{I''}(y_1^{I''},\dots,y_{n_{I''}}^{I''} ) \, = \, \intn_i \, ,
\end{gather*}
where $D_{\rm exp}$ is the multiplicatively closed set generated by 
\[
E_1(\bZ),\dots,E_l(\bZ) \in \field \langle \bss(\bvv),\bvvbase \rangle \{ \bZ \}.
\]
Proposition~\ref{cor:exprforexpandint} and Assumption~\ref{assumption2} \ref{itemNew2:assumption2} imply that the  elements 
\begin{equation*}  
E_1(y_1^{I''},\dots,y_{n_{I''}}^{I''}) \, = \, \expsolass_1, \quad \dots, \quad E_l(y_1^{I''},\dots,y_{n_{I''}}^{I''}) \, = \, \expsolass_l
\end{equation*}
are in $R_{\rm rest}$ and so we can apply $\sigred$ to them. 
Proposition~\ref{prop:exponentialandRiccati} implies that the exponential solutions $\expsolass_1,\dots,\expsolass_l$ of the associated equations 
are products of powers of $\exp_1,\dots,\exp_l$ with exponents in $\Z$. Since $\exp_1,\dots,\exp_l$ are units in 
$$
D^{-1} \difffield \{ \bvv \}[\bexp,\bexp^{-1}, \intn_i \mid \beta_i \in \leviroots^-] \, , 
$$
we conclude that $\expsolass_1,\dots,\expsolass_l$ do not lie in the kernel of $\sigred$. Hence, we can 
extend $\eta$ to the localization of $R_{\rm rest}$ by the multiplicatively closed subset generated by $\expsolass_1,\dots,\expsolass_l$. This localization has now the important property that it contains the elements $\bexp$, $\bvv$ and $\intn_i$ with $\beta_i \in \leviroots^-$.
\end{remark}

\begin{lemma}\label{lem:condcijhom}
For an invertible matrix of constants $(\overline{c}_{i,j}) \in \GL_{n_{I''}}(\field)$ define
\[
\begin{array}{rcl}
\eta\colon D^{-1} \field\{ \bss(\bvv) , \bvvbase \} \{ y_1^{I''},\dots,y_{n_{I''}}^{I''} \} \! & \! \to \! & \! \difffield [\GL_{n_{I''}}] / Q\,,\\[0.5em]
\bss(\bvv) \! & \! \mapsto \! & \! \bsq\,,\\[0.5em]
\bvvbase \! & \! \mapsto \! & \! \bvvqbase\,,\\[0.5em]
y_i^{I''}  \! & \! \mapsto \! & \! \sum_{j=1}^{n_{I''}} \overline{c}_{j,i} \overline{X}_{1,j}\,.
\end{array}
\]
Then $\eta$ is a differential homomorphism
if and only if we have
\begin{equation}\label{eq:RELinQ}
\begin{array}{lcl}
\mathrm{REL}_i(\sum \overline{c}_{j,1} \overline{X}_{1,j} ,\dots,\sum \overline{c}_{j,n_{I''}} \overline{X}_{1,j} ,\bsq,\bvvqbase) \! & \! = \! & \! 0\,,
\quad  i=1,\dots,k\,,\\[0.5em]
\mathrm{REL}^{\neq}(\sum \overline{c}_{j,1} \overline{X}_{1,j} ,\dots,\sum \overline{c}_{j,n_{I''}} \overline{X}_{1,j} ,\bsq,\bvvqbase) \! & \! \neq \! & \! 0\,,
\end{array}
\end{equation}
where $\mathrm{REL}_i(\bZ,\bsq,\bvvqbase)$ and $\mathrm{REL}^{\neq}(\bZ,\bsq,\bvvqbase)$  are obtained from  $\mathrm{REL}_i(\bZ)$ and $\mathrm{REL}^{\neq}(\bZ)$ respectively as in Proposition~\ref{prop:computationRELnew} by applying $\siginter$ to their coefficients.
\end{lemma}

\begin{proof}
Extending the differential homomorphism $\varphi$ from
Proposition~\ref{prop:computationRELnew} to
\[
\varphi\colon D^{-1} \field\{ \bss(\bvv),\bvvbase \} \{ \bZ \} \to D^{-1} \field\{ \bss(\bvv),\bvvbase \} \{ y_1^{I''},\dots, y^{I''}_{n_{I''}}\}\,,
\]
the differential homomorphism
\[
\begin{array}{rcl}
\widehat{\eta}\colon D^{-1} \field\{ \bss(\bvv) , \bvvbase \} \{ Z_1,\dots, Z_{n_{I''}} \} \! & \! \to \! & \! \difffield [\GL_{n_{I''}}] / Q\,,\\[0.5em]
\bss(\bvv) \! & \! \mapsto \! & \! \bsq\,,\\[0.5em]
\bvvbase \! & \! \mapsto \! & \! \bvvqbase\,,\\[0.5em]
Z_i \! & \! \mapsto \! & \! \sum_{j=1}^{n_{I''}} \overline{c}_{j,i} \overline{X}_{1,j}
\end{array}
\]
induces a differential homomorphism $\eta$ as required if and only if
it factors as indicated in the following diagram:
\[
\begin{tikzcd}
D^{-1} \field\{ \bss(\bvv) , \bvvbase \} \{ Z_1,\dots, Z_{n_{I''}} \} \arrow{r}{\varphi}
\arrow[swap]{ddr}{\widehat{\eta}} & D^{-1} \field\{ \bss(\bvv) , \bvvbase \} \{ y_1^{I''},\dots,y_{n_{I''}}^{I''} \} \arrow{dd}{\eta} \\ \\
  & \difffield [\GL_{n_{I''}}] / Q
\end{tikzcd} 
\]
This is equivalent to
\[
p\big( \sum_{j=1}^{n_{I''}} \overline{c}_{j,1} \overline{X}_{1,j}, \ldots,
\sum_{j=1}^{n_{I''}} \overline{c}_{j,n_{I''}} \overline{X}_{1,j}, \bsq, \bvvqbase
\big) \, = \, 0
\quad \text{for all} \ p \in \ker(\varphi)\,,
\]
which by Proposition~\ref{prop:computationRELnew} is equivalent to \eqref{eq:RELinQ}.
\end{proof}

We are ready now to present an algorithm which computes representatives 
\[\label{eqn:definnitionofhatvariables}
\widehat{\bvv} \, := \, (\widehat{v}_1, \ldots, \widehat{v}_l), \quad \bratexp \, := \, (\ratexp_1, \ldots, \ratexp_l) \quad \mathrm{and} \quad \widehat{\intn}_i \quad \mbox{for } \beta_i \in \leviroots^-
\]
in $\difffield(\GL_{n_{I''}})$ of residue classes in $\extfieldred$
defining a specialization $\sigred$ as in Proposition~\ref{prop:sigmared}.
Lemma~\ref{lem:condcijhom} gives a criterion such that the map $\eta$ associating to the generic basis a basis of the solution space of
\[
\overline{\LCLM}(\bsq,\bvvqbase, \partial) \, y \, = \, 0
\]
in $\extfieldred$ is a differential homomorphism.
Once we know the differential homomorphism $\eta$, we obtain 
residue classes 
$\widehat{\bvv} + Q$, $\bratexp + Q$ and $\widehat{\intn}_i + Q$ as the images under $\eta$ (applied to numerator and denominator) of 
\[
\begin{array}{l}
\exp_1 \, = \, \exprforexp_1^{I''}(y_1^{I''},\dots,y_{n_{I''}}^{I''}), \quad \dots, \quad \exp_l \, = \, \exprforexp_l^{I''}(y_1^{I''},\dots,y_{n_{I''}}^{I''}),\\[0.5em]
\qquad
\intn_i \, = \, \exprforint_i^{I''}(y_1^{I''},\dots,y_{n_{I''}}^{I''}) \quad \text{for} \ \beta_i \in \leviroots^-
\end{array}
\]
and of
\[
     v_1 \, = \,V_1^{I''}(y_1^{I''},\dots,y_{n_{I''}}^{I''}), \quad \ldots, \quad  
     v_l \, = \, V_l^{I''}(y_1^{I''},\dots,y_{n_{I''}}^{I''}).
\]
To determine the differential homomorphism $\eta$ we make for each $y_i^{I''}$ the ansatz
\[
c_{i,1} X_{1,1} + \dots + c_{i,n_{I''}} X_{1,n_{I''}} 
\] 
with $(c_{i,j})$ an invertible matrix of constant indeterminates.
Then we use the differential Thomas decomposition to compute polynomial conditions on $c_{i,j}$ such that 
\[
\begin{array}{l}
\mathrm{REL}_s(\sum c_{j,1} X_{1,j},\dots,\sum c_{j,n_{I''}} X_{1,j},\bsq,\bvvqbase) = 0   \ \text{mod} \ Q \quad \text{for} \ s=1,\dots, k\,,\\[0.5em]
\mathrm{REL}^{\neq}(\sum c_{j,1} X_{1,j} ,\dots,\sum c_{j,n_{I''}} X_{1,j} ,\bsq,\bvvqbase) \neq 0 \  \text{mod} \ Q\,.
\end{array}
\]

\begin{algorithm} 
\DontPrintSemicolon
\KwInput {\begin{itemize}
    \item Generators $q_1,\dots,q_{s}$ of $Q \unlhd \difffield[\GL_{n_{I''}}]$ from Definition~\ref{def:Q},
    \item the specialized least common left multiple $\overline{\LCLM}(\bsq,\bvvqbase,\partial)$,
    \item the specialized $\mathrm{REL}_s(\bZ,\bsq,\bvvqbase)$ 
    and $\mathrm{REL}^{\neq}( \bZ ,\bsq,\bvvqbase)$ for $s=1,\dots,k$,
    \item $\exprforexp_j^{I''}(\bZ)$ and $V_j^{I''}(\bZ)$ for $j = 1$, \ldots, $l$ and 
     $\exprforint_i^{I''}(\bZ)$ for $i=1,\dots,m$ with $\beta_i \in \leviroots^-$.
\end{itemize}   }
\KwOutput{Representatives $\widehat{\bvv}$, $\widehat{\bexp}$ and $\ratint_i$ in $\difffield(\GL_{n_{I''}})$ of the residue classes in $\Frac(\difffield[\GL_{n_{I''}}]/Q)$ corresponding to $\bvv$, $\bexp$ and $\intn_i$ for $i=1,\dots,m$ with $\beta_i \in \leviroots^-$ as in Proposition~\ref{prop:sigmared}.}
Let $S_{\mathrm{eqs}}=\emptyset$ and $S_{\mathrm{ineqs}}=\emptyset$.\\
Let $x_1,\dots,x_{n_{I''}}$ and $c_{i,j}$ with $i,j=1,\dots,n_{I''}$ be differential indeterminates over $\difffield$. We will compute a differential Thomas decomposition of a system of equations and inequations in $\difffield\{x_1,\dots,x_{n_{I''}},c_{i,j}\}$\\
Apply the substitution
\[
X_{i,j} \mapsto x_j^{(i-1)}
\]
to the generators $q_1,\dots,q_{s}$ of $Q$ and append the resulting differential polynomials to $S_{\mathrm{eqs}}$.\\
 For $j=1,\dots,n_{I''}$ append $\overline{\LCLM}(\bsq,\bvvqbase,\partial) \, x_j$ to $S_{\mathrm{eqs}}$.\\
For $i, j = 1, \ldots, n_{I''}$ append  $c_{i,j}'$ to $S_{\mathrm{eqs}}$ and $\det(c_{i,j})$ to $S_{\mathrm{ineqs}}$.\\ 
Append the differential polynomials 
\begin{gather*}
\mathrm{REL}_s(\sum_r c_{1,r} x_r,\dots,\sum_r c_{n_{I''},r} x_r,\bsq,\bvvqbase) \quad \text{for} \ s=1,\dots,k
\end{gather*}
to $S_{\mathrm{eqs}}$ and $\mathrm{REL}^{\neq}( \sum_r c_{1,r} x_r,\dots,\sum_r c_{n_{I''},r} x_r ,\bsq,\bvvqbase)$ to $S_{\mathrm{ineqs}}$.\\
For $j=1,\dots,l$ and for $i=1,\dots,m$ with $\beta_i \in \leviroots^-$ append to $S_{\mathrm{ineqs}}$ the denominators of  
\[
\begin{array}{l}
    \exprforexp_j^{I''}(\sum_r c_{1,r} x_r,\dots,\sum_r c_{n_{I''},r} x_r,\bsq,\bvvqbase ) \, ,\\[0.5em] 
    V_j^{I''}(\sum_r c_{1,r} x_r,\dots,\sum_r c_{n_{I''},r} x_r,\bsq,\bvvqbase ) \quad \mathrm{and} \quad \\[0.5em]
    \exprforint_i^{I''}( \sum_r c_{1,r} x_r,\dots,\sum_r c_{n_{I''},r} x_r,\bsq,\bvvqbase ) .
\end{array}
\]
\\
With respect to an elimination ranking satisfying 
\[
x_1 > \dots > x_{n_{I''}} \gg c_{1,1} > \dots >c_{1,n_{I''}} > \dots > c_{l,1}>\dots >c_{l,n_{I''}}
\]
compute a Thomas decomposition of the differential system with equations $S_{\mathrm{eqs}}$ and inequations $S_{\mathrm{ineqs}}$.\\
Choose a simple system and determine a solution $\overline{c}_{i,j}$ of the equations and inequations only involving $c_{i,j}$.\\    
Define for $j=1,\dots,l$ and for $i=1,\dots,m$ with $\beta_i \in \leviroots^-$ the rational functions 
\begin{eqnarray*}
   \ratexp_j & := & \exprforexp_j^{I''}(\sum_r \overline{c}_{1,r} X_{1,r},\dots,\sum_r \overline{c}_{n_{I''},r} X_{1,r}) \, , \\
   \widehat{v}_j & := & V_j^{I''}(\sum_r \overline{c}_{1,r} X_{1,r},\dots,\sum_r \overline{c}_{n_{I''},r} X_{1,r}) \, , \\
   \ratint_i & := & \exprforint_i^{I''}(\sum_r \overline{c}_{1,r} X_{1,r},\dots,\sum_r \overline{c}_{n_{I''},r} X_{1,r})) \, .
\end{eqnarray*}
\Return($\widehat{\bvv}=(\widehat{v}_1,\dots,\widehat{v}_l)$, $\widehat{\bexp}=(\ratexp_1,\dots,\ratexp_l)$, $\ratint_i  \ \text{with} \ \beta_i \in \leviroots^-$)
\caption{Match Bases\label{alg:matchbases}}
\end{algorithm}

\begin{proposition}\label{prop:algmatchbases}
    Algorithm~\ref{alg:matchbases} is correct and terminates.
\end{proposition}

\begin{proof}
The algorithm terminates, since the  Thomas decomposition terminates.

Since the map 
\[
\begin{array}{rcl}
    \difffield\{ x_1,\dots,x_{n_{I''}} \}/Q_1 \! & \! \to \! & \! \difffield[\GL_{n_{I''}}]/Q\,,\\[0.5em]
    \overline{x}_j^{(i-1)} \! & \! \mapsto \! & \! \overline{X}_{i,j}\,,
\end{array}
\]
where $Q_1$ is the differential ideal generated by the differential polynomials obtained from substituting $X_{i,j}$ in $q_1,\dots, q_s$ by 
$x_j^{(i-1)}$  and the differential polynomials $\overline{\LCLM}(\bsq,\bvvqbase, \partial) \, x_j = 0$, is a differential isomorphism, the ideal $Q_1$ is a maximal differential ideal.
Let $\widehat{Q}$ be the differential ideal in $\difffield\{x_1,\dots,x_{n_{I''}},c_{i,j}\}$ generated by the elements of $S_{\rm eqs}$ obtained in step~6, i.e.\ by the generators of $Q_1$ and $c_{i,j}'$ and
\begin{gather*}
\mathrm{REL}_s(\sum_r c_{1,r} x_r,\dots,\sum_r c_{n_{I''},r} x_r,\bsq,\bvvqbase) \quad \text{for} \ s=1,\dots,k.
\end{gather*}
We are going to show that $\widehat{Q}$ is a proper differential ideal which does not contain the inequations $S_{\rm ineqs}$ so that the Thomas decomposition consists of at least one simple differential system.
To this end, observe first that $\widehat{Q}$ contains the ideal $(Q_1)$ generated by $Q_1$. Now Proposition~\ref{prop:sigmared} implies the existence of a matrix $(\overline{c}_{i,j}) \in \GL_{n_{I''}}(\field)$ such that the restriction of $\sigred$ to 
    $$D^{-1} \field\{ \bss(\bvv) , \bvvbase \} \{ y_1^{I''},\dots,y_{n_{I''}}^{I''} \}$$ yields a differential homomorphism $\eta$ as in Lemma~\ref{lem:condcijhom} (cf.\ $\varphi$ in the proof of Proposition~\ref{prop:sigmared}).
    It follows now from  Lemma~\ref{lem:condcijhom} that substituting $c_{i,j}$ by $\overline{c}_{i,j}$ maps the differential ideal $\widehat{Q}$ of $\difffield\{x_1,\dots,x_{n_{I''}},c_{i,j}\}$ to the differential ideal $Q_1$ of $\difffield\{x_1,\dots,x_{n_{I''}} \}$ showing that $\widehat{Q}$ is proper.
    We now address the inequalities.
    Lemma~\ref{lem:condcijhom} implies that substituting $c_{i,j}$ by $\overline{c}_{i,j}$ does not map 
    $$\mathrm{REL}^{\neq}( \sum_r c_{1,r} x_r,\dots,\sum_r c_{n_{I''},r} x_r ,\bsq,\bvvqbase)$$ 
    into the differential ideal $Q_1$ and hence this differential polynomial does not lie in  $\widehat{Q}$.  Moreover, it follows from Remark~\ref{rem:specializeEXPVINT} that the denominators of 
\[
\exprforexp_j^{I''}(y_1^{I''},\dots,y_{n_{I''}}^{I''}), \ 
V_j^{I''}(y_1^{I''},\dots,y_{n_{I''}}^{I''}) \quad \text{and} \quad 
\exprforint_i^{I''}(y_1^{I''},\dots,y_{n_{I''}}^{I''})
\]
do not lie in the kernel of $\eta$. In other words, if we substitute $c_{i,j}$by $\overline{c}_{i,j}$ in the denominators of
\[
\begin{array}{l}
    \exprforexp_j^{I''}(\sum_r c_{1,r} x_r,\dots,\sum_r c_{n_{I''},r} x_r,\bsq,\bvvqbase ) \, ,\\[0.5em] 
    V_j^{I''}(\sum_r c_{1,r} x_r,\dots,\sum_r c_{n_{I''},r} x_r,\bsq,\bvvqbase ) \quad \mathrm{and} \quad \\[0.5em]
    \exprforint_i^{I''}( \sum_r c_{1,r} x_r,\dots,\sum_r c_{n_{I''},r} x_r,\bsq,\bvvqbase )\,,
\end{array}
\]
the obtained differential polynomials in $\difffield\{ x_1,\dots,x_{n_{I''}}\}$ do not lie $Q_1$.
Therefore, these denominators do not lie in $\widehat{Q}$. Finally, since $\widehat{Q} \cap F\{ c_{i,j}\} =(0)$, we have that $c_{i,j}'$ and $\det(c_{i,j})$ are not contained in $\widehat{Q}$. Overall, we conclude that the  Thomas decomposition applied to $S_{\rm eqs}$ and $S_{\rm ineqs}$ returns at least one simple system.

Since the $c_{i,j}$ are ranked lowest, each simple system returned by the differential Thomas decomposition computed in step~8 has the following elimination property:
each $(\overline{c}_{i,j}) \in \GL_{n_{I''}}(\field)$ satisfying all equations and inequations of the simple system that only involve the indeterminates $c_{i,j}$ can be extended to a solution $(\overline{x}_k, \overline{c}_{i,j})$ of the simple system,
where $\overline{x}_k$ belong to some differential extension field of $\difffield$.
Thus each such $(\overline{c}_{i,j})$ yields a surjective differential homomorphism 
\[
    \varphi\colon \difffield\{x_1,\dots,x_{n_{I''}},c_{i,j}\} \to \difffield\{ x_1,\dots,x_{n_{I''}} \}, c_{i,j} \mapsto \overline{c}_{i,j}
\]
with the property that $\varphi(\widehat{Q})$ is a proper differential ideal of $\difffield\{x_1,\dots,x_{n_{I''}} \}$
    containing $Q_1$. Since $Q_1$ was maximal, we have $Q_1 = \varphi(\widehat{Q})$ and so
    \begin{gather*}
\mathrm{REL}_s(\sum_r \overline{c}_{1,i} x_r,\dots,\sum_r \overline{c}_{n_{I''},i} x_r,\bsq,\bvvqbase)  \quad \text{for} \ s=1,\dots,k
\end{gather*}
    are elements of $Q_1$. Moreover, by the definition of $S_{\rm ineqs}$ we have that 
    $$
    \mathrm{REL}^{\neq}(\sum_r \overline{c}_{1,i} x_r,\dots,\sum_r \overline{c}_{n_{I''},i} x_r,\bsq,\bvvqbase) \notin Q_1 .
    $$
    We conclude that 
\[
\begin{array}{l}
\mathrm{REL}_s(\sum_r \overline{c}_{1,i} X_{1,r},\dots,\sum_r \overline{c}_{n_{I''},i} X_{1,r},\bsq,\bvvqbase)  \in Q \\[0.5em]
\mathrm{REL}^{\neq}(\sum_r \overline{c}_{1,i} X_{1,r},\dots,\sum_r \overline{c}_{n_{I''},i} X_{1,r},\bsq,\bvvqbase) \notin Q
\end{array}
\]
    and so Lemma~\ref{lem:condcijhom} implies that the map 
   \begin{eqnarray*}
    \eta\colon D^{-1}\field\{ \bss(\bvv) , \bvvbase \} \{ y_1^{I''},\dots,y_{n_{I''}}^{I''} \} & \to & \difffield [\GL_{n_{I''}}] / Q\,,\\
    \bss(\bvv) & \mapsto & \bsq\,,\\
    \bvvbase & \mapsto & \bvvqbase\,,\\
    y_i^{I''} & \mapsto & \sum_{j=1}^{n_{I''}} \overline{c}_{j,i} \overline{X}_{1,j} 
    \end{eqnarray*}
    is a differential homomorphism. 
    By the definition of $S_{\rm ineqs}$ the denominators of
    $\exprforexp_j^{I''}( y_1^{I''},\dots,y_{n_{I''}}^{I''})$,  $V_j^{I''}( y_1^{I''},\dots,y_{n_{I''}}^{I''})$ for $j = 1, \ldots,l$ and $\exprforint_i^{I''}( y_1^{I''},\dots,y_{n_{I''}}^{I''})$ for $1\leq i \leq m$ with $\beta_i \in \leviroots^-$ do not lie in the kernel of $\eta$. Thus, we can extend $\eta$ to localizations of its domain and codomain respectively such that these denominators 
    and their images are contained in the respective multiplicatively closed subsets and denote this map again by $\eta$.  Then
    \[
    \begin{array}{rcll}
    \eta(\exprforexp_j^{I''}( y_1^{I''},\dots,y_{n_{I''}}^{I''})) \! & \! = \! & \! \ratexp_j + Q,\\[0.5em]
    \displaystyle\eta (V_j^{I''}( y_1^{I''},\dots,y_{n_{I''}}^{I''})) \! & \! = \! & \! \widehat{v}_j + Q,\\[0.5em]
    \eta(\exprforint_i^{I''}(y_1^{I''},\dots,y_{n_{I''}}^{I''})) \! & \! = \! & \! \ratint_i + Q.
   \end{array}
   \]
 \end{proof}

\section{The Structure of the Reductive Part}\label{sec:structureofreductivepart}

In this section we extend the specialization $\sigred$ computed in Section~\ref{sec:specializingreductivepart} to a specialization 
\[
  \sigul\colon  D^{-1} \difffield \{ \bvv \}[\bexp,\bexp^{-1}, \bint ]   
  \to \underline{E},
\]
where the differential field $\underline{E}$ is the field of fractions of a certain differential integral domain $\underline{R}$ containing $\extfieldred$.
For maximal differential ideals $\Imax \unlhd \underline{R}$ we construct Picard-Vessiot extensions $\overline{\generalext} = \Frac(\underline{R}/\Imax)$ of $\difffield$,
each of which in turn allows us to construct the announced specialization
\[
\sigPV\colon D^{-1} \difffield \{ \bvv \}[\bexp,\bexp^{-1}, \bint ]   
 \to \overline{\generalext}.
\]

We study the relationship between the group $\underline{H}$ of differential $\difffield$-automorphisms of $\underline{E}$ and the differential Galois group $H$ of $\overline{\generalext}$ depending on the choice of $\Imax$.  
The group $\underline{H}$ is the semidirect product of the unipotent radical $\unirad(\parabolic_J)$ of the parabolic group $\parabolic_J$ computed in Section~\ref{sec:bound} and a Levi factor of the differential Galois group $H$.
The choice of $\Imax$ determines the unipotent radical $\unirad(H) \leq \unirad(\parabolic_J)$ of $H$ as well as which Levi groups of $\underline{H}$ are Levi groups of $H$.
More precisely, the Levi groups of $\underline{H}$ are all conjugate by elements of $\unirad(\parabolic_J)$. In case $\unirad(H)$ is properly contained in $\unirad(\parabolic_J)$, the set of Levi groups may decompose into different orbits under $\unirad(H)$. The Levi groups of $H$ will correspond to one such orbit.

For the construction of the extension $\sigul$ and $\underline{E}$ let $\intrad_i$ for $\beta_i \in \roots^- \setminus \leviroots^-$ be differential indeterminates over $\extfieldred$ and let $\bratint$ be the $m$-tuple whose $i$-th entry, for $\beta_i \in \leviroots^-$, is a rational function $\widehat{\intn}_i$ as in  Proposition~\ref{prop:sigmared},
and whose $i$-th entry with $\beta_i \in \roots^- \setminus \leviroots^-$ is the indeterminate $\intrad_i$. Let
\[
 \Iuni \, := \, \langle \intrad_i' - \text{integrand of} \, \intn_i \rangle  
\]
be the differential ideal in $\extfieldred\{ \intrad_i \mid \beta_i \in \roots^- \setminus \leviroots^-\}$, 
where in the integrand of $\intn_i$, that is $\intn_i'$, one substitutes $\bexp$, $\bvv$ and $\intn_i$ with $\beta_i \in \leviroots^-$ by the respective residue classes of $\bratexp$, $\widehat{\bvv}$ and $\widehat{\intn}_i$ in $\extfieldred$ and $\intn_i$ with $\beta_i \in \roots^- \setminus \leviroots^-$ by the respective
indeterminate $\intrad_i$.
We define
\[
\underline{R} \, := \, \extfieldred\{ \intrad_i \mid \beta_i \in \roots^- \setminus \leviroots^- \}/ \Iuni.
\]
\begin{remark}\label{rem:Rintegraldomain}
The differential ring $\underline{R}$ is an integral domain, since $\Iuni$
is generated by differential polynomials which involve the pairwise distinct highest derivatives $\intrad_i'$ only linearly with constant coefficient and the remaining terms lie in 
$$\extfieldred[ \intrad_j \mid \beta_j \in \roots^- \setminus \leviroots^- \ \text{with} \ | \height(\beta_j) | <  | \height(\beta_i) |]$$
implying triangularity.
We denote its field of fractions by $\underline{E}=\Frac(\underline{R})$.
\end{remark}
Then $\sigred$ extends to a differential homomorphism 
\begin{equation}\label{eqn:specializationreductivepart}
\begin{array}{rcl}
\sigul\colon D^{-1} \difffield \{ \bvv \}[\bexp,\bexp^{-1}, \bint ] \! & \! \to \! & \! \underline{E} \, , \\[0.5em]
\intn_i \! & \! \mapsto \! & \! \underline{\intrad}_i \quad \text{for} \quad  \beta_i \in \roots^- \setminus \leviroots^-.
\end{array}
\end{equation}
Since $\sigul$ is a differential homomorphism, the specialization of the generic fundamental matrix
\[\label{eqn:imageoffmundersigul}
\fmuq \, := \, \sigul(\fm)  \in \group(\underline{R})
\]
satisfies 
\[
\partial(\fmuq) \, = \, A_{\group}(\bsq) \, \fmuq.
\]
If we denote by $\underline{\bvv}$, $\underline{\bff}$, $\underline{\bexp}$ and $\underline{\bint}$ the residue classes of $\bvv$, $\bff$, $\bexp$ and $\bint$ in $\underline{R}$, respectively, then the matrix $\fmuq=\sigul(\fm)$ decomposes explicitly as 
\begin{equation}\label{eqn:bruhatoffundmatrixprechapter10}
\fmuq \, = \, \buu(\underline{\bvv},\underline{\bff}) \, n(\overline{w}) \, \btt(\underline{\bexp}) \, \buu(\underline{\bint}) .
\end{equation}
Note that $\underline{\bvv}$, $\underline{\bff}$, $\underline{\bexp}$ and $\underline{\intn}_i$ with $\beta_i \in \leviroots^-$ are elements of $\extfieldred$.

The differential ring $\underline{R}$ is not necessarily differentially simple, unless $\Iuni$ is a maximal differential ideal. If this is not the case, we can choose a maximal differential ideal $\Imax$ in $\underline{R}$
and then take the quotient
\[\label{eqn:definitionRq}
\overline{R} \, := \, \underline{R}/\Imax   
\]
obtaining a differentially simple ring $\overline{R}$. 
Since $\Imax$ is prime, $\overline{R}$ is an integral domain and we define 
\[\label{eqn:definitionEq}
 \overline{\generalext} \, := \, \Frac(\overline{R}).
\]
From $\sigul$ we construct our final extension of the differential homomorphism $\sigred$ to the differential homomorphism 
\begin{equation}\label{eqn:specializationreductivepartfinal}
\begin{array}{rcl}
\sigPV\colon  D^{-1} \difffield \{ \bvv \}[\bexp,\bexp^{-1}, \bint ] \! & \! \to \! & \! \overline{\generalext},\\[0.5em]
\intn_i \! & \! \mapsto \! & \! \underline{\intrad}_i + \Imax \quad \text{for} \quad  \beta_i \in \roots^- \setminus \leviroots^-.
\end{array}
\end{equation}
Moreover, we denote by $\bvq,\overline{\bff}$, $\overline{\bexp}$ and
$\overline{\bint}$ the residue classes of $\underline{\bvv},\underline{\bff}$, $\underline{\bexp}$ and $\underline{\bint}$ in $\overline{R}$, respectively. 
Note that the elements $\bvq, \overline{\bff}$, $\overline{\bexp}$ and
$\overline{\intn}_i$ for $\beta_i \in \leviroots^-$ are elements of $\extfieldred$ and
$\overline{\intn}_i = \overline{\intrad}_i$ for $\beta_i \in \roots^- \setminus \leviroots^-$ lies in  
$\overline{R} \setminus \extfieldred$. We let
\begin{equation}\label{eqn:bruhatoffundmatrixchapter10}
\fmspec \, = \, \buu(\bvq,\overline{\bff}) \, n(\overline{w}) \, \btt(\overline{\bexp})\, \buu(\overline{\bint})\,.  
\end{equation}

Recall from Section~\ref{sec:intoLieAlgParabolic} that we denoted the roots of $\leviroots^-$ by $\beta_{j_1},\dots,\beta_{j_k}$
and the roots of the complement $\roots^- \setminus \leviroots^-$ by 
$\beta_{j_{k+1}},\dots,\beta_{j_m}$. 
Similarly as in Remark~\ref{rem:genericYredYraddecomp} we apply Lemma~\ref{lem:separationradicalnew} to the matrix
$\buu(\underline{\bint})$ and $\buu(\overline{\bint})$ and obtain the decompositions
\begin{eqnarray*}
    \buu(\underline{\bint}) & = & u_{\beta_{j_1}}(\underline{\intn}_{j_1}) \cdots u_{\beta_{j_k}}(\underline{\intn}_{j_k}) \cdot u_{\beta_{j_{k+1}}}(\underline{y}_{j_{k+1}}) \cdots u_{\beta_{j_m}}(\underline{y}_{j_m}) \\
    \buu(\overline{\bint}) & = & u_{\beta_{j_1}}(\overline{\intn}_{j_1}) \cdots u_{\beta_{j_k}}(\overline{\intn}_{j_k}) \cdot u_{\beta_{j_{k+1}}}(\overline{y}_{j_{k+1}}) \cdots u_{\beta_{j_m}}(\overline{y}_{j_m})
\end{eqnarray*}
with $\underline{y}_{j_{k+1}},\dots,\underline{y}_{j_m} \in \underline{R} \setminus \extfieldred$ and $\overline{y}_{j_{k+1}},\dots,\overline{y}_{j_m} \in \overline{R}$. 
Thus the matrices 
  \begin{equation*}\label{eqn:definitionofYredradobuq}
  \begin{array}{rcl}
 \fmreduq & := &  \buu(\underline{\bvv},\underline{\bff}) \, n(\overline{w}) \, \btt(\underline{\bexp}) \, u_{j_1}(\underline{\intn}_{j_1}) \cdots u_{j_k}(\underline{\intn}_{j_k}) \in \group(\extfieldred) \, ,\\[0.5em]
 \fmredq & := &  \buu(\bvq,\overline{\bff}) \, n(\overline{w}) \, \btt(\overline{\bexp}) \, u_{j_1}(\overline{\intn}_{j_1}) \cdots u_{j_k}(\overline{\intn}_{j_k}) \in \group(\extfieldred) \, , \\[0.5em]
 \fmraduq & := & u_{\beta_{j_{k+1}}}(\underline{y}_{j_{k+1}}) \cdots u_{\beta_{j_m}}(\underline{y}_{j_m}) \in \unirad(\parabolic_J)(\underline{R}) \, ,\\[0.5em]
 \fmradq & := & u_{\beta_{j_{k+1}}}(\overline{y}_{j_{k+1}}) \cdots u_{\beta_{j_m}}(\overline{y}_{j_m}) \in \unirad(\parabolic_J)(\overline{R}) 
 \end{array}
 \end{equation*}
satisfy $\fmuq = \fmreduq \, \fmraduq$ and $\fmspec = \fmredq \, \fmradq$. Note that $\fmreduq=\fmredq$.

\begin{proposition}\label{prop:galaction} Using the above notation we have:
\begin{enumerate}
    \item\label{prop:galaction(a)}
    The differential field $\extfieldred$ is generated as a field over $\difffield$ by $\underline{\bvv}$, $\underline{\bff}$, $\underline{\bexp}$ and $\underline{\intn}_i$  with $\beta_i \in \leviroots^-$ and we have 
    \[
    \extfieldred \subset \Frac(\difffield[\fmuq,\det(\fmuq)^{-1}]) \, = \, \underline{E}\,.
    \]
    \item\label{prop:galaction(b)}
    The differential field $\extfieldred$ is generated as a field over $\difffield$ by $\bvq$, $\overline{\bff}$, $\overline{\bexp}$ and $\overline{\intn}_i$ with $\beta_i \in \leviroots^-$ and we have 
    $$ \extfieldred \subset \Frac(\difffield[\fmspec,\det(\fmspec)^{-1}]) \, = \, \overline{\generalext}.$$
    \item\label{prop:galaction(c)}
    The differential ring $\difffield[\fmspec,\det(\fmspec)^{-1}]$
    is a Picard-Vessiot ring over $\difffield$ for $A_{\group}(\bsq)$ with fundamental matrix $\fmspec$. 
    Its differential Galois group $H(\field)$ is a subgroup of $\parabolic_J(\field)$.
    \item\label{prop:galaction(d)}
    The differential ring $\extfieldred[\fmspec,\det(\fmspec)^{-1}]$ is a Picard-Vessiot ring over $\extfieldred$ for $A_{\group}(\bsq)$ with fundamental matrix $\fmspec$. Its differential Galois group is a subgroup of $\unirad(\parabolic_J)(\field)$.
\end{enumerate}
\end{proposition}

\begin{proof}
\ref{prop:galaction(a)}\
Recall from Theorem~\ref{cor:extensionsbyPI} that $\generalext^{\unirad(\parabolic_J)}$ is a Picard-Vessiot extension of $\difffield\langle \bss(\bvv)\rangle(\bpp)$ for 
\begin{equation}\label{eqn:LCLMsection10}
\LCLM(\bss(\bvv),\bvvbase,\partial) \, y \, = \, 0
\end{equation}
and that $\generalext^{\unirad(\parabolic_J)}$ is generated as a field by $\bvv$, $\bff$, $\bexp$ and $\intn_i$ with $\beta_i \in \leviroots^-$ over $\difffield\langle \bss(\bvv) \rangle(\bpp)$.
As a consequence we have that 
the basis elements $y^{I''}_1,\dots,y^{I''}_{n_{I''}}$ of the solution space of \eqref{eqn:LCLMsection10} and their derivatives can be expressed as rational functions in $\bvv$, $\bff$, $\bexp$ and $\intn_i$ with $\beta_i \in \leviroots^-$ over $\difffield\langle \bss(\bvv) \rangle (\bpp)$. 
From Proposition~\ref{prop:specializefactorization} \ref{item:specializefactorization3} we obtain that $\sigPV$ specializes this basis to a basis in $\extfieldred$ of the specialized least common left multiple.  
Since the specialized basis and its derivatives generate $\extfieldred$ as a field over $\difffield$, the same is true for $\overline{\bvv}=\underline{\bvv}$, $\overline{\bff}=\underline{\bff}$, $\overline{\bexp}=\underline{\bexp}$ and $\overline{\intn}_i=\underline{\intn}_i$ with $\beta_i \in \leviroots^-$.
From \cite[Lemma~4.2]{Seiss_Generic} and the Bruhat decomposition in \eqref{eqn:bruhatoffundmatrixprechapter10} we conclude that all parameters $\underline{\bvv}, \underline{\bff}$, $\underline{\bexp}$ and $\underline{\bint}$  are in $\Frac(\difffield[\fmuq,\det(\fmuq)^{-1}])$ and so $\extfieldred$ is contained in $\Frac(\difffield[\fmuq,\det(\fmuq)^{-1}])$.
Since $\underline{E}$ is generated as a field by $\underline{\intn}_i$ with $\beta_i \in \roots^- \setminus \leviroots^-$ over $\extfieldred$ and these elements are contained in $\Frac(\difffield[\fmuq,\det(\fmuq)^{-1}])$, it follows that 
$\underline{E}=\Frac(\difffield[\fmuq,\det(\fmuq)^{-1}])$.

\ref{prop:galaction(b)}\
Since $\difffield[\fmspec,\det(\fmspec)^{-1}] \subset \overline{\generalext}$, the ring
$\difffield[\fmspec,\det(\fmspec)^{-1}]$ is an integral domain and so we can consider $\Frac(\difffield[\fmspec,\det(\fmspec)^{-1}])$. As above we conclude from  \cite[Lemma~4.2]{Seiss_Generic} and the Bruhat decomposition in \eqref{eqn:bruhatoffundmatrixchapter10} that all parameters $\bvq, \overline{\bff}$, $\overline{\bexp}$ and $\overline{\bint}$ are in $\Frac(\difffield[\fmspec,\det(\fmspec)^{-1}])$.
Since $\overline{\bff}$, $\overline{\bexp}$ and $\overline{\intn}_i$ with $\beta_i \in \leviroots^-$ are the same elements in $\extfieldred$ as the elements $\underline{\bvv}$, $\underline{\bff}$, $\underline{\bexp}$ and $\underline{\intn}_i$ with $\beta_i \in \leviroots^-$, the first statement follows from \ref{prop:galaction(a)}. Since for $\beta_i \in \leviroots^-$ we have that $\overline{\intn}_i \in \Frac(\difffield[\fmspec,\det(\fmspec)^{-1}])$ and these elements generate $\overline{\generalext}$ as a field over $\extfieldred$, we conclude that $\overline{\generalext}=\Frac(\difffield[\fmspec,\det(\fmspec)^{-1}])$.

\ref{prop:galaction(c)}\
Recall that the constants of $\extfieldred$ are $\field$. Since $\underline{R}/\Imax=\overline{R}$ is differentially simple and finitely generated over $\extfieldred$, the constants of $\overline{\generalext}=\Frac(\difffield[\fmspec,\det(\fmspec)^{-1}])$ are $\field$.
Clearly $\fmspec$ is a fundamental matrix for $A_{\group}(\bsq)$ and its entries are contained in 
$\Frac(\difffield[\fmspec,\det(\fmspec)^{-1}])$.
Obviously, $\Frac(\difffield[\fmspec,\det(\fmspec)^{-1}])$ is generated as a field over $\difffield$ by the entries of $\fmspec$ and so
$\Frac(\difffield[\fmspec,\det(\fmspec)^{-1}])$ is a Picard-Vessiot extension of $\difffield$ for $A_{\group}(\bsq)$ with fundamental matrix $\fmspec$. It follows now from the proof of \cite[Proposition~1.22]{vanderPutSinger} that $\difffield[\fmspec, \det(\fmspec)^{-1}]$ is a Picard-Vessiot ring for $A_{\group}(\bsq)$ over $\difffield$ with fundamental matrix $\fmspec$.

Since by the construction of $\fmspec$ the parameters $\vq_{i_{r+1}},\dots,\vq_{i_{l}}$ are in $\difffield$, 
Theorem~\ref{thm:fixedfieldparabolic} implies that the differential Galois group $H$ is contained in $\parabolic_J$.

\ref{prop:galaction(d)} 
Since $\extfieldred \subset \Frac(\difffield[\fmspec,\det(\fmspec)^{-1}])$ by \ref{prop:galaction(b)}, we conclude that
\[
\Frac(\difffield[\fmspec,\det(\fmspec)^{-1}]) \, = \, \Frac(\extfieldred[\fmspec,\det(\fmspec)^{-1}])
\]
and so it follows from \ref{prop:galaction(c)}\ 
that $\Frac(\extfieldred[\fmspec,\det(\fmspec)^{-1}])$ 
is a Picard-Vessiot extension of $\extfieldred$ for $A_{\group}(\bsq)$ with fundamental matrix $\fmspec$.
Again, the proof of \cite[Proposition~1.22]{vanderPutSinger} shows that $ \extfieldred[\fmspec, \det(\fmspec)^{-1}]$ 
is a Picard-Vessiot ring over $\extfieldred$ for $A_{\group}(\bsq)$.
Let $\gamma$ be a differential $\extfieldred$-automorphism of $\Frac(\extfieldred[\fmspec,\det(\fmspec)^{-1}])$ and $C_{\gamma} \in \parabolic_J(\field)$ such that $\gamma(\fmspec) = \fmspec \, C_{\gamma}$.
Then we obtain
\[
 \gamma(\fmredq \, \fmradq) \, = \, \fmredq \, \gamma( \fmradq) \, = \, \fmredq \, \fmradq \, C_{\gamma} \, ,
\]
which is equivalent to $\fmradq^{-1} \gamma( \fmradq) = C_{\gamma}$. Since the algebraic group $\unirad(\parabolic_J)$ is defined over $\field$, we obtain that $\gamma(\fmradq) \in \unirad(\parabolic_J)(\Frac(\extfieldred[\fmspec,\det(\fmspec)^{-1}]))$ and so $C_{\gamma} \in \unirad(\parabolic_J)(\field)$. 
\end{proof}

Next we prove that $\fmredq$ is a fundamental matrix of a matrix differential equation over $\difffield$ and that it induces a representation of $\Gal_{\partial}(\extfieldred/ \difffield)$ which is contained in the standard Levi group of $\parabolic_J$.

\begin{proposition}\label{prop:reductivepartYred}
    Let 
    \[
    \Aprered \, := \, \dlog(\fmredq).  
    \]
    Then $\Aprered \in \Lie(\group)(\difffield)$ and $\extfieldred=\Frac(\difffield[\fmredq,\det(\fmredq)^{-1}])$ is a Picard-Vessiot extension of $\difffield$ for 
    $\Aprered$ with fundamental matrix $\fmredq$. Its differential Galois group $\Lred(\field)$ in the representation induced by $\fmredq$ is contained in the standard Levi group $\levi_J(\field)$ of $\parabolic_J(\field)$.
\end{proposition}

\begin{proof}
    We show that $\Aprered = \dlog (\fmredq) \in \Lie(\group)(\difffield)$.
    To this end we use a result from Section~\ref{sec:unipotentradical} where we shall compute reduction matrices to reduce $A_{\group}(\bsq)$ over an algebraic extension of $\difffield$. Let
    \[
    \overline{g}_1 \, := \, n(\overline{w})^{-1} \, u_{k_s}(\overline{x}_{k_s}) \cdots u_{k_1}(\overline{x}_{k_1}) \in \group(\difffield)
    \]
    be as in Proposition~\ref{prop:gaugetoAPJ}. Since $\fmreduq = \fmredq$, it follows with Proposition~\ref{prop:gaugetoAPJ} that $\overline{g}_1 \fmspec \in \parabolic_J(\overline{\generalext})$ and that it decomposes into a product
    \[
     \overline{g}_1 \, \fmspec \, = \, ( \overline{g}_1 \, \fmredq) \, \fmradq
    \]
    with $\overline{g}_1  \fmredq \in \levi_J(\extfieldred)$ and $\fmradq \in \unirad(\parabolic_J)(\overline{R})$.
    Its logarithmic derivative is
    \[
    \begin{array}{rcl}
    \dlog( (\overline{g}_1 \fmredq ) \fmradq)
    \! & \! = \! & \! \dlog(\overline{g}_1 \fmredq ) + 
    (\overline{g}_1  \fmredq ) \, \dlog(\fmradq) (\overline{g}_1   \fmredq )^{-1}\\[0.5em]
    \! & \! = \! & \! \overline{g}_1.A_{\group}(\bsq)\\[0.5em]
    \! & \! \in \! & \! \Lie(\group)(\difffield) \, = \, \Lie(\levi_J)(\difffield) \oplus \Lie(\unirad(\parabolic_J))(\difffield).
    \end{array}
    \]
    Since $\dlog(\overline{g}_1 \fmredq ) \in \Lie(\levi_J)$ and $(\overline{g}_1 \fmredq ) \, \dlog(\fmradq) (\overline{g}_1  \fmredq )^{-1} \in \Lie(\unirad(\parabolic_J))$ and the sum decomposition is direct, it follows from the $\difffield$-rationality of $\gauge{\overline{g}_1}{A_{\group}(\bsq)}$ that 
    $\dlog(\overline{g}_1 \fmredq ) \in \Lie(\levi_J)(\difffield)$. Gauge transforming $\dlog(\overline{g}_1 \fmredq )$ with $\overline{g}_1^{-1} \in \group(\difffield)$ yields that $\dlog(\fmredq)\in \Lie(\group)(\difffield)$.

    Clearly we have $\Frac(\difffield[\fmredq, \det(\fmredq)^{-1}]) \subset \extfieldred$. 
    The Bruhat decomposition of $\fmredq$ is given by 
    \[
    \fmredq \, = \, \buu(\bvq,\overline{\bff}) \, n(\overline{w}) \, \btt(\overline{\bexp}) \, u_{j_1}(\overline{\intn}_{j_1}) \cdots u_{j_k}(\overline{\intn}_{j_k})
    \]
    and it lies in the big cell. It follows then from \cite[Lemma~4.2]{Seiss_Generic} that all parameters
    $\bvq$, $\overline{\bff}$, $\overline{\bexp}$
    and $\overline{\intn}_i$ with $\beta_i \in \leviroots^-$ are elements of $\Frac(\difffield[\fmredq, \det(\fmredq)^{-1}])$. Since by Proposition~\ref{prop:galaction} \ref{prop:galaction(b)} these elements generate $\extfieldred$, we have that
    \[
    \Frac(\difffield[\fmredq, \det(\fmredq)^{-1}]) \, = \, \extfieldred\,.
    \]
    We conclude that $\extfieldred$ is a Picard-Vessiot extension of $\difffield$ for $\Aprered$ with fundamental matrix $\fmredq$.
    
    Let $\gamma$ be a differential $\difffield$-automorphism of $\extfieldred$ and let $C_{\gamma} \in \GL_n(\field)$ be such that $\gamma(\fmredq) = \fmredq C_{\gamma}$. Then we have $\gamma(\overline{g}_1 \, \fmredq) = \overline{g}_1 \, \fmredq \, C_{\gamma}$
    and since the standard Levi group $\levi_J$ of $\parabolic_J$ is defined over $\field$, we obtain from
    $\overline{g}_1 \, \fmredq \in \levi_J$ that $\gamma(\overline{g}_1 \, \fmredq)$ is also an element of $\levi_J$.
    Hence, $C_{\gamma} \in \levi_J(\field)$ and so the induced representation $\Lred(\field)$ of the differential Galois group is contained in $\levi_J(\field)$.
\end{proof}

For an $n \times n$ matrix $\widehat{\fm} =(\widehat{\fm}_{i,j})$ of indeterminates over $\difffield$ we consider now the substitution homomorphisms
\[
\begin{array}{rclrcl}
    \ulphi\colon \difffield[\widehat{\fm},\det(\widehat{\fm})^{-1}]
    \! & \! \to \! & \! \difffield[\fmuq,\det(\fmuq)^{-1}]\,, & \widehat{\fm}_{i,j} \! & \! \mapsto \! & \! \fmuq_{i,j}\,,\\[0.5em]
    \ovphi\colon \difffield[\widehat{\fm},\det(\widehat{\fm})^{-1}]
    \! & \! \to \! & \! \difffield[\fmspec,\det(\fmspec)^{-1}]\,, & \widehat{\fm}_{i,j} \! & \! \mapsto \! & \! \fmspec_{i,j} \,,\\[0.5em]
    \ovphi_{\rm red}\colon \difffield[\widehat{\fm},\det(\widehat{\fm})^{-1}]
    \! & \! \to \! & \! \difffield[ \fmredq,\det( \fmredq)^{-1}]\,, & \widehat{\fm}_{i,j} \! & \! \mapsto \! & \! (\fmredq)_{i,j}
\end{array}
\]
and we denote their kernels by $\ulQ$, $\ovQ$ and $\ovQ_{\rm red}$, respectively.

\begin{proposition}\label{prop:statementsaboutstabilzers}
Denote by $\Stab(\ulQ)$ and $\Stab(\ovQ)$ and $\Stab(\ovQred)$ the
stabilizer of the ideal $\ulQ$ and $\ovQ$ and $\ovQred$, respectively,
in $\GL_n(\field)$ for the action $g \mapsto \widehat{\fm}g$.
\begin{enumerate}
    \item\label{prop:statementsaboutstabilzers(a)} 
    Then $H(\field) = \Stab(\ovQ)$ and $\Lred(\field) = \Stab(\ovQred)$ are the differential Galois groups of the Picard-Vessiot 
    rings $\difffield[\fmspec,\det(\fmspec)^{-1}]$ and $\difffield[ \fmredq,\det( \fmredq)^{-1}]$,  
    respectively.
    \item\label{prop:statementsaboutstabilzers(b)}
    We define  
$$\underline{H}(\field) \, := \, \Stab(\ulQ)\, .$$ 
Then $\underline{H}(\field)$ is a linear algebraic group and it is the group of differential $\difffield$-automorphisms of $\difffield[\fmuq,\det(\fmuq)^{-1}]$. Moreover, we have $\ulQ \subset \ovQ$ and $\parabolic_J(\field) \geq \underline{H}(\field)\geq H(\field)$.
\end{enumerate}
\end{proposition}

\begin{proof}
\ref{prop:statementsaboutstabilzers(a)}\
Extending the derivation of $\difffield$ to $\difffield[\widehat{\fm},\det(\widehat{\fm})^{-1}]$ 
by $\partial(\widehat{\fm}) = A_{\group}(\bsq) \, \widehat{\fm}$ and by $\partial(\widehat{\fm}) = \Aprered \, \widehat{\fm}$ respectively, turns the surjective homomorphisms $\ovphi$ and $\ovphired$ into differential homomorphisms.
Thus, the ideals $\ovQ$ and $\ovQred$ are maximal differential ideals and so their stabilizers define the respective differential Galois groups.
Since the induced differential isomorphisms send fundamental matrices to fundamental matrices, we obtain that $H(\field) = \Stab(\ovQ)$ and $\Lred(\field) =\Stab(\ovQred)$.

\ref{prop:statementsaboutstabilzers(b)}\
Using the same arguments as in the proof which shows that the differential Galois group is a linear algebraic group (cf.\ \cite[Thm.~1.27 (1)]{vanderPutSinger}), we obtain that $\underline{H}$ is a linear algebraic group. 
By definition,
the differential homomorphism $\ovphi$ factors as $\psi \circ \ulphi$:
\[
\begin{tikzcd}
\difffield[\widehat{\fm},\det(\widehat{\fm})^{-1}] \arrow{r}{\ulphi}
\arrow[swap]{dr}{\ovphi} & \difffield[\fmuq,\det(\fmuq)^{-1}] \arrow{d}{\psi} \\
  & \difffield[\fmspec,\det(\fmspec)^{-1}]
\end{tikzcd} 
\]
Thus we have $\ulQ \subset \ovQ$ and so $\Stab(\ulQ) \geq \Stab(\ovQ)$.

Every matrix $g \in \Stab(\ulQ)$ induces by right multiplication on $\widehat{\fm}$ a differential $\difffield$-automorphism of $\difffield[\widehat{\fm},\det(\widehat{\fm})^{-1}]/\ulQ$. 
Since
\[
\rho\colon \difffield[\widehat{\fm},\det(\widehat{\fm})^{-1}]/\ulQ \to \difffield[\fmuq,\det(\fmuq)^{-1}], \ \widehat{\fm}_{i,j} + \ulQ \mapsto \fmuq_{i,j}
\]
is a differential $\difffield$-isomorphism mapping $\widehat{\fm} + \ulQ$ to $\fmuq$, the matrix $g$ also induces a differential $\difffield$-automorphism on $\difffield[\fmuq,\det(\fmuq)^{-1}]$ by $\fmuq \mapsto  \fmuq \, g$, i.e., $\rho$ is $\underline{H}$-equivariant.
Thus the values $\vq_{i_{r+1}},\dots,\vq_{i_l}$ of the parameters of the root groups $U_{-\alpha_{i_{r+1}}},\dots,U_{-\alpha_{i_l}}$ in the Bruhat decomposition of $\fmuq$ have to be the same in the Bruhat decomposition of $\fmuq \, g$.
By Lemma~\ref{lem:invariantreflections} this is only possible for elements of $\parabolic_J(\field)$ and so $\Stab(\ulQ)\leq \parabolic_J(\field)$.
\end{proof}

\begin{lemma}\label{lem:forrestrictiontobeinjective}
For a field $K \supset \field$ and $\bxx_1 \in K^m$, $\bxx_2 \in (K^{\times})^l$ and $\bxx_3 \in K^m$ let 
\[
Y \, = \, \buu(\bxx_1) \, n(\overline{w}) \, \btt(\bxx_2) \, \buu(\bxx_3) \, .
\]
Moreover, for $p_1$, $p_2 \in \parabolic_J(\field)$ let 
$\ell_1$, $\ell_2 \in \widetilde{\levi}(\field)$ and $\urelem_1$, $\urelem_2 \in \unirad(\parabolic_J)(\field)$ be the unique elements such that $p_1 = \ell_1 \, \urelem_1$ and $p_2 = \ell_2 \, \urelem_2$, where $\widetilde{\levi}$ is a Levi group of $\parabolic_J$. 
Assume that the coefficients $\baa_1,\baa_3,\boldsymbol{b}_1,\boldsymbol{b}_3 \in K^m$ and $\baa_2,\boldsymbol{b}_2 \in (K^{\times})^{l}$ in the Bruhat decompositions
\[
Y p_1 \, = \, \buu(\baa_1) \, n(\overline{w}) \, \btt(\baa_2) \, \buu(\baa_3) \qquad \text{and} \qquad 
Y p_2 \, = \, \buu(\boldsymbol{b}_1) \, n(\overline{w}) \, \btt(\boldsymbol{b}_2) \, \buu(\boldsymbol{b}_3)
\]
satisfy $\baa_1 = \boldsymbol{b}_1$, 
$\baa_2 = \boldsymbol{b}_2$ and 
$a_{3,i} = b_{3,i}$ for all $\beta_i \in \leviroots^-$.  
Then $\ell_1 = \ell_2$.
\end{lemma}

\begin{proof}
    Recall that $\buu(\baa_3)$ (resp.\ $\buu(\boldsymbol{b}_3)$) means a product of root group elements of fixed order with parameter values $\baa_3$ (resp.\ $\boldsymbol{b}_3$).
    Applying Lemma~\ref{lem:separationradicalnew} to $\buu(\baa_3)$ and $\buu(\boldsymbol{b}_3)$ we obtain
    \[
    \begin{array}{rcl}
        \buu(\baa_3) \! & \! = \! & \! u_{\beta_{j_1}}(a_{3,j_1}) \cdots u_{\beta_{j_k}}(a_{3,j_k}) \cdot u_{\beta_{j_{k+1}}}(y_{j_{k+1}}) \cdots u_{\beta_{j_m}}(y_{j_m}) \quad \text{and}\\[0.5em] 
        \buu(\boldsymbol{b}_3) \! & \! = \! & \! u_{\beta_{j_1}}(b_{3,j_1}) \cdots u_{\beta_{j_k}}(b_{3,j_k}) \cdot u_{\beta_{j_{k+1}}}(x_{j_{k+1}}) \cdots u_{\beta_{j_m}}(x_{j_m})\,,
    \end{array}
    \]
    where  
    \[
    \begin{array}{rcl}
    u_{\beta_{j_{k+1}}}(y_{j_{k+1}}) \cdots u_{\beta_{j_m}}(y_{j_m}) \! & \! \in \! & \! \unirad(\parabolic_J)(K) \quad \text{and}\\[0.5em]
    u_{\beta_{j_{k+1}}}(x_{j_{k+1}}) \cdots u_{\beta_{j_m}}(x_{j_m}) \! & \! \in \! & \! \unirad(\parabolic_J)(K) .
    \end{array}
    \]
    Since 
    $a_{3,i} = b_{3,i}$ for all $\beta_i \in \leviroots^-$, we have that 
    \[
    u_{\beta_{j_1}}(a_{3,j_1}) \cdots u_{\beta_{j_k}}(a_{3,j_k}) \, = \, u_{\beta_{j_1}}(b_{3,j_1}) \cdots u_{\beta_{j_k}}(b_{3,j_k})\,.
    \]
    Normality of $\unirad(\parabolic_J)(K)$ implies now that there exists $\widetilde{\urelem} \in \unirad(\parabolic_J)(K)$ such that
\[
\buu(\baa_3) \, \buu(\boldsymbol{b}_3)^{-1} \, = \, \widetilde{\urelem} \, \iff \, \buu(\baa_3) \, = \, \widetilde{\urelem} \, \buu(\boldsymbol{b}_3)\,.
\]
    Normality again implies that there exists $\urelem \in \unirad(\parabolic_J)(K)$ such that $\buu(\baa_3) = \buu(\boldsymbol{b}_3) \, \urelem$ and so we find that
    \begin{equation}\label{eqn:samereductivepart}
    Y p_1 \, = \, Y p_2 \, \urelem \, \iff \, \ell_1 \, \urelem_1 \, = \, \ell_2 \, \urelem_2 \, \urelem \, \iff \, \ell_2^{-1} \, \ell_1 \, = \, \urelem_2 \, \urelem \, \urelem_1^{-1}\,.
    \end{equation}
    Since for the two factors of the semidirect product $\parabolic_J(K) = \widetilde{\levi}(K)\ltimes \unirad(\parabolic_J)(K)$, we have that 
    $\widetilde{\levi}(K) \cap \unirad(\parabolic_J)(K) = \{ \mathrm{id} \}$ and $\ell_2^{-1} \ell_1 \in \widetilde{\levi}(K)$ and $\urelem_2 \, \urelem \, \urelem_1^{-1} \in \unirad(\parabolic_J)(K)$, it follows from \eqref{eqn:samereductivepart} that 
    $\ell_1 = \ell_2$ and $\urelem = \urelem_2^{-1} \, \urelem_1$.
\end{proof}

The following lemma shows that the group of differential $\difffield$-automorphisms of $\underline{E}$ allows to recover the differential Galois group of the reductive part $\extfieldred$ over $\difffield$.

\begin{lemma}\label{lem:restrictionEred}
Denote by $\Aut_{\partial}(\underline{E}/\difffield)$ the group of differential $\difffield$-automorphisms of $\underline{E}$ induced by right multiplication of elements of $\underline{H}(\field)$ on $\fmuq$. 
Then every $\gamma \in \Aut_{\partial}(\underline{E}/\difffield)$ restricts to a differential $\difffield$-automorphism of $\extfieldred$. Moreover, the map
\begin{equation}\label{eqn_surjectivrestrictionmapofgaloisgrps}
\mathrm{Aut}_{\partial}(\underline{E}/\difffield) \twoheadrightarrow \Gal_{\partial}(\extfieldred/\difffield) , \ \gamma \mapsto \gamma \big|_{\extfieldred}
\end{equation}
is a surjective group homomorphism.
\end{lemma}

\begin{proof}
Since the extension $\generalext$ of $\generalext^{\parabolic_J}$ is a Picard-Vessiot extension with differential Galois group $\parabolic_J(\field)$ and since $\unirad(\parabolic_J)(\field)$ is normal in $\parabolic_J(\field)$, every differential $\generalext^{\parabolic_J}$-automorphism $\gamma$ of $\generalext$ restricts to a differential $\generalext^{\parabolic_J}$-automorphism of $\generalext^{\unirad(\parabolic_J)}$.
Let $\bxx_1 \in \generalext^m$, $\bxx_2 \in (\generalext^{\times})^l$ and $\bxx_3 \in \generalext^m$ be the coefficients of the Bruhat decomposition  
\begin{eqnarray}
\nonumber
\gamma(\fm) \! & \! = \! & \! \fm \, g \, = \, \buu(\bvv,\bff) \, n(\overline{w}) \, \btt(\bexp) \, \buu(\bint) \, g \\ \label{eq:specializebruhat}
 \! & \! = \! & \! \buu(\bxx_1) \, n(\overline{w}) \, \btt(\bxx_2) \, \buu(\bxx_3) \\
 \nonumber
 \! & \! = \! & \! \buu(\gamma(\bvv,\bff)) \, n(\overline{w}) \, \btt(\gamma(\bexp)) \, \buu(\gamma(\bint)) \, .
\end{eqnarray}
Since by Theorem~\ref{cor:extensionsbyPI} \ref{cor:extensionsbyPI(b)}\
the elements $\bvv$, $\bff$, $\bexp$ and $\intn_i$ with $\beta_i \in \leviroots^-$ generate $\generalext^{\unirad(\parabolic_J)}$ over $\generalext^{\parabolic_J}$ and so in particular lie in $\generalext^{\unirad(\parabolic_j)}$, we obtain that $\bxx_1$, $\bxx_2$ and $x_{3,i}$ with $\beta_i \in \leviroots^-$ are again elements of $\generalext^{\unirad(\parabolic_J)}$, i.e., they are rational expressions over $\generalext^{\parabolic_J}$ in  $\bvv$, $\bff$, $\bexp$ and $\intn_i$ with $\beta_i \in \leviroots^-$.

We are going to show below that for every $g \in \underline{H}(\field) \leq \parabolic_J(\field)$ we can specialize the elements $\bxx_1$, $\bxx_2$ and $\bxx_3$ to $\underline{E}$. Since $\underline{\bvv}$, $\underline{\bff}$, $\underline{\bexp}$ and $\underline{\intn}_i$
with $\beta_i \in \leviroots^-$ are elements of $\extfieldred$, it will then follow that the specialized elements $\bxx_1$, $\bxx_2$ and $x_{3,i}$ with $\beta_i \in \leviroots^-$ are again in $\extfieldred$.
Since according to Proposition~\ref{prop:galaction} \ref{prop:galaction(a)}\
the elements $\underline{\bvv}$, $\underline{\bff}$ $\underline{\bexp}$ and $\underline{\intn}_i$ with $\beta_i \in \leviroots^-$ generate $\extfieldred$ over $\difffield$, we conclude that the differential $\difffield$-automorphism of $\underline{E}$ induced by $g$ restricts to a differential $\difffield$-automorphism of $\extfieldred$. 

Recall from \eqref{eqn:specializationreductivepart} the specialization
\[
\sigul\colon D^{-1}\difffield\{\bvv\}[\bexp,\bexp^{-1},\bint] \to \underline{E}\,.
\]
Since the entries of $\fm$ are contained in $D^{-1} \difffield\{\bvv\}[\bexp,\bexp^{-1},\bint]$ and the entries of $\fmuq$ are contained in the image of $\sigul$, we obtain a specialization 
\[
\begin{array}{rcl}
\widetilde{\sigma}\colon  \difffield[\fm ,\det (\fm )^{-1}] \! & \! \to \! & \! \Frac(\difffield[\fmuq,\det (\fmuq)^{-1}]) \, = \, \underline{E}\,,\\[0.5em]
\fm \, = \, \buu(\bvv,\bff) \, n(\overline{w}) \, \btt(\boldsymbol{\exp}) \, \buu(\bint) \! & \! \mapsto \! & \!
\buu(\underline{\bvv},\underline{\bff}) \, n(\overline{w}) \, \btt(\underline{\bexp}) \, \buu(\underline{\bint}) \, = \, \fmuq\,.
\end{array}
\]
Since for $g\in \underline{H}(\field) \leq \parabolic_J(\field)$ and $a(\fm) \in \difffield[\fm ,\det (\fm )^{-1}]$ we have  
\[
\widetilde{\sigma}(\gauge{a(\fm)}{g}) \, = \, \widetilde{\sigma}(a(\fm g)) \, = \, a(\fmuq g) \, = \, \gauge{a(\fmuq)}{g} \, = \, \widetilde{\sigma}(a(\fm)).g \, ,
\] 
we conclude that $\widetilde{\sigma}$ is an $\underline{H}$-equivariant differential $\difffield$-homomorphism.
Thus, the kernel $\widetilde{Q}$ of $\widetilde{\sigma}$ is stabilized by $\underline{H}(\field)$.

Let $\field[\group] = \field[\overline{X},\det(\overline{X})^{-1}]$ be the coordinate ring of $\group$.
We are going to show that the Bruhat decomposition
\[
\overline{X} \, = \, \buu(\bxx) \, n(\overline{w}) \, \btt(\bzz) \, \buu(\byy)
\]
specializes to the Bruhat decompositions of both $\fm$ and $\fmuq$, where $\bxx$, $\bzz$ and $\byy$ are rational functions over $\field$ in the coordinates $\overline{X}$ of $\group$.
The varieties $\unipotent^-\times \torus \times \unipotent^-$ and $\group$ are birationally equivalent; more precisely, according to the proof of \cite[Lemma~4.2]{Seiss_Generic} the product morphism
\[
\varphi\colon \unipotent^- \times \torus \times \unipotent^- \to \group , \, (u_1,t,u_2) \mapsto u_1 \, n(\overline{w}) \, t \, u_2
\]
is an isomorphism onto the open subset $\unipotent^-n(\overline{w}) \borel^-$ of $\group$. Thus, the denominators of these rational functions do not vanish at any point of $\unipotent^-(K) \, n(\overline{w}) \, \borel^-(K)\subset \group(K)$ for any field extension $K$ of $\field$.
Since 
\[
\fm \in \unipotent^-(\generalext) \, n(\overline{w}) \, \borel^-(\generalext)\subset \group(\generalext) \quad \text{and} \quad \fmuq \in \unipotent^-(\underline{E}) \, n(\overline{w}) \, \borel^-(\underline{E})\subset \group(\underline{E}) \, ,
\]
the coefficients $\bvv,\bff,\boldsymbol{\exp},\bint$ and $\underline{\bvv},\underline{\bff},\underline{\boldsymbol{\exp}},\underline{\bint}$ of the respective Bruhat decomposition are obtained by evaluating $\bxx$, $\bzz$ and $\byy$ at $\fm$ and $\fmuq$, respectively.

Let $\widetilde{p}$ be one of the components of $\bxx$, $\bzz$, $\byy$ and let $a_1(\overline{X})$, $a_2(\overline{X}) \in \field[\overline{X},\det(\overline{X})^{-1}]$ such that
$\widetilde{p} = a_1(\overline{X})/a_2(\overline{X})$. Then by the above we obtain the respective component $p$ of $\bvv,\bff,\bexp$, $\bint$ and the respective component $\underline{p}$ of $\underline{\bvv},\underline{\bff},\underline{\bexp},\underline{\bint}$ by evaluating $\widetilde{p}$ at $\fm$ and $\fmuq$ respectively, i.e., we have  
\[
p \, = \, \widetilde{p}(\fm) \, = \, \frac{a_1(\fm)}{a_2(\fm)} 
\quad \text{and} \quad  \underline{p} \, = \, \widetilde{p}(\fmuq) \, = \, \frac{a_1(\fmuq)}{a_2(\fmuq)}
\]
with $a_2(\fm)\neq 0$ and $a_2(\fmuq)\neq 0$.
Since $\widetilde{\sigma}(\fm)=\fmuq$, we have $\widetilde{\sigma}(a_2(\fm))=a_2(\fmuq) \neq 0$ implying that $a_2(\fm) \notin \widetilde{Q}$.
For $g \in \underline{H}(\field)$ the induced differential $\difffield$-automorphism $\gamma$ maps $a_1(\fm)/a_2(\fm)$ to
\[
\gamma(p) \, = \,
\gamma(a_1(\fm)/a_2(\fm)) \, = \, x_{k,i} \, = \, a_1(\fm g)/ a_2(\fm g) \quad \text{(cf.\ \eqref{eq:specializebruhat})}.
\]
Since $\widetilde{Q}$ is stabilized by $\underline{H}(\field)$ and $a_2(\fm) \notin \widetilde{Q}$, we conclude that for every $g \in \underline{H}(\field)$ the element $a_2(\fm g) \notin \widetilde{Q}$ and so $\widetilde{\sigma}(a_2(\fm g))\neq 0$. Hence, for every element of $\underline{H}(\field)$ we can extended $\widetilde{\sigma}$ to a localization of its domain containing $\bxx_1$, $\bxx_2$ and $\bxx_3$.

It is left to show the surjectivity of the restriction map in \eqref{eqn_surjectivrestrictionmapofgaloisgrps}.
Since $\unirad(H)(\field)$ is normal in $H(\field)$ and $\overline{\generalext}^{\unirad(H)} = \extfieldred$, we obtain that the map
\[
\Gal_{\partial}(\overline{\generalext}/\difffield) \twoheadrightarrow \mathrm{Gal}_{\partial}(\extfieldred/\difffield) , \ \gamma \mapsto \gamma \big|_{\extfieldred}
\]
is surjective.
So for $\gamma_{\rm red} \in \mathrm{Gal}_{\partial}(\extfieldred/\difffield)$ let $g \in H(\field)$ be such that the automorphism $\gamma \in \Gal_{\partial}(\overline{\generalext}/\difffield)$ induced by $g$ restricts to $\gamma|_{\extfieldred}=\gamma_{\rm red}$.
Since $g \in H(\field) \leq \underline{H}(\field) \leq \parabolic_J(\field)$, the induced differential $\difffield$-automorphism $\underline{\gamma}$ of $\underline{E}$ restricts by the first part of our proof to a differential $\difffield$-automorphism $\underline{\gamma}|_{\extfieldred}$ of $\extfieldred$.
We show that $\underline{\gamma}|_{\extfieldred}$ coincides with $\gamma|_{\extfieldred} = \gamma_{\rm red}$ proving surjectivity.
To this end we prove that the images of the generators
of $\extfieldred$ under $\underline{\gamma}|_{\extfieldred}$ and $\gamma|_{\extfieldred}$ agree using the uniqueness of the Bruhat decomposition.
Consider the Bruhat decompositions of $\underline{\gamma}(\fmuq) = \fmuq \, g$ and $\gamma(\fmspec)=\fmspec \, g$, i.e.\
\[
\begin{array}{rcl}
\fmuq \, g \!& \! = \! & \! \buu(\underline{\bxx}_1) \, n(\overline{w}) \, \btt(\underline{\bxx}_2) \, \buu(\underline{\bxx}_3)\\[0.2em]
\! & \! = \! & \! \buu(\underline{\gamma}(\underline{\bvv},\underline{\bff})) \, n(\overline{w}) \, \btt(\underline{\gamma}(\underline{\bexp})) \, \buu(\underline{\gamma}(\underline{\bint}))
\qquad \text{and}\\[0.5em]
\fmspec \, g \! & \! = \! & \! \buu(\overline{\bxx}_1) \, n(\overline{w}) \, \btt(\overline{\bxx}_2) \, \buu(\overline{\bxx}_3) \\[0.2em]
\! & \! = \! & \! \buu(\gamma(\overline{\bvv},\overline{\bff})) \, n(\overline{w}) \, \btt(\gamma(\overline{\bexp})) \, \buu(\gamma(\overline{\bint}))\,,
\end{array}
\]
and the differential $\difffield$-homomorphism 
\[
\widehat{\sigma}\colon \difffield[\fmuq,\det(\fmuq)^{-1}] \to \difffield [\fmspec,\det(\fmspec)^{-1}], \, \fmuq \mapsto \fmspec
\]
completing the following commutative diagram:
\[
\begin{tikzcd}
\difffield[\fm,\det(\fm)^{-1}] \arrow{rrr}{\sigul|_{\difffield[\fm,\det(\fm)^{-1}]}}
\arrow[swap]{drrr}{\sigPV|_{\difffield[\fm,\det(\fm)^{-1}]}} & & & \difffield[\fmuq,\det(\fmuq)^{-1}] \arrow[dashed]{d}{\widehat{\sigma}} \\
& & & \difffield[\fmspec,\det(\fmspec)^{-1}]
\end{tikzcd} 
\]
As in case of $\widetilde{\sigma}$ one proves that $\widehat{\sigma}$ is $H$-equivariant and that one can extend $\widehat{\sigma}$ to the parameters $\underline{\bvv}$, $\underline{\bff}$, $\underline{\bexp}$, $\underline{\bint}$ and $\underline{\bxx}_1$, $\underline{\bxx}_2$, $\underline{\bxx}_3$  of the Bruhat decompositions of $\fmuq$ and $\fmuq g$ respectively.
Since $\widehat{\sigma}(\fmuq)=\fmspec$, we conclude with the uniqueness of the Bruhat decomposition that 
\[
\widehat{\sigma}(\underline{\bvv})= \bvq, \quad \widehat{\sigma}(\underline{\bff}) \, = \, \overline{\bff} , \quad \widehat{\sigma}(\underline{\bexp}) \, = \, \overline{\bexp}, \quad
\widehat{\sigma}(\underline{\bint}) \, = \, \overline{\bint}\,.
\]
By construction of  $\underline{E}$ and $\overline{\generalext}$ we have 
\[
\underline{\bvv} \, = \, \bvq, \quad
\underline{\bff} \, = \, \overline{\bff} , \quad
\underline{\bexp} \, = \, \overline{\bexp}, \quad
\underline{\intn}_i \, = \, \overline{\intn}_i \quad \text{with} \ \beta_i \in \leviroots^- 
\]
in $\extfieldred$ and since these elements generate $\extfieldred$, it follows that $\widehat{\sigma}$ is the identity on $\extfieldred$. Hence, we obtain 
\[
\begin{array}{rcl}
\widehat{\sigma}(\fmuq \, g) \! & \! = \! & \! \buu(\widehat{\sigma}(\underline{\bxx}_1)) \, n(\overline{w}) \, \btt(\widehat{\sigma}(\underline{\bxx}_2)) \, \buu(\widehat{\sigma}(\underline{\bxx}_3)) \, = \,  
\buu(\underline{\bxx}_1) \, n(\overline{w}) \, \btt(\underline{\bxx}_2) \, \buu(\widehat{\sigma}(\underline{\bxx}_3))\\[0.5em]
\! & \! = \! & \! \fmspec \, g \, = \, \buu(\overline{\bxx}_1) \, n(\overline{w}) \, \btt(\overline{\bxx}_2) \, \buu(\overline{\bxx}_3),  
\end{array}
\]
where the $i$-th entry in $\widehat{\sigma}(\underline{\bxx}_3)$ such that $\beta_i \in \leviroots^-$
satisfies $\widehat{\sigma}(\underline{x}_{3,i})=\underline{x}_{3,i}$.
It follows then from the uniqueness of the Bruhat decomposition that the images of the generators of $\extfieldred$ under  $\underline{\gamma}|_{\extfieldred}$ and $\gamma|_{\extfieldred}$ coincide.
\end{proof}

\begin{remark}\label{rem:fiexedelementsunderRu}
Let $\urelem \in \unirad(\parabolic_J)(\field)$. It follows from Lemma~\ref{lem:bruhatunipotent} that the coefficients
$\bxx_1$, $\bxx_2$ and $\bxx_3$ in the Bruhat decomposition
\[
\fmuq \, \urelem  \, = \, \buu(\underline{\bvv},\underline{\bff}) \, n(\overline{w}) \, \btt(\underline{\bexp}) \, \buu(\underline{\bint}) \, \urelem \, = \, \buu(\bxx_1) \, n(\overline{w}) \, \btt(\bxx_2) \, \buu(\bxx_3)
\]
satisfy $\bxx_1 = (\underline{\bvv},\underline{\bff})$, $\bxx_2 = \underline{\bexp}$ and $x_{3,i}=\underline{\intn}_i$ 
with $\beta_i \in \leviroots^-$. Hence, the unipotent radical $\unirad(\parabolic_J)(\field)$ fixes the elements $\underline{\bvv}$, $\underline{\bff}$, $\underline{\bexp}$ and $\underline{\intn}_i$ 
with $\beta_i \in \leviroots^-$ and so leaves $\extfieldred$ elementwise fixed.
\end{remark}

 The following theorem shows that the Levi groups of $\underline{H}(\field)$ are conjugate by elements of $\unirad(\parabolic_J)(\field)$ to the Levi groups of $H(\field)$. But not all Levi groups of $\underline{H}(\field)$ are Levi groups of $H(\field)$. Only those Levi groups of $\underline{H}(\field)$ appear as Levi groups of $H(\field)$ which stabilize the ideal $\Imax \lhd \underline{R}$, where $g \in \underline{\levi}(\field)$ acts on  $\underline{R}$ as usually by
 \[
\fmuq \, g \, = \, \buu(\underline{\bvv},\underline{\bff}) \, n(\overline{w}) \, \btt(\underline{\bexp}) \, \buu(\underline{\bint}) \, g  
\, = \, \buu(\bxx_1) \, n(\overline{w}) \, \btt(\bxx_2) \, \buu(\bxx_3). 
\]

\begin{theorem}\label{thm:Levigroupsconj}
Let $\underline{\levi}(\field)$ and $\levi(\field)$ be Levi groups of $\underline{H}(\field)$ and $H(\field)$, respectively. Let $\Lred$ be as in Proposition~\ref{prop:reductivepartYred}.
\begin{enumerate}
    \item\label{thm:Levigroupsconj(a)} 
    We have $\unirad(\underline{H})(\field) = \unirad(\parabolic_J)(\field)$ and $\underline{H}(\field) = \underline{\levi}(\field) \ltimes \unirad(\parabolic_J)(\field)$ is a Levi decomposition of $\underline{H}(\field)$.
    \item\label{thm:Levigroupsconj(b)} 
    The group $\Lred(\field)$ is contained in $\underline{H}(\field)$.
    \item\label{thm:Levigroupsconj(c)}
    The groups $\Lred(\field)$ and $\levi(\field)$ are Levi groups of $\underline{H}(\field)$, i.e., the groups $\underline{\levi}(\field)$, $\levi(\field)$ and $\Lred(\field)$ are conjugate by elements in $\unirad(\parabolic_J)(\field)$.
\end{enumerate}
\end{theorem}

\begin{proof}
    \ref{thm:Levigroupsconj(a)}\
    We need only to show that $\unirad(\underline{H})(\field) = \unirad(\parabolic_J)(\field)$. 
    According to Proposition~\ref{prop:parabolicbound} \ref{prop:E(d)} the group $\levi(\field)$ is $\widetilde{\levi}(\field)$-irreducible for a Levi group $\widetilde{\levi}(\field)$ of $\parabolic_J(\field)$. 
    As $\levi(\field)$ is a reductive subgroup of $\underline{H}(\field)$ by Proposition~\ref{prop:statementsaboutstabilzers} \ref{prop:statementsaboutstabilzers(b)}, there is a Levi group $\widetilde{\underline{\levi}}(\field)$ such that
    \[
    \levi(\field) \leq \widetilde{\underline{\levi}}(\field) \leq \parabolic_J(\field).
    \]
    Minimality of $\parabolic_J(\field)$ for $\levi(\field)$ implies minimality of $ \parabolic_J(\field)$ for $\widetilde{\underline{\levi}}(\field)$. Thus,  $\widetilde{\underline{\levi}}(\field)$ is $\widetilde{\levi}(\field)$-irreducible for a Levi group $\widetilde{\levi}(\field)$ of $\parabolic_J(\field)$ and thus $\unirad(\underline{H})(\field) \leq \unirad(\parabolic_J)(\field)$.
    The group $\unirad(\parabolic_J)(\field)$ acts on $\underline{R}$ leaving $\extfieldred$ elementwise fixed by Remark~\ref{rem:fiexedelementsunderRu}. Since $\underline{\intn}_i \in\underline{E}$ for $\beta_i \in \roots^- \setminus \leviroots^-$ are transcendental
   over $\extfieldred$, we conclude that 
    $\unirad(\parabolic_J)(\field) \leq \underline{H}(\field)$ and so $\unirad(\underline{H})(\field) = \unirad(\parabolic_J)(\field)$. 

\ref{thm:Levigroupsconj(b)}\
Let $\ell \in \Lred$. Then $\ell$ induces a differential $\difffield$-automorphism $\gamma$ of $\extfieldred$ by right multiplication on $\fmreduq = \fmredq$.
It follows from Lemma~\ref{lem:restrictionEred} that there exists $g \in \underline{H}$ such that the restriction of the induced differential $\difffield$-automorphism $\underline{\gamma}$ of $\underline{E}$ is equal to $\gamma$, that is $\underline{\gamma}|_{\extfieldred} = \gamma$. Thus 
\[
\underline{\gamma}|_{\extfieldred}(\fmreduq) \, = \, \gamma(\fmreduq) \, = \, \fmreduq \, \ell\,.
\]
We compute 
\[
\underline{\gamma}(\fmreduq \, \fmraduq) \, = \, \underline{\gamma}|_{\extfieldred}(\fmreduq) \, \underline{\gamma}( \fmraduq) \, = \, \fmreduq \, \ell \, \underline{\gamma} (\fmraduq) \, = \,  \fmreduq \, \fmraduq \, g
\]
and obtain
\[
\ell \underline{\gamma} (\fmraduq) \, = \, \fmraduq \, g
\, \iff \, \ell \, = \, \fmraduq \, g \, \underline{\gamma} (\fmraduq)^{-1} \, .
\]
Since the defining ideal of $\unirad(\parabolic_J)$ is defined over $\field$ and $\fmraduq \in \unirad(\parabolic_J)(\underline{E})$, we conclude that $\underline{\gamma} (\fmraduq)^{-1} \in \unirad(\parabolic_J)(\underline{E})$.
Since $\unirad(\parabolic_J) \leq \underline{H}$ according to
\ref{thm:Levigroupsconj(a)}, it follows that $\ell \in \underline{H}(\field)$. 

\ref{thm:Levigroupsconj(c)}\
We start by showing first that any Levi group $\underline{\levi}(\field)$ of $\underline{H}(\field)$ is isomorphic to
$\Gal_{\partial}(\extfieldred/\difffield)$. As in Lemma~\ref{lem:restrictionEred} denote by $\Aut_{\partial}(\underline{E}/\difffield)$ the group of differential $\difffield$-automorphisms of $\underline{E}$ induced by right multiplication of elements of $\underline{H}(\field)$ on $\fmuq$.
We obtain from Lemma~\ref{lem:restrictionEred} that the group homomorphism
\[
\Aut_{\partial}(\underline{E}/\difffield) \twoheadrightarrow \Gal_{\partial}(\extfieldred/\difffield), \ \underline{\gamma} \mapsto \underline{\gamma} \big|_{\extfieldred}
\]
is surjective. We determine its kernel.
Suppose that $g_1,g_2 \in \underline{H}(\field)$ induce
$\underline{\gamma}_1$ and $\underline{\gamma}_2$, respectively, such that 
$\underline{\gamma}_1|_{\extfieldred} \, = \, \underline{\gamma}_2|_{\extfieldred}$.
Then the Bruhat decompositions of $\fmuq g_1$ and $\fmuq g_2$ satisfy the conditions of Lemma~\ref{lem:forrestrictiontobeinjective}.
The same lemma then implies that $g_1^{-1}g_2 \in \unirad(\parabolic_J)(\field)$. Moreover, for every $g \in \unirad(\parabolic_J)(\field) \leq \underline{H}(\field)$ the induced differential $\difffield$-automorphism restricts to the identity on $\extfieldred$.
Thus, the kernel consists of all $\underline{\gamma}$
with $\underline{\gamma}(\fmuq) = \fmuq g$ for some $g \in \unirad(\parabolic_J)(\field)$.
 Hence, 
 \[
 \underline{\levi}(\field) \, \cong \, \underline{H}(\field)/\unirad(\parabolic_J)(\field) \, \cong \, \Gal_{\partial}(\extfieldred/\difffield)\,.
 \]

Since $\Lred(\field)$ and $\levi(\field)$ are reductive subgroups of $\underline{H}(\field)$, there exist Levi groups $\underline{\levi}_1(\field)$ and $\underline{\levi}_2(\field)$ of $\underline{H}(\field)$ such that $\Lred(\field) \leq \underline{\levi}_1(\field)$ and $\levi(\field) \leq \underline{\levi}_2(\field)$. 
We have
\[
\Lred(\field) \, \cong \, \Gal_{\partial}(\extfieldred/\difffield) \quad \text{and} \quad \levi(\field) \, \cong \, H(\field)/\unirad(\parabolic_J)(\field) \, \cong \, \Gal_{\partial}(\extfieldred/\difffield)
\]
according to Proposition~\ref{prop:reductivepartYred} and the Fundamental Theorem of Differential Galois Theory, respectively.
Since by the first part of this proof any Levi group of $\underline{H}(\field)$ is isomorphic to $\Gal_{\partial}(\extfieldred/\difffield)$, we obtain 
$\Lred (\field) = \underline{\levi}_1 (\field)$ and $\levi (\field) = \underline{\levi}_2 (\field)$.
The last assertion is simply the fact that all Levi groups of $\underline{H}(\field)$ are conjugate by elements of $\unirad(H) (\field)= \unirad(\parabolic_J)(\field)$.
\end{proof}

The following proposition shows that every Levi group of $\underline{H}(\field)$ can be realized as a Levi group of $H$ by the choice of a suitable maximal differential ideal $\Imax \lhd \underline{R}$ in the construction of $\overline{\generalext}$ introduced at the beginning of Section~\ref{sec:structureofreductivepart}.
\begin{proposition}\label{cor:ForOurLeviGrpThereisanidealnew} $\,$
\begin{enumerate}
\item\label{cor:ForOurLeviGrpThereisanidealnew(a)} 
Every $\urelem \in \unirad(\parabolic_J)(\field)$ induces a differential $\extfieldred$-automorphism
\[
\varphi_{\urelem}\colon \underline{R} \to \underline{R}, \ \underline{\intn}_i \mapsto x_{3,i} \quad \text{for all} \quad \beta_i \in \roots^- \setminus \leviroots^- \, ,
\]
where $x_{3,i}$ are the respective coefficients of the Bruhat decomposition
\[
\fmuq \, \urelem \, = \, \buu(\bxx_1)\, n(\overline{w}) \, \btt(\bxx_2) \, \buu(\bxx_3) 
\qquad
\text{with} \, \,
\bxx_1, \bxx_3 \in \underline{E}^m \, \, \text{and} \, \, \bxx_2 \in (\underline{E}^{\times})^l\,.
\]
\item\label{cor:ForOurLeviGrpThereisanidealnew(b)} 
For every Levi group $\underline{\levi}(\field)$ of $\underline{H}(\field)$ 
there exists a maximal differential ideal $\Imax^{(1)} \lhd \underline{R}$ such that $\underline{\levi}(\field)$ is a Levi group of the differential Galois group $H^{(1)}(\field)$ of the respective Picard-Vessiot extension $\overline{\generalext}^{(1)}=\Frac(\underline{R}/\Imax^{(1)})$ of $\difffield$ with fundamental matrix $\fmspec^{(1)}$.
\item\label{cor:ForOurLeviGrpThereisanidealnew(c)}
Let $\Imax^{(1)}$ and $\Imax^{(2)}$ be two maximal differential ideals in $\underline{R}$ and denote by $\overline{\generalext}^{(1)}=\Frac(\underline{R}/\Imax^{(1)})$ and $\overline{\generalext}^{(2)}=\Frac(\underline{R}/\Imax^{(2)})$ the respective Picard-Vessiot extension of $\difffield$ with fundamental matrices $\fmspec^{(1)}$ and $\fmspec^{(2)}$.
Then there exists $\urelem \in \unirad(\parabolic_J)(\field)$ such that
\[
\varphi\colon \overline{\generalext}^{(1)} \to \overline{\generalext}^{(2)}, \ \fmspec^{(1)} \mapsto \fmspec^{(2)} \, \urelem
\]
is a differential $\difffield$-isomorphism with $\varphi(\fmspec_{\rm red}^{(1)}) = \fmspec_{\rm red}^{(2)}$. Moreover, the map $\varphi_{\urelem}$ from \ref{cor:ForOurLeviGrpThereisanidealnew(a)}\
maps $\Imax^{(1)}$ to $\Imax^{(2)}$.
\end{enumerate}
\end{proposition}

\begin{proof}
\ref{cor:ForOurLeviGrpThereisanidealnew(a)}\
Since by Theorem~\ref{thm:Levigroupsconj} \ref{thm:Levigroupsconj(a)}\
we have $\unirad(\underline{H})(\field)=\unirad(\parabolic_J)(\field)$, the element $\urelem$ induces a differential $\difffield$-automorphism $\varphi_{\urelem}$ of $\underline{E}$.
Since the fixed field of $\underline{E}$ under the action of $\unirad(\parabolic_J)(\field)$ is $\extfieldred$, it is a differential $\extfieldred$-automorphism of $\underline{E}$. Moreover, because $\underline{\intn}_i \in \underline{R}$ for all $1 \leq i \leq m$, the elements $x_{3,i}$ are also in $\underline{R}$. 

\ref{cor:ForOurLeviGrpThereisanidealnew(b)}\
In the construction of the Picard-Vessiot extension $\overline{\generalext} = \Frac(\underline{R}/\Imax)$ of $\difffield$ with differential Galois group $H(\field)$ at the beginning
of Section~\ref{sec:structureofreductivepart} we chose a maximal differential ideal $\Imax \lhd \underline{R}$.
Let $\levi(\field)$ be a Levi group of $H(\field)$.
By Theorem~\ref{thm:Levigroupsconj} \ref{thm:Levigroupsconj(c)}\  
there exists $\urelem \in \unirad(\parabolic_J)(\field)$ such that $\underline{\levi}(\field) = \urelem \levi(\field) \urelem^{-1}$.
According to \ref{cor:ForOurLeviGrpThereisanidealnew(a)}, 
the map $\varphi_{\urelem}$ is a differential $\extfieldred$-automorphism of $\underline{R}$ and so the image $\Imax^{(1)}$ of $\Imax$ under $\varphi_{\urelem}$ is a maximal differential ideal of $\underline{R}$. Then, $\urelem$ induces the differential $\difffield$-isomorphism 
\[
\overline{\generalext} \to \overline{\generalext}^{(1)}, \ \fmspec \mapsto \fmspec^{(1)} \urelem\,.
\]
The differential Galois group $H^{(1)}(\field)$ of $\overline{\generalext}^{(1)}$ then satisfies $H^{(1)}(\field) = \urelem H(\field) \urelem^{-1}$.
Hence, with  $\underline{\levi}(\field) = \urelem \levi(\field) \urelem^{-1}$ we conclude that $\underline{\levi}(\field)$ is a Levi group of $H^{(1)}(\field)$.  

\ref{cor:ForOurLeviGrpThereisanidealnew(c)}\
According to Proposition~\ref{prop:galaction} \ref{prop:galaction(d)}\
the differential fields $\overline{\generalext}^{(1)}$ and $\overline{\generalext}^{(2)}$ are also Picard-Vessiot extensions of $\extfieldred$ for $A_{\group}(\bsq)$.
Thus, there exists $\urelem \in \GL_n(\field)$ such that the map
\[
\varphi\colon \overline{\generalext}^{(1)} \to \overline{\generalext}^{(2)}, \ \fmspec^{(1)} \mapsto \fmspec^{(2)} \urelem
\]
is a differential $\extfieldred$-isomorphism. Since by construction $\fmspec^{(1)}_{\rm red} = \fmspec^{(2)}_{\rm red}$, we conclude that $\varphi(\fmspec^{(1)}_{\rm red}) = \fmspec^{(2)}_{\rm red}$.
Then, from
\[
\fmspec^{(2)}_{\rm red} \, \fmspec^{(2)}_{\rm rad} \, \urelem \, = \, \varphi (\fmspec^{(1)}_{\rm red} \, \fmspec^{(1)}_{\rm rad}) \, = \, \varphi(\fmspec^{(1)}_{\rm red}) \, \varphi (\fmspec^{(1)}_{\rm rad} ) \, = \,  \fmspec^{(2)}_{\rm red} \, \varphi (\fmspec^{(1)}_{\rm rad} )
\]
we obtain that $\varphi(\fmspec^{(1)}_{\rm rad}) = \fmspec^{(2)}_{\rm rad} \, \urelem$ and so $\urelem \in \unirad(\parabolic_J)(\field)$.
The last assertion follows from the fact that the kernel of the differential $\difffield$-homomorphism
\[
\widetilde{\varphi}_{\urelem}\colon \underline{R} \to \underline{R}/\Imax^{(2)}, \ \underline{\intn}_i \mapsto x_{3,i} \quad \text{for all} \ \beta_i \in \roots^- \setminus \leviroots^-
\]
is $\Imax^{(1)}$, where $x_{3,i}$ is as in \ref{cor:ForOurLeviGrpThereisanidealnew(a)}.
Indeed, composing $\widetilde{\varphi}_{\urelem}$ with the restriction of $\varphi^{-1}$ to $\underline{R}/\Imax^{(2)}$, we obtain a differential $\difffield$-homomorphism which maps $\fmuq$ to $\fmspec^{(1)}$.
\end{proof}

\section{Computing the Reductive Part of the differential Galois Group} \label{sec:reductivepart}

In this section we present an algorithm which computes the ideal $\ulQ$ of algebraic relations between the entries of $\fmuq$ and an ideal $I_{\underline{H}}$ of the ring $\field[\GL_n]$ defining the group $\underline{H}(\field) = \Stab(\ulQ)(\field)$ contained in $\parabolic_J(\field) \leq \GL_n(\field)$ by Proposition~\ref{prop:statementsaboutstabilzers} \ref{prop:statementsaboutstabilzers(b)}.
Our algorithm will also compute the ideal $\ovQred$ of algebraic relations between the entries of $\fmspec_{\rm red}$ and an ideal $I_{\Lred}$ in $\field[\GL_n]$ defining the differential Galois group $\Lred(\field) = \Stab(\ovQred)(\field)$, which is a Levi group of $\underline{H}(\field)$ contained in the standard Levi group of $\parabolic_J(\field)$ according to Theorem~\ref{thm:Levigroupsconj} \ref{thm:Levigroupsconj(c)}.
 
In Section~\ref{sec:specializingreductivepart} we computed representatives $\bratv=(\ratv_1,\dots,\ratv_l)$ in
\[
\difffield(\GL_{n_{I''}}) \, = \, \difffield(X) \, = \, \difffield(X_{i,j}\mid i,j=1,\dots,n_{I''})
\]
of the residue classes in $\extfieldred$ which are the images of $\bvv$ under the specialization $\sigma$.
Since the derivation on $\difffield(X)$ is defined by
\[
\partial(X) \, = \, A_{\rm comp}X  \qquad (\text{cf.\ Definition~\ref{def:Q}} ),
\] 
we can also compute representatives in $\difffield(X)$ of the derivatives of $\sigma(\bvv)$ in $\extfieldred$. Thus, we can also determine representatives in $\difffield(X)$ of the images in $\extfieldred$ of the $m-l$ differential polynomials $\bff$ appearing in the Bruhat decomposition of $\fm$. We denote these elements in $\difffield(X)$ by
$$\bratf \, = \, (\ratf_{l+1},\dots,\ratf_m).$$ 

We substitute in the Bruhat decomposition of
$\fm$ the generic solutions $\bvv$, $\bff$, $\bexp$, $\bint$ by 
$\bratv ,\, \bratf, \, \bratexp, \, \bratint,$
where we recall from the beginning
of Section~\ref{sec:structureofreductivepart} that the entry $\ratint_i$ in $\bratint$ with $\beta_i \in \roots^-\setminus \leviroots^-$ is equal to $\intrad_i$, and obtain the matrix product
\begin{equation}\label{eqn:fm_in_hatelements}
\buu(\bratv,\bratf) \, n(\overline{w}) \, \btt(\bratexp) \, \buu(\bratint)\,.
\end{equation}
Note that the entries of this matrix product lie in the ring
\[
\difffield(X)[\intrad_i \mid \beta_i \in \roots^- \setminus \leviroots^-],
\]
which for the current purpose is not considered as a differential ring.
We denote by $\mathcal{D}$ the multiplicatively closed subset of $\difffield[X ,\ratfm,\intrad_i \mid \beta_i \in \roots^-\setminus \leviroots^-]$ which is generated by the denominators of
$\bratv$, $\bratf$, $\bratexp$ and of $\widehat{\intn}_i$ with $\beta_i \in \leviroots^-$ and the determinants $\det(X)$ and $\det(\ratfm)$.
We consider now the localization
\[
\locali \, := \, \mathcal{D}^{-1}\difffield[X,\ratfm,\intrad_i \mid \beta_i \in \roots^-\setminus \leviroots^-]
\]
and the ideal $\widetilde{Q}$ of $\locali$
generated by the numerators of the entries of the matrix
\[
\ratfm - \buu(\bratv, \bratf) \, n(\overline{w}) \, \btt(\bratexp) \, \buu(\bratint) \in \locali^{n \times n}
\]
and the generators of $Q$.
We can use now Gr\"obner basis methods to compute the intersection
\[
\ulQ \, := \, \widetilde{Q} \cap \difffield[\ratfm,\det(\ratfm)^{-1}]\,. 
\]
Again we can use Gr\"obner basis methods to compute the generators of the ideal $I_{\underline{H}}$ in $\field[\GL_n]$ which defines the stabilizer $\underline{H}(\field) = \Stab(\ulQ)(\field)$ in $\parabolic_J(\field) \leq \GL_n(\field)$ of the ideal $\ulQ$. Similarly, we consider in the localization 
\[
\locali_{\rm red} \, := \, \mathcal{D}^{-1}\difffield[X,\ratfm]
\]
the ideal $\widetilde{Q}_{\rm red}$ generated by the numerators of the entries of the matrix 
\[
\ratfm - \buu(\bratv, \bratf) \, n(\overline{w}) \, \btt(\bratexp) \, 
u_{j_1}(\widehat{\intn}_{j_1}) \cdots u_{j_k}(\widehat{\intn}_{j_k}) \in \locali_{\rm red}^{n \times n}
\]
and the generators of $Q$. One uses Gr\"obner basis methods to compute the intersection 
\[
\ovQred \, := \,  \widetilde{Q}_{\rm red} \cap \difffield[\ratfm,\det(\ratfm)^{-1}] 
\]
and generators of the ideal $I_{\Lred}$ defining the stabilizer $\Lred (\field)= \Stab(\ovQred)(\field)$ of the ideal $\ovQred$.
We summarize these steps in Algorithm~\ref{alg:reductivepart}.
\begin{algorithm} 
\DontPrintSemicolon
\KwInput { The matrix $\buu(\bratv,\bratf) \, n(\overline{w}) \, \btt(\bratexp) \, \buu(\bratint)$ and the generators of the ideal $Q \unlhd \difffield[X]$. }
\KwOutput{
\begin{itemize}
    \item A generating set of an ideal $I_{\underline{H}} \unlhd \field[\GL_n]$ defining the stabilizer $\underline{H}(\field) = \Stab(\ulQ)(\field)$. 
    The stabilizer has a Levi decomposition $\underline{\levi}(\field) \ltimes \unirad(\parabolic_J)(\field)$ for a Levi group $\underline{\levi}(\field)$  and $\underline{\levi}(\field)$ is conjugate by an element of $\unirad(\parabolic_J)(\field)$ to a Levi group of $H(\field)$.
    \item  A generating set of an ideal $\ovQred \unlhd \difffield[\widehat{\fm},\det(\widehat{\fm})^{-1}]$ which is a maximal differential ideal for $\Aprered$.
    \item A generating set of an ideal $I_{\Lred} \unlhd \field[\GL_n]$ defining the stabilizer $\Lred(\field)  = \Stab(\ovQred)(\field)$
    which is a Levi group of $\underline{H} (\field)= \Stab(\ulQ)(\field)$ contained in the standard Levi group of $\parabolic_J(\field)$.
\end{itemize}
}
Let $\widetilde{Q}$ be the ideal in $\difffield[X,\ratfm,\intrad_i, \det(\ratfm)^{-1} \mid \beta_i \in \roots^-\setminus \leviroots^-]$
generated by the numerators of the entries of the matrix 
\[
\ratfm - \buu(\bratv, \bratf) \, n(\overline{w}) \, \btt(\bratexp) \, \buu(\bratint) \in \locali^{n \times n}
\]
and the generators of $Q$.\\
Let $\widetilde{Q}_{\rm red}$ be the ideal in $\difffield[X,\ratfm, \det(\ratfm)^{-1} ]$
generated by the numerators of the entries of the matrix 
\[
\ratfm - \buu(\bratv, \bratf) \, n(\overline{w}) \, \btt(\bratexp) \, 
u_{j_1}(\widehat{\intn}_{j_1}) \cdots u_{j_k}(\widehat{\intn}_{j_k}) \in \locali_{\rm red}^{n \times n}
\]
and the generators of $Q$.\\
Compute with Gr\"obner basis methods a generating set of 
\[
 \ulQ \, = \, \widetilde{Q} \cap \difffield[\ratfm,\det(\ratfm)^{-1}] \quad \mathrm{and} \quad  \ovQred \, = \, \widetilde{Q}_{\rm red} \cap \difffield[\ratfm,\det(\ratfm)^{-1}].
\] \\
Compute with Gr\"obner basis methods generating sets of the  defining ideals 
\[
I_{\underline{H}} \unlhd \field[\GL_n] \quad \mathrm{and} \quad  I_{\Lred} \unlhd \field[\GL_n]
\]
of the stabilizers of $\ulQ$ and $\ovQred$ in $\GL_n(\field)$.\\
\Return(the generating sets of $I_{\underline{H}}$, $\ovQred$ and $I_{\Lred}$)
\caption{ComputeReductivePart\label{alg:reductivepart}}
\end{algorithm}

\begin{proposition}
    Algorithm~\ref{alg:reductivepart} is correct and terminates.
\end{proposition}

\begin{proof}
    Because the Gr\"obner basis computations terminate, the algorithm also terminates.
        
    We first show that the computed ideal $\ulQ$ in step~3 is the kernel of the substitution homomorphism
    \[
    \ulphi \colon \difffield[\widehat{\fm},\det(\widehat{\fm})^{-1}] \to \difffield[\fmuq,\det(\fmuq)^{-1}], \ \widehat{\fm}_{i,j} \to \fmuq_{i,j} \, .
    \]
    Then the properties of $\underline{H}(\field)=\Stab(\ulQ)(\field)$ stated in the output of the algorithm follow from Theorem~\ref{thm:Levigroupsconj}.

    The ideal $\ulQ$ is the kernel of the ring homomorphism 
    \begin{eqnarray*}
    \difffield[\ratfm,\det(\ratfm)^{-1}] & \to & 
    \locali / \widetilde{Q}, \\ \ratfm_{i,j} & \mapsto & \ratfm_{i,j} + \widetilde{Q} \, \equiv \, (\buu(\bratv, \bratf) \, n(\overline{w}) \, \btt(\bratexp) \, \buu(\bratint) )_{i,j} + \widetilde{Q}
    \end{eqnarray*}
    and so the ring homomorphism 
 \begin{eqnarray*}
    \varphi\colon \difffield[\ratfm ,\det(\ratfm)^{-1}]/\ulQ & \to & 
    \locali /\widetilde{Q}, \\
    \ratfm_{i,j} + \ulQ & \mapsto & (\buu(\bratv, \bratf) \, n(\overline{w}) \, \btt(\bratexp) \, \buu(\bratint))_{i,j} + \widetilde{Q}
\end{eqnarray*}
is a monomorphism. Its image is the subring generated over $\difffield$ by the entries
\[
(\buu(\bratv, \bratf) \, n(\overline{w}) \, \btt(\bratexp) \, \buu(\bratint))_{i,j} + \widetilde{Q} \, ,
\]
which is isomorphic as a ring to $\difffield[\fmuq,\det(\fmuq)^{-1}]$.

Finally, we have to show that the computed ideal $\ovQred$ in step~3 is the kernel of the substitution homomorphism
\[
    \ovphired \colon \difffield[\widehat{\fm},\det(\widehat{\fm})^{-1}] \to \difffield[ \fmspec_{\rm red},\det( \fmspec_{\rm red})^{-1}], \ \widehat{\fm}_{i,j} \mapsto (\fmspec_{\rm red})_{i,j}.
    \]
    Then, the statements about $\Lred(\field) = \Stab(\ovQred)(\field)$ will follow from Propositions~\ref{prop:reductivepartYred} and \ref{prop:statementsaboutstabilzers} \ref{prop:statementsaboutstabilzers(a)} and Theorem~\ref{thm:Levigroupsconj} \ref{thm:Levigroupsconj(c)}. 
    The ideal $\ovQred$ is the intersection  of $\widetilde{Q}_{\rm red}$ with $\difffield[\ratfm,\det(\ratfm)^{-1}]$ and so it is the kernel of the ring homomorphism 
    \begin{eqnarray*}
    \difffield[\ratfm ,\det(\ratfm)^{-1}] & \to & \locali_{\rm red}/\widetilde{Q}_{\rm red}, \\ \ratfm_{i,j} & \mapsto & \ratfm_{i,j} + \widetilde{Q}_{\rm red} \, \equiv \, (\buu(\bratv, \bratf) \, n(\overline{w}) \, \btt(\bratexp) \cdot \\ 
    & & \quad u_{j_1}(\widehat{\intn}_{j_1}) \cdots u_{j_k}(\widehat{\intn}_{j_k}))_{i,j} + \widetilde{Q}_{\rm red}.
    \end{eqnarray*}
    Thus, the ring homomorphism
    \[
    \begin{array}{rcl}
        \varphi\colon \difffield[\ratfm ,\det(\ratfm)^{-1}]/\ovQred & \to & \locali_{\rm red}/\widetilde{Q}_{\rm red},\\[0.5em]
        \ratfm_{i,j} + \ovQred &\mapsto & \ratfm_{i,j} + \widetilde{Q}_{\rm red} \, \equiv \, (\buu(\bratv, \bratf) \, n(\overline{w}) \, \btt(\bratexp) \cdot \\[0.5em]
        & & \quad u_{j_1}(\widehat{\intn}_{j_1}) \cdots u_{j_k}(\widehat{\intn}_{j_k}))_{i,j} + \widetilde{Q}_{\rm red} 
    \end{array}
    \]
    is a monomorphism. The image of $\varphi$ is the subring generated over $\difffield$ by the entries  
    \[
    (\buu(\bratv, \bratf) \, n(\overline{w}) \, \btt(\bratexp) \, u_{j_1}(\widehat{\intn}_{j_1}) \cdots u_{j_k}(\widehat{\intn}_{j_k}))_{i,j} + \widetilde{Q}_{\rm red} \, = \, (\fmspec_{\rm red})_{i,j}
    \]
    and so $\image(\varphi)$ is isomorphic to $\difffield[\fmspec_{\rm red},\det(\fmspec_{\rm red})^{-1}]$. 
\end{proof}

\section{Computing the Unipotent Radical}\label{sec:unipotentradical}

We start using the results of Section~\ref{sec:intoLieAlgParabolic} to show that we can compute a matrix $\overline{g}_1 \in \group(\difffield)$ such that $\overline{g}_1 \, \fmuq \in \parabolic_J(\underline{R})$ and $\overline{g}_1 \fmreduq \in \levi_J(\extfieldred)$ and such that $\overline{g}_1$ gauge transforms $A_{\group}(\bsq)$ into the Lie algebra of $\parabolic_J$, that is
\begin{equation}\label{eqn:definitionA_PJ}
\gauge{\overline{g}_1}{A_{\group}(\bsq)} \, =: \, A_{\parabolic_J} \in \Lie(\parabolic_J)(\difffield)\,.
\end{equation}

\begin{proposition}\label{prop:gaugetoAPJ}
In the notation of Section~\ref{sec:intoLieAlgParabolic}
let $\beta_{k_1},\dots,\beta_{k_s}$ be the roots in $\roots^-\setminus ( \Phi_1'^- \cup \dots \cup \Phi_d'^- )$.
We can compute  
$\overline{x}_{k_1},\dots, \overline{x}_{k_s} \in \difffield$ such that $A_{\group}(\bsq) $ is gauge equivalent by
\[
\overline{g}_1 \, := \, n(\overline{w})^{-1} \, u_{k_s}(\overline{x}_{k_s}) \cdots u_{k_1}(\overline{x}_{k_1}) \in \group(\difffield)
\]
to a matrix $A_{\parabolic_J}$ in $\Lie(\parabolic_J)(\difffield)$.
Moreover, we have
\begin{eqnarray*}
\overline{g}_1 \, \fmuq & \in & \parabolic_J(\underline{R})\,,\\
\overline{g}_1 \, \buu(\underline{\bvv},\underline{\bff}) \, n(\overline{w}) &\in& \unipotent^+_{\Psi}(\extfieldred)\,,\\
\overline{g}_1 \fmreduq & \in & \levi_J(\extfieldred)
\end{eqnarray*}
with $\fmuq = \fmreduq \, \fmraduq$, $\underline{R}$, $\underline{\bvv}$ and $\underline{\bff}$ as introduced in the beginning of Section~\ref{sec:structureofreductivepart}.
\end{proposition}

\begin{proof}
According to Proposition~\ref{prop:transformationintoPJgen} and Remark~\ref{rem:transformationintoPJgen} we can compute
differential polynomials $x_{k_1},\dots,x_{k_s} \in \field\{ \bss(\bvv),\bvvbase \}$ such that the matrix 
\[
g_1 \, = \, n(\overline{w})^{-1} \, u_{k_s}(x_{k_s}) \cdots u_{k_1}(x_{k_1})
\]
satisfies:
\begin{eqnarray}
    \gauge{g_1}{A_{\group}(\bss(\bvv))} \, =: \, A_{\parabolic_J}^{\rm gen} & \in & \Lie(\parabolic_J)(\field\{\bss(\bvv),\bvvbase\}) \, ,\label{eqn:gaugetransformationPJgeneric}\\
    g_1 \, \fm & \in & \parabolic_J(\field\{\bvv\}[\bexp,\bexp^{-1},\bint ]) \, , \label{eqn:multiplicationintoPJgeneric} \\
    g_1 \, \buu(\bvv,\bff) \, n(\overline{w}) & \in & \unipotent^+_{\Psi}(\field\{\bvv\} ) \,. \label{eqn:genericelementinUPsi}
\end{eqnarray}
We are going to specialize the generic results to the corresponding results over $\underline{R}$.

Let  
\[
\pi_{\mathrm{inter}}\colon \difffield\{ \bss(\bvv),\bvvbase \} \to \difffield\{ \bss(\bvv),\bvvbase \}/\idealinter \, \cong \, \difffield
\]
be the surjective homomorphism of differential rings obtained from $\siginter$ in Section~\ref{sec:specializingreductivepart} by extending scalars.
Then, we set
\[
\overline{x}_{k_s} \, = \, \pi_{\mathrm{inter}}(x_{k_s}), \quad \dots, \quad \overline{x}_{k_1} \, = \, \pi_{\mathrm{inter}}(x_{k_1}) \, .
\]
We explain how to compute $\overline{x}_{k_1},\dots, \overline{x}_{k_s} \in \difffield$ from $x_{k_1},\dots, x_{k_s} \in \field\{ \bvv \}$. 
We use differential elimination to express $x_{k_1},\dots, x_{k_s}$ as differential polynomials in $\bss(\bvv)$ and $\bvvbase$. Then one simply substitutes in these expressions 
$\bss(\bvv)$ and $\bvvbase$ by $\bsq$ and $\bvvqbase$.
Let $\baa=(a_1,\dots,a_l)$ be differential indeterminates over $\field\{\bvv\}$. One uses the differential Thomas decomposition to compute the normal form of $x_{k_1},\dots, x_{k_s}$ with respect to the differential ideal generated by
\[
 a_1-s_1(\bvv), \quad \dots, \quad a_l-s_l(\bvv)  
\]
and an elimination ranking satisfying $\bvv \gg \baa$. We obtain expressions for $x_{k_1},\dots, x_{k_s}$ in $\baa$ and $\bvvbase$.
Substituting $\baa$ and $\bvvbase$ by $\bsq$ and $\bvvqbase$, respectively, we obtain $\overline{x}_{k_1},\dots, \overline{x}_{k_s} \in \difffield$.

Applying $\pi_{\rm inter}$ to \eqref{eqn:gaugetransformationPJgeneric} we obtain
\[
\begin{array}{rcl}
    \pi_{\mathrm{inter}}(A_{\parabolic_J}^{\rm gen}) \! & \! = \! & \! \pi_{\mathrm{inter}} \big(n(\overline{w})^{-1} u_{k_s}(x_{k_s}) \cdots u_{k_1}(x_{k_1}).A_{\group}(\bss(\bvv)) \big)\\[0.2em]
    \! & \! = \! & \! n(\overline{w})^{-1} u_{k_s}(\pi_{\mathrm{inter}}(x_{k_s})) \cdots \gauge{u_{k_1}(\pi_{\mathrm{inter}}(x_{k_1}))}{A_{\group}(\pi_{\mathrm{inter}}(\bss(\bvv)))}\\[0.2em]
    \! & \! = \! & \! n(\overline{w})^{-1} u_{k_s}(\overline{x}_{k_s}) \cdots u_{k_1}(\overline{x}_{k_1}).A_{\group}( \bsq) \, =: \, A_{\parabolic_J} \in \Lie(\parabolic_J)(\difffield).
\end{array}
\]

Let $\pi$ be the surjective homomorphism of differential rings
\[
\pi\colon \difffield \{\bvv\}[\bexp,\bexp^{-1} ,\bint ] \to \underline{R}
\]
obtained from $\sigul$ in \eqref{eqn:specializationreductivepart} in Section~\ref{sec:structureofreductivepart} by restriction to $\difffield \{\bvv\}[\bexp,\bexp^{-1} ,\bint ]$. Since the restriction of $\sigul$ to $\field\{\bss(\bvv),\bvvbase\}$ is $\siginter$, we conclude that the restriction of $\pi$ to $\difffield\{\bss(\bvv),\bvvbase\}$ coincides with $\pi_{\mathrm{inter}}$ and so $\pi(g_1) = \pi_{\mathrm{inter}}(g_1)$. Applying $\pi$ to \eqref{eqn:multiplicationintoPJgeneric} we obtain  
\[
\pi(g_1\fm) \, = \, \pi(g_1) \, \pi(\fm) \, = \, \pi_{\mathrm{inter}}(g_1) \, \fmuq \, \in \, \parabolic_J(\underline{R})\,.
\]
We apply $\pi$ to \eqref{eqn:genericelementinUPsi} and obtain
\[
\begin{array}{rcl}
\pi ( g_1 \, \buu(\bvv,\bff) \, n(\overline{w}) ) \! & \! = \! & \! \pi_{\mathrm{inter}}(g_1) \, \buu( \pi(\bvv), \pi(\bff)) \, n(\overline{w})\\[0.2em]
\! & \! = \! & \! \pi_{\mathrm{inter}}(g_1) \, \buu( \underline{\bvv}, \underline{\bff}) \, n(\overline{w}) \in \unipotent^+_{\Psi}(\extfieldred).
\end{array}
\]
The last assertion simply follows now from the definition of $\fmreduq$, i.e.\ from
\[
\fmreduq \, = \, \buu(\underline{\bvv},\underline{\bff}) \, n(\overline{w}) \, \btt(\underline{\bexp}) \, u_{j_1}(\underline{\intn}_{j_1}) \cdots u_{j_k}(\underline{\intn}_{j_k}) \in \group(\extfieldred) \, , 
\]
and from the fact that $u_{j_1}(\underline{\intn}_{j_1}) \cdots u_{j_k}(\underline{\intn}_{j_k}) \in U^-_{\Psi}(\extfieldred)$.
\end{proof}

Recall from Definition~\ref{def:Q} that $\extfieldred = \Frac(\difffield[X,\det(X)^{-1}]/Q)$ is a Picard-Vessiot extension of $\difffield$ for the companion matrix $A_{\rm comp}$ for
$\overline{\LCLM}(\bsq,\bvvqbase,\partial) \, y = 0$.
Moreover, $\Hred(\field) := \Stab(Q)(\field) \leq \GL_{n_{I''}}(\field)$ is the differential Galois group of $\extfieldred$ over $\difffield$ in the representation induced by the fundamental matrix 
$$X + Q$$
 The Fundamental Theorem of Differential Galois Theory (cf.\ \cite[Proposition~1.34]{vanderPutSinger}) implies that the fixed field $\extfieldred^{\cocomp{\Hred}} $ is a finite Galois extension of $\difffield$ with Galois group $\Hred/\cocomp{\Hred}$.

\begin{definition}\label{def:difffieldalg}
We denote the finite Galois extension $\extfieldred^{\cocomp{\Hred}}$ of $\difffield$ by $\difffieldalg$.
\end{definition}

\begin{proposition}\label{prop:primelem}
We can compute a primitive element
\[
p \in \left( \difffield[X,\det(X)^{-1}]/Q \right)^{\cocomp{\Hred}}
\]
for the algebraic extension $\difffieldalg$ of $\difffield$.
\end{proposition}
\begin{proof}
A proof is given in Appendix~\ref{prop:primelemappendix}.
\end{proof}

Recall from Section~\ref{sec:structureofreductivepart} the decomposition $\fmuq = \fmreduq \, \fmraduq$ and the maximal differential ideal $\ovQred \unlhd \difffield[\widehat{\fm},\det(\widehat{\fm})^{-1}]$.
Its stabilizer is the differential Galois group $\Lred(\field) = \Stab(\ovQred)(\field)$ of the Picard-Vessiot extension $\extfieldred$ over $\difffield$ with respect to the fundamental matrix $\fmreduq$ for $\Aprered$. Recall that $\Hred(\field) \cong\Lred(\field)$ and that we have $\extfieldred^{\cocomp{\Lred}}=\difffieldalg $.

\begin{proposition}\label{prop:reduceredivtivepartnew}
Let $\overline{g}_1$ be as in Proposition~\ref{prop:gaugetoAPJ} and for an $n\times n$ matrix  $\widehat{X}=(\widehat{X}_{i,j})$ of indeterminates let 
$$\ovQred' \unlhd \difffield[\widehat{X},\det(\widehat{X})^{-1}]$$ 
be the ideal obtained from $\ovQred \unlhd \difffield[\widehat{\fm},\det(\widehat{\fm})^{-1}]$ by applying the transformation $\widehat{X} = \overline{g}_1 \, \widehat{\fm}$.
Then there exists an $\difffieldalg$-rational point $\overline{g}_2$ of $\ovQred'$ such that  
\[
\overline{g}_2 \in \levi_J(\difffieldalg) \quad \text{and} \quad \overline{g}_2 \overline{g}_1 \fmreduq \in \cocomp{\Lred}(\extfieldred) \, .
\] 
In particular, the matrix
\[
\gauge{\overline{g}_2 \overline{g}_1}{\Aprered} \, = \, \dlog( \overline{g}_2 \overline{g}_1 \fmreduq) \, =: \, \Ared \in \Lie(\Lred)(\difffieldalg)
\]
is in reduced form.
\end{proposition}

\begin{proof}
It follows from Proposition~\ref{prop:reductivepartYred} that $\difffield[\fmreduq,\det(\fmreduq)^{-1}]$ is a Picard-Vessiot ring for $\Aprered$ over $\difffield$ with fundamental matrix $\fmreduq$ and differential Galois group $\Lred(\field)$. 
Since $\overline{g}_1 \in \group(\difffield)$, the ring $\difffield[\fmreduq,\det(\fmreduq)^{-1}]$ is also a Picard-Vessiot ring for $\gauge{\overline{g}_1}{\Aprered}$ over $\difffield$ with fundamental matrix
$\overline{g}_1 \fmreduq$.
Moreover, $\overline{g}_1 \in \group(\difffield)$ implies that the representation of the differential Galois group with respect to $\overline{g}_1 \, \fmreduq$ is again $\Lred(\field)$.

We extend now the derivation of $\difffield$ to $\difffield[\widehat{X},\det(\widehat{X})^{-1}]$ by $\partial(\widehat{X})=(\overline{g}_1.\Aprered)\widehat{X}$.
Since the derivation on $\difffield[\widehat{\fm},\det(\widehat{\fm})^{-1}]$ is defined by $\Aprered$, we conclude that 
\[
\varphi\colon \difffield[\widehat{X},\det(\widehat{X})^{-1}] \to \difffield[\fmreduq,\det(\fmreduq)^{-1}] , \ \widehat{X} \mapsto \overline{g}_1 \, \fmreduq
\]
is a surjective differential $\difffield$-homomorphism. Because 
\[
\difffield[\widehat{\fm},\det(\widehat{\fm})^{-1}]/\ovQred \to \difffield[\fmreduq,\det(\fmreduq)^{-1}], \ \widehat{\fm} + \ovQred \mapsto \fmreduq
\]
is a differential $\difffield$-isomorphism, it follows that $\kernel(\varphi)=\ovQred'$.
Since the ring $\difffield[\fmreduq,\det(\fmreduq)^{-1}]$ is differentially simple, $\ovQred'$ is a maximal differential ideal and so 
\[
\difffield[\widehat{X},\det(\widehat{X})^{-1}]/\ovQred'
\]
is a Picard-Vessiot ring over $\difffield$ with fundamental matrix $\widehat{X} + \ovQred'$. Moreover, since $\varphi$ maps the fundamental matrix $\widehat{X} + \ovQred'$ to the fundamental matrix $\overline{g}_1 \, \fmreduq$, the differential Galois group of $\difffield[\widehat{X},\det(\widehat{X})^{-1}]/\ovQred'$ is also $\Lred(\field)$.

Since 
\[
\difffieldalg \, = \, \extfieldred^{\cocomp{\Lred}} \, = \, \Frac(\difffield[\fmreduq,\det(\fmreduq)^{-1}])^{\cocomp{\Lred}} \, = \, \difffield[\fmreduq,\det(\fmreduq)^{-1}]^{\cocomp{\Lred}},
\]
the differential ring 
\[
\difffield[\fmreduq,\det(\fmreduq)^{-1}] \, = \, \difffieldalg[\fmreduq,\det(\fmreduq)^{-1}]
\]
is also a Picard-Vessiot ring over $\difffieldalg$ for $\overline{g}_1 \Aprered$ with differential Galois group the connected component $\cocomp{\Lred}$ according to \cite[Proposition~1.34]{vanderPutSinger}. 
Consider now the surjective differential homomorphism
\[
\eta \colon \difffieldalg[\widehat{X},\det(\widehat{X})^{-1}] \to \difffieldalg[\fmreduq,\det(\fmreduq)^{-1}], \ \widehat{X} \mapsto \overline{g}_1 \, \fmreduq
\]
and denote its kernel by $\ovQred'' := \kernel(\eta)$.
Since $\difffieldalg[\fmreduq,\det(\fmreduq)^{-1}]$ is differentially simple, $\ovQred''$ is a maximal differential ideal and one easily checks that   
\begin{equation}\label{eqn:inclusionidealsQ3}
(\ovQred') \subset \ovQred'' \subset \difffieldalg[\widehat{X},\det(\widehat{X})^{-1}]\,,
\end{equation}
where $(\ovQred')$ denotes the ideal in $\difffieldalg[\widehat{X},\det(\widehat{X})^{-1}]$ generated by $\ovQred'$.
Hence,
\[
\overline{\eta}\colon \difffieldalg[\widehat{X},\det(\widehat{X})^{-1}]/\ovQred'' \to \difffieldalg[\fmreduq,\det(\fmreduq)^{-1}], \ \widehat{X} + \ovQred'' \mapsto \overline{g}_1 \, \fmreduq
\]
is a differential $\difffieldalg$-isomorphism of Picard-Vessiot rings.
Since it maps the fundamental matrix $\widehat{X} + \ovQred''$ to the fundamental matrix $\overline{g}_1 \, \fmreduq$, we conclude that the differential Galois group of the Picard-Vessiot ring 
$\difffieldalg[\widehat{X},\det(\widehat{X})^{-1}]/\ovQred''$ over $\difffieldalg$ is also $\cocomp{\Lred}(\field)$. 
Note that since $\difffield$ is a $C_1$-field and because finite algebraic extensions of a $C_1$-field are again $C_1$-fields (cf.\ \cite{SerreGaloisCohomology}), the field 
 $\difffieldalg$ is again a $C_1$-field.
Since $\cocomp{\Lred}(\field)$ is connected and the base field $\difffieldalg$ is a $C_1$-field, it follows that the torsor
\begin{equation}\label{eqn:trivialtorsor}
\max(\difffieldalg[\widehat{X},\det(\widehat{X})^{-1}]/\ovQred'' )
\end{equation}
is trivial and therefore $\ovQred''$ has an $\difffieldalg$-rational point $\overline{g}_2$.
The first inclusion in \eqref{eqn:inclusionidealsQ3} implies that $\overline{g}_2$ is also an $\difffieldalg$-rational point of $\ovQred'$.

According to Proposition~\ref{prop:gaugetoAPJ} we have that $\overline{g}_1 \fmreduq \in \levi_J(\extfieldred)$ and since the substitution homomorphism $\eta$ maps $\widehat{X}$ to $\overline{g}_1 \fmreduq$, 
we conclude that the ideal $\ovQred''$ contains the ideal $(I_{\levi_J})$ of $ \difffield[\widehat{X},\det(\widehat{X})^{-1}]$,  which is generated by the defining ideal $I_{\levi_J}$ in $\field[\widehat{X},\det(\widehat{X})^{-1}]$ of $\levi_J$. Thus, $\overline{g}_2 \in \levi_J(\difffieldalg)$.

We have shown that $\overline{g}_2$ is an $\difffieldalg$-rational point of the trivial torsor \eqref{eqn:trivialtorsor},
which means that
\[
\overline{g}_2 \, (\widehat{X} + \ovQred'') \in \cocomp{\Lred}(\Frac( \difffieldalg[\widehat{X},\det(\widehat{X})^{-1}]/\ovQred'' ) ) \, .
\]
Using the $\difffieldalg$-isomorphism $\overline{\eta}$ we conclude that
\[
\overline{g}_2 \overline{g}_1 \fmreduq \in \cocomp{\Lred}(\extfieldred) \,. 
\]
Finally, by Remark~\ref{remark4}, we obtain that
\[
\overline{g}_2 \overline{g}_1.\Aprered \, = \, \dlog (\overline{g}_2 \overline{g}_1 \fmreduq) \in \Lie(\cocomp{\Lred})(\difffieldalg) \, ,
\] 
because $\overline{g}_2 \overline{g}_1 \in \group(\difffieldalg)$.
Since $\cocomp{\Lred}(\field)$ is the differential Galois group of $\extfieldred$ over $\difffieldalg$, the matrix $\overline{g}_2 \overline{g}_1.\Aprered$ is in reduced form.
\end{proof}

\begin{algorithm} 
\DontPrintSemicolon
\KwInput {
\begin{enumerate}
    \item Generators $l_1,\dots,l_b$ of the ideal $I_{\levi_J} \unlhd \field[\GL_n] = \field[\coord,\det(\coord)^{-1}]$ defining $\levi_J$.
    \item Generators of the ideal $I_{\Lred} \unlhd \field[\coord,\det(\coord)^{-1}]$ defining $\Lred$.
    \item A generator $p \in \difffield[X,\det(X)^{-1}]$ of $\difffieldalg$ over $\difffield$ and the degree $\delta$ of the extension.
    \item Generators $q_1,\dots,q_s$ of $Q \unlhd \difffield[X,\det(X)^{-1}]$ (cf.\ Definition~\ref{def:Q}).
    \item The matrix $\overline{g}_1 \fmreduq$.
\end{enumerate}
    }
\KwOutput{
A matrix $\overline{g}_2 \in \levi_J(\difffieldalg)$ such that $\overline{g}_2\overline{g}_1 \fmreduq \in \cocomp{\Lred}(\extfieldred)$
}
Compute a primary decomposition of $I_{\Lred} \unlhd \field[\coord,\det(\coord)^{-1}]$ and find the primary ideal $I_{\cocomp{\Lred}}$ representing the connected component $\cocomp{\Lred}$ by testing the membership of the identity matrix.
Let $\ell_1,\dots,\ell_c \in \field[\coord,\det(\coord)^{-1}]$ form a generating set of $I_{\cocomp{\Lred}}$.
\\
Let $h_1,\dots,h_a$ be the generators of the ideal $I$ in
\[
\difffield[X,\coord,\det(X)^{-1},\det(\coord)^{-1}]
\]
generated by $l_1,\dots,l_b$ and the numerators of 
\[
\ell_1(\coord\overline{g}_1 \fmreduq),\dots,\ell_c(\coord\overline{g}_1 \fmreduq) \in \Frac(\difffield[X,\det(X)^{-1}])[\coord,\det(\coord)^{-1}] \, .
\]
\\
For $r \in \N$ and for each $\coord_{i,j}$ make the ansatz 
\[
\frac{c_0+ c_1p+ c_2 p^2+ \dots + c_{\delta} p^{\delta} }{\widetilde{c}_{0}+\widetilde{c}_{1}p+\widetilde{c}_{2}p^2+ \dots + \widetilde{c}_{\delta} p^{\delta}}
\]
where $c_k = c_{k,0}+c_{k,1}z+\dots+c_{k,r} z^r$
and $\widetilde{c}_k = \widetilde{c}_{k,0}+\widetilde{c}_{k,1}z+\dots+\widetilde{c}_{k,r} z^r$
for $k = 0, \ldots, \delta$ with constant coefficients $c_{r,s}$ and $\widetilde{c}_{r,s}$, respectively, and substitute the so obtained matrix $\mathcal{Z}$ in $h_1,\dots,h_a$.\\
Compute the normal forms of the numerators of $h_1(\mathcal{Z}),\dots,h_a(\mathcal{Z})$ modulo $Q$. 
Compute a Gr\"obner basis of the system of equations in $\field[c_{r,s},\tilde{c}_{r,s}]$ obtained by comparing coefficients with respect to the monomials in $z$ and $p$.\\
If the system is consistent compute a solution $\boldsymbol{c}$ and set $\overline{g}_2:=\mathcal{Z}(\boldsymbol{c})$.
If the system is not consistent, increase $r$ and repeat.

\Return(the matrix $\overline{g}_2$)
\caption{ComputeRationalPoint\label{alg:RationalPoint}}
\end{algorithm}

\begin{proposition}\label{prop:algorithmcomputeg2iscorrect}
    Algorithm~\ref{alg:RationalPoint} terminates and is correct.
\end{proposition}

\begin{proof}
    According to Proposition~\ref{prop:reduceredivtivepartnew},  there exists a solution over $\difffieldalg$ and the ansatz exhausts all elements of $\difffieldalg$ with increasing degree bound.
    Therefore, the algorithm terminates. 
    
    Since among the generators of $I$ are the generators of $I_{\levi_J}$ the found solution belongs to $\levi_J(\difffieldalg)$.
    Since the other generators are the numerators of
    \[
    \ell_1(\coord \overline{g}_1 \fmreduq),\dots,\ell_c(\coord \overline{g}_1 \fmreduq) \,,
    \]
    the generators $\ell_1,\dots , \ell_c$ of the defining ideal $I_{\cocomp{\Lred}}$ of $\cocomp{\Lred}$ vanish on $\overline{g}_2 \overline{g}_1 \fmreduq$ implying 
    that $\overline{g}_2 \overline{g}_1 \fmreduq \in \cocomp{\Lred}(\difffieldalg)$.
\end{proof}

\begin{lemma}\label{lem:partiallyreducedAP_J}
Let $\Ared$ and $\overline{g}_2 \overline{g}_1 \in \group(\difffieldalg)$ be as in Proposition~\ref{prop:reduceredivtivepartnew}. 
Then the gauge transform $\gauge{(\overline{g}_2 \overline{g}_1)}{A_{\group}(\bsq)}$ has the direct sum decomposition
\[
\gauge{(\overline{g}_2 \overline{g}_1)}{A_{\group}(\bsq)} \, = \, \Ared + \Aprerad \in \Lie(\cocomp{\Lred})(\difffieldalg) \oplus \Lie(\unirad(\parabolic_J))(\difffieldalg) \, . 
\]
In particular, we have
\[
\Aprerad \, = \, \Ad(\overline{g}_2 \overline{g}_1 \fmreduq) ( \dlog(\fmraduq)) \, .
\]
\end{lemma}

\begin{proof}
Using Proposition~\ref{prop:reduceredivtivepartnew} we compute
\begin{eqnarray*}
\gauge{(\overline{g}_2 \overline{g}_1)}{A_{\group}(\bsq)}
\! & \! = \! & \! \dlog (\overline{g}_2 \overline{g}_1 \, \fmreduq \, \fmraduq)\\
\! & \! = \! & \!
\dlog (\overline{g}_2 \overline{g}_1) + \Ad(\overline{g}_2 \overline{g}_1) \big(\dlog (\fmreduq) \big) + \Ad(\overline{g}_2 \overline{g}_1 \fmreduq) \big( \dlog (\fmraduq) \big)\\
\! & \! = \! & \!
\Ared + \Ad(\overline{g}_2 \overline{g}_1 \fmreduq) \big( \dlog(\fmraduq) \big) \,.
\end{eqnarray*}
Since $\unirad(\parabolic_J)$ is normal in $\parabolic_J$, for all $g \in \parabolic_J(\underline{E})$ we have 
\[
\Ad(g)(\Lie(\unirad(\parabolic_J))(\underline{E})) \, \subseteq \, \Lie(\unirad(\parabolic_J))(\underline{E})\, .
\]
 Moreover, from Remark~\ref{remark4} it follows that $\dlog(\fmraduq) \in \Lie(\unirad(\parabolic_J))(\underline{E})$ and so we conclude that 
\[
\Ad( \overline{g}_2 \overline{g}_1 \fmreduq) \big( \dlog (\fmraduq) \big) \, =: \, \Aprerad \, \in \, \Lie(\unirad(\parabolic_J))(\underline{E}).
\]
Finally, since the sum of Lie subalgebras in the assertion of the lemma is direct and $\gauge{(\overline{g}_2 \overline{g}_1)}{A_{\group}(\bsq)}$ has entries in $\difffieldalg$, it follows that $\Aprerad \in \Lie(\unirad(\parabolic_J))(\difffieldalg)$.
\end{proof}

By applying the fourth step of the algorithm presented in \cite[Subsection~5.2]{DreyfusWeil} by T.~Dreyfus and J.-A.\ Weil, we achieve the following reduction. (The first step is automatically achieved by the transformation into the parabolic subgroup $\parabolic_J$, the second step, that is the reduction of the reductive part, is achieved by Proposition~\ref{prop:reduceredivtivepartnew} and Algorithm~\ref{alg:RationalPoint} and the third step is simply Lemma~\ref{lem:partiallyreducedAP_J}.)

\begin{proposition}\label{prop:completereduction}
Recall the direct sum decomposition
\begin{equation}\label{eqn:directdecompLieP_J}
    \Lie(\parabolic_J) \, = \, \Lie(\levi_J)\oplus \Lie(\unirad(\parabolic_J))
\end{equation}
and suppose that $\Lie(\levi_J)$ and $\Lie(\unirad(\parabolic_J))$ consist of block diagonal matrices respectively unipotent lower triangular matrices.
\begin{enumerate}
\item\label{item:completereductiona}
We can compute $\overline{g}_3 \in \unirad(\parabolic_J)(\difffieldalg)$ such that
\[
\overline{g}_3\overline{g}_2\overline{g}_1.A_{\group}(\bsq) \, = \, \Ared + \Arad  
\]
is in reduced form, that is $\Ared + \Arad$ lies in a
Lie algebra $\Lie_{\rm red}(\difffieldalg)$ such that there is a connected algebraic group $\GalConn$ with Lie algebra $\Lie_{\rm red}$ and $\GalConn(\field)$ is a differential Galois group for $\Ared + \Arad$.
\item\label{item:completereductionb}
The differential Galois group $\GalConn(\field)$ has
Levi decomposition $\GalConn (\field) = \cocomp{\Lred}(\field) \ltimes R_1(\field)$ for some $R_1(\field) \leq \unirad(\parabolic_J)(\field)$.
\item\label{item:completereductionc}
We can compute generators $f_1,\dots,f_a$ of the defining ideal $I_{R_1}$ in $\field[\GL_n]$ of $R_1$.
\end{enumerate}
\end{proposition}

\begin{proof}
\ref{item:completereductiona}\
The assumption on the block diagonal structure guarantees that we can apply the algorithm in \cite[Subsection~5.2]{DreyfusWeil} to $\gauge{\overline{g}_1}{A_{\group}(\bsq)}$.
The block diagonal matrix $\gauge{\overline{g}_1}{A_{\group}(\bsq)} \in \Lie(\parabolic_J)(\difffield)$ is irreducible, since otherwise a gauge transformation into a smaller parabolic subgroup would be possible. 
The next step is guided by the transformation of
the diagonal block matrix $\Aprered$ over $\difffieldalg\cong \extfieldred^{\cocomp{\Lred}}$ into reduced form $\Ared$,
which was achieved in Proposition~\ref{prop:reduceredivtivepartnew}.
Algorithm~\ref{alg:RationalPoint} computes $\overline{g}_2$ performing such a
gauge transformation over $\difffieldalg$.
The effect of this gauge transformation on the block off-diagonal part is given by the matrix computed in Lemma~\ref{lem:partiallyreducedAP_J}, that is
\[
\gauge{\overline{g}_2 \overline{g}_1}{A_{\group}(\bsq)} \, = \, \Ared + \Aprerad \in \Lie(\cocomp{\Lred})(\difffieldalg) \oplus \Lie(\unirad(\parabolic_J))(\difffieldalg) .
\]
We can apply now step four of the algorithm presented in \cite[Subsection~5.2]{DreyfusWeil} to $\Ared + \Aprerad$ with $\boldsymbol{k}_0$ replaced by $\difffieldalg$, since we performed the reduction of the block diagonal part over $\difffieldalg$.
This yields a matrix $\overline{g}_3 \in \unirad(\parabolic_J)(\difffieldalg)$ such that
\[
\gauge{(\overline{g}_3\overline{g}_2\overline{g}_1)
}{A_{\group}(\bsq)} \, = \, \Ared + \Arad
\]
is reduced with $\Arad \in \Lie(\unirad(\parabolic_J))(\difffieldalg)$.
   
\ref{item:completereductionb}\
We consider now a Wei-Norman decomposition
\[
\Ared + \Arad \, = \, \sum a_i M_i
\]
of $\Ared + \Arad$, where $M_i \in \gl_n(\field)$ 
and $a_i \in \difffieldalg$ form a basis of the $\field$-vector space
spanned by the entries of $\Ared + \Arad$.
Now we can compute a basis of the smallest Lie subalgebra $\Lie_{\rm red}$ of $\gl_n(\field)$ which contains all matrices $M_i$, that is the algebraic envelope of the Lie algebra generated by all $M_i$ (cf.\ \cite[Definition~1.8]{DreyfusWeil}). 
 Since $\Ared + \Arad$ is in reduced form, it follows from \cite[Remark~1.9]{DreyfusWeil} that   
\begin{equation}\label{eqn:envelopeequalsLiealgebra}
\Lie(\GalConn)(\field) \, = \, \Lie_{\rm red}(\field)\,.   
\end{equation}
Since the smallest Lie algebra which contains $\Ared$ is $\Lie(\cocomp{\Lred})(\difffieldalg)$ and since $\Ared$ and $\Arad$ lie in the two different components of the direct decomposition \eqref{eqn:directdecompLieP_J}, we conclude that $\Lie(\cocomp{\Lred}) \subset \Lie(\GalConn)$ and so 
\begin{equation}\label{eqn:LhatsubgroupHcon}
    \cocomp{\Lred}(\field) \leq \GalConn (\field). 
\end{equation}

For a maximal differential ideal $\Imax$ of $\underline{R}$ we construct as in Section~\ref{sec:structureofreductivepart} (cf.\ in particular Proposition~\ref{prop:galaction})   
the Picard-Vessiot extension $\Frac(\difffield[\fmspec,\det(\fmspec)^{-1}])$ of $\difffield$ for 
$A_{\group}(\bsq)$ with differential Galois group $\widetilde{H}(\field)=\widetilde{\levi}(\field) \ltimes \widetilde{R}(\field) \leq \parabolic_J(\field)$, where, according to Theorem~\ref{thm:Levigroupsconj}, the group $\widetilde{\levi}(\field)$ is a Levi group of $\underline{H}(\field)$ and $\widetilde{R}(\field) \leq \unirad(\parabolic_J)(\field)$.

According to \cite[Proposition~1.34.3]{vanderPutSinger}
\[
\Frac(\difffield[\fmspec,\det(\fmspec)^{-1}])^{\cocomp{\widetilde{H}}}
\]
is the algebraic closure of $\difffield$ in $\Frac(\difffield[\fmspec,\det(\fmspec)^{-1}])$. Moreover, since
$\widetilde{\levi}(\field)$ and $\Lred(\field)$ are both Levi groups of $\underline{H}(\field)$, we obtain that
\[
\widetilde{H}(\field) / \cocomp{\widetilde{H}}(\field) \, \cong \, \widetilde{\levi}(\field) / \cocomp{\widetilde{\levi}}(\field) \, \cong \, \Lred(\field) / \cocomp{\Lred}(\field) \,. 
\]
Thus, we have  
\[
\Frac(\difffield[\fmspec,\det(\fmspec)^{-1}])^{\cocomp{\widetilde{H}}} \, = \, \extfieldred^{\cocomp{\Lred}} \, = \, \difffieldalg\,.
\]
Now the Galois correspondence implies that $\Frac(\difffield[\fmspec,\det(\fmspec)^{-1}])$ is a Picard-Vessiot extension of $\difffieldalg$
for $A_{\group}(\bsq)$ with differential Galois group $\cocomp{\widetilde{H}}(\field)$.
Because $\overline{g}_3 \overline{g}_2 \overline{g}_1 \in \group(\difffieldalg)$, we conclude that $\Frac(\difffield[\fmspec,\det(\fmspec)^{-1}])$ is also a Picard-Vessiot extension of $\difffieldalg$
for $\Ared + \Arad$ with differential Galois group $\cocomp{\widetilde{H}}(\field)$ and fundamental matrix $\overline{g}_3\overline{g}_2\overline{g}_1 \fmspec$.

Since $\Ared + \Arad$ is in reduced form and $\Lie(\GalConn)(\field)   = \Lie_{\rm red}(\field)$, the
defining ideal $I_{\GalConn}$ of $\GalConn$ in $\field[X,\det(X)^{-1}]$ generates a maximal differential ideal $(I_{\GalConn})$  in $\difffieldalg[X,\det(X)^{-1}]$, where the derivation on $\difffieldalg[X,\det(X)^{-1}]$ is defined by $\partial(X) = (\Ared + \Arad) X$, and so the differential field 
$$\Frac(\difffieldalg[X,\det(X)^{-1}]/(I_{\GalConn})) $$
is also a Picard-Vessiot extension of $\difffieldalg$ for $\Ared + \Arad$ with Galois group $\GalConn(\field)$.  

By \cite[Proposition~1.20.3]{vanderPutSinger} the two Picard-Vessiot rings are isomorphic, that is, there exists $g \in \GL_n(\field)$ such that the map
\[
\begin{array}{rcl}
\Frac(\difffield[\fmspec,\det(\fmspec)^{-1}]) \! & \! \to \! & \! \Frac(\difffieldalg[X,\det(X)^{-1}]/(I_{\GalConn})), \\[0.2em]
\overline{g}_3\overline{g}_2\overline{g}_1 \fmspec
\! & \! \mapsto \! & \! (X + (I_{\GalConn})) \, g
\end{array}
\]
is a differential $\difffieldalg$-isomorphism.
This isomorphism implies that the two differential Galois groups are conjugate by $g$, i.e.\
\[
\GalConn (\field) \, = \, g \, \cocomp{\widetilde{H}}(\field) \, g^{-1}\,.
\]
The first consequence of this conjugation is that from the connectedness of $\cocomp{\widetilde{H}}(\field)$ the connectedness of $\GalConn(\field)$ follows. 
As a second consequence, we obtain a Levi decomposition
\[
\GalConn (\field) \, = \, g \, \cocomp{\widetilde{H}}(\field) \, g^{-1} \, = \, g \, (\cocomp{\widetilde{\levi}}(\field) \ltimes \widetilde{R}(\field) ) \, g^{-1} \, = \, g \, \cocomp{\widetilde{\levi}}(\field) \, g^{-1} \ltimes g\widetilde{R}(\field) \, g^{-1}
\]
of $\GalConn(\field)$ with Levi group $g \, \cocomp{\widetilde{\levi}}(\field) \, g^{-1}$ and unipotent radical $g \, \widetilde{R}(\field) \, g^{-1}$.
Since both fundamental matrices $\overline{g}_3\overline{g}_2\overline{g}_1 \fmspec$ and $X + (I_{\GalConn})$ are elements of $\parabolic_J$, we conclude that $g \in \parabolic_J(\field)$ and so
\[
R_1(\field) \, := \, g \, \widetilde{R}(\field) \, g^{-1} \, \leq \, \unirad(\parabolic_J)(\field)
\]
and $g \, \cocomp{\widetilde{\levi}}(\field) \, g^{-1} \leq \parabolic_J(\field)$.
Since both Levi groups $\Lred(\field)$ and $\widetilde{\levi}(\field)$ of $\underline{H}(\field)$ are 
conjugate, the same holds for $\cocomp{\Lred}(\field)$ and $\cocomp{\widetilde{\levi}}(\field)$.
Thus, $\cocomp{\Lred}(\field)$ and $g \, \cocomp{\widetilde{\levi}}(\field) \, g^{-1}$ 
are also conjugate.
We conclude with \eqref{eqn:LhatsubgroupHcon} that $\cocomp{\Lred}(\field)$ is also a maximal reductive subgroup of $\GalConn(\field)$, i.e.,
$\cocomp{\Lred}(\field)\ltimes R_1(\field)$ is a Levi decomposition of $\GalConn(\field)$. 

\ref{item:completereductionc}\
Since we know a basis of $\Lie(\unirad(\parabolic_J))(\field)$, we can compute now a basis of the intersection
\[
\Lie(\GalConn)(\field) \cap \Lie(\unirad(\parabolic_J))(\field)\,,
\]
which is a basis of $\Lie(R_1) (\field)$. Indeed, from \ref{item:completereductionb}\  and \eqref{eqn:directdecompLieP_J} we obtain
\[
\Lie(R_1) (\field) \subset \Lie(\unirad(\parabolic_J))(\field) \quad \text{and} \quad  \Lie(\unirad(\parabolic_J))(\field) \cap \Lie(\cocomp{\Lred})(\field) \, = \, 0\,. 
\]
Using the exponential map, we can compute a generating set of one-parameter unipotent subgroups of $R_1(\field)$.
Using these generating matrices, we can compute now generators $f_1,\dots,f_a$ of the defining ideal $I_{R_1}$ in
$\field[\GL_n] = \field[\coord,\det(\coord)^{-1}] $
of $R_1(\field)$ with \cite[Algorithm~1, page~367]{DerksenKoiran}.
More precisely, we compute with this algorithm the Zariski closure of the group generated by the finitely many matrices obtained from the one-parameter unipotent subgroups of $R_1(\field)$ by specializing the parameter to $1 \in \field$.
\end{proof}

\begin{remark}
We can also compute generators of the defining ideal of $\GalConn(\field)$ in $\field[\GL_n]$. Indeed, we can
compute a primary decomposition of the defining ideal $I_{\Lred}$ of $\Lred(\field)$ and find among the primary ideals the ideal $I_{\cocomp{\Lred}}$ defining the connected component $\cocomp{\Lred}(\field)$ of $\Lred(\field)$ by checking the vanishing on the identity matrix. 
We consider the defining ideals $I_{\cocomp{\Lred}}$ and $I_{R_1}$ as ideals in the polynomial ring $\field[\coord^{(1)},\det(\coord^{(1)})^{-1}]$ and 
$\field[\coord^{(2)},\det(\coord^{(2)})^{-1}]$, respectively,
where $\coord^{(1)}$ and $\coord^{(2)}$ are as usually $n\times n$ matrices of respective indeterminates $\coord^{(1)}_{i,j}$ and $\coord^{(2)}_{i,j}$.
Then the coordinate ring of the semidirect product $\cocomp{\Lred}(\field)\ltimes R_1(\field)$ is
\[
\begin{array}{rcl}
\field[\cocomp{\Lred}] \otimes \field[R_1]
\! & \! \cong \! & \! \field[\cocomp{\Lred} \times R_1]\\[0.5em]
\! & \! = \! & \! \field[\coord^{(1)},\coord^{(2)},\det(\coord^{(1)})^{-1},\det(\coord^{(2)})^{-1}]/\langle I_{{\cocomp{\Lred}}},I_{R_1} \rangle\,.
\end{array}
\]
The multiplication map
\[
\mu\colon \cocomp{\Lred}(\field) \ltimes R_1(\field) \to \GalConn(\field), \ (g_1,g_2) \mapsto g_1 g_2
\]
is an isomorphism of affine varieties and so the map
\[
\mu^*\colon \field[\GalConn] \to \field[\cocomp{\Lred} \times R_1], f \mapsto f \circ \mu
\]
is an isomorphism of $\field$-algebras. We can use Gr\"obner basis methods to compute the kernel of the map
\[
\begin{array}{rcl}
 \field[\coord,\det(\coord)^{-1}] \! & \! \to \! & \! \field[\coord^{(1)},\coord^{(2)},\det(\coord^{(1)})^{-1},\det(\coord^{(2)})^{-1}]/\langle I_{\cocomp{\Lred}},I_{R_1} \rangle,\\[0.2em]
 \coord_{i,j} \! & \! \mapsto \! & \! (\coord^{(1)}\cdot \coord^{(2)})_{i,j} + \langle I_{\cocomp{\Lred}},I_{R_1} \rangle\,,
\end{array}
\]
which is equal to the defining ideal of $\GalConn(\field)$.
\end{remark}

Using the Lie structure, Proposition~\ref{prop:droppingtraingularblock} below shows the existence of reduction matrices $\overline{g}_2 \in \levi_{\parabolic_j}(\difffieldalg)$ and $\overline{g}_3 \in \unirad(\parabolic_J)(\difffieldalg)$ for $A_{\parabolic_J} = \gauge{\overline{g}_1}{A_{\group}(\bsq)}$ (cf.\ \eqref{eqn:definitionA_PJ}) for an arbitrary Levi group $\widetilde{\levi}$ of $\parabolic_J$ and independently of whether $\overline{g}_2. A_{\parabolic_J}$ is in triangular block form or not.
Its proof is very similar to the proof of \cite[Theorem~2.4]{DreyfusWeil}. Proposition~\ref{prop:droppingtraingularblock} is not needed later.
\begin{proposition}\label{prop:droppingtraingularblock} 
    Let $\widetilde{\levi}$ be an arbitrary Levi group of $\parabolic_J$ and $A_{\parabolic_J}$ as in \eqref{eqn:definitionA_PJ}.
    \begin{enumerate}
        \item\label{prop:droppingtraingularblock(a)}
        There exists $\overline{g}_2 \in \widetilde{\levi}(\difffieldalg)$ such that $\widetilde{A}_{\rm red}$ in the direct sum decomposition
    \begin{equation}\label{eqn:existenceofreductionLieAlgebra}
    \gauge{\overline{g}_2}{A_{\parabolic_J}} \, = \, \widetilde{A}_{\rm red} + \widetilde{A}_{\rm rad}^{\rm pre} \in \Lie(\widetilde{\levi})(\difffieldalg) \oplus \Lie(\unirad(\parabolic_J))(\difffieldalg)
    \end{equation}
    is in reduced form. 
    \item\label{prop:droppingtraingularblock(b)}
    There exists $\overline{g}_3 \in \unirad(\parabolic_J)(\difffieldalg)$ such that 
    \[
    \gauge{\overline{g}_3\overline{g}_2}{A_{\parabolic_J}} \, = \, \widetilde{A}_{\rm red} + \widetilde{A}_{\rm rad} \in \Lie(\widetilde{\levi})(\difffieldalg) \oplus \Lie(\unirad(\parabolic_J))(\difffieldalg)
    \]
    is in reduced form.
    \end{enumerate}
\end{proposition}

\begin{proof}
\ref{prop:droppingtraingularblock(a)}
Since $\unirad(\underline{H})(\field) = \unirad(\parabolic_J)(\field)$, there exists a Levi group $\underline{\levi}(\field)$ of $\underline{H}(\field)$ such that $\underline{\levi}(\field)\leq \widetilde{\levi}(\field)$.
According to Proposition~\ref{cor:ForOurLeviGrpThereisanidealnew} there exists a maximal differential ideal $\Imax \unlhd \underline{R}$ such that $\underline{\levi}(\field)$ is a Levi group of the differential Galois group $H(\field)$ of the
Picard-Vessiot extension $\overline{\generalext} = \Frac(\underline{R}/\Imax)$ of $\difffield$ constructed with respect to $\Imax$.
The group $H(\field)$ has a Levi decomposition $H(\field) = \underline{\levi}(\field) \ltimes \unirad(H)(\field)$ and since $\unirad(\cocomp{H})(\field) = \unirad(H)(\field)$, its connected component has Levi decomposition 
\[
\cocomp{H}(\field) \, = \, \cocomp{\underline{\levi}}(\field) \ltimes \unirad(H)(\field).
\]
According to the Fundamental Theorem, $\overline{\generalext}^{\cocomp{H}}$ is a finite algebraic extension of $\difffield$ with Galois group $H(\field)/\cocomp{H}(\field)$. Since $\underline{\levi}(\field) \cong \Stab(Q)(\field)$ and $\extfieldred \subset \overline{\generalext}$,
we conclude that $\difffieldalg = \overline{\generalext}^{\cocomp{H}}$. Since $\difffieldalg$ is a finite algebraic extension of the $C_1$-field $\difffield$, it is again a $C_1$-field by \cite{SerreGaloisCohomology}. 
Hence, the Kolchin-Kovacic Reduction Theorem (cf.\ \cite[Proposition~1.31]{vanderPutSinger}) implies that there exists 
$\overline{g} \in \parabolic_J(\difffieldalg)$ such that 
\[
\overline{g}. A_{\parabolic_J} \, =: \, \widetilde{A}_{\rm red} + \widetilde{A}_{\rm rad}  \in \Lie(\cocomp{H})(\difffieldalg) \, = \, \Lie(\cocomp{\underline{\levi}})(\difffieldalg) \oplus \Lie(\unirad(H))(\difffieldalg)\,.
\]
Since $\overline{g} \in \parabolic_J(\difffieldalg) = \unirad(\parabolic_J)(\difffieldalg) \cdot \widetilde{\levi}(\difffieldalg)$
there are uniquely determined $\overline{g}_2 \in \widetilde{\levi}(\difffieldalg)$ and $u \in \unirad(\parabolic_J)(\difffieldalg)$ such that $\overline{g} = u \, \overline{g}_2$.
Let $A_{\parabolic_J} = A_{1}+A_{2}$ be the decomposition according to 
\[
\Lie(\widetilde{\levi})(\difffield) \oplus \Lie(\unirad(\parabolic_J))(\difffield).
\]
 First observe that since $\unirad(\parabolic_J)(\difffieldalg)$ is a normal subgroup of $\parabolic_J(\difffieldalg)$,  for any $g \in \parabolic_J(\difffieldalg)$ the Lie algebra automorphism $\Ad (g)$ of $\Lie(\parabolic_J)(\difffieldalg)$ 
 stabilizes the Lie subalgebra $\Lie(\unirad(\parabolic_J))(\difffieldalg)$. 
 Next observe that if for $ \beta  \in  \roots^- \setminus \leviroots^-$ and $\alpha \in \leviroots$ and $k\geq 1$ the sum $\alpha + k \beta$ is a root of $\roots$, then it lies in $\roots^- \setminus \leviroots^-$ and so it follows  
 with Remark~\ref{remark3} that for any $A \in \Lie(\widetilde{\levi})(\difffieldalg)$ and any $g \in \unirad(\parabolic_J)(\difffieldalg)$  
the image of $A$ under $\Ad(g)$  lies in the plane $A+\Lie(\unirad(\parabolic_J))(\difffieldalg)$, that is 
\begin{equation}\label{eqn:reductiononlywithLiestructure}
    \Ad(g)(A) \in A+\Lie(\unirad(\parabolic_J))(\difffieldalg) .
\end{equation}
Thus, these two observations and Remark~\ref{remark4} imply together with  
\begin{gather*}
     \widetilde{A}_{\rm red} + \widetilde{A}_{\rm rad} \, = \, \overline{g}.A_{\parabolic_J}=\Ad(\overline{g})(A_{\parabolic_J})+\dlog(\overline{g}) \, = \, \\
     \Ad(u\overline{g}_2)(A_{1})+\Ad(u\overline{g}_2)(A_{2}) + \dlog(u) + \Ad(u)(\dlog(\overline{g}_2))
\end{gather*}
that only the matrices $\Ad(u\overline{g}_2)(A_{1})$ and $\Ad(u)(\dlog(\overline{g}_2))$ contribute to the part of $\overline{g}.A_{\parabolic_J}$ which lies in  $\Lie(\cocomp{\underline{\levi}})(\difffieldalg)$, that is to $\widetilde{A}_{\rm red}$.
We actually have that 
 \begin{eqnarray*}
 \Ad(u\overline{g}_2)(A_{1}) \! & \! \in \! & \! \Ad(\overline{g}_2)(A_{1})+\Lie(\unirad(\parabolic_J)) \quad \text{and}\\
 \Ad(u)(\dlog(\overline{g}_2)) \! & \! \in \! & \! \dlog(\overline{g}_2) + \Lie(\unirad(\parabolic_J)).
 \end{eqnarray*}
We conclude that
 \[
 \overline{g}_2.A_{\parabolic_J} \, = \, \widetilde{A}_{\rm red} + \widetilde{A}_{\rm rad}^{\rm pre}  
 \]
 with some suitable $\widetilde{A}_{\rm rad}^{\rm pre} \in \Lie(\unirad(\parabolic_J))(\difffieldalg)$.

\ref{prop:droppingtraingularblock(b)}
For the second assertion assume that we have a reduction matrix $\overline{g}_2$ such that \eqref{eqn:existenceofreductionLieAlgebra} holds with $\widetilde{A}_{\rm red}$ in reduced form, meaning that $\widetilde{A}_{\rm red}$ lies in the Lie algebra of a Levi group $\levi(\cocomp{H})$ of a potential differential Galois group $\cocomp{H}(\field)=\levi(\cocomp{H})(\field)\ltimes \unirad(\cocomp{H})(\field)$.
In other words we have
\[
\gauge{\overline{g}_2}{A_{\parabolic_J}} \, = \, \widetilde{A}_{\rm red} + \widetilde{A}_{\rm rad}^{\rm pre} \in 
\Lie(\levi(\cocomp{H}))(\difffieldalg)\oplus \Lie(\unirad(\parabolic_J))(\difffieldalg)
\]
from which we conclude with the Kolchin-Kovacic reduction theorem that there exists a matrix 
$$\overline{g} \in \levi(\cocomp{H})(\difffieldalg) \cdot \unirad(\parabolic_J)(\difffieldalg) = \levi(\cocomp{H})(\difffieldalg)\ltimes \unirad(\parabolic_J)(\difffieldalg)$$
 such that $\gauge{\overline{g}}{(\widetilde{A}_{\rm red} + \widetilde{A}_{\rm rad}^{\rm pre})}$ lies in  $\Lie(\cocomp{H})(\difffieldalg)$.
 Let $g = \ell \, \overline{g}_3$ be the product decomposition with uniquely determined matrices $\ell \in \levi(\cocomp{H})(\difffieldalg)$ and $\overline{g}_3 \in \unirad(\parabolic_J)(\difffieldalg)$.
 Since $\ell^{-1} \in \levi(\cocomp{H})(\difffieldalg) \leq \cocomp{H}(\difffieldalg)$ is an $\difffieldalg$-rational point of the differential Galois group and since $\Lie(\cocomp{H})(\difffieldalg)$ is closed under gauge transformation by elements of $\cocomp{H}(\difffieldalg)$, it follows that 
 \[
 \ell^{-1}.(g .(\widetilde{A}_{\rm red} + \widetilde{A}_{\rm rad}^{\rm pre})) \, = \, \ell^{-1}g .(\widetilde{A}_{\rm red} + \widetilde{A}_{\rm rad}^{\rm pre})) \, = \, \overline{g}_3 .(\widetilde{A}_{\rm red} + \widetilde{A}_{\rm rad}^{\rm pre})   
 \]
 still lies in the Lie algebra $\Lie(\cocomp{H})(\difffieldalg)$. Hence, $\overline{g}_3 \in \unirad(\parabolic_J)(\difffieldalg)$ completely reduces $\widetilde{A}_{\rm red} + \widetilde{A}_{\rm rad}^{\rm pre}=\overline{g}_2.A_{\parabolic_J}$. 
 The same arguments made in the proof of \ref{prop:droppingtraingularblock(a)} to show \eqref{eqn:reductiononlywithLiestructure} imply that $\Ad(\overline{g}_3)$ maps $\widetilde{A}_{\rm red}$ into the plane $\widetilde{A}_{\rm red} + \Lie(\unirad(\parabolic_J))(\difffieldalg)$.
 Together with the fact that 
 \[
 \Ad(\overline{g}_3)(\Lie(\unirad(\parabolic_J))(\difffieldalg)) \subset \Lie(\unirad(\parabolic_J))(\difffieldalg)
 \]
 and Remark~\ref{remark4} we conclude that the reduced form 
$\overline{g}_3 .(\widetilde{A}_{\rm red} + \widetilde{A}_{\rm rad}^{\rm pre})$ is
 \[
 \overline{g}_3\overline{g}_2.A_{\parabolic_J} \, = \, \overline{g}_3 .(\widetilde{A}_{\rm red} + \widetilde{A}_{\rm rad}^{\rm pre}) \, = \, \widetilde{A}_{\rm red} + \widetilde{A}_{\rm rad} \, .
 \] 
\end{proof}

\section{Computing the Differential Galois Group}\label{sec:compGalois}

We start this section with the proof that the product of the reduction matrix $\overline{g}_3 \overline{g}_2 \overline{g}_1$ from Proposition~\ref{prop:completereduction} with the partially identified fundamental matrix $\fmuq$ can be decomposed into the product of a matrix in $\cocomp{\Lred}(\extfieldred)$ and one in $\unirad(\parabolic_J)(\underline{R})$. 

\begin{proposition}\label{prop:decompredrad}
Suppose we are in the situation of Proposition~\ref{prop:completereduction}.
We can effectively decompose $\overline{g}_3 \overline{g}_2 \overline{g}_1 \fmuq$ as
\[
\overline{g}_3 \overline{g}_2 \overline{g}_1 \fmuq \, = \, \fmreduqhat \, \fmraduqhat
\]
with
(uniquely determined) matrices $\fmreduqhat \in \cocomp{\Lred}(\extfieldred)$ and $\fmraduqhat \in \unirad(\parabolic_J)(\underline{R})$.
\end{proposition}

\begin{proof}
It follows from Proposition~\ref{prop:reduceredivtivepartnew} that in 
\[
\overline{g}_3 \overline{g}_2 \overline{g}_1 \fmuq \, = \, \overline{g}_3 (\overline{g}_2 \overline{g}_1 \fmreduq) \fmraduq
\]
with $\fmraduq \in \unirad(\parabolic_J)(\underline{R})$ we have that $\overline{g}_2 \overline{g}_1 \fmreduq \in \cocomp{\Lred}(\extfieldred)$.
Since we have $\overline{g}_3 \in \unirad(\parabolic_J)(\difffieldalg)$ and since $\unirad(\parabolic_J)$ is normal in $\parabolic_J$, we conclude that there exists a matrix $u \in \unirad(\parabolic_J)(\extfieldred)$ such that 
\[
\overline{g}_3 \, (\overline{g}_2 \overline{g}_1 \, \fmreduq) \, \overline{g}_3^{-1} \, = \, ( \overline{g}_2 \overline{g}_1 \fmreduq) \, u \, .  
\]
Hence, with $\fmreduqhat := \overline{g}_2 \overline{g}_1 \, \fmreduq \in \cocomp{\Lred}(\extfieldred)$ and $\fmraduqhat := u \, \overline{g}_3 \fmraduq \in \unirad(\parabolic_J)(\underline{R})$ we obtain the Levi decomposition
\[
\overline{g}_3 \overline{g}_2 \overline{g}_1 \fmuq \, = \, \fmreduqhat \, \fmraduqhat \,.
\]
Note that the factors of a Levi decomposition are unique.  
Clearly, these matrix multiplications can be computed, where the matrix $u$ is simply read off. 
\end{proof}

Let $\overline{g}_3\overline{g}_2 \overline{g}_1 \fmuq = \fmreduqhat \,  \fmraduqhat$ be the decomposition of Proposition~\ref{prop:decompredrad}. Clearly, $ \fmreduqhat \,  \fmraduqhat$ satisfies 
\[
(\fmreduqhat \,  \fmraduqhat)' \, = \, (\overline{g}_3\overline{g}_2 \overline{g}_1.A_{\group}(\bsq)) (\fmreduqhat \,  \fmraduqhat) \, = \, (\Ared + \Arad)  (\fmreduqhat \,  \fmraduqhat) .
\]
Recall from the beginning of Section~\ref{sec:structureofreductivepart} the maximal differential ideal $\Imax$ in $\underline{R}$, the projection $\pi\colon \underline{R} \to \underline{R}/\Imax=\overline{R}$ and the image $\fmspec$ of $\fmuq$ under $\pi$.
We denote now the image of the matrices $\fmreduqhat$ and $\fmraduqhat$ under $\pi$ by 
\[
\widehat{\fmspec}_{\rm red} \, := \, \pi(\fmreduqhat) \quad \text{and} \quad 
\widehat{\fmspec}_{\rm rad} \, := \, \pi(\fmraduqhat).
\]
Since $\pi$ is a differential homomorphism and the identity on $\extfieldred$, we obtain now the decomposition 
\begin{equation}\label{eqn:decompositionreducedfm}
     \overline{g}_3 \overline{g}_2 \overline{g}_1 \fmspec \, = \, 
     \fmredqhat \, \fmradqhat 
\end{equation}
satisfying 
\[
(\widehat{\fmspec}_{\rm red} \,  \widehat{\fmspec}_{\rm rad})' \, = \, (\overline{g}_3\overline{g}_2 \overline{g}_1.A_{\group}(\bsq)) (\widehat{\fmspec}_{\rm red} \,  \widehat{\fmspec}_{\rm rad})  .
\]
In the following we are going to compute the generators of a maximal differential ideal $\Imax$ in $\underline{R}$ such that
\[
\fmradqhat \in R_1(\underline{R}/\Imax) \leq \unirad(\parabolic_J)(\underline{R}/\Imax),
\]
where $R_1$ is the unipotent group of Proposition~\ref{prop:completereduction}. The purpose of this choice of $\Imax$ is to match the reduction of $A_{\group}(\bsq)$ by $\overline{g}_3\overline{g}_2\overline{g}_1$ with the (Lie algebra of the) differential Galois group of $\overline{\generalext} = \Frac(\underline{R}/\Imax)$. But first we will show the following lemma.
\begin{lemma}\label{lem:ringsoverFalg} $\,$
\begin{enumerate}
\item\label{lem:ringsoverFalg(a)}
The field $\difffieldalg$ is contained in $\difffield[\fmspec,\det(\fmspec)^{-1}]$ and we have 
\begin{equation}\label{eqn:equalityofrings}
\difffield[\fmspec,\det(\fmspec)^{-1}] \, = \, \difffieldalg[ \fmredqhat \fmradqhat,\det(\fmredqhat \fmradqhat)^{-1}] 
\end{equation}
as differential rings.
\item\label{lem:ringsoverFalg(b)}
The ring $\difffieldalg[\fmredqhat \fmradqhat,\det(\fmredqhat \fmradqhat)^{-1}]$ is a Picard-Vessiot ring over $\difffieldalg$ for $\Ared + \Arad$ with fundamental matrix $\fmredqhat \fmradqhat$.
\end{enumerate}
\end{lemma}

\begin{proof}
\ref{lem:ringsoverFalg(a)}\ 
Recall from Proposition~\ref{prop:galaction} that 
$\difffield[ \fmspec,\det(\fmspec)^{-1}]$ is a Picard-Vessiot ring for $A_{\group}(\bsq)$ and that $\extfieldred \subset \Frac(\difffield[\fmspec,\det(\fmspec)^{-1}]) = \overline{\generalext}$. 
Let $p$ be a primitive element for the algebraic extension $\difffieldalg$ of $\difffield$ (cf.\ Proposition~\ref{prop:primelem}). Then the orbit of $p$ under the differential Galois group $H$ of $\overline{\generalext}$ over $\difffield$ is finite and so the $\field$-vector space spanned by the elements of the orbit is finite dimensional. It follows from \cite[Corollary~1.38]{vanderPutSinger} that $p \in \difffield[\fmspec,\det(\fmspec)^{-1}]$ and so $\difffieldalg \subset  \difffield[\fmspec,\det(\fmspec)^{-1}]$. 
This proves the first statement of \ref{lem:ringsoverFalg(a)}.
Note that
\[
\fmredqhat \, \fmradqhat  \, = \, \overline{g}_3 \overline{g}_2 \overline{g}_1 \, \fmspec
\]
with $\overline{g}_3 \overline{g}_2 \overline{g}_1 \in \group(\difffieldalg)$
shows that the entries of $\fmredqhat \fmradqhat$ and the entries of $\fmspec$
can be expressed in terms of one another as homogeneous polynomials of degree one with coefficients in $\difffieldalg$.
Therefore, the first statement of \ref{lem:ringsoverFalg(a)} implies the second one.
The equality as differential rings follows from the fact that the derivation of $\difffield$ uniquely extends to $\difffieldalg$ and that matrices $A_{\group}(\bsq)$ and $\Ared +\Arad$ defining the derivative of $\fmspec$ and $\fmredqhat \, \fmradqhat$, respectively, are gauge equivalent over $\difffieldalg$. 

\ref{lem:ringsoverFalg(b)}\ Since the Picard-Vessiot ring 
$\difffield[\fmspec,\det(\fmspec)^{-1}]$ is differentially simple, we conclude with \eqref{eqn:equalityofrings} that
\[
 \difffieldalg[\fmredqhat \fmradqhat,\det(\fmredqhat \fmradqhat)^{-1}]
\]
is also differentially simple. Moreover, since 
\[
\dlog(\fmredqhat \fmradqhat) \, = \,
\dlog(\overline{g}_3\overline{g}_2\overline{g}_1 \fmspec) \, = \, (\overline{g}_3\overline{g}_2\overline{g}_1).
\dlog(\fmspec) \, = \, (\overline{g}_3\overline{g}_2\overline{g}_1).A_{\group}(\bsq) \, = \, \Ared + \Arad,
\]
the matrix $\fmredqhat \, \fmradqhat$ is a fundamental matrix for $\Ared + \Arad$
and so
$$\difffieldalg[\fmredqhat \fmradqhat,\det(\fmredqhat \fmradqhat)^{-1}]$$ 
is a Picard-Vessiot ring over $\difffieldalg$ for $\Ared + \Arad$. 
\end{proof}

\begin{proposition}
Let $\GalConn (\field) = \cocomp{\Lred} (\field) \ltimes R_1 (\field)$ be as in Proposition~\ref{prop:completereduction} and 
suppose there exists a maximal differential ideal $\Imax$ in $\underline{R}$
defining $\fmradqhat$ with the property
\[
\fmradqhat \in R_1(\underline{R}/\Imax)  .
\]
As earlier denote by $H (\field)$ the differential Galois group of the Picard-Vessiot ring
$\difffield[\fmspec,\det(\fmspec)^{-1}]$
over $\difffield$ for $A_{\group}(\bsq)$ constructed with respect to $\Imax$.
Then $\unirad(H) (\field) = R_1 (\field)$ and there exists a Levi group $\levi(\field)$ of $H (\field)$ such that $\cocomp{\levi} (\field) = \cocomp{\Lred} (\field)$.
In particular, $\cocomp{H} (\field) = \GalConn (\field)$ and $H (\field) = \levi(\field) \ltimes R_1 (\field)$.
\end{proposition}

\begin{proof}
   According to Lemma~\ref{lem:ringsoverFalg} the differential ring 
   \[
   \difffieldalg[\fmredqhat \fmradqhat,\det(\fmredqhat \fmradqhat)^{-1}]
   \]
   is a Picard-Vessiot ring for $\Ared + \Arad$ over $\difffieldalg$. The assumption on $\Imax$ implies that its
   fundamental solution matrix satisfies  
\begin{equation}\label{eq:Hconinclusion}
    \overline{g}_3 \overline{g}_2 \overline{g}_1 \fmspec \, = \, \fmredqhat \fmradqhat \in \cocomp{\Lred}(\extfieldred) \cdot R_1(\overline{R}) \, \subseteq \, \cocomp{\Lred}(\overline{R}) \cdot R_1(\overline{R}) \, = \, \GalConn(\overline{R})\,,
\end{equation}
where we recall $\overline{R}=\underline{R}/\Imax$ and $\extfieldred \subset \overline{R}$. 
   Consider now the defining ideal  
   $$I_{\GalConn} \lhd \field[X,\det(X)^{-1}]$$
   of $\GalConn (\field)=\cocomp{\Lred} (\field) \ltimes R_1 (\field)$ and the ideal $(I_{\GalConn})$ in $\difffieldalg[X,\det(X)^{-1}]$
   generated by $I_{\GalConn}$. We extend the derivation of $\difffieldalg$ to  $\difffieldalg[X,\det(X)^{-1}]$ by 
   $$\partial(X) = (\Ared + \Arad)X.$$ 
   The ideal $(I_{\GalConn})$ is then a maximal differential ideal, since $\Ared + \Arad$ belongs to $\Lie(\GalConn)(\difffieldalg)$ and the differential Galois group for $\Ared + \Arad$ is $\GalConn (\field)$.
   The kernel of the surjective differential $\difffieldalg$-homomorphism 
   \[
   \Phi\colon \difffieldalg[X,\det(X)^{-1}] \to \difffieldalg[\fmredqhat \fmradqhat,\det(\fmredqhat \fmradqhat)^{-1}], \, X \mapsto  \fmredqhat \fmradqhat 
   \]
   contains the differential ideal $(I_{\GalConn})$, because $\fmredqhat \fmradqhat \in \GalConn(\overline{R})$ (cf.\ \eqref{eq:Hconinclusion}). Since $(I_{\GalConn})$ is a maximal differential ideal, it follows that $(I_{\GalConn})=\ker(\Phi)$ and so
   \[
   \begin{array}{rcl}
   \difffieldalg[X,\det(X)^{-1}]/(I_{\GalConn}) & \to & \difffieldalg[\fmredqhat \fmradqhat,\det(\fmredqhat \fmradqhat)^{-1}], \\[0.2em]
   X + (I_{\GalConn}) & \mapsto & \fmredqhat \fmradqhat
   \end{array}
   \]
   is a differential $\difffieldalg$-isomorphism of Picard-Vessiot rings. Since the differential Galois group of the first ring is $\GalConn(\field)$ and this differential $\difffieldalg$-isomorphism maps the fundamental matrix $X + (I_{\GalConn})$ 
   to the fundamental matrix $\fmredqhat \fmradqhat$, we conclude that the differential Galois group of
   \[
   \difffieldalg[ \fmredqhat \fmradqhat,\det(\fmredqhat \fmradqhat)^{-1}]
   \]
   over $\difffieldalg$ is also $\GalConn(\field)=\cocomp{\Lred}(\field) \ltimes R_1(\field)$.
   
   According to Theorem~\ref{thm:Levigroupsconj} and Proposition~\ref{prop:parabolicbound} the differential Galois group $H(\field)$ of the Picard-Vessiot ring $\difffield[\fmspec,\det(\fmspec)^{-1}]$ over $\difffield$ for $A_{\group}(\bsq)$ has a Levi decomposition 
   \[
   H(\field) \, = \, \underline{\levi}(\field) \ltimes R_2(\field),
   \]
   with $\underline{\levi}(\field)$ a Levi group of $\underline{H}(\field)$ and $R_2(\field) \leq \unirad(\parabolic_J(\field))$ its unipotent radical.
   Since by Lemma~\ref{lem:ringsoverFalg} we have the inclusions 
   \[
   \difffield \subset \difffieldalg \subset\overline{\generalext}=\Frac(\difffield[\fmspec,\det(\fmspec)^{-1}])
   \]
   the Fundamental Theorem of Differential Galois Theory implies that the differential Galois group of the Picard-Vessiot extension $\overline{\generalext}$ of $\difffieldalg$ is the subgroup of $H(\field)$ which leaves $\difffieldalg$ fixed.
   By Lemma~\ref{lem:ringsoverFalg} \ref{lem:ringsoverFalg(a)}\
   the rings $\difffield[\fmspec,\det(\fmspec)^{-1}]$ and
   $\difffieldalg[ \fmredqhat \fmradqhat,\det(\fmredqhat \fmradqhat)^{-1}]$ have the same $\difffieldalg$-automorphisms.
   From
   \[
   \overline{g}_3 \overline{g}_2 \overline{g}_1 \fmspec \, = \, \fmredqhat \fmradqhat
   \]
we conclude that for every $\difffieldalg$-automorphism $\gamma$
there exists $h \in \GL_n(\field)$ satisfying both $\gamma(\fmspec)=\fmspec h$ and $\gamma (\fmredqhat \fmradqhat)=\fmredqhat \fmradqhat h$, i.e., the representations of $\gamma$ induced by $\fmspec$ and $\fmredqhat \fmradqhat$ coincide. 
Hence, the subgroup of $H(\field)$ fixing $\difffieldalg$ is $\GalConn(\field)$.
   
   The inclusion $\GalConn(\field) \leq H(\field)$ implies the inclusion $\cocomp{\Lred}(\field) \leq H(\field)$ and so, since $\cocomp{\Lred}(\field)$ is reductive, there exists a Levi group $\levi(\field)$ of $H(\field)$, which is also a Levi group of $\underline{H}$ by Theorem~\ref{thm:Levigroupsconj}, such that $\cocomp{\Lred}(\field) \leq \levi(\field)$.
   The conjugacy of $\Lred$ and $\levi$
   implies that $\cocomp{\Lred}(\field) = \cocomp{\levi}(\field)$ and so we obtain
   \[
   \cocomp{H}(\field) \, = \, (\levi(\field) \ltimes \cocomp{R_2(\field))} \, = \, \cocomp{\levi}(\field) \ltimes R_2(\field) \, = \, \cocomp{\Lred}(\field) \ltimes R_2(\field).
   \]
   Since $\Frac(\difffield[\fmspec,\det(\fmspec)^{-1}])^{\cocomp{H}}$ is the algebraic closure of $\difffield$ in $\Frac(\difffield[\fmspec,\det(\fmspec)^{-1}])$, it follows that 
   \[
   \difffieldalg \, = \, \Frac(\difffield[\fmspec,\det(\fmspec)^{-1}])^{\GalConn} \subset \Frac(\difffield[\fmspec,\det(\fmspec)^{-1}])^{\cocomp{H}}
   \]
   and so $\GalConn(\field) \geq \cocomp{H}(\field)$.
   With $\GalConn(\field) \leq H(\field)$ we conclude now that  
   $\cocomp{H}(\field) = \GalConn(\field)$ and so $R_2(\field) = R_1(\field)$.
\end{proof}

\begin{proposition}\label{prop:generatorsforImax}
Suppose we are in the situation of Proposition~\ref{prop:completereduction}.
Evaluating the generators of the ideal $I_{R_1}$ in $\field[\GL_n]$ at $\fmraduqhat$ gives generators
$$f_1(\fmraduqhat), \dots, f_a(\fmraduqhat)$$
of a maximal differential ideal $\Imax$ in $\underline{R}$. This ideal has the property that
\[
\fmradqhat \in R_1(\underline{R}/\Imax)\,.
\]
\end{proposition}

\begin{proof}
 According to Proposition~\ref{cor:ForOurLeviGrpThereisanidealnew} ~\ref{cor:ForOurLeviGrpThereisanidealnew(b)} there exists a maximal differential ideal $\Imax^{(1)}$ in $\underline{R}$ such that $\Lred(\field)$ is a Levi group of the differential Galois group $H^{(1)}(\field)$ of the Picard-Vessiot ring $\difffield[\fmspec^{(1)},\det(\fmspec^{(1)})^{-1}]$ constructed for $\Imax^{(1)}$. The differential Galois group has then a Levi decomposition $H^{(1)}(\field)=\Lred(\field) \ltimes R_2(\field)$ with $R_2(\field) \leq \unirad(\parabolic_J)(\field)$ by Proposition~\ref{prop:parabolicbound}. The decomposition of Proposition~\ref{prop:decompredrad} reduces modulo $\Imax^{(1)}$ to a decomposition 
 \[
    \overline{g}_3 \overline{g}_2 \overline{g}_1 \fmspec^{(1)} = \fmredqhat^{(1)} \fmradqhat^{(1)}
 \]
 with $\fmredqhat^{(1)} \in \cocomp{\Lred}(\extfieldred)$ and $\fmradqhat^{(1)} \in \unirad(\parabolic_J)(\underline{R}/\Imax^{(1)})$.
 According to Lemma~\ref{lem:ringsoverFalg} \ref{lem:ringsoverFalg(b)} the differential ring
\[
  \difffieldalg[\fmredqhat^{(1)} \fmradqhat^{(1)},\det(\fmredqhat^{(1)} \fmradqhat^{(1)})^{-1}] 
\]
is a Picard-Vessiot ring over $\difffieldalg$ for $\Ared + \Arad$.
We will prove that its differential Galois group is the connected component $\cocomp{(H^{(1)})}(\field)$ of $H^{(1)}(\field)$. 
The fixed field 
$\overline{\generalext}^{\cocomp{(H^{(1)})}}$  
is the algebraic closure of $\difffield$ in $\overline{\generalext}$, where as usually $\overline{\generalext}$ denotes $\Frac(\difffield[\fmspec^{(1)},\det(\fmspec^{(1)})^{-1}])$, and so it contains $\difffieldalg$.
Since $H^{(1)}(\field)/\cocomp{(H^{(1)})}(\field) \cong \Lred(\field) / \cocomp{\Lred}(\field)$, both algebraic extensions have the same degree, forcing
$$\difffieldalg=\overline{\generalext}^{\cocomp{(H^{(1)})}}$$
and so the differential Galois group of $\difffield[\fmspec^{(1)},\det(\fmspec^{(1)})^{-1}]$ over $\difffieldalg$ is $\cocomp{(H^{(1)})}(\field)$. 
Hence, for every $\gamma \in \Gal_{\partial}(\overline{\generalext}/\difffieldalg)$ there exists $g \in \cocomp{(H^{(1)})}(\field)$ such that $\gamma(\fmspec^{(1)})=\fmspec^{(1)} g$. It follows from Lemma~\ref{lem:ringsoverFalg} \ref{lem:ringsoverFalg(a)} that $\Gal_{\partial}(\overline{\generalext}/\difffieldalg)$ is also the group of differential $\difffieldalg$-automorphisms of  $$\difffieldalg[\fmredqhat^{(1)} \fmradqhat^{(1)},\det(\fmredqhat^{(1)} \fmradqhat^{(1)})^{-1}].$$
We conclude with
\[
\gamma (\fmredqhat^{(1)} \fmradqhat^{(1)}) \, = \,
\gamma (\overline{g}_1 \overline{g}_2 \overline{g}_3 \fmspec^{(1)}) \, = \,
\overline{g}_1 \overline{g}_2 \overline{g}_3 \gamma(\fmspec^{(1)}) \, = \,
\overline{g}_1 \overline{g}_2 \overline{g}_3 \fmspec^{(1)} g \, = \, \fmredqhat^{(1)} \fmradqhat^{(1)} \, g
\]
that its representation with respect to $\fmredqhat^{(1)} \fmradqhat^{(1)}$ is also $\cocomp{(H^{(1)})}(\field)$, i.e.\ the differential Galois group of $\difffieldalg[\fmredqhat^{(1)} \fmradqhat^{(1)},\det(\fmredqhat^{(1)} \fmradqhat^{(1)})^{-1}]$ over $\difffieldalg$ is $\cocomp{(H^{(1)})}(\field)=\cocomp{\Lred}(\field) \ltimes R_2(\field)$ as stated.

Next we extend the derivation of $\difffieldalg$ to $\difffieldalg[X,\det(X)^{-1}]$ by 
\[
\partial(X) \, = \, (\Ared + \Arad) \, X \, .
\]
Then, Proposition~\ref{prop:completereduction} implies that the defining ideal $I_{\GalConn}$ in $\field[X,\det(X)^{-1}]$ generates a maximal differential ideal $(I_{\GalConn})$ in $\difffieldalg[X,\det(X)^{-1}]$. We obtain that
\[
 \difffieldalg[X,\det(X)^{-1}]/(I_{\GalConn})
\]
is a Picard-Vessiot ring over $\difffieldalg$ with differential Galois group $\GalConn (\field)= \cocomp{\Lred}(\field) \ltimes R_1(\field)$
(cf.\ Proposition~\ref{prop:completereduction}).
By construction we trivially have
\[
\overline{X} \, := \, X + (I_{\GalConn}) \in \GalConn(\difffieldalg[X,\det(X)^{-1}]/(I_{\GalConn}))\,.
\]
Since the multiplication map
\[
\mu\colon \cocomp{\Lred}\times R_1 \to \GalConn, \ (\ell,n) \mapsto \ell n
\]
is an isomorphism of varieties, its comorphism
\[
\mu^*\colon \field[\GalConn] \to \field[\cocomp{\Lred}\times R_1], \ \overline{X}_{i,j} \mapsto \overline{X}_{i,j} \circ \mu
\]
is a ring isomorphism. Combining it with the isomorphism  
\begin{gather*}
    \field[\cocomp{\Lred}\times R_1] \, \cong \, 
    \field[\coord^{(1)},\coord^{(2)},\det(\coord^{(1)})^{-1},\det(\coord^{(2)})^{-1}]/ (I_{\cocomp{\Lred}} , I_{R_1}) \, ,
\end{gather*}
where $I_{\cocomp{\Lred}}\unlhd \field[\coord^{(1)},\det(\coord^{(1)})^{-1}]$ and $I_{R_1} \unlhd \field[\coord^{(2)},\det(\coord^{(2)})^{-1}]$ are the defining ideals of $\cocomp{\Lred}$ and $R_1$, respectively, we obtain a ring isomorphism 
\[
 \field[\GalConn] \to \field[\coord^{(1)},\coord^{(2)},\det(\coord^{(1)})^{-1},\det(\coord^{(2)})^{-1}]/ (I_{\cocomp{\Lred}} , I_{R_1}), \ \overline{X} \mapsto  \overline{\coord}^{(1)} \overline{\coord}^{(2)} \, ,
\]
where the factors of the matrix product are $\overline{\coord}^{(i)}:=\coord^{(i)} + (I_{\cocomp{\Lred}} , I_{R_1})$ with $i=1,2$.
Applying its inverse to $\overline{\coord}^{(1)} \overline{\coord}^{(2)}$ we conclude that there exist
\begin{eqnarray*}
    \overline{X}_{\rm red} & \in & \cocomp{\Lred}(\difffieldalg[X,\det(X)^{-1}]/(I_{\GalConn})) \quad \text{and} \quad \\ \overline{X}_{\rm rad} & \in & R_1(\difffieldalg[X,\det(X)^{-1}]/(I_{\GalConn}))
\end{eqnarray*}
such that $\overline{X} = \overline{X}_{\rm red}\overline{X}_{\rm rad}$.

Since $\difffieldalg[\fmredqhat^{(1)} \fmradqhat^{(1)},\det(\fmredqhat^{(1)} \fmradqhat^{(1)})^{-1}]$ and $\difffieldalg[X,\det(X)^{-1}]/(I_{\GalConn})$ are Picard-Vessiot rings over $\difffieldalg$ for $\Ared + \Arad$,
there exists a matrix $g \in \group(\field)$ such that
\[
\begin{array}{rcl}
  \psi\colon \difffieldalg[\fmredqhat^{(1)} \fmradqhat^{(1)},\det(\fmredqhat^{(1)} \fmradqhat^{(1)})^{-1}] \! & \! \to \! & \! \difffieldalg[X,\det(X)^{-1}]/(I_{\GalConn}), \\[0.2em]
    \fmredqhat^{(1)} \fmradqhat^{(1)} \! & \! \mapsto \! & \! \overline{X}_{\rm red} \, \overline{X}_{\rm rad} \, g
\end{array}
\]
is a differential $\difffieldalg$-isomorphism.
Since both fundamental matrices $\fmredqhat^{(1)} \fmradqhat^{(1)}$ and $\overline{X}_{\rm red}\overline{X}_{\rm rad}$ are elements of the group $\cocomp{\Lred} \cdot \unirad(\parabolic_J)$, we conclude that $g \in \cocomp{\Lred}(\field) \cdot \unirad(\parabolic_J)(\field)$. Hence, there exist $\ell \in \cocomp{\Lred}(\field)$ and $u \in \unirad(\parabolic_J)(\field)$ such that $g = \ell \, u$. The choice of $\Imax^{(1)}$ implies now that 
$\ell^{-1} \in \cocomp{\Lred} \leq \cocomp{(H^{(1)})}$ induces a differential $\difffieldalg$-isomorphism
\[
\begin{array}{rcl}
\gamma_{\ell^{-1}}\colon \difffieldalg[\fmredqhat^{(1)} \fmradqhat^{(1)},\det(\fmredqhat^{(1)} \fmradqhat^{(1)})^{-1}] \! & \! \to \! & \! \difffieldalg[\fmredqhat^{(1)} \fmradqhat^{(1)},\det(\fmredqhat^{(1)} \fmradqhat^{(1)})^{-1}] ,\\[0.2em]
\fmredqhat^{(1)} \, \fmradqhat^{(1)} \! & \! \mapsto \! & \! \fmredqhat^{(1)} \, \fmradqhat^{(1)} \, \ell^{-1}.
\end{array}
\]
Thus, the composition $\overline{\psi} = \psi \circ \gamma_{\ell^{-1}}$ is the differential $\difffieldalg$-isomorphism
\[
\begin{array}{rcl}
\overline{\psi}\colon \difffieldalg[\fmredqhat^{(1)} \fmradqhat^{(1)},\det(\fmredqhat^{(1)} \fmradqhat^{(1)})^{-1}] \! & \! \to \! & \! \difffieldalg[X,\det(X)^{-1}]/(I_{\GalConn} ) , \\[0.2em]
\fmredqhat^{(1)} \fmradqhat^{(1)} \! & \! \mapsto \! & \! \overline{X}_{\rm red}\overline{X}_{\rm rad} u_1
\end{array}
\]
with $u_1 := \ell u \ell^{-1} \in \unirad(\parabolic_J)(\field)$.
Its inverse $\overline{\psi}^{-1}$ maps $\overline{X}_{\rm red}\overline{X}_{\rm rad}$
to $\fmredqhat^{(1)} \fmradqhat^{(1)} u_1^{-1}$ and so we obtain  
\[
 (\fmredqhat^{(1)})^{-1} \overline{\psi}^{-1}(\overline{X}_{\rm red}) \, = \, \fmradqhat^{(1)} u_1^{-1} \overline{\psi}^{-1}(\overline{X}_{\rm rad})^{-1}
\]
and since the left hand side and the right hand side of this equality are contained in $\cocomp{\Lred}$ and $\unirad(\parabolic_J)$ respectively, we conclude with $\cocomp{\Lred} \cap \unirad(\parabolic_J) = \{ \mathrm{id} \}$ that 
\[
\overline{\psi}^{-1}(\overline{X}_{\rm rad}) \, = \, \fmradqhat^{(1)} u_1^{-1} \, .
\]
Since $R_1$ is defined over $\field$ and since $\overline{\psi}^{-1}$ is an $\difffieldalg$-isomorphism, the fact that $\overline{X}_{\rm rad} \in R_1$ implies that  
\begin{equation}\label{eqn:wichtigeGleichung}
\overline{\psi}^{-1}(\overline{X}_{\rm rad}) \, = \, \fmradqhat^{(1)} u_1^{-1} \in R_1(\underline{R}/\Imax^{(1)}). 
\end{equation}
By Proposition~\ref{cor:ForOurLeviGrpThereisanidealnew} ~\ref{cor:ForOurLeviGrpThereisanidealnew(a)} the element $u_1$ induces a differential $\extfieldred$-isomorphism $\varphi_{u_1}\colon \underline{R} \to \underline{R}$ and so the image $\Imax^{(2)} := \varphi_{u_1}(\Imax^{(1)})$ is a maximal differential ideal of $\underline{R}$. According to Proposition~\ref{cor:ForOurLeviGrpThereisanidealnew} ~\ref{cor:ForOurLeviGrpThereisanidealnew(c)} the map
\[
\varphi\colon \Frac(\underline{R}/\Imax^{(1)}) \to \Frac(\underline{R}/\Imax^{(2)}) , \quad \fmspec^{(1)} \mapsto \fmspec^{(2)} u_1
\]
is a differential $\difffield$-isomorphism which is also the identity on $\extfieldred$.
Combining the surjective differential $\extfieldred$-homomorphism 
\[
\pi\colon \underline{R} \to \underline{R}/\Imax^{(2)}
\]
with $\varphi^{-1}$ we obtain a surjective differential $\extfieldred$-homomorphism 
\[
\widehat{\pi}\colon \underline{R} \to \underline{R}/\Imax^{(1)} , \ \fmuq \mapsto \fmspec^{(1)} u_1^{-1}
\]
with kernel $\Imax^{(2)}$.
From $\widehat{\pi}(\fmreduqhat) = \fmredqhat^{(1)}$ and \eqref{eqn:wichtigeGleichung} and
\[
 \widehat{\pi}(\fmreduqhat ) \widehat{\pi}( \fmraduqhat) \, = \, \widehat{\pi}(\overline{g}_1\overline{g}_2 \overline{g}_3 \fmuq) \, = \,
\overline{g}_1\overline{g}_2 \overline{g}_3 \widehat{\pi}(\fmuq) \, = \, \overline{g}_1\overline{g}_2 \overline{g}_3\fmspec^{(1)} u_1^{-1} = \fmredqhat^{(1)} \fmradqhat^{(1)} u_1^{-1},
\]
we obtain $\widehat{\pi}( \fmraduqhat) = \fmradqhat^{(1)} u_1^{-1} \in R_1$ and so $\Imax^{(2)} = \langle f_1(\fmraduqhat), \dots, f_a(\fmraduqhat) \rangle$. 
\end{proof}

Note that the generators $f_1(\fmraduqhat),\dots, f_a(\fmraduqhat)$ of the maximal differential ideal $\Imax$ in 
\[
\underline{R} \, = \,
\extfieldred\{\intrad_i \mid \beta_i \in \roots^- \setminus \leviroots^- \} / \Iuni \, = \, \extfieldred[\underline{\intrad}_i \mid \beta_i \in \roots^-\setminus \leviroots^- ]
\]
from Proposition~\ref{prop:generatorsforImax} have coefficients in
\[
\Frac(\difffield[X,\det(X)^{-1}]/Q) \, = \, \extfieldred.
\]
We denote by $\underline{\bintrad}$ the tuple whose entries are the algebraic indeterminates $\underline{\intrad}_i$ with $\beta_i \in \roots^- \setminus \leviroots^-$.
\begin{definition}\label{def:widetildef}
We denote by $\widetilde{f}_1,\dots, \widetilde{f}_a$ the polynomials in
\[
\Frac(\difffield[X])[\underline{\bintrad}] \, ,
\]
which modulo $Q$ are equal to $f_1(\fmraduqhat),\dots, f_a(\fmraduqhat)$.
\end{definition}

\begin{algorithm} 
\DontPrintSemicolon
\KwInput { 
\begin{enumerate}
\item The matrix $\buu(\bratv,\bratf) \, n(\overline{w}) \, \btt(\bratexp) \, \buu(\bratint)$ in 
$\group(\Frac(\difffield[X])[\underline{\bintrad} ] )$ (cf.\ \eqref{eqn:fm_in_hatelements}).
\item A generating set of the ideal $Q \unlhd \difffield[\GL_{n_{I''}}] = \difffield[X,\det(X)^{-1}]$ (cf.\ Definition~\ref{def:Q}).
\item The polynomials $\widetilde{f}_1 ,\dots, \widetilde{f}_a$ in 
\[
\Frac(\difffield[X])[\underline{\bintrad} ] \, , 
\]
which generate modulo $Q$ the ideal $\Imax\lhd \underline{R}$.
\end{enumerate}
 }
\KwOutput{
\begin{enumerate}
    \item A generating set of a maximal differential ideal $\ovQ$ of $\difffield[\ratfm,\det(\ratfm)^{-1}]$ for $A_{\group}(\bsq)$.
    \item A generating set of the defining ideal $I_{H} \unlhd \field[\GL_n]$ of the differential Galois group for the Picard-Vessiot ring $\difffield[\ratfm,\det(\ratfm)^{-1}]/\ovQ$ over $\difffield$.
\end{enumerate}
}
Let $\widetilde{Q}$ be the ideal in 
$$\difffield[X,\det(X)^{-1},\underline{\bintrad},\ratfm, \det(\ratfm)^{-1}]$$
generated by the numerators of the entries of the matrix 
\[
\ratfm - \buu(\bratv, \bratf) \, n(\overline{w}) \, \btt(\bratexp) \, \buu(\bratint) \in \Frac(\difffield[X])[\ratfm,\underline{\bintrad} ]^{n \times n} ,
\]
the generators of $Q$ and the numerators  of $\widetilde{f}_1,\dots, \widetilde{f}_a$ in 
    $\difffield[X][\underline{\bintrad}]$.
     \\
 Compute with Gr\"obner basis methods a generating set of 
 \[\ovQ \, = \, \widetilde{Q} \cap \difffield[\ratfm,\det(\ratfm)^{-1}]\] \\
Compute with Gr\"obner basis methods a generating set of the  defining ideal 
$$I_{H} \unlhd \field[\GL_n] \ $$ of the stabilizer of $\ovQ $ in $\GL_n(\field)$.\\
\Return(the generating sets of $\ovQ$ and $I_{H}$)
\caption{ComputeDifferentialGaloisGroup\label{alg:DiffGaloisgroup}}
\end{algorithm}

\begin{proposition}
    Algorithm~\ref{alg:DiffGaloisgroup} terminates and is correct.
\end{proposition}

\begin{proof}
    Since Gröbner basis computations terminate, we conclude that the algorithm terminates.

    For the proof of the correctness of the algorithm let $\Imax$ be the maximal differential ideal of $\underline{R}$ generated by $f_1(\fmraduqhat),\dots, f_a(\fmraduqhat)$ from Proposition~\ref{prop:generatorsforImax} and consider the Picard-Vessiot ring $\difffield[\fmspec,\det(\fmspec)^{-1}]$ over $\difffield$ for $A_{\group}(\bsq)$ constructed with respect to $\Imax$ (cf.\ the beginning of Section~\ref{sec:structureofreductivepart} and Proposition~\ref{prop:galaction}). 
    Extending now the derivation of $\difffield$ to $\difffield[\ratfm,\det(\ratfm)^{-1}]$ 
    by $\partial(\ratfm) = A_{\group}(\bsq) \ratfm$
    the substitution homomorphism 
    \[
    \ovphi\colon \difffield[\ratfm,\det(\ratfm)^{-1}] \to \difffield[\fmspec,\det(\fmspec)^{-1}], \ 
    \ratfm \mapsto \fmspec 
    \]
    is a surjective differential $\difffield$-homomorphism.  
    Being a Picard-Vessiot ring over $\difffield$, $\difffield[\fmspec,\det(\fmspec)^{-1}]$ is differentially simple and so  $\ker(\ovphi)$ is a maximal differential ideal in $\difffield[\ratfm,\det(\ratfm)^{-1}]$. 
    Assuming that we have proved that the ideal $\ovQ$ computed in step~2 is equal to $\ker(\ovphi)$, it follows that 
    the ring 
    $$\difffield[\ratfm,\det(\ratfm)^{-1}]/\ovQ$$ 
    is a Picard-Vessiot ring over $\difffield$ and so the ideal $I_{H}$ computed in step~3 of the algorithm, i.e.\
    the defining ideal of the stabilizer of $\ovQ$ in $\GL_n(\field)$, is the defining ideal of the differential Galois group of $\difffield[\ratfm,\det(\ratfm)^{-1}]/\ovQ$ over $\difffield$ for $A_{\group}(\bsq)$.

    It is left to show that $\ker(\ovphi)=\ovQ$. As a maximal differential ideal in the ring $\difffield[X,\det(X)^{-1}]$ the ideal $Q$ is a prime ideal and so the ideal $(Q)$ generated by $Q$ in 
    \[
    \difffield[X,\det(X)^{-1},\underline{\bintrad},\ratfm, \det(\ratfm)^{-1}]
    \]
is also a prime ideal. We consider now the localization
    \[
   \mathcal{R}_{\rm loc} \, := \, \difffield[X,\det(X)^{-1},\underline{\bintrad},\ratfm, \det(\ratfm)^{-1}]_{(Q)}
    \]
of $\difffield[X,\det(X)^{-1},\underline{\bintrad},\ratfm, \det(\ratfm)^{-1}]$ at $(Q)$. Since the denominators of the elements
$\bratv$, $\bratf$, $\bratexp$ and $\ratint_i$ with $\beta_i \in \leviroots^-$ in $\Frac(\difffield[X,\det(X)^{-1}])$ do not vanish modulo $Q$, the parameters of the Bruhat decomposition 
\[
\buu(\bratv, \bratf) \, n(\overline{w}) \, \btt(\bratexp) \, \buu(\bratint)
\]
are contained in $\mathcal{R}_{\rm loc}$.
Denote by $I$ the ideal in $\mathcal{R}_{\rm loc}$ generated by the numerators of the entries of the matrix 
\[
\widehat{\fm} - \buu(\bratv, \bratf) \, n(\overline{w}) \, \btt(\bratexp) \, \buu(\bratint). 
\]
Consider the canonical projection 
\[
\pi_1\colon \mathcal{R}_{\rm loc} \to \mathcal{R}_{\rm loc}/I
\]
for $I$ and then the canonical projection 
\[
\pi_2\colon \mathcal{R}_{\rm loc}/I \to (\mathcal{R}_{\rm loc}/ I )/\pi_1\big((Q)\big)
\]
for the proper ideal $\pi_1((Q))$ in $\mathcal{R}_{\rm loc}/I$. Since the generators of $I$ are linear in $\widehat{\fm}_{i,j}$ and since 
\[
\difffield[X,\det(X)^{-1}]_{Q}/Q' \, \cong \, \Frac(\difffield[X,\det(X)^{-1}]/Q) \, = \, \extfieldred , 
\]
where $\difffield[X,\det(X)^{-1}]_{Q}$ is the localization at the prime ideal $Q$ and $Q'$ is the ideal generated by $Q$ 
in $\difffield[X,\det(X)^{-1}]_{Q}$, we conclude that  
\[
(\mathcal{R}_{\rm loc}/I)/ \pi_1 \big( (Q) \big) \, = \, \extfieldred[\underline{\bintrad}] \, = \, \underline{R} .
\]
Thus, the map 
\[
\pi_2 \circ \pi_1\colon \mathcal{R}_{\rm loc} \to \underline{R}
\]
is a surjective $\difffield$-algebra homomorphism and its kernel is generated by the generators of $I$ and $(Q)$. The preimage 
$(\pi_2 \circ \pi_1)^{-1}(\Imax)$
of the ideal $\Imax \lhd \underline{R}$ is the ideal of $\mathcal{R}_{\rm loc}$ which is generated by the generators of $I$ and $(Q)$, since it clearly contains $\ker(\pi_2 \circ \pi_1 )$, and by the numerators of
$\widetilde{f}_1,\dots,\widetilde{f}_a$. We observe that $(\pi_1 \circ \pi_2)^{-1}(\Imax)=(\widetilde{Q})$, that is the ideal in $\mathcal{R}_{\rm loc}$ generated by $\widetilde{Q}$, and so we obtain an $\difffield$-algebra isomorphism 
\[
\pi\colon \mathcal{R}_{\rm loc}/(\widetilde{Q}) \to \underline{R}/\Imax \, ,
\]
which by construction maps the entries  
\[
\widehat{\fm}_{i,j} + (\ovQ) \equiv \big( \buu(\bratv, \bratf) \, n(\overline{w}) \, \btt(\bratexp) \, \buu(\bratint) \big)_{i,j} + (\ovQ)
\]
to the entries of 
\[
\fmspec \, = \, \buu(\bvq,\overline{\bff}) \, n(\overline{w}) \, \btt(\overline{\bexp})\, \buu(\overline{\bint})\,. 
\]
The ideal $\ovQ$ in step~2 is the kernel of the $\difffield$-algebra homomorphism 
\[
\phi\colon \difffield[\widehat{\fm},\det(\widehat{\fm})^{-1}] \to \mathcal{R}_{\rm loc}/(\widetilde{Q}), \ \widehat{\fm}_{i,j} \mapsto \widehat{\fm}_{i,j} + (\widetilde{Q})\,.
\]
Composing $\phi$ with the isomorphism $\pi$ we obtain the differential $\difffield$-algebra homomorphism $\ovphi$ and conclude that its kernel is $\ovQ$.
\end{proof}

\part{Appendix}\label{part:IV}

\appendix

\section{The Normal Form Matrix and Operator}
\label{sec:correspondenceA_GandL_G}
In this section we describe the transformation matrix $\BNFcomp \in \GL_n(\field\langle \bss(\bvv) \rangle)$ which defines a gauge transformation of the normal form matrix $A_{\group}(\bss(\bvv))$ to a companion matrix $\ANFcomp$, whose entries in the last row are the coefficients of the normal form operator $L_{\group}(\bss(\bvv),\partial)$.
\begin{definition}
    Let $\group$ be one of the groups $\SL_{l+1}$, $\SP_{2l}$, $\SO_{2l+1}$ or $\Gzwei$ in their natural representation. Define the transformation matrix $\BNFcomp  \in \GL_n(\field\langle \bss(\bvv) \rangle)$ as the one corresponding to the cyclic vector presented in 
    \cite{Seiss} in Section~7, 8, 9 and 11, respectively.
\end{definition}

\begin{proposition}\label{prop:PropertiesOfBcomp}
 Let $\group$ be one of the groups $\SL_{l+1}$, $\SP_{2l}$, $\SO_{2l+1}$ or $\Gzwei$. Then $\BNFcomp \in \GL_n(\field\{ \bss(\bvv) \})$.
\end{proposition}
\begin{proof}
In case $\group$ is one of the groups $\SL_{l+1}$, $\SP_{2l}$ or $\SO_{2l+1}$ the shape of the respective normal form matrix is 
\[
A_{\group}(\bss(\bvv))= \left( \begin{array}{ccccc}
0 & e_1 & 0 & \ldots & 0 \\
p_{2,1} & 0 & e_2 & \ldots & 0 \\
p_{3,1} & p_{3,2} & 0 & \ddots & \vdots \\
\vdots & & & \ddots & e_{n-1} \\
p_{n,1} & \ldots & \ldots & p_{n,n-1} & 0
\end{array} \right),
\]
where $e_1,\dots,e_{n-1} \in \field^{\times}$ and $p_{i,j} \in \{ 0, \tilde{e}_1 s_1(\bvv) , \dots , \tilde{e}_{l} s_l(\bvv) \}$ with  $\tilde{e}_1,\dots, \tilde{e}_l \in \field^{\times}$ (cf.\ \cite[Section~7, 8 and 9]{Seiss} respectively).
Let $y_1,\dots , y_n$ be a basis of the differential module defined by $A_{\group}(\bss(\bvv))$.
Then, in the respective section of \cite{Seiss}, we chose $y_1$ as a cyclic vector leading to the respective normal form equation.
The specific shape of $A_{\group}(\bss(\bvv))$ implies that the matrix $\BNFcomp$ describing the change of basis
\[
\BNFcomp \, (y_1,\dots, y_n)^{\rm tr}=(y_1,y_1',\dots, y_1^{(n-1)})^{\rm tr}
\]
is a unipotent lower triangular matrix with entries in $\field\{ \bss(\bvv) \}$.
Hence, we have $\BNFcomp \in \GL_n(\field\{ \bss(\bvv) \})$. 

In case $\group$ is the group $\Gzwei$ one checks that the matrix
\[
\BNFcomp=
\begin{pmatrix}
    0 & 1 & 0 & 0 & 0 & 0 & 0 \\
    0 & 0 & 0 & 0 & 0 & 0 & 1\\
    0 & 0 & 0 & 0 & 0 & -1 & 0 \\ 
     \sqrt{2} & 0 & 0 & 0 & 0 & 0 & s_1 \\
    0 & 0 & 2 & 0 & 0 & -s_1 & s_1' \\
     \sqrt{2} s_1 & 0 & 0 & 2 & 0 & -2 s_1' &  s_1'' +s_1^2 \\
    3  \sqrt{2} s_1' & 2 s_2 & 4 s_1 & 0 & -2 & -s_1^2-3  s_1''& 4 s_1 s_1'+ s_1''' 
\end{pmatrix}, 
\]
where $s_i = s_i(\bvv)$ for $i=1,2$, gauge transforms $A_{\group}(\bss(\bvv))$ (cf.\ \cite[Section~11]{Seiss} for its explicit definition) to companion form of the normal form equation.
The entries of $\BNFcomp$ are in $\field\{ \bss(\bvv)\}$ and $\det(\BNFcomp)=8 \sqrt{2}$, implying $\BNFcomp \in \GL_n(\field\{ \bss(\bvv)\})$.  
\end{proof}

\section{The Factorization of the Normal Form Operators}\label{sec:exponentialsol}

In this section we prove that the normal form operators $L_{\group}(\bss(\bvv),\partial) \in  \field\{\bss(\bvv) \}[\partial]$ for the classical groups $\SL_{l+1}$, $\SP_{2l}$, $\SO_{2l+1}$ and $\Gzwei$ factorize over $\field \langle \bvv \rangle$ into a product of operators of order one, where the factors depend linearly on the indeterminates $\bvv$ over $\field$, i.e., they have shape
\[
\partial + c_1 v_1 + \dots + c_l v_l
\]
with $c_i \in \field$ not all zero. 

\begin{lemma}\label{lem:factorization}
Let $R = \field \langle \bvv \rangle[\partial]$ and
\[
A \, = \, \left( \begin{array}{ccccc}
d_1 & e_1 & 0 & \ldots & 0 \\
0 & d_2 & e_2 & \ldots & 0 \\
0 & 0 & d_3 & \ddots & \vdots \\
\vdots & & \ddots & \ddots & e_{n-1} \\
0 & \ldots & \ldots & \ldots & d_n
\end{array} \right) \in R^{n \times n},
\]
where $e_1$, \ldots, $e_{n-1} \in \{ -1, 1,2 \}$.
Then we have
\[
R^{1 \times n} / R^{1 \times n} A \, \cong \,
R / R \, r
\]
with
\begin{equation}\label{eq:op}
r \, = \, (-1)^n (e_1 \cdots e_{n-1})^{-1} d_n d_{n-1} \cdots d_1\,. 
\end{equation}
More precisely, the residue class of
the first standard basis vector $(1, 0, \ldots, 0)$ in $R^{1 \times n} / R^{1 \times n} A$ is a cyclic vector whose annihilator is generated by $r$.
\end{lemma}

\begin{proof}
For $j \in \{ 1, \ldots, n-1 \}$ let
\[
p_j \, = \, (-1)^j (e_{n-1} e_{n-2} \cdots e_{n-j} )^{-1}d_n d_{n-1} \cdots d_{n-j+1} .
\]
Then, by induction, we have
\[
\left( \begin{array}{ccccc}
1 & 0 & \ldots & 0 & 0\\
0 & 1 & \ldots & 0 & 0\\
\vdots & & \ddots & & \vdots\\
0 & & & 1 & 0\\
p_{n-1} & p_{n-2} & \ldots & p_1 & 1
\end{array} \right) A
\, = \, \left( \begin{array}{ccccc}
d_1 & e_1 & 0 & \ldots & 0 \\
0 & d_2 & e_2 & \ldots & 0 \\
\vdots & & \ddots & \ddots & \vdots \\
0 & & & d_{n-1} & e_{n-1} \\
r & 0 & \ldots & 0 & 0
\end{array} \right).
\]
The matrix on the right hand side of the previous equation defines a presentation of the same module $R^{1 \times n} / R^{1 \times n} A$ as the one presented by $A$.
Since $e_1$, \ldots, $e_{n-1}$ are units in $R$, this module is isomorphic to $R / R \, r$ and the claim follows.
\end{proof}

\begin{proposition}\label{prop:factorization}
Let $\group$ be one of the groups $\SL_{l+1}$, $\SP_{2l}$ or $\SO_{2l+1}$. Let $a_k$ be the coefficient of $E_{k,k}$ in the linear representation of 
\[
\gauge{n(\overline{w})}{\ALiou(\bvv)} \, = \, \gauge{n(\overline{w})}{( \sum_{i=1}^l g_i(\bvv) H_i + \sum_{i=1}^l c_i X_{\beta_i})}
\]
(cf.\ Theorem~\ref{thm:RobertzSeissNormalForms}) with respect to the standard basis $E_{i,j}$ of $\field^{n \times n}$.
    Then the normal form operator $L_{\group}(\bss(\bvv),\partial)$ has the factorization over $\difffield \langle  \bvv \rangle$ in first order operators
\[
L_{\group}(\bss(\bvv),\partial) \, = \, (\partial - a_n)  \cdots  (\partial - a_1)
\]
with $a_k \in \field[\bvv]$ homogeneous of degree one.
\end{proposition}

\begin{proof}
Recall that the normal form equation $L_{\group}(\bss(\bvv),\partial) \, y = 0$ is equivalent to the matrix differential equation defined by the normal form matrix $A_{\group}(\bss(\bvv))$. 
Applying the gauge transformation with the inverse of 
\[
\buu(\bvv,f_{l+1}(\bvv), \dots , f_m(\bvv) )
\]
to $A_{\group}(\bss(\bvv))$
we obtain
\[
\gauge{n(\overline{w})}{\ALiou(\bvv)} \, = \, \gauge{n(\overline{w})}{( \sum_{i=1}^l g_i(\bvv) H_i + \sum_{i=1}^l c_i X_{\beta_i})} \in \mathfrak{b}^-(\field [\bvv])\,,
\]
where $g_i(\bvv) \in \field[\bvv]$ are homogeneous of degree one by Theorem~\ref{thm:RobertzSeissNormalForms}. 
For the groups $\SL_{l+1}$, $\SP_{2l}$ and $\SO_{2l+1}$ we used for the construction of $\ALiou(\bvv)$ the respective representations of the Lie algebras presented in \cite[Section 7, 8 and 9]{Seiss} and so 
\[
\gauge{n(\overline{w})}{\ALiou (\bvv)} \, = \, \sum_{k=1}^n a_k E_{k,k} + \sum_{i=1}^{n-1} e_i E_{i,i+1}
\]
is an upper triangular matrix with $e_i \in \{ +1,-1,+2\}$ and $a_k \in \field[\bvv]$ homogeneous of degree one.
Hence, the differential equation
$\partial(\byy) = \gauge{n(\overline{w})}{\ALiou(\bvv)} \, \byy$ under consideration is equivalent to
$A \, \byy = 0$ with operator matrix
\[
A \, = \, \partial I_n - \gauge{n(\overline{w})}{\ALiou (\bvv)} \in \field\langle \bvv \rangle[\partial]^{n \times n} ,
\] 
which has the same shape as the matrix $A$ of Lemma~\ref{lem:factorization} with $d_k = \partial - a_k$. Thus the statement of the proposition follows by applying Lemma~\ref{lem:factorization} and observing that the left hand side is monic.
\end{proof}

Finally we present the factorization for the group $\Gzwei$.
\begin{proposition}\label{prop:factorizationG2}
   Let $\group$ be the exceptional group $\Gzwei$. Then the normal form operator $L_{\Gzwei}(\bss(\bvv),\partial)$ has the factorization over $\difffield\langle \bvv\rangle$ in first order operators
   \[
   L_{\Gzwei}(\bss(\bvv),\partial) \, = \, (\partial+v_1) (\partial - v_1+v_2)(\partial +2v_1-v_2) \partial (\partial - 2v_1+v_2) (\partial +v_1 -v_2)(\partial - v_1)\,.
   \]
\end{proposition}

\begin{proof}
According to \cite[Lemma~11.2]{Seiss} the normal form equation for $\Gzwei$ is
\begin{equation*}
\begin{array}{l}
L_{\Gzwei}(\bss(\bvv),\partial)y \, = \, \\[0.5em] \quad 
y^{(7)} - 2 s_2(\bvv) y' - 2 (s_2(\bvv) y )' - ( s_1(\bvv) y^{(4)})' - (s_1(\bvv) y')^{(4)} + (s_1(\bvv) (s_1(\bvv) y')')'.
\end{array}
\end{equation*} 
We explain briefly how one can compute $s_1(\bvv)$ and $s_2(\bvv)$.   
One uses the representation of the Lie algebra presented in \cite[Section~11]{Seiss} to derive parametrized generators of the torus and the root groups of $\Gzwei$ from it.
Moreover, using the representatives 
\begin{eqnarray*}
    n(w_1) \! & \! = \! & \! -E_{1,1} - E_{2,7} - E_{3,6} + E_{4,5} - E_{5,4} - E_{6,3} + E_{7,2} \quad \text{and} \\
    n(w_2) \! & \! = \! & \! E_{1,1} + E_{2,2} - E_{3,4} + E_{4,3}+E_{5,5} -E_{6,7} + E_{7,6}
\end{eqnarray*}
of the two Weyl group generators, one obtains by matrix multiplication a representative $n(\overline{w})=(n(w_2)n(w_1))^3$ of the longest Weyl group element.
Using these matrices, one follows the construction of the fundamental matrix $\fm$ for $A_{\Gzwei}(\bss(\bvv))$ as introduced in \cite{Seiss_Generic} (cf.\ also Theorem~\ref{thm:RobertzSeissNormalForms}).
Finally, one computes the logarithmic derivative $A_{\Gzwei}(\bss(\bvv))$ of $\fm$ and simply reads off $s_1(\bvv)$ and $s_2(\bvv)$. One will find 
\begin{equation*}
\begin{array}{l}
    s_1(\bvv) \, = \,  3 v_1' + 3 v_1^2 - 3 v_1 v_2 + v_2' + v_2^2 \,,\\[0.5em]
    s_2(\bvv) \, = \, \frac{1}{4} ( 2v_1^{(5)} + 4 v_1 v_1^{(4)} - 2v_1  v_2^{(4)} + (18 v_1 v_2 -14 v_1^2  - 4 v_2^2 + 2 v_1' - 4 v_2')  v_1''' + \\[0.5em] \quad \quad 
      (2 v_1^2 - 4 v_1 v_2 - 6 v_1') v_2''' + (v_1'')^2 + ((36 v_2-72 v_1) v_1' -6 v_2'' 
    + (30 v_1 - 12 v_2)  \\[0.5em] \quad \quad 
    v_2' - 28 v_1^3 + 30 v_1^2 v_2 - 6 v_1 v_2^2) v_1'' + ((26 v_1 - 12 v_2) v_1' 
    + 14 v_1^3 - 14 v_1^2 v_2 + 2 v_1 v_2^2 \\[0.5em] \quad \quad  
    - 10 v_1 v_2') v_2'' - 24 (v_1')^3 + (36 v_1 v_2 -60 v_1^2 
    - 2 v_2^2 + 34 v_2') (v_1')^2 + ( (68 v_1^2  - 38 \\[0.5em] \quad \quad 
    v_1 v_2 + 4 v_2^2) v_2' -10 (v_2')^2  + 12 v_1^4 
     - 48 v_1^3 v_2 + 50 v_1^2 v_2^2 - 18 v_2^3 v_1 + 2 v_2^4) v_1'  + \\[0.5em] \quad \quad (4 v_1 v_2 
     -13 v_1^2) (v_2')^2 
    + (16 v_1^3 v_2-2 v_1^4  - 16 v_1^2 v_2^2 + 4 v_2^3 v_1) v_2' + 4v_1^6 - 12 v_1^5 v_2 + \\[0.5em] \quad \quad
     13 v_1^4 v_2^2  - 6 v_1^3 v_2^3 + v_1^2 v_2^4 ).
\end{array}
\end{equation*}
Now one checks by computation that $L_{\Gzwei}(\bss(\bvv),\partial)$ factors as stated.
\end{proof}

\section{The Associated Equations and Their Riccati Equations}
\label{app:associated&riccati}
Following \cite[Section~3.2.1]{Singer_Reducibility} we define the \emph{$i$-th associated equation}
\[
L^{\det(i)}(\bss(\bvv),\partial)y \, = \, 0 
\]
for the normal form equation $L_{\group}(\bss(\bvv),\partial) \, y = 0$ (cf.\ Definition~\ref{def:associatedRic}) as
the differential equation of lowest order whose solution space $V^{\det(i)}$ is spanned by the elements of the set
 \begin{equation*}
 \big\{ \det(\wronski(y_{j_1},\dots , y_{j_i}) ) \mid \{j_1,\dots,j_i\} \subset \{1,\ldots,n\} \big\} , 
 \end{equation*}
where $y_1,\dots,y_n$ is a basis of the solution space $V \subset \generalext$ of $L_{\group}(\bss(\bvv),\partial) \, y = 0$.
The vector space $V^{\det(i)}$ is left invariant under the action of $\group$ and therefore the associated equation has coefficients in $\field\langle \bss(\bvv) \rangle$. The map
\begin{equation*}
\bigwedge\nolimits^i V \to V^{\det(i)}, \quad y_{j_1} \wedge \dots \wedge y_{j_i} \mapsto \det(\wronski(y_{j_1},\dots , y_{j_i}) )
\end{equation*}
is a surjective $\group$-homomorphism. It is a $\group$-isomorphism if and only if the order of 
the associated equation is $\binom{n}{i}$.

\begin{remark}
It is also explained in \cite[Section~3.2.1]{Singer_Reducibility} how one can construct the $i$-th associated equation
$L^{\det(i)}(\bss(\bvv),\partial) \, y = 0$ from the normal form equation $L_{\group}(\bss(\bvv),\partial) \, y = 0$ and a basis $y_1,\dots,y_n$ of its solution space. In fact, one differentiates $\binom{n}{i}$ times the Wronskian determinant
\[
w \, = \, \det(\wronski(y_{1},\dots , y_{i}) )
\]
and uses in each step the relation $L_{\group}(\bss(\bvv),\partial) \, y_k = 0$ to eliminate the $n$-th derivative of $y_k$ for $1 \leq k \leq i$.
Finally, one determines the $\field\langle \bss(\bvv) \rangle$-linear dependencies among
$$w, w', \dots, w^{\left(\binom{n}{i}\right)}$$
and one picks a dependency relation whose maximum differentiation order is minimal
for the definition of the associated equation.
\end{remark}

\begin{proposition}\label{prop:assoc_operator}
Let $\group$ be one of the classical groups $\SL_{l+1}$, $\SP_{2l}$, $\SO_{2l+1}$  
or $\Gzwei$ and let $\generalext/ \difffield \langle \bss(\bvv) \rangle$ be the respective general extension defined by the normal form equation 
\[
L_{\group}(\bss(\bvv),\partial) \, y \, = \, 0
\]
with solution space $V$ in $\generalext$. Then the $i$-th associated equation 
\[
L^{\det(i)}(\bss(\bvv),\partial) \, y \, = \, 0
\]
has the following exponential as a solution:
\begin{enumerate}
    \item\label{prop:assoc_operator_sl}
    In case of $\SL_{l+1}$ the exponential is 
    \[
    \exp(\smallint v_{i}) \quad \mbox{for} \ i=1,\dots,l\,.
    \]
    \item\label{prop:assoc_operator_sp}
    In case of $\SP_{2l}$ the exponential is
    \[
    \exp(-\smallint v_{i}) \quad \mbox{for} \ i=1,\dots ,l\,.
    \]
    \item\label{prop:assoc_operator_so_odd}
    In case of $\SO_{2l+1}$ the exponential is
    \[
    \exp(-\smallint v_{i}) \quad \mbox{for} \ i=1,\dots ,l-1 \quad \mbox{and} \quad \exp(-\smallint 2 v_{l})
    \quad \mbox{for} \ i=l\,.
    \]
    \item\label{prop:assoc_operator_g2}
    In case of $\Gzwei$ the exponential is 
    $\exp(\smallint v_1)$ for $i=1$ and $\exp(\smallint v_2)$ for $i = 2$. 
\end{enumerate}
Moreover, the logarithmic derivatives of the exponential solutions of the associated equations generate the same $\Z$-module as $g_1(\bvv),\dots,g_l(\bvv)$.
\end{proposition}
\begin{proof}
It is well known (cf.\ \cite[Ex.~1.14, 5.~(b)]{vanderPutSinger}) that for a linear differential equation
\[
y^{(n)} \, = \, c_{n-1}y^{(n-1)}+\dots + c_0 y
\]
with coefficients in some differential field $\widetilde{F}$ and for a basis $\widetilde{y}_1,\dots,\widetilde{y}_n$ of its solution space in a Picard-Vessiot extension $\widetilde{E}$ of $\widetilde{F}$, the determinant of the Wronskian matrix $\wronski(\widetilde{y}_1,\dots,\widetilde{y}_n)$ satisfies the first order linear differential equation 
\[
y' \, = \, c_{n-1}y\,.
\]
Having said that, we consider for each of the above groups $\group$ the intermediate products
\[
\mathcal{L}^{\group}_i(\partial) \, := \, (\partial-a_{1+i}) \cdots (\partial-a_1) \qquad \text{for} \ i=0,\dots,l-1
\]
of the factorization of its normal form operator $L_{\group}(\bss(\bvv),\partial)$ from Proposition~\ref{prop:factorization} and Proposition~\ref{prop:factorizationG2}.
The solution space $\mathcal{V}_i^{\group}$ of the equation $\mathcal{L}^{\group}_i(\partial) \, y = 0$ is a subspace of the solution space $V$ of $L_{\group}(\bss(\bvv),\partial) \, y = 0$
and we denote a basis of $\mathcal{V}_i^{\group}$ by $\widetilde{y}_1,\dots,\widetilde{y}_{i+1}$.
Moreover, the coefficient of the second highest derivative in $\mathcal{L}^{\group}_i(\partial)$ is the sum
$$-a_{1+i} - \dots - a_1.$$
Applying now the  result mentioned at the beginning of this proof, we obtain
\[
\det(\wronski(\widetilde{y}_1,\dots,\widetilde{y}_{i+1}))' \, = \, (a_{1+i}+\dots +a_1) \, \det(\wronski(\widetilde{y}_1,\dots,\widetilde{y}_{i+1})).
\]
Hence, $0 \neq \det(\wronski(\widetilde{y}_1,\dots,\widetilde{y}_{i+1}))$ is
an element of the solution space $V^{\det(i)}$ of the $i$-th associated operator $L^{\det(i)}(\bss(\bvv),\partial)$, whose logarithmic derivative is
\[
\det(\wronski(\widetilde{y}_1,\dots,\widetilde{y}_{i+1}))' \det(\wronski(\widetilde{y}_1,\dots,\widetilde{y}_{i+1}))^{-1} \, = \, a_{1+i}+\dots +a_1 \, .
\]

In the following we compute the explicit values of $a_1,\dots,a_n$ for each of the groups $\SL_{l+1}$, $\SP_{2l}$, $\SO_{2l+1}$ and $\Gzwei$ separately.
The assertion then follows by applying the above result to these explicit values. 
We recall from \cite[Section~5]{Seiss} (note that there $g_i(\bvv)$ is $\overline{g}_i(\bvv)$) that the diagonal entries of $\gauge{n(\overline{w})}{\ALiou(\bvv)}$ are equal to the negatives of the diagonal entries of $\Ad(\buu(\bvv,\bff))(A_0^+)$ (cf.\ Section~\ref{sec:normalform}).
Thus, Remark~\ref{remark3} implies that it is sufficient to compute the diagonal entries of
$$\Ad(u_1(v_1)\cdots u_l(v_l))(A_0^+) . $$
In the following we write $E$ for the identity matrix $E_{1,1}+ \dots + E_{n,n}.$

\ref{prop:assoc_operator_sl}\
For the construction of the normal form matrix for $\SL_{l+1}$  we used the representation of  $\Lie(\SL_{l+1})$ introduced in \cite[Section~7]{Seiss}. Using this representation we find 
\[
A_0^+ \, = \, \sum_{i=1}^l E_{i,i+1} \quad \text{and} \quad 
u_{-\alpha_i}(v_i) \, = \, E + E_{i+1,i}v_i  \quad \text{for} \ 1 \leq i \leq l.
\]
Carrying out the respective matrix multiplications we obtain 
\[
A_0^+ + \sum_{i=1}^l- v_i H_i +r \, = \, \sum_{i=1}^l E_{i,i+1} + \sum_{i=1}^l -v_i(E_{i,i}-E_{i+1,i+1}) +r
\]
with $r$ a lower triangular nilpotent matrix in $\gl_n(C)$. 
Hence, the diagonal entries of $\gauge{n(\overline{w})}{\ALiou(\bvv)}$ are 
\[
a_1 \, = \, v_1\,, \quad
a_2 \, = \, -v_{1} +v_2\,, \quad \dots, \quad
a_l \, = \, -v_{l-1} + v_l\,, \quad
a_{l+1} \, = \, -v_l \,. 
\]
We conclude that for each $1\leq i \leq l$ the associated equation $L^{\det(i)}(\bss(\bvv),\partial) \, y = 0$ has the exponential solution $\exp(\int v_i)$.

\ref{prop:assoc_operator_sp}\
The construction of the normal form matrix in case of the group $\SP_{2l}$  uses the representation of the Lie algebra presented in \cite[Section~8]{Seiss}. Note that there we renumbered the rows and columns of matrices in $\field^{2l\times 2l}$ using the range $(1,2,\dots,l,-l,\dots,-2,-1)$.
In this representation we have 
\begin{gather*}
A_0^+ = ( \sum_{i=1}^{l-1} E_{i,i+1}-E_{-i-1,-i}) + E_{l,-l} \ , \quad 
u_{-\alpha_i}(v_i) = E + v_i(-E_{i+1,i}+E_{-i,-i-1}) \\ \text{for} \ 1 \leq i \leq l-1 \quad  \text{and} \quad 
u_{-\alpha_l}(v_l) = E + v_l E_{-l,l} \, .
\end{gather*}
Carrying out the matrix multiplications we find 
\[
A_0^+ + \sum_{i=1}^l(\tilde{g}_i H_i) +r = A_0^+ + \sum_{i=1}^l \tilde{g}_i(E_{i,i}-E_{-i,-i}) +r,
\]
where $\tilde{g}_1=v_1$ and $\tilde{g}_i=v_i-v_{i-1}$ for $2\leq i \leq l$ and $r \in \gl_n(C)$ a lower triangular nilpotent matrix. We conclude that the first $l$ diagonal entries of $\gauge{n(\overline{w})}{\ALiou(\bvv)}$ are 
$$
a_1 = -v_1, \ a_2 = -v_2 + v_{1} , \dots, a_l= -v_l+v_{l-1}
$$
and so for $1\leq i \leq l$ the $i$-th associated equation $L^{\det(i)}(\bss(\bvv),\partial)y=0$ has the exponential solution $\exp(\int -v_i)$.

\ref{prop:assoc_operator_so_odd}\
For the construction of the normal form matrix for $\SO_{2l+1}$ we used the representation of $\Lie(\SO_{2l+1})$ introduced in \cite[Section 9]{Seiss}. There, we renumbered the rows and columns of matrices in 
$\gl_{2l+1}(\field)$ using the range $(1,\dots,l,0,-l,\dots,-1)$. 
In this representation we find  
\begin{gather*}
\begin{array}{rcl}
A_0^+ & = &  \displaystyle \sum_{i=1}^{l-1} (E_{i,i+1}-E_{-i-1,-i}) + 2E_{l,0} +E_{0,-l} \ , \\[0.5em] 
u_{-\alpha_i}(v_i)& =&  E +v_i (-E_{i+1,i}+ E_{-i,-i-1}) \qquad \text{for} \ 1 \leq i \leq l-1  \quad \text{and} \\[0.5em]
u_{-\alpha_l}(v_l)&=& E +v_l (-E_{0,l}-2 E_{-l,0}) +v_l^2 E_{-l,l} \, .
\end{array}
\end{gather*}
Multiplying out the conjugation of $A_0^+$ by  $u_1(v_1)\cdots u_l(v_l)$ we obtain 
\[
A_0^+ + \sum_{i=1}^l \tilde{g}_i H_i +r \, = \, A_0^++ \sum_{i=1}^l \tilde{g}_i (E_{i,i}-E_{-i,-i}) +r\,,
\]
where $\tilde{g}_1=v_1$ and $\tilde{g}_i=v_i-v_{i-1}$ for $2\leq i\leq l-1$ and $\tilde{g}_l=2v_l-v_{l-1}$. As before, $r \in \gl_n(\field)$ is a suitable lower triangular nilpotent matrix. Thus, the first $l$ diagonal entries of $\gauge{n(\overline{w})}{\ALiou(\bvv)}$ are 
$$
a_1 \, = \, -v_1, \, \,
a_2 \, = \, -v_2+v_1, \, \, \dots, \, \,
a_{l-1} \, = \, -v_{l-1}+v_{l-2}, \, \,
a_l \, = \, -2v_l+v_{l-1}\,.
$$
Hence, for $1\leq i \leq l-1$ the $i$-th associated equation has the exponential solution $\exp(\int - v_i)$ and the $l$-th associated equation has the exponential solution $\exp(-2v_l)$.

\ref{prop:assoc_operator_g2}\
According to Proposition~\ref{prop:factorizationG2} the last $l=2$ factors in the factorization of the normal form operator $L_{\Gzwei}(\bss(\bvv),\partial)$ are $\partial +v_1 -v_2$ and $\partial - v_1$.
Hence, the first associated equation has the exponential solution $\exp(\int v_1)$ and the second one has the exponential solution $\exp(\int v_2)$. 

It is left to show the assertion of the supplement. Since the diagonal entries of $\ALiou(\bvv)$ are $\Z$-linear combinations of $g_1(\bvv), \dots,g_l(\bvv)$ and conjugation with $n(\overline{w})$ permutes those entries and potentially changes their signs, we conclude that the diagonal entries of  $n(\overline{w}).\ALiou(\bvv)$ are also $\Z$-linear combinations of $g_1(\bvv), \dots,g_l(\bvv)$. Because the logarithmic derivatives of the exponential solutions of the associated operators are the partial sums of these diagonal entries, it follows that they are contained in the $\Z$-span of $g_1(\bvv), \dots,g_l(\bvv)$. The reversed inclusion  simply follows from the fact that $g_i(\bvv)$ are $\Z$-linear combinations of the $v_i$ (in case $\group = \SO_{2l+1}$ only one $g_i$ involves $v_l$ and it has the form $\pm (2 v_l-v_{l-1})$) and that the logarithmic derivatives of the exponential solutions of the associated operators have the form $\pm v_i$ (in case $\group = \SO_{2l+1}$ one logarithmic derivative is $-2v_l$).
\end{proof}

Recall that in Definition~\ref{def:associatedRic} we denoted by $\Ric{i}{\bss(\bvv)}{y} = 0$ the Riccati equation for the 
$i$-th associated equation $L^{\det(i)}(\bss(\bvv),\partial) \, y = 0$ for the normal form equation $L_{\group}(\bss(\bvv),\partial) \, y = 0$. 
More generally, for a definition and construction of a Riccati equation associated to a scalar differential equation, cf.\ e.g.\ \cite[Definition~4.6]{vanderPutSinger}.
From Proposition~\ref{prop:assoc_operator} we immediately obtain the following corollary.
\begin{corollary}\label{cor:riccati_assoc_operator} $\,$
\begin{enumerate}
    \item Case $\SL_{l+1}$: For $i=1,\dots,l$ the Riccati equation $\Ric{i}{\bss(\bvv)}{y} = 0$ for the $i$-th associated equation $L^{\det(i)}(\bss(\bvv),\partial) \, y = 0$ has $v_{i}$ as a solution.
    \item Case $\SP_{2l}$: For $i=1,\dots ,l$ the Riccati equation $\Ric{i}{\bss(\bvv)}{y} = 0$ for the $i$-th associated equation $L^{\det(i)}(\bss(\bvv),\partial) \, y = 0$ has $-v_{i}$ as a solution.
    \item Case $\SO_{2l+1}$: For $i=1,\dots ,l-1$ the Riccati equation $\Ric{i}{\bss(\bvv)}{y} = 0$ for the $i$-th associated equation $L^{\det(i)}(\bss(\bvv),\partial) \, y = 0$ has $-v_{i}$ as a solution
    and the Riccati equation $\Ric{l}{\bss(\bvv)}{y} = 0$ for the $l$-th associated equation $L^{\det(l)}(\bss(\bvv),\partial) \, y = 0$ has
    $-2 v_{l}$ as a solution.
    \item Case $\Gzwei$: The Riccati equations $\Ric{1}{\bss(\bvv)}{y} = 0$ and $\Ric{2}{\bss(\bvv)}{y} = 0$ for the first and second associated equations, i.e.\ for 
    \[
    L^{\det(1)}(\bss(\bvv),\partial) \, y \, = \, L_{\Gzwei}(\bss(\bvv),\partial) \, y \, = \, 0
    \]
    and for $L^{\det(2)}(\bss(\bvv),\partial) \, y = 0$, have $v_{1}$ and $v_2$ as solutions respectively. 
\end{enumerate}
\end{corollary}

\section{The Denominators in the Bruhat Decomposition}\label{sec:denom}
In this section we are going to analyze in more detail the coefficients in $\field(\group)$ of the Bruhat decomposition 
\[
\overline{Y}  \, = \, \buu(\bxx) \, n(\overline{w}) \, \btt(\boldsymbol{z}) \, \buu(\bww) 
\]
where the entries $\overline{Y}_{i,j}$ are the residue classes of $Y_{i,j}$ in $C[G]=C[Y_{i,j},\det(Y)^{-1}]/I_G$.

\begin{lemma}\label{lem:sollinearsystem}
For $n\in \N$ let $Y=(Y_{i,j})$ be an $n \times n$ matrix whose entries are indeterminates $Y_{i,j}$ over $\field$.
For further indeterminates $c_1,\dots,c_n$ and $y_{i,j}$ over $\field$ with $2 \leq i \leq n $ and $1 \leq j <i$ consider the matrices
\[
W \, := \, \left( \begin{array}{cccc}
0 & \ldots & 0 & c_1 \\
\vdots & \addots & c_2 & 0 \\
0 & \addots & \addots & \vdots \\
c_n & 0 & \ldots & 0
\end{array}
\right), \qquad
L \, := \, \left( \begin{array}{ccccc}
1 & 0 & \ldots & \ldots & 0 \\
y_{2,1} & 1 & 0 & \ldots & 0 \\
y_{3,1} & y_{3,2} & 1 & \ddots & \vdots \\
\vdots & \ddots & \ddots & \ddots & 0 \\
y_{n,1} & \ldots & y_{n,n-2} & y_{n,n-1} & 1
\end{array} \right).
\]
Moreover, for $1 \leq k \leq n-1$ define the matrices and vectors
\begin{equation*}
\begin{array}{rcl}
\displaystyle
Y^{(k)} & = & (Y_{i,j})_{{1 \le i \le n-k} \atop {k+1 \le j \le n}} \ \text{and} \\[0.5em] 
\displaystyle
v_k & = & ( -Y_{n-k+1,k+1},- Y_{n-k+1,k+2}. \dots,- Y_{n-k+1,n}  ).
\end{array}
\end{equation*}
\begin{enumerate}
    \item\label{item(a):sollinearsystem}
    The linear system of equations $(W L Y)_{i,j} =  0$ with $1\leq i \leq n-1$ and $ i < j \leq n$ in the indeterminates  $y_{i,j}$ over $\field(c_1,\dots,c_n,Y_{i,j})$ is equivalent to the conjunction of the $(n-1)$ linear systems
    \begin{equation}\label{eq:sollinearsystem1}
    \left\{
    \begin{array}{lcl} 
    \displaystyle
    c_k \sum_{s=1}^{n-k}  y_{n+1-k,s} Y_{s,k+1} &=& - c_k Y_{n-k+1,k+1} \\[0.5em] 
    \displaystyle
    c_k \sum_{s=1}^{n-k}  y_{n+1-k,s} Y_{s,k+2} &=& - c_kY_{n-k+1,k+2} \\[0.5em]
    \qquad \qquad \vdots && \qquad \vdots \\[0.5em] 
    \displaystyle
    c_k \sum_{s=1}^{n-k}  y_{n+1-k,s} Y_{s,n} &=& -c_k Y_{n-k+1,n} 
    \end{array}
    \right\}
\quad \text{with}  \quad 1 \leq k \leq n-1.
\end{equation}
    \item\label{item(b):sollinearsystem} For $1 \leq k \leq n-1$ the systems of equations in \ref{item(a):sollinearsystem} have the unique solutions 
    \[
    y_{n+1-k,s} \, = \, \frac{\det(Y_s^{(k)})}{\det(Y^{(k)})} \qquad \text{with} \ 1 \leq s \leq n-k\,,
    \]
    where $Y_s^{(k)}$ is the matrix obtained by replacing the $s$-th row by $v_k$.
    \item\label{item(c):sollinearsystem}
    Substituting the solutions of \ref{item(b):sollinearsystem} into $WLY$ the diagonal entries $(W L Y)_{i,i}$ become 
    \begin{gather*}
    (-1)^{n-1} c_1 \frac{\det(Y)}{\det(Y^{(1)})}, \ (-1)^{n-2}c_2 \frac{\det(Y^{(1)})}{\det(Y^{(2)})}, \ \dots , \\
    \qquad \qquad \qquad \qquad \qquad \qquad -c_{n-1} \frac{\det(Y^{(n-2)})}{\det(Y^{(n-1)})}  , \ c_n \det(Y^{(n-1)}) . 
    \end{gather*}
\end{enumerate}
\end{lemma}
\begin{proof}
\ref{item(a):sollinearsystem}\
The matrix $WL$ is the matrix whose rows are  
\[
c_1 L_{n,*}, \quad c_2 L_{n-1,*} , \quad \dots , \ c_{n-1} L_{2,*} , \quad c_n L_{1,*}\,,
\]
where $L_{i,*}$ denotes the $i$-th row of $L$. Now for $1 \leq k \leq n-1$ the $k$-th system in \eqref{eq:sollinearsystem1} is obtained by multiplying the $k$-th row of $WL$ with the columns $Y_{*,k+1}, \dots, Y_{*,n}$ of $Y$ and then equating these products to zero.
The right hand sides of the so obtained system are those entries of the $k$-th row of  $WLY$ which lie above the diagonal.\\
\ref{item(b):sollinearsystem}\
Observe that the systems in \eqref{eq:sollinearsystem1} are equivalent to 
\[
L'_k Y^{(k)}  = v_k \qquad \text{for} \ 1 \leq k \leq n-1\,,
\]
where $L'_k$ denotes the vector of indeterminates
$$ L'_k =(y_{n+1-k,1},\dots,y_{n+1-k,n-k}). $$
The assertion then follows from applying Cramer's rule.\\
\ref{item(c):sollinearsystem}\
Applying the argumentation of \ref{item(a):sollinearsystem}\ to the diagonal entries $(W L Y)_{i,i}$ of $WLY$, we observe that they are
\begin{gather*}
\begin{array}{c} 
\displaystyle
c_1 (\sum_{s=1}^{n-1}  y_{n,s} Y_{s,1} + Y_{n,1}),  \ \dots , \
c_k (\sum_{s=1}^{n-k}  y_{n+1-k,s} Y_{s,k} + Y_{n+1-k,k}) \ \dots ,  \\[0.5em] 
\displaystyle
c_{n-1} (  y_{2,1} Y_{1,n-1} + Y_{2,n-1}) \quad \text{and} \quad c_n Y_{1,n}=c_n \det(Y^{(n-1)}) .
\end{array}
\end{gather*}
Let $1 \leq k \leq n-1$. Substituting the solutions of \ref{item(b):sollinearsystem}\ into the $k$-th diagonal element 
\[
c_k (\sum_{s=1}^{n-k}  y_{n+1-k,s} Y_{s,k} + Y_{n+1-k,k}),
\]
we obtain
\[
c_k (\sum_{s=1}^{n-k} \frac{\det(Y_s^{(k)})}{\det(Y^{(k)})}  Y_{s,k} + Y_{n+1-k,k}) =
c_k (\sum_{s=1}^{n-k} \frac{\det(Y_s^{(k)})}{\det(Y^{(k)})}  Y_{s,k} + \frac{\det(Y^{(k)})}{\det(Y^{(k)})} Y_{n+1-k,k})\,.
\]
Observe that
\[
Y^{(k-1)} \, = \, (Y_{i,j})_{{1 \le i \le n-k+1} \atop {k \le j \le n}} \, = \, 
\left(
\begin{array}{c|ccccc}
    Y_{1,k} &&&&& \\
    Y_{2,k} &&&&&\\
    \vdots  &&&Y^{(k)}&& \\
    &&&&& \\
    Y_{n-k,k} &&&&& \\ 
    \hline
    Y_{n-k+1,k} & & &-v_k & &
\end{array}
\right).
\]
Developing the determinant of $Y^{(k-1)}$ for the first column, we obtain 
\begin{equation}\label{eq:sollinearsystem2}
\det(Y^{(k-1)})=\sum_{s=1}^{n-k} (-1)^{s-1}Y_{s,k} \det(\widehat{Y}^{(k-1)}_s) + (-1)^{n-k} Y_{n-k+1,k} \det(Y^{(k)}),
\end{equation}
where $\widehat{Y}^{(k-1)}_s$ is the matrix obtained by canceling the $s$-th row of $\begin{pmatrix}
    Y^{(k)} \\ -v_k
\end{pmatrix}$. 
Swapping over the last row of $\widehat{Y}^{(k-1)}_s$ until it becomes the $s$-th row, the determinant of the so obtained matrix is  
\[
(-1)^{n-k-s} \det(\widehat{Y}^{(k-1)}_s)=-\det(Y^{(k)}_s)\,,
\] 
which is equivalent to 
$$
\det(\widehat{Y}^{(k-1)}_s)=(-1)^{-n+k+s+1}  \det(Y^{(k)}_s).
$$
Substituting now in equation \eqref{eq:sollinearsystem2} the expression for $\det(\widehat{Y}^{(k-1)}_s)$, we obtain 
\begin{eqnarray*}
\det(Y^{(k-1)})
& = & \sum_{s=1}^{n-k} (-1)^{-n+k+2s} Y_{s,k} \det(Y^{(k)}_s) + (-1)^{n-k} Y_{n-k+1,k} \det(Y^{(k)}) \\
& = & \sum_{s=1}^{n-k} (-1)^{-n+k} Y_{s,k} \det(Y^{(k)}_s) + (-1)^{n-k} Y_{n-k+1,k} \det(Y^{(k)}) \\
& = & (-1)^{n-k}( \sum_{s=1}^{n-k}  Y_{s,k} \det(Y^{(k)}_s) +   Y_{n-k+1,k} \det(Y^{(k)}))
\end{eqnarray*}
and so the $k$-th diagonal entry becomes 
\begin{equation*}
    c_k (\sum_{s=1}^{n-k} \frac{\det(Y_s^{(k)})}{\det(Y^{(k)})}  Y_{s,k} + \frac{\det(Y^{(k)})}{\det(Y^{(k)})} Y_{n+1-k,k}) = (-1)^{n-k} c_k\frac{\det(Y^{(k-1)})}{\det(Y^{(k)})}.
\end{equation*}
\end{proof}

\begin{proposition}\label{prop:BruhatParametersofX}
Let $\group$ be one of the groups $\SL_{l+1}$, $\SP_{2l}$ or $\SO_{2l+1}$ and let 
\[
\field[\group] \, = \, \field[\overline{Y}_{i,j} \mid i,j=1, \dots, n] \, = \, \field[Y_{i,j}\mid i,j=1, \dots, n]/I_{\group}
\]
the coordinate ring of $\group$, where $I_{\group}$ is the defining ideal of $\group$ and $n$ denotes the dimension of the representation of $\group$, that is  $n=l+1$, $n=2l$ and  $n=2l+1$ respectively. Then there exist $e_1,\dots,e_l \in \field[\group]$ and  
$$\boldsymbol{z} = (\prod_{j=1}^l e_j^{a_{1,j}},\dots, \prod_{j=1}^l e_j^{a_{l,j}}) \quad \text{with} \ a_{i,j} \in \Z$$  
and $\bxx:=(x_1,\dots,x_m)$ and $\bww:=(w_1,\dots,w_m)$ in the localization $\mathcal{M}^{-1} \field[\group]$ of $\field[\group]$ by the multiplicatively closed subset  $\mathcal{M}$ generated by $e_1,\dots,e_l$ such that 
    \[
    \overline{Y}  \, = \, \buu(\bxx) \, n(\overline{w}) \, \btt(\boldsymbol{z}) \, \buu(\bww)  
    \]
    is the Bruhat decomposition of $\overline{Y}:=(\overline{Y}_{i,j})$.  
\end{proposition}
\begin{proof}
    Since we used the representations of $\SL_{l+1}$, $\SP_{2l}$ and $\SO_{2l+1}$ presented in Sections~7, 8 and 9 of \cite{Seiss}, the maximal unipotent subgroups $\unipotent^-$ of these groups consist of lower triangular matrices and their maximal tori $\torus$ of diagonal matrices.
    In particular, $\borel^-$ consists of lower triangular matrices, too. The inverses of the representatives $n(\overline{w})$ of the longest Weyl group elements for these groups are matrices of the form
    \[ 
\left( \begin{array}{cccc}
0 & \ldots & 0 & c_1 \\
\vdots & \addots & c_2 & 0 \\
0 & \addots & \addots & \vdots \\
c_n & 0 & \ldots & 0
\end{array}
\right)
\quad \text{with} \ c_i \in \{\pm 1 \}.
    \]
    
    On the one hand, according to \cite[Lemma~4.2]{Seiss_Generic}, there exist $\bxx=(x_1,\dots,x_m)$, $\bzz=(z_1,\dots, z_l)$ and $\bww=(w_1,\dots,w_m)$ in the field of fractions $\field(\group)$ of the coordinate ring $\field[\group]$ such that 
    \[
    \overline{Y} \, = \, \buu(\bxx) \, n(\overline{w}) \, \btt(\boldsymbol{z}) \, \buu(\bww)\,.
    \]
    On the other hand, Lemma~\ref{lem:sollinearsystem} applied with $W = n(\overline{w})^{-1}$ to $Y:=(Y_{i,j})$ yields a unipotent lower triangular matrix $L^{-1}$ and a lower triangular matrix $B$ such that 
    \[
    Y \, = \, L^{-1} \, n(\overline{w}) \, B
    \]
    with the following properties. The diagonal entries of $B$ are as in
    Lemma~\ref{lem:sollinearsystem} \ref{item(c):sollinearsystem} and the remaining entries of $B$ and the entries of $L$ are in the localization of $\field[Y_{i,j}]$ by the multiplicatively closed subset generated by
    \begin{equation}\label{eqn:1appendixD}
    \det(Y^{(n-1)}), \ \det(Y^{(n-2)}), \ \dots , \ \det(Y^{(1)}), \  \det(Y) \in \field [Y_{i,j}]\,. 
    \end{equation}
    Since the evaluation of the elements in \eqref{eqn:1appendixD}
     at the point $n(\overline{w})$ does not vanish, we conclude that these elements are not in the defining ideal $I_{\group}$ of $\group$. Thus, the entries of $L^{-1}$ and $B$ map to the localization of $\field[\group]$ by the multiplicatively closed subset generated by the residue classes
\begin{gather*}
   d_n \, := \, \det(Y^{(n-1)}) + I_{\group}, \quad d_{n-1} \, := \, \det(Y^{(n-2)})+I_{\group}, \quad \ldots , \\ d_2 \, := \, \det(Y^{(1)})+I_{\group}, \quad d_1 \, := \, \det(Y) +I_{\group}\,.
\end{gather*}
     We denote the respective image matrices by $\overline{L}^{-1}$ and $\overline{B}$. We obtain\
    \[
    \buu(\bxx) \, n(\overline{w}) \, \btt(\boldsymbol{z}) \, \buu(\bww) \, = \, \overline{Y} \, = \, \overline{L}^{-1} \, n(\overline{w}) \, \overline{B}\,,
    \]
    which is equivalent to
    \[
    \overline{L}\buu(\bxx) \, = \, n(\overline{w}) \, \overline{B} \, (\btt(\boldsymbol{z}) \, \buu(\bww))^{-1} \, n(\overline{w})^{-1}\,.
    \] 
    The left hand side of the last equality is a unipotent lower triangular matrix and since conjugation by $n(w)$ maps a lower triangular matrix to an upper triangular one, the right hand side is an upper triangular matrix.
    This implies that $\overline{L}^{-1} = \buu(\bxx)$ and that $\overline{B}=\btt(\boldsymbol{z}) \, \buu(\bww)$.

    According to \cite[Theorem~1.28]{vanderPutSinger} (note that the torsor here is trivial) there exists a differential $\field\langle \bss(\bvv)\rangle$-isomorphism of Picard-Vessiot rings
    \[
    \psi\colon \field\langle \bss(\bvv)\rangle \otimes_{\field} \field[\group] \to \field\langle \bss(\bvv)\rangle[\fm], \quad  1 \otimes \overline{Y}_{i,j} \mapsto \fm_{i,j} ,
    \]  
    which extends to a differential $\field\langle \bss(\bvv)\rangle$-isomorphism of Picard-Vessiot fields.
    The exponential solutions $\expsolass_1,\dots, \expsolass_l$ of the associated equations (cf.\ Proposition~\ref{prop:assoc_operator}) trivially satisfy linear differential equations over $\field\langle \bss(\bvv)\rangle$ and so,
    according to \cite[Corollary~1.38]{vanderPutSinger}, they lie in the Picard-Vessiot ring $\field\langle \bss(\bvv) \rangle [\fm]$.
    By Proposition~\ref{prop:assoc_operator} 
    there exist $b_{i,j} \in \Z$ such that 
    \[
     \expsolass_i \, = \, \prod_{j=1}^l \exp_j^{b_{i,j}}\,.
    \]
Moreover, since $\psi^{-1}$ sends $\exp_j$ to $z_j \in \field(\group)$, we conclude that $\psi^{-1}(\expsolass_i) =: e_i$ lies in $\field(\group)$.
Because $\expsolass_i$ lies in the Picard-Vessiot ring $\field\langle \bss(\bvv) \rangle[\fm]$,
we finally obtain that $e_i \in \field[\group]$.
Since by Proposition~\ref{prop:assoc_operator} there exist $a_{i,j} \in \Z$ such that  
\[
\exp_i \, = \, \prod_{j=1}^l (\expsolass_j)^{a_{i,j}}\,,
\]
we obtain by applying $\psi^{-1}$ that  
\[
z_i \, = \, \prod_{j=1}^l e_j^{a_{i,j}}\,.
\]
Denote by $\check{z}_1,\dots,\check{z}_n$ the diagonal entries of $t(\bzz)$ and note that these are products of $e_1,\dots,e_l$ with exponents in $\Z$ and a non-zero constant.
Comparing the diagonal entries of the left and right hand side of 
$\overline{B}=\btt(\boldsymbol{z}) \buu(\bww)$ we obtain
\begin{gather*}
    \diag( \check{z}_1,\check{z}_2,\dots,\check{z}_{n-1},\check{z}_n) =
    \diag ( \frac{d_1}{d_2},\frac{d_2}{d_3}, \dots,\frac{d_{n-1}}{d_n} ,d_n)\,,
\end{gather*}
which is equivalent to 
    \begin{gather*}
   d_n \, = \, \check{z}_n, \frac{d_{n-1}}{d_n} =  \check{z}_{n-1}, \dots, \frac{d_2}{d_3}= \check{z}_2 , \frac{d_1}{d_2} =  \check{z}_1 .
    \end{gather*} 
We conclude that $d_i$ is the product of $e_1,\dots,e_l$ with exponents in $\Z$ and a non-zero constant. Combining this with Lemma~\ref{lem:sollinearsystem} it follows that the entries of $\overline{L}=\buu(\bxx)$ and $\overline{B}=\btt(\bzz) \buu(\bww)$ are in the localization of $\field[\group]$ by the multiplicatively closed subset $\mathcal{M}$ generated by $e_1,\dots,e_l$.
Since the entries of $\btt(\bzz)^{-1}$ are also contained in $\mathcal{M}^{-1} \field[\group]$, the same is true for $\buu(\bww)$. Since producing the standard factorization of $\buu(\bxx)$ and $\buu(\bww)$ into root group elements only involves operations in $\mathcal{M}^{-1} \field[\group]$, we conclude that $\bxx$ and $\bww$ are in $\mathcal{M}^{-1} \field[\group]$
\end{proof}

\begin{proposition} \label{prop:BruhatParametersofXforG2}
    Let $\group = \Gzwei \subset \SO_{7}$. In the notation of
    Proposition~\ref{prop:BruhatParametersofX} there exist $e_1,e_2 \in \field[\group]$ and  
$$\boldsymbol{z} = (e_1,e_2^{-1})$$   
and $\bxx:=(x_1,\dots,x_6)$ and $\bww:=(w_1,\dots,w_6)$ in the localization $\mathcal{M}^{-1} \field[\group]$ of $\field[\group]$ by the multiplicatively closed subset  $\mathcal{M}$ generated by $e_1$ and $e_2$ such that 
    \[
    \overline{Y}  \, = \, \buu(\bxx) \, n(\overline{w}) \, \btt(\boldsymbol{z}) \, \buu(\bww)  
    \]
    is the Bruhat decomposition of $\overline{Y}:=(\overline{Y}_{i,j})$. 
\end{proposition}
\begin{proof}
In the proof of Proposition~\ref{prop:factorizationG2} we explained how one can compute the explicit Bruhat decomposition of the fundamental matrix 
 \begin{equation*} 
 \fm \, = \, \buu(\bvv, \bff) \, n(\overline{w}) \, \btt(\bexp) \, \buu(\bint)
 \end{equation*}
 for $A_{\Gzwei}(\bss(\bvv))$. Carrying out this construction one finds that 
 \[
 \bexp =( \exp_1, \exp_2)=( e^{\int v_1}, e^{\int -v_2}).
 \] 
Now, by \cite[Theorem~1.28]{vanderPutSinger} the map 
\[ 
\psi\colon  \field \langle \bss(\bvv)\rangle \otimes \field[\Gzwei] \to \field \langle \bss(\bvv)\rangle [\fm], \ \overline{Y}_{i,j} \mapsto \fm_{i,j}
\]
is a differential $\field \langle \bss(\bvv)\rangle$-isomorphism of Picard-Vessiot rings which extends to a differential $\field \langle \bss(\bvv)\rangle$-isomorphism 
\[ 
\psi\colon  \field \langle \bss(\bvv)\rangle \otimes \field(\Gzwei) \rightarrow \generalext, \ \overline{Y}_{i,j} \mapsto \fm_{i,j}
\]
of Picard-Vessiot fields. According to \cite[Lemma~4.2]{Seiss_Generic}, there exist 
\begin{equation} \label{eqn:G2computation}
\bxx=(\frac{x_{1,1}}{x_{1,2}},\dots,\frac{x_{6,1}}{x_{6,2}}), \quad \bzz=(\frac{z_{1,1}}{z_{1,2}}, \frac{z_{2,1}}{z_{2,2}}) \quad \text{and} \quad \bww=(\frac{w_{1,1}}{w_{1,2}},\dots,\frac{w_{6,1}}{w_{6,2}})
\end{equation}
in the field of fractions $\field(\group)$ of the coordinate ring $\field[\group]$ such that 
\[
\overline{Y} \, = \, \buu(\bxx) \, n(\overline{w}) \, \btt(\boldsymbol{z}) \, \buu(\bww) \,,
\]
where we assume that these fractions are completely reduced.
First, we show that $\bzz=(e_1,e_2^{-1})$ with $e_1$, $e_2 \in C[\Gzwei]$. We apply $\psi$ to $\bzz$ and obtain with
Proposition~\ref{prop:assoc_operator} \ref{prop:assoc_operator_g2} that
\[
\psi(\frac{z_{1,1}}{z_{1,2}}) \, = \, \frac{\psi(z_{1,1})}{\psi(z_{1,2})} \, = \, \exp_1 \, = \, \frac{\expsolass_1}{1} \ \text{and} \ \psi(\frac{z_{2,1}}{z_{2,2}}) \, = \, \frac{\psi(z_{2,1})}{\psi(z_{2,2})} \, = \, \exp_2 \, = \, \frac{1}{\expsolass_2}.
\]
Observe that since $z_{i,1}/z_{i,2}$ with $i=1,2$ is completely reduced, so is $\psi(z_{i,1})/\psi(z_{i,2})$. Moreover, since $\psi(z_{i,1})$, $\psi(z_{i,2})$ and $\expsolass_i$ are elements of the Picard-Vessiot ring $\field \langle \bss(\bvv)\rangle [\fm]$ (the  $\expsolass_i$  because of \cite[Corollary 1.38]{vanderPutSinger}), we conclude that $z_{1,2} = z_{2,1} = 1$ without loss of generality.
Next we are going to show that the fractions $\bxx$ and $\bww$ lie in the localization of $\field[\Gzwei]$ by the multiplicatively closed subset generated by $z_{1,1}$ and $z_{2,2}$. A computation shows that
\[
\begin{array}{rcl}
    v_1 \expsolass_1 &=& -\fm_{7,5}\,,\\[0.2em]
    v_2 \expsolass_2 &=& \expsolass_1 \, \fm_{6,4} + \fm_{2,4} \, \fm_{6,5}\,,\\[0.2em]
    f_3 \expsolass_1 &=& \fm_{6,5}\,,\\[0.2em]
    f_4 (\expsolass_1)^2 &=& -\frac{\sqrt{2}}{2} \, \expsolass_1 \, \fm_{1,5} + \fm_{6,5} \, \fm_{7,5}\,,\\[0.2em]
    f_5 (\expsolass_1)^3 &=& -(\expsolass_1)^2 \, \fm_{3,5} \, +\\[0.2em]
    & & \fm_{7,5} \, (2 \, f_4 (\expsolass_1)^2 + v_1 \expsolass_1 \, f_3 \expsolass_1)\,,\\[0.2em]
    f_6 (\expsolass_1)^4 \expsolass_2 &=& -(\expsolass_1)^3 \expsolass_2 \, \fm_{4,5} - \expsolass_1 \, v_2 \expsolass_2 \, f_5 (\expsolass_1)^3 \, -\\[0.2em]
    & & 2 \, \expsolass_1 \expsolass_2 \, f_3 \expsolass_1 \, f_4 (\expsolass_1)^2\,,\\[0.2em]
    \intn_1 \expsolass_1 &=& \fm_{2,4}\,,\\[0.2em]
    \intn_2 \expsolass_2 &=& -\expsolass_1 \, \fm_{7,3} + v_1 \expsolass_1 \, \fm_{2,3}\,,\\[0.2em]
    \intn_3 \expsolass_1 \expsolass_2 &=& -\expsolass_2 \, \fm_{2,3} - \fm_{2,4} \, ( \expsolass_1 \, \fm_{7,3} - v_1 \expsolass_1 \, \fm_{2,3} )\,,\\[0.2em]
    \intn_4 \expsolass_2 &=& -\expsolass_1 \, \fm_{7,6} - \fm_{2,6} \, \fm_{7,5}\,,\\[0.2em]
    \intn_5 \expsolass_1 \expsolass_2 &=& \expsolass_2 \, \fm_{2,6} + \fm_{2,4} \, ( \expsolass_1 \, \fm_{7,6} + \fm_{2,6} \, \fm_{7,5} )\,,\\[0.2em]
    \intn_6 (\expsolass_1)^3 (\expsolass_2)^2 &=& (\expsolass_1)^2 \, (\expsolass_2)^2 \, \fm_{2,7} \, +\\[0.2em]
    & & (\expsolass_1)^2 \, (\intn_1 \expsolass_1 \, \intn_2 \expsolass_2 \, -\\[0.2em]
    & & \intn_3 \expsolass_1 \expsolass_2) \, \intn_4 \expsolass_2 \, +\\[0.2em]
    & & \intn_1 \expsolass_1 \, (\intn_3 \expsolass_1 \expsolass_2)^2\,.
\end{array}
\]
Since $\fm_{i,j}$ and $\expsolass_1$ and $\expsolass_2$ are in the Picard-Vessiot ring $\field \langle \bss(\bvv)\rangle [\fm]$, also  the left hand sides of these equations are contained in $\field \langle \bss(\bvv)\rangle [\fm]$. Let $a$ be one of the elements  $v_1,v_2$, $f_3,f_4,f_5,f_6$ and $\intn_1,\dots,\intn_6$ and let $\frac{b_1}{b_2}$ be its counterpart among the elements 
in \eqref{eqn:G2computation}, that is the preimage of $a$ under $\psi$.
Let $(\expsolass_1)^{c_1}(\expsolass_2)^{c_2}$ with $c_1,c_2 \in \Z_{\geq 0}$ be the respective factor from above such that $a(\expsolass_1)^{c_1}(\expsolass_2)^{c_2}$ is contained in $\field \langle \bss(\bvv)\rangle [\fm]$. Applying $\psi^{-1}$ we obtain  
\[
\psi^{-1}(a (\expsolass_1)^{c_1}(\expsolass_2)^{c_2}) \, = \, \frac{b_{1}}{b_{2}} z_{1,1}^{c_1} z_{2,2}^{c_2}\in \field \langle \bss(\bvv) \rangle [\Gzwei]\,.
\]
Since $b_1$ and $b_2$ have no common devisor and the right hand side lies in the Picard-Vessiot ring, we conclude that $b_2$ divides $z_{1,1}^{c_1} z_{2,2}^{c_2}$. Since $\expsolass_i$ are irreducible in $\field \langle \bss(\bvv)\rangle [\fm]$, it follows that $z_{1,1}$ and $z_{2,2}$ are irreducible in $\field[\Gzwei]$ and so $b$ lies in the localization of $\field[\Gzwei]$ by the multiplicatively closed subset generated by $z_{1,1}$ and $z_{2,2}$. 
\end{proof}

\section{Computation of a Primitive Element for $\difffieldalg$}
\label{appendix:primitveElement}

In the last part of the appendix we explain how one can compute a primitive element for the algebraic extension $\difffieldalg$ of $\difffield$ corresponding under the Galois correspondence to the connected component of the differential Galois group of the completely reducible part (cf.\ Definition~\ref{def:difffieldalg}). 

\begin{proposition}\label{prop:primelemappendix}
We can compute a primitive element
\[
p \in \left( \difffield[X,\det(X)^{-1}]/Q \right)^{\cocomp{\Hred}}
\]
for the algebraic extension $\difffieldalg$ of $\difffield$.
\end{proposition}

\begin{proof}
    Since
    \[
    \extfieldred^{\cocomp{\Hred}(\field)}  \, = \, \Frac(\difffield[X,\det(X)^{-1}]/Q)^{\cocomp{\Hred}(\field)}
    \]
    is an algebraic extension of $\difffield$ and every element that is algebraic over $\difffield$ lies in the Picard-Vessiot ring $\difffield[X,\det(X)^{-1}]/Q$, we have that 
    \[
    \Frac(\difffield[X,\det(X)^{-1}]/Q)^{\cocomp{\Hred}(\field)} \, = \, (\difffield[X,\det(X)^{-1}]/Q)^{\cocomp{\Hred}(\field)}.
    \]
    Indeed, if an element $a \in \Frac(\difffield[X,\det(X)^{-1}]/Q)$ is algebraic over $\difffield$, then its orbit is finite and so the $\field$-vector space spanned by the elements in the orbit is finite dimensional.
    It follows from \cite[Corollary~1.38]{vanderPutSinger} that $a$ lies in $\difffield[X,\det(X)^{-1}]/Q$. Thus we have to compute generators of the invariant ring 
    \[
    (\difffield[X,\det(X)^{-1}]/Q)^{\cocomp{\Hred}(\field)}.
    \]
    Consider the $n_{I''}^2$-dimensional vector space 
    \[
    \mathcal{V} \, = \, \overline{\difffield}^{n_{I''} \times n_{I''}} = \mathbb{A}^{n_{I''}^2}(\overline{\difffield})
    \]
    over the algebraic closure $\overline{\difffield}$ of $\difffield$ and let
    \[
    \cocomp{\Hred}(\overline{\difffield}) \to \GL(\mathcal{V}), \quad g \mapsto (v \mapsto vg)
    \]
    be the respective rational representation of $\cocomp{\Hred}(\overline{\difffield})$. 
    Let $(Q)$ be the ideal in $\overline{\difffield}[X,\det(X)^{-1}]$ generated by $Q \lhd \difffield[X,\det(X)^{-1}]$ and consider the variety
    \[
    \mathcal{U} \, = \, \{ v \in \mathcal{V} \mid f(v)=0 \ \text{for all} \, f \in (Q) \}\,.
    \]
    The stability of $Q$ under $\Hred(\field)$ implies stability of 
    $Q$ under $\Hred(\overline{\difffield})$ and so $(Q)$ is stable under $\cocomp{\Hred}(\overline{\difffield})$. Thus $\mathcal{U}$ is an affine variety stable under $\cocomp{\Hred}(\overline{\difffield})$. We obtain an $\cocomp{\Hred}(\overline{\difffield})$-equivariant embedding 
    \[
    i\colon \mathcal{U} \hookrightarrow \mathcal{V} 
    \]
    and so an $\cocomp{\Hred}(\overline{\difffield})$-equivariant surjective ring homomorphism
    \[
    i^*\colon \overline{\difffield}[\mathcal{X}] \to \overline{\difffield}[X,\det(X)^{-1}]/(Q), \ 
    \mathcal{X}_{i,j} \mapsto X_{i,j} + (Q) , 
    \]
    where $\overline{\difffield}[\mathcal{X} ]$ is the coordinate ring of $\mathcal{V}$ with $\mathcal{X}$ an $(n_{I''}\times n_{I''})$ matrix whose entries $\mathcal{X}_{i,j}$ are indeterminates over $\overline{\difffield}$.
    According to \cite[Corollary~2.2.9]{DerksenKemper} (recall that $\cocomp{\Hred} (\overline{\difffield})$ is reductive) we have 
    \[
    i^*(\overline{\difffield}[\mathcal{X}]^{\cocomp{\Hred}(\overline{\difffield})}) \, = \, (\overline{\difffield}[X,\det(X)^{-1}]/(Q))^{\cocomp{\Hred}(\overline{\difffield})} .
    \]
    Since $\overline{\difffield}[\cocomp{\Hred}]\cong \overline{\difffield} \otimes_{\field} \field[\cocomp{\Hred}]$, we conclude that 
    \[
    \overline{\difffield}[\mathcal{X}]^{\cocomp{\Hred}(\overline{\difffield})} \, = \, \overline{\difffield}[\mathcal{X}]^{\cocomp{\Hred}(\field)}  
    \]
    and that 
    \[
    (\overline{\difffield}[X,\det(X)^{-1}]/(Q))^{\cocomp{\Hred}(\overline{\difffield})} \, = \, (\overline{\difffield}[X,\det(X)^{-1}]/(Q))^{\cocomp{\Hred}(\field)}.
    \]
    Thus, we obtain 
    \[
    i^*(\overline{\difffield}[\mathcal{X}]^{\cocomp{\Hred}(\field)}) \, = \, (\overline{\difffield}[X,\det(X)^{-1}]/(Q))^{\cocomp{\Hred}(\field)} .
    \]
    The multiplication maps
    \begin{gather*}
    \mu_1\colon \overline{\difffield} \otimes_{\difffield} \difffield[\mathcal{X}] \to \overline{\difffield}[\mathcal{X}], \ f \otimes h_1 \mapsto f \, h_1 
    \end{gather*}
    and 
    \begin{gather*}
    \mu_2\colon \overline{\difffield} \otimes_{\difffield} \difffield[X,\det(X)^{-1}]/Q \to \overline{\difffield}[X,\det(X)^{-1}]/(Q), \ f \otimes h_2 \mapsto f \, h_2
    \end{gather*}
    are $\cocomp{\Hred}(\field)$-equivariant isomorphisms. Hence, we have 
    \begin{gather*}
    \overline{\difffield} \otimes_{\difffield} \difffield[\mathcal{X}]^{\cocomp{\Hred}(\field)} \, \cong \, (\overline{\difffield} \otimes_{\difffield} \difffield[\mathcal{X}])^{\cocomp{\Hred}(\field)} \, \underset{{\rm via} \, \mu_1}\cong \, \overline{\difffield}[\mathcal{X}]^{\cocomp{\Hred}(\field)}
    \end{gather*}
    and 
    \begin{gather*}
    \overline{\difffield} \otimes_{\difffield} (\difffield[X,\det(X)^{-1}]/Q)^{\cocomp{\Hred}(\field)} \, \cong \, (\overline{\difffield} \otimes_{\difffield} \difffield[X,\det(X)^{-1}]/Q)^{\cocomp{\Hred}(\field)} \\ 
    \underset{{\rm via} \, \mu_2}\cong  \, (\overline{\difffield}[X,\det(X)^{-1}]/(Q))^{\cocomp{\Hred}(\field)}.
    \end{gather*}
    Combining this with $i^*$ we obtain a surjective ring homomorphism
    \[
    \begin{array}{rcl}
     \overline{\difffield} \otimes_{\difffield} \difffield[\mathcal{X}]^{\cocomp{\Hred}(\field)} \! & \! \to \! & \! \overline{\difffield} \otimes_{\difffield} (\difffield[X,\det(X)^{-1}]/Q)^{\cocomp{\Hred}(\field)},\\[0.2em]
    f \otimes h \! & \! \mapsto \! & \! f \otimes \widehat{\rho}(h) \, ,
    \end{array}
    \]
    where
    \[
    \widehat{\rho}\colon \difffield[\mathcal{X}] \to \difffield[X,\det(X)^{-1}]/Q, \  \mathcal{X}_{i,j} \mapsto X_{i,j} + Q.
    \]
    We conclude that $\widehat{\rho}$ restricts to a surjective ring homomorphism 
    \[
    \rho\colon \difffield[\mathcal{X}]^{\cocomp{\Hred}(\field)} \to (\difffield[X,\det(X)^{-1}]/Q)^{\cocomp{\Hred}(\field)}, \  \mathcal{X}_{i,j} \mapsto X_{i,j} + Q.
    \]
    Moreover, we have that 
    \[
    \difffield[\mathcal{X}]^{\cocomp{\Hred}(\field)} \cong \difffield \otimes_{\field} \field[\mathcal{X}]^{\cocomp{\Hred}(\field)}
    \]
    and so a generating set of $\field[\mathcal{X}]^{\cocomp{\Hred}(\field)}$ over $\field$ maps to a generating set of 
    $$ (\difffield[X,\det(X)^{-1}]/Q)^{\cocomp{\Hred}(\field)}$$ 
    over $\difffield$.
    We are going to compute a generating set of $\field[\mathcal{X}]^{\cocomp{\Hred}(\field)}$ using \cite[Algorithm~4.1.9]{DerksenKemper}.
    To this end we need generators of the defining ideal of $\cocomp{\Hred}$. We compute with Gr\"obner basis methods a minimal primary decomposition of the defining ideal of $\Hred$ and determine by evaluating the defining polynomials of the components at the identity matrix the defining ideal of $\cocomp{\Hred}$. We can compute now with \cite[Algorithm~4.1.9]{DerksenKemper} generators of $\field[\mathcal{X}]^{\cocomp{\Hred}(\field)}$ and obtain generators of $(\difffield[X,\det(X)^{-1}]/Q)^{\cocomp{\Hred}(\field)}$ by applying $\rho$ to them.
    Now one needs to find an $\difffield$-linear combination 
    $$p \in \difffield[X,\det(X)^{-1}]/Q$$ 
    of the  generators of $(\difffield[X,\det(X)^{-1}]/Q)^{\cocomp{\Hred}(\field)}$ such that its minimal polynomial over $\difffield$ has degree equal to the order of $\Hred/\cocomp{\Hred}$ (the number of components in the minimal primary decomposition).
    Note that a generic choice of such a linear combination has the desired property.
    Then $p$ is a primitive element generating the algebraic extension $\extfieldred^{\cocomp{\Hred}} = \difffieldalg$ of $\difffield$.  
\end{proof}

\newpage
\section*{Notation} 
\begin{center}
\begin{longtable}{lll}
  $\field$   & a computable alg.~closed field of char.~zero & \pageref{sec:classicalgroups}  \\
  $\group$   & a classical group  & \pageref{sec:classicalgroups}\\
  $\difffield$   & the rational function field $\field(z)$ with derivation $\frac{d}{dz}$  & \pageref{sec:normalform}\\
  $\sigma_0$ & specialization of $\bss(\bvv)$ to $\bsq \in \difffield^l$ & \pageref{eqn:defsigma0} \\
  $A_{\group}(\bsq)$ & specialized normal form matrix& \pageref{eqn:defspecnormalform}\\
  $H$ & the differential Galois group of $\overline{\generalext}$ over $\difffield$ for $A_{\group}(\bsq)$ & \pageref{prop:parabolicbound} \\
  $l$ & the Lie rank of the classical group $\group$ & \pageref{sec:classicalgroups}\\
  $I$ & the set of indices $\{1, \dots, l\}$ & \pageref{def:partitions}\\
  $J$ & a subset of indices $J \subseteq I$ & \pageref{def:partitions} \\
  $\parabolic_J$ & the standard parabolic subgroup corresponding to $J$ & \pageref{eqn:standardparabolic} \\
  $\unirad(\widetilde{G})$ & the unipotent radical of a linear algebraic group $\widetilde{G}$ & \pageref{thm:levidecomposition} \\
  $\levi$ & a Levi group of a linear algebraic group & \pageref{thm:levidecomposition} \\
  $\levi_J$ &the standard Levi group of $\parabolic_J$ & \pageref{eqn:standardLevi} \\
  $\bvvbase $ & the indeterminates among $\bvv$ fixed by $\parabolic_J$ & \pageref{def:partitions} \\
$ \bvvext$ & the indeterminates among $\bvv$ not fixed by $\parabolic_J$ & \pageref{def:partitions}  \\
  $I'$ & the indices  $\{ i_1,\dots,i_r \} \subset I$ corresponding to $\bvvext$  & \pageref{def:partitions} \\
  $ I''$ &the indices $ \{ i_{r+1},\dots,i_l \} \subset I$ corresponding to $\bvvbase$  & \pageref{def:partitions} \\
  $\beta_1,\dots,\beta_m$ & the roots of $\roots^-$ enumerated in a specific way & \pageref{eqn:numberingnegativeroots} \\
  $\leviroots \subset \roots$ & the root system of the standard Levi group $\levi_J$ of $\parabolic_J$ & \pageref{eqn:definitionLeviroots} \\
  $\beta_{j_1},\dots,\beta_{j_k}$ & the roots in $\leviroots^-$ & \pageref{eqn:rootsofpsiminus} \\
  $\beta_{j_{k+1}},\dots,\beta_{j_m}$  & the roots in $\roots^- \setminus \leviroots^-$ corresponding to $\unirad(\parabolic_J)$& \pageref{eqn:rootsofradical} \\
  $L_{\group}(\bss(\bvv), \partial)$ &normal form operator in $\field \{ \bss(\bvv)\}$ corresponding &  \pageref{de:normalformequation} \\ & to $A_{\group}(\bsq)$& \\
  $L^{\det(i)}(\bss(\bvv),\partial)$ & for $1 \leq i \leq l$ the associated operator with & \pageref{def:associatedRic} \\ 
  & solution $\exp(\int b_i v_i)$&  \\
  $\expsolass_i$ & exponential solution $\exp(\int b_i v_i)$ of the $i$-th associated equation & \pageref{prop:exponentialandRiccati} \\
  $\Ric{i}{\bss(\bvv)}{y}=0$ &  for $1 \leq i \leq l$ the Riccati equation for $L^{\det(i)}(\bss(\bvv),y)=0$ & \pageref{def:associatedRic} \\
  $L_{i}(\bss(\bvv),\bvvbase,\partial)$ & for $1\leq i \leq k$ the irreducible factors of an irreducible  & \pageref{eqn:irreduciblefactorization} \\
  & factorization of  $L_{\group}(\bss(\bvv), \partial)$ over $\difffield\langle \bss(\bvv),\bvvbase \rangle$& \\
  $L_i(\partial)$ & short notation for $L_{i}(\bss(\bvv),\bvvbase,\partial)$ & \pageref{eqn:shortdefofirredfactors} \\
  $\LCLM(\bss(\bvv),$ & the least common left multiple of the $L_{i}(\bss(\bvv),\bvvbase,\partial)$ & \pageref{def:lclm} \\ $\qquad \quad \bvvbase,\partial)$& &\\ 
  $n_{I''}$ & the order of $\LCLM(\bss(\bvv),\bvvbase,\partial)$ & \pageref{def:lclm}  \\
  $y^{I''}_1,\dots,y^{I''}_{n^{I''}}$ & a fixed basis of $\LCLM(\bss(\bvv),\bvvbase,\partial)$ in $\generalext^{\unirad(\parabolic_J)}$ & \pageref{rem:basisLCLM} \\
  $\bZ $ & differential indeterminates $Z_1,\dots, Z_{n^{I''}}$ over $\generalext^{\parabolic_J}$ & \pageref{cor:exprforexpandint} \\   
  $\exprforexp^{I''}_i(\bZ)$ & for $1 \leq i \leq l$ a differential rational function in $\generalext^{\parabolic_J}\langle \bZ \rangle$ & \pageref{cor:exprforexpandint} \\ 
  & such that $\exprforexp^{I''}_i(y^{I''}_1,\dots,y^{I''}_{n^{I''}})=\exp_i$ & \\
  $V^{I''}_i(\bZ)$ & for $1 \leq i \leq l$ a differential rational function in $\generalext^{\parabolic_J}\langle \bZ \rangle$ & \pageref{cor:exprforexpandint} \\ 
  & such that $V^{I''}_i(y^{I''}_1,\dots,y^{I''}_{n^{I''}})=v_i$  &  \\
  $\exprforint^{I''}_j(\bZ)$ & for $1\leq j \leq m$ with $\beta_j \in \leviroots^-$ a differential rational function & \pageref{cor:exprforexpandint} \\
  & in $\generalext^{\parabolic_J}\langle \bZ \rangle$ such that $\exprforint^{I''}_j(y^{I''}_1,\dots,y^{I''}_{n^{I''}}) =\intn_j$  & \\
  $\mathrm{REL}_i$ & differential polynomial in $\field \{ \bss(\bvv),\bvvbase\}\{ \bZ \}$ & \pageref{prop:computationRELnew} \\
  & representing a relation between $y^{I''}_1,\dots,y^{I''}_{n^{I''}}$     & \\
  $\mathrm{REL}^{\neq}$ & product of the initials and separants of the $\mathrm{REL}_i$  & \pageref{prop:computationRELnew} \\
  $ \fmred \, \fmrad$ & the factorization of $\fm$ with & \pageref{rem:genericYredYraddecomp} \\
  & $\fmred \in \group(\generalext^{\unirad(\parabolic_J)})$ and $\fmrad \in \unirad(\parabolic_J)(\generalext)$ & \\
  $g_1 $ & a matrix in $\group(\generalext^{\parabolic_J})$ satisfying $g_1\fmred \in \levi_J(\generalext^{\unirad(\parabolic_J)})$& \pageref{prop:transformationintoPJgen} \\ 
  $\overline{\generalext}$ &specialized Picard-Vessiot extension of $\difffield$ for $A_{\group}(\bsq)$; & \pageref{prop:parabolicbound} \\
   & the field of fractions of $\overline{R}$ & \pageref{eqn:definitionEq}\\
$\overline{\fm}$ & specialized fundamental matrix for $A_{\group}(\bsq)$ defining $\overline{\generalext}$;  & \pageref{prop:parabolicbound}\\
&image of $\fm$ under $\sigPV$ & \pageref{eqn:bruhatoffundmatrixchapter10} \\
$\bvvqbase $ & the specialization of $\bvvbase$ in $\difffield^{|I''|}$ under $\siginter$ & \pageref{defn:sigmainter} \\
$\siginter$ &an extension of $\sigma_0$ specializing consistently $\bvvbase$ to $\bvvqbase$ & \pageref{defn:sigmainter} \\
$\idealinter$  & the differential ideal in $\difffield\{\bvv\}$ defining the specialization $\siginter$& \pageref{defn:idealSinter} \\
$L_{\group}(\bsq,\partial)$ &specialized normal form operator & \pageref{def:specfactorsLCLM} \\
$\overline{L}_i(\partial)$ & for $1 \leq i \leq k$ the image of $L_{i}(\bss(\bvv),\bvvbase,\partial)$ under $\siginter$  & \pageref{item3:assumption2} \\
$\overline{\LCLM}(\bsq,$& the image of $\LCLM(\bss(\bvv),\bvvbase,\partial)$ under $\siginter$  & \pageref{def:specfactorsLCLM} \\
 $\qquad \quad \bvvqbase,\partial)$ && \\
$A_{\rm comp}$ & the companion matrix for $\overline{\LCLM}(\bsq,\bvvqbase,\partial)$ & \pageref{def:Q} \\
$\difffield[\GL_{n_{I''}}]$ & the differential ring $ \difffield[X_{i,j},\det(X_{i,j})^{-1}]$ with derivation & \pageref{def:Q} \\ 
& defined by $X'=A_{\rm comp} X$ & \\
$Q $ & a maximal differential ideal in $\difffield[\GL_{n_{I''}}]$  & \pageref{def:Q} \\ 
$\extfieldred$ & the Picard-Vessiot extension $  \Frac(\difffield[\GL_{n_{I''}}]/Q)$ & \pageref{def:Q} \\ 
& for $\overline{\LCLM}(\bsq,\bvvqbase,\partial)$& \\
$\Hred$  & the stabilizer $\Stab(Q) $ of $Q$ in $\GL_{n_{I''}}(\field)$ & \pageref{def:difffieldalg} \\
$D$ & multiplicatively closed subset in $\field\{ \bss(\bvv),\bvvbase \}$ & \pageref{rem:extensigma}\\
$\sigred$ & specialization  
of the reductive part to $\extfieldred$ & \pageref{prop:sigmared} \\
$\widehat{v}_i$ & for $1\leq i\leq l$ a function in $\Frac(\difffield[\GL_{n_{I''}}])$ such & \pageref{eqn:definnitionofhatvariables} \\ 
& that $\sigred(v_i)=\widehat{v_i} +Q $ &  \\
$\widehat{\bvv}$ & the tuple $(\widehat{v}_1,\dots,\widehat{v}_l)$ in $\Frac(\difffield[\GL_{n_{I''}}])^l$   & \pageref{eqn:definnitionofhatvariables} \\
$\ratexp_i$ &  for $1\leq i\leq l$ a function in $\Frac(\difffield[\GL_{n_{I''}}])^{\times}$ such & \pageref{eqn:definnitionofhatvariables}  \\
 & that $\sigred(\exp)=\ratexp_i +Q $  & \\
$\bratexp$ & the tuple $(\ratexp_1,\dots,\ratexp_l)$ in $(\Frac(\difffield[\GL_{n_{I''}}])^{\times})^l$ & \pageref{eqn:definnitionofhatvariables} \\
$\ratint_i$ & for $1\leq i \leq m$ such that $\beta_i \in \leviroots^-$ a function in  & \pageref{eqn:definnitionofhatvariables} \\ 
& $\Frac(\difffield[\GL_{n_{I''}}])$ such that  $\sigred(\intn_i)=\ratint_i + Q$  & \\
$\intrad_i$ & for $1\leq i \leq m$ with $\beta_i \in \roots^- \setminus \leviroots^-$ a differential &  \pageref{sec:structureofreductivepart}  \\ & indeterminate over $\extfieldred$ & \\
 $\Iuni$ & differential ideal in $\extfieldred\{ \intrad_i \mid \beta_i \in \roots^- \setminus \leviroots^-\}$& \pageref{sec:structureofreductivepart} \\
$\bratint$ & the $m$-tuple where the $i$-th entry is $\ratint_i$ if $\beta_i \in \leviroots^-$ and & \pageref{sec:structureofreductivepart}\\ 
&$\intrad_i$ if $\beta_i \in \roots^- \setminus \leviroots^-$ & \\
$\underline{\intrad}_i$ & for $1\leq i \leq m$ with $\beta_i \in \roots^- \setminus \leviroots^-$ the residue class & \pageref{eqn:specializationreductivepart} \\ 
& of $\intrad_i$ modulo $\Iuni$ &  \\
$\overline{\intrad}_i$ & for $1\leq i \leq m$ with $\beta_i \in \roots^- \setminus \leviroots^-$ the residue class & \pageref{eqn:specializationreductivepartfinal} \\
&  of $\underline{\intrad}_i$ modulo $\Imax$ & \\
$\underline{R}$ & the differential ring  $ \extfieldred[\underline{\intrad}_i \mid \beta_i \in \roots^- \setminus \leviroots^-]$ & \pageref{sec:structureofreductivepart} \\
$\underline{E}$ & the differential field $\Frac(\underline{R})$& \pageref{rem:Rintegraldomain} \\
$\sigul$ & the specialization 
of the parameters to $\underline{E}$ extending $\sigred$  & \pageref{eqn:specializationreductivepart} \\
$\Imax$  & a maximal differential ideal in $\underline{R}$ & \pageref{sec:structureofreductivepart} \\
$\overline{R}$ & the quotient of $\underline{R}$ by $\Imax$  & \pageref{eqn:definitionRq}\\
$\fmuq$ & the image of $\fm$ under $\sigul$ in $\group(\underline{R}) $& \pageref{eqn:imageoffmundersigul} \\
$\fmreduq\, \fmraduq$ & the factorization of $\fmuq$ with $\fmreduq \in \group(\extfieldred)$ and & \\ & $\fmraduq \in \unirad(\parabolic_J)(\underline{R})$& \pageref{eqn:definitionofYredradobuq}\\
$\fmredq\, \fmradq$ & the factorization of $\fmspec$ with $\fmredq \in \group(\extfieldred)$ and & \\ & $\fmradq \in \unirad(\parabolic_J)(\overline{R})$& \pageref{eqn:definitionofYredradobuq}\\
$\Aprered$  & the logarithmic derivative $\dlog (\fmredq) \in \Lie(\group)(\difffield)$ & \pageref{prop:reductivepartYred} \\
$\Lred$ & representation of $\Gal_{\partial}(\extfieldred/\difffield)$ induced by $\fmredq$ & \pageref{prop:reductivepartYred} \\
$\underline{H}$ & representation of $\Aut_{\partial}(\underline{E}/\difffield)$ induced by $\fmuq$ & \pageref{prop:statementsaboutstabilzers(b)} \\
$I_{\Lred}$ & the defining ideal of $\Lred$ in $\field[\GL_n] = \field[\coord,\det(\coord)^{-1}]$& \pageref{sec:reductivepart} \\
$I_{\underline{H}}$ & the defining ideal of $\underline{H}$ in $\field[\GL_n] = \field[\coord,\det(\coord)^{-1}]$ & \pageref{sec:reductivepart} \\
$\overline{g}_1$ &a matrix in $\group(\difffield)$ satisfying $\overline{g}_1.A_{\group}(\bsq) \in \Lie(\parabolic_J)(\difffield)$ & \pageref{prop:gaugetoAPJ} \\
$A_{\parabolic_J} $ & the matrix $\gauge{\overline{g}_1}{A_{\group}(\bsq)}\in \Lie(\parabolic_J)(\difffield) $& \pageref{eqn:definitionA_PJ}\\
$\difffieldalg$ & the algebraic closure of $\extfieldred$ in $\difffield$    & \pageref{def:difffieldalg} \\
$p$ &a primitive element generating $\difffieldalg$ over $\difffield$ & \pageref{prop:primelem} \\
$\overline{g}_2$ & a matrix in $\levi_J(\difffieldalg)$ such that $\gauge{\overline{g}_2 \overline{g}_1}{\Aprered} \in \Lie(\Lred)(\difffieldalg)$ & \pageref{prop:reduceredivtivepartnew} \\
$\Ared$ &  the matrix $\gauge{\overline{g}_2 \overline{g}_1}{\Aprered}$ & \pageref{prop:reduceredivtivepartnew}  \\
$\Aprerad$ & the matrix in $\Lie(\unirad(\parabolic_J))(\difffieldalg)$ such & \pageref{lem:partiallyreducedAP_J} \\
& that $\gauge{\overline{g}_2 \overline{g}_1}{A_{\group}(\bsq)} \, = \, \Ared + \Aprerad$ &  \\
$\overline{g}_3$ & a matrix in $\unirad(\parabolic_J)(\difffieldalg)$ such that $\overline{g}_3\overline{g}_2\overline{g}_1.A_{\group}(\bsq)$ is  & \pageref{prop:completereduction} \\ & in reduced form & \\ 
$\Arad$ & the matrix such that $\gauge{\overline{g}_3\overline{g}_2 \overline{g}_1}{A_{\group}(\bsq)}= \Ared + \Arad$ & \pageref{prop:completereduction} \\
$\Lie_{\rm red}(\difffieldalg)$ & smallest Lie algebra containing $\Ared + \Arad$  & \pageref{prop:completereduction} \\
$\GalConn$ & connected linear algebraic group with Lie algebra $\Lie_{\rm red}(\difffieldalg)$  & \pageref{prop:completereduction} \\
$R_1$ & unipotent radical of $\GalConn$ & \pageref{prop:completereduction} \\
$I_{R_1}$ & defining ideal of $R_1$ in $\field[\GL_n]=\field[\coord,\det(\coord)^{-1}]$ & \pageref{prop:completereduction} \\
$f_1,\dots,f_a$ & generators of $I_{R_1}$ & \pageref{prop:completereduction}  \\
$\fmreduqhat \fmraduqhat$ & the factorization of $\overline{g}_3\overline{g}_2 \overline{g}_1 \fmuq$ with $\fmreduqhat \in \cocomp{\Lred}(\extfieldred)$ & \pageref{prop:decompredrad}  \\ & and $\fmraduqhat \in \unirad(\parabolic_J)(\underline{R})$ & \\
$f_i(\fmraduqhat)$ & the element in $\underline{R}$ obtained by evaluating $f_i$ at $\fmraduqhat$    & \pageref{prop:generatorsforImax} \\
$\fmredqhat \fmradqhat$ & the factorization of $\overline{g}_3 \overline{g}_2 \overline{g}_1 \fmspec$ with $\fmredqhat \in \cocomp{\Lred}(\extfieldred)$ & \pageref{eqn:decompositionreducedfm} \\ 
&and $\fmradqhat \in R_1(\overline{\generalext})$ & \\
$I_{H}$ &the defining ideal of $H$ in $\field[\GL_n]=\field[\coord,\det(\coord)^{-1}]$ & \pageref{alg:DiffGaloisgroup}\\
$\widetilde{f}_i$ & polynomial in $ \difffield(\GL_{n_{I''}})[\underline{\intrad}_i \mid \beta_i \in \leviroots^-]$ such & \pageref{def:widetildef} \\
& that $\widetilde{f}_i+Q=f_i(\fmraduqhat)$  & \\
$E_{i,j}$ & standard basis element of $C^{n \times n}$& \pageref{prop:factorization}
\end{longtable}
\end{center}

\newpage

\bibliographystyle{alpha}
\bibliography{main}

\end{document}